\documentclass[12pt]{amsart}
\usepackage{amsmath,amstext,amsfonts,amsthm,amssymb,hyperref, amscd,wasysym,stmaryrd}
\usepackage{eucal}
\usepackage[all]{xy}                                        %
\DeclareSymbolFont{bbold}{U}{bbold}{m}{n}
\DeclareSymbolFontAlphabet{\mathbbold}{bbold}

\swapnumbers
\newtheorem{thm}[subsubsection]{Theorem}
\newtheorem{lem}[subsubsection]{Lemma}
\newtheorem{prp}[subsubsection]{Proposition}
\newtheorem{crl}[subsubsection]{Corollary}

\newtheorem*{Lem}{Lemma}
\newtheorem*{Prp}{Proposition}
\newtheorem*{Crl}{Corollary}

{  \theoremstyle{definition}
           \newtheorem{dfn}[subsubsection]{Definition}
           \newtheorem{exm}[subsubsection]{Example}
           \newtheorem{rem}[subsubsection]{Remark}
           \newtheorem{rems}[subsubsection]{Remarks}

           \newtheorem*{Dfn}{Definition}
           \newtheorem*{Exm}{Example}
           \newtheorem*{Rem}{Remark}
           \newtheorem*{Rems}{Remarks}
}

\newcommand{\Ar}{\mathtt{Ar}}
\newcommand{\Ass}{\mathtt{Ass}}

\newcommand{\BM}{\mathtt{BM}}
\newcommand{\BMod}{\mathtt{BMod}}

\newcommand{\BiFun}{\mathtt{BiFun}}

\newcommand{\bis}{\mathit{bis}}
\newcommand{\bPr}{\bold{Pr}}
\newcommand{\cart}{\mathrm{cart}}
\newcommand{\Cat}{\mathtt{Cat}}
\newcommand{\cat}{\mathtt{cat}}

\newcommand{\CM}{\mathtt{CM}}
\newcommand{\Coc}{\mathtt{Coc}}
\newcommand{\coc}{\mathit{coc}}

\newcommand{\colim}{\operatorname{colim}}
\newcommand{\Com}{\mathtt{Com}}

\newcommand{\Cor}{\mathtt{Cor}}

\newcommand{\Cut}{\mathrm{Cut}}
\newcommand{\deco}{{\eighthnote}}

\newcommand{\End}{\operatorname{End}}

\newcommand{\Env}{\operatorname{Env}}

\newcommand{\eq}{\mathit{eq}}
\newcommand{\ev}{\mathit{ev}}
\newcommand{\Fib}{\mathtt{Fib}}
\newcommand{\Fin}{\mathit{Fin}}
\newcommand{\Fun}{\operatorname{Fun}}
\newcommand{\FUN}{\mathtt{FUN}}
\newcommand{\Funop}{\operatorname{Funop}}

\newcommand{\Hom}{\mathrm{Hom}} 

\newcommand{\id}{\mathrm{id}}

\newcommand{\LM}{\mathtt{LM}}
\newcommand{\LLM}{\mathtt{LLM}}
\newcommand{\LMod}{\mathtt{LMod}}
\newcommand{\RM}{\mathtt{RM}}
\newcommand{\RMod}{\mathtt{RMod}}
\newcommand{\lax}{\mathrm{lax}}
\newcommand{\laxNC}{\mathrm{laxNC}}

\newcommand{\Lt}{\mathtt{Left}}

\newcommand{\Map}{\operatorname{Map}}
\newcommand{\Mon}{\mathtt{Mon}}

\newcommand{\og}{\ogreaterthan}
\newcommand{\one}{\mathbbold{1}}
\newcommand{\Ob}{\operatorname{Ob}}
\newcommand{\op}{\mathrm{op}}
\newcommand{\Op}{\mathtt{Op}}

\newcommand{\PCat}{\mathtt{PCat}}

\newcommand{\Quiv}{\operatorname{Quiv}}
\newcommand{\quiv}{\operatorname{quiv}}

\newcommand{\rlarrows}{\substack{\longrightarrow\\ \longleftarrow}}

\newcommand{\rep}{\mathit{rep}}
\newcommand{\rev}{\mathit{rev}}
\newcommand{\Seg}{\mathtt{Seg}}

\newcommand{\Set}{\mathtt{Set}}

\newcommand{\sSet}{\mathtt{sSet}}

\renewcommand{\top}{\mathit{top}}

\newcommand{\Tw}{\operatorname{Tw}}

\newcommand{\wt}{\widetilde}

\newcommand{\Alg}{\mathtt{Alg}}

\newcommand{\weak}{\mathrm{weak}}

\newcommand{\bC}{\mathbb{C}}

\newcommand{\bO}{\mathbb{O}}

\newcommand{\cA}{\mathcal{A}}
\newcommand{\cB}{\mathcal{B}}
\newcommand{\cC}{\mathcal{C}}
\newcommand{\cD}{\mathcal{D}}
\newcommand{\cE}{\mathcal{E}}
\newcommand{\cF}{\mathcal{F}}
\newcommand{\cG}{\mathcal{G}}

\newcommand{\cJ}{\mathcal{J}}
\newcommand{\cK}{\mathcal{K}}
\newcommand{\cL}{\mathcal{L}}
\newcommand{\cM}{\mathcal{M}}
\newcommand{\cN}{\mathcal{N}}
\newcommand{\cO}{\mathcal{O}}
\newcommand{\cP}{\mathcal{P}}
\newcommand{\cQ}{\mathcal{Q}}
\newcommand{\cR}{\mathcal{R}}
\newcommand{\cS}{\mathcal{S}}
\newcommand{\cT}{\mathcal{T}}
\newcommand{\cW}{\mathcal{W}}
\newcommand{\cX}{\mathcal{X}}
\newcommand{\cY}{\mathcal{Y}}
\newcommand{\cZ}{\mathcal{Z}}

\newcommand{\Fam}{\mathtt{Fam}}

\newcommand{\Q}{\mathbb{Q}}

\begin{document}

\title[]{Yoneda lemma for enriched $\infty$-categories}
\author{Vladimir Hinich}
\address{Department of Mathematics, University of Haifa,
Mount Carmel, Haifa 3498838,  Israel}
\email{hinich@math.haifa.ac.il}
\begin{abstract}
We continue the study of enriched $\infty$-categories, 
using a definition equivalent to that of Gepner and 
Haugseng.
In our approach enriched $\infty$-categories are 
associative monoids in an especially designed monoidal 
category of enriched quivers. We prove that, in the 
case where the  monoidal structure in the  basic 
category $\cM$ comes from the direct product, our 
definition is essentially equivalent to the approach via 
Segal objects. Furthermore, we compare our notion with 
the notion of category left-tensored over
$\cM$, and prove a version of Yoneda lemma in this 
context. We apply the Yoneda lemma to the study of 
correspondences of enriched (for instance, higher) 
$\infty$-categories.
\end{abstract}
\maketitle

\section{Introduction}

\subsection{Overview}
The objective of this paper is the study of enriched 
$\infty$-categories, examples of which include 
$A_\infty$-categories, DG categories and higher 
$\infty$-categories.

Enriched $\infty$-categories abound in derived algebraic 
geometry, mirror symmetry, and so on. However, category
theory is unthinkable without the Yoneda lemma and this 
is what is lacking in the existing approaches~
\footnote{See, however, \cite{GR},  Appendix, for the 
Yoneda lemma for $(\infty,2)$-categories.}. 

Our approach to the notion of enriched category is based 
on the following observation. It is clear how to define 
an $\cM$-enriched $\infty$-category with one object 
for a monoidal $\infty$-category $\cM$; this is
just an associative algebra in $\cM$. This means that
one can expect to define enriched $\infty$-categories
as associative algebra objects in a certain category of 
enriched quivers.

Our definition of  enriched category is very 
close to that of Gepner-Haugseng~\cite{GH} 
\footnote{The paper \cite{H2} of R. Haugseng contains
a basically equivalent approach to enriched categories 
via  a monoidal structure on enriched quivers, see 4.1 
{\sl op. cit.}}.
The definitions
are equivalent, but we work with enriched categories 
slightly differently. The details are below.

\subsubsection{} Given a space $X$, Gepner and Haugseng 
construct a planar operad which we denote as $\Ass_X$ 
such that $\Ass_X$-algebras in a monoidal category $\cM$ 
are precisely enriched $\cM$-(pre)categories with the 
space of objects $X$.

Our definition is based on the observation
that direct product with a flat planar operad
(see Definition~\ref{dfn:flatoperad}) admits a right 
adjoint which we will denote as $\Funop_\Ass$ in this 
paper. The planar operad $\Ass_X$ is flat, so this 
allows us to define a planar operad 
$\Funop_\Ass(\Ass_X,\cM)$ which is the operad of $\cM$-
quivers with the space of objects $X$. Enriched 
precategories are just associative algebras in it.

\subsubsection{}The above construction makes sense for 
any planar operad $\cM$ and for any category $X$. It 
gives a monoidal category if $\cM$ is a monoidal 
category having enough colimits. The definition is 
equivalent to that of \cite{GH} when $X$ is a space.

\subsubsection{}
Simultaneously with defining a planar operad of quivers,
we define its (weak) action on the category $\Fun(X,\cM)$
compatible with the right $\cM$-action. In the most general 
form, our construction gives, for any ($\infty$-) category 
$X$ and any $\BM$-operad 
$\cM$ \footnote{This means, a pair of planar operads weakly 
acting on the left and on the right on a category in a 
compatible way.}, a $\BM$-operad $\Quiv_X^\BM(\cM)$.

 This is important for several reasons. First, it turns out that, under mild
restrictions on $\cM$ and $X$, the monoidal category of quivers 
can be defined as an endomorphism object for the right $\cM$-module
$\Fun(X,\cM)$. We use this characterization to compare the notion of enriched category with the Segal-type definition, in the case where the monoidal structure in $\cM$ is given by direct product.

Second, we use this extended construction to
define $\cM$-functors from an enriched $\cM$-precategory $\cA$
to a left $\cM$-module $\cB$. This latter is the basis for our
approach to the Yoneda lemma, described for the convential enriched categories in the short note \cite{H.Y}. It is also a basis 
of the construction assigning, when possible, an $\cM$-enriched
category to a left $\cM$-module.

The detailed content of the sections is presented below.

\subsection{ $\infty$-categories}

In this paper we use a model-independent language of 
$\infty$-categories. The idea of this approach is that, once we 
understand the $\infty$-categorical Yoneda lemma, we can reformulate 
all the theory without mentioning a specific model.

Section 2 of the present paper is devoted to developing 
this language, and so, most of it consists of well-known 
results and constructions.

In it we reformulate in the model-independent language the 
standard notions of theory of $\infty$-categories (which can be found in \cite{L.T}) as well as the language of operads from \cite{L.HA}.

In this section we sketch the basic notions of the language of infinity categories (left fibrations, cocartesian fibrations, the
Yoneda lemma, localizations) and the language of operads 
(operads, operads over a given operad, approximation,
flat operads, internal mapping object, tensor product). Some of the notions differ slightly 
from their original version in \cite{L.HA}.

We also present a number of operads and their approximations 
important for the present work. 

On the technical level, we reevaluate the role of the 
conventional category $\Delta$ of finite totally ordered sets in the theory of infinity categories. This category usually appears
in a model category (of simplicial or bisimplicial sets)
describing infinity categories. We use it in the 
identification
of the $\infty$-category $\Cat$ as a (Bousfield) localization
of the $\infty$-category of simplicial spaces. 

In this way most of the constructions with infinity categories
(for instance, the opposite  category $\cC^\op$ or the twisted 
arrows category $\Tw(\cC)$) can be described purely in terms of 
conventional categories (by functors $\Delta\to\Delta$ or 
similar).

Note a few places of Section 2 containing a less 
standard material.
In Section~\ref{ss:tensoroperad} we develop a notion of tensor product of operads
slightly more general that the notion described in \cite{L.HA}, 3.2.4. In 
Section~\ref{ss:sieves} we study operadic sieves which describe a nice class
of morphisms of operads leading to a cartesian fibration  of the corresponding 
categories of operad algebras. In \ref{sss:funop} we present an explicit description of the internal Hom in operads,
$\Funop_\cC(\cP,\cQ)$, in the case when $\cP$ and $\cQ$ are
$\cC$-monoidal categories.

\subsection{Quivers}

In Section~\ref{sec:quivers} we present the construction of the categories
of enriched quivers. There are different versions of the 
construction, the most general among them assigns to an $\infty$-category 
$X$ (of objects) and to a $\BM$-operad $\cM$, a $\BM$-operad $\Quiv^\BM_X(\cM)$. The 
construction is represented by a (strict) functor assigning
to any simplex $\sigma:[n]\to\BM$ a poset $\cF^\BM(\sigma)$
describing the combinatorics of compositions of arrows in a 
category. This allows us to define, for any $X\in\Cat$,
a $\BM$-operad denoted $\BM_X$, as a functor
$(\Delta_{/\BM})^\op\to\cS$, carrying $\sigma:[n]\to\BM$ to
the space $\Map(\cF^\BM(\sigma),X)$.
We also verify that for any $X$ the $\BM$-operad
$\BM_X$ is flat.

In Section~\ref{sec:quivers-mon} we study conditions on 
$\cM$ and $X$ that 
ensure that the enriched quivers form a monoidal category. Roughly speaking, $\cM$ is required to be monoidal, with the tensor product commuting with enough colimits (with respect to the size of $X$), see Theorem~\ref{thm:Quiv-monoidal}. We also provide the description of $\Quiv_X(\cM)$
as the category of endomorphisms of $\Fun(X,\cM)$, see 
Proposition~\ref{prp:quiv-end}.

\subsection{ The case where $\cM$ has a cartesian monoidal structure}
\label{ss:intro-cart}
In Section~\ref{sec:cart} we compare our notion of enriched precategory over $\cM$ with the existing Segal-type definition, in the case where the monoidal structure on $\cM$ comes from the direct product.

Here we proceed as follows. To any category $\cM$ with fiber products we assign a family of $\BM$-monoidal categories,
cartesian over $\cM\times\cM$. Its fiber at $(X,Y)\in
\cM\times\cM$ consists of
a pair of monoidal categories, $\cM_{/X\times X}$ and 
$\cM_{/Y\times Y}$,
acting on the left and on the right on the category $\cM_{/X\times Y}$.

This construction yields a monoidal structure on 
$\cM_{/X\times X}$; associative algebras in $\cM_{/X\times X}$ 
identify with the simplicial objects 
$A:\Delta^\op\to\cM$ 
satisfying Segal condition and having $A_0=X$.

In Section~\ref{ss:prototopoi} we define prototopoi,
categories with fiber products satisfying some
weakened topos properties. In particular, 
topoi, as well as the categories of $(\infty,n)$
categories, are prototopoi. 
 
Finally, in~\ref{quiv=seg}, in the case when $\cM$ is a prototopos  and $X$ is in the image of the standard colimit preserving functor $\cS\to\cM$, 
we identify the triple
$(\cM_{/X\times X},\cM_{/X},\cM)$ with $\Quiv^\BM_X(\cM)$.
This implies that enriched precategories in this case are equivalent to simplicial objects $A:\Delta^\op\to\cM$,
satisfying the Segal condition and having $A_0$ in the image of $\cS$, 
see~\ref{cor:precat=seg}.
In proving this result we use the description of
$\Quiv_X(\cM)$ as the endomorphism object of $\Fun(X,\cM)$
and the full $\BM$-category structure on the quivers.

\subsection{The Yoneda Lemma}

In Section~\ref{sec:yoneda} we develop the notion of enriched presheaves and
prove a version of the Yoneda lemma.
Let us try to imagine what a Yoneda lemma could mean for enriched 
categories. Let $\cM$ be a monoidal category and $\cA$ be 
an $\cM$-enriched precategory. Enriched preshaves should be enriched 
functors $F:\cA^\op\to\cM$. Note that $\cA^\op$ is an enriched
precategory over $\cM^\rev$, the monoidal category with reverse
multiplication. On the other hand, it is not at all obvious
that $\cM$ is enriched over $\cM^\rev$ (or over $\cM)$. 

These problems  can already be seen in the conventional setting.
As we showed in \cite{H.Y}, they disappear if one carefully distingushes between two different types of enrichment:
$\cM$ is not necessarily enriched over itself, but it is
definitely a left (and a right) $\cM$-module.

The basis of our approach to the
Yoneda lemma is a notion of an $\cM$-functor $F:\cA\to\cB$
from an $\cM$-enriched precategory $\cA$ to a left $\cM$-module
$\cB$. Let $\cA\in\Quiv_X(\cM)$. An $\cM$-enriched functor
$F:\cA\to\cB$ is given by a functor $f:X\to\cB$, together with some extra data. The formalism of quiver categories provides
an action of the monoidal category $\Quiv_X(\cM)$ on the category of functors $\Fun(X,\cB)$. The extra data on $f:X\to\cB$ is precisely the $\cA$-module structure on $f$, see \ref{sss:functor}.

The notion of $\cM$-functor described above is exactly what is
needed for the Yoneda lemma. Any associative algebra gives rise to a bimodule in an appropriate sense. Applying this general principle to an associative algebra $\cA$ in $\Quiv_X(\cM)$,
we get a bimodule which can be interpreted as an $\cM\times\cM^\rev$-functor from $\cA\boxtimes\cA^\op$ to $\cM$, or,
even better, as an $\cM$-functor from $\cA$ to 
the category of $\cM$-presheaves 
$P_\cM(\cA)=\Fun_{\cM^\rev}(\cA^\op,\cM)$.
The Yoneda lemma~\ref{sss:EYL} claims that this $\cM$-functor is 
fully faithful in an appropriate sense.

The passage from $A$-$B$-bimodules to modules over the tensor
product $A\otimes B^\op$ that we used in the above explanation 
seems very plain; but it is not competely obvious in Higher 
Algebra. The corresponding general construction is presented
in \ref{ss:folding}. The construction is presented as a ``folding functor'' carrying a $\BM$-operad with components 
$(\cP_a,\cP_m,\cP_b)$ to an $\LM$-operad with components
$(\cP_a\times\cP_b^\rev,\cP_m)$. 

\subsubsection{}

The same notion of $\cM$-functor allows one, when possible,
to convert a left $\cM$-module $\cB$ into an $\cM$-enriched
category. In general, for any pair of objects $x,y\in\cB$,
the left $\cM$-module structure on $\cB$ gives rise to a presheaf
$\hom_\cB(x,y)$ (in the usual, non-enriched sense) on 
$\cM$. Given a functor
$F:X\to\cB$ such that for any $x,y\in X$ $\hom_\cB(F(x),F(y))$
is representable, we can construct a universal $\cM$-morphism
$\cA\to\cB$ where $\cA$ is an $\cM$-enriched precategory with the category of objects $X$.

\subsubsection{Completeness}

According to~\ref{ss:intro-cart}, the notion of $\cS$-enriched 
precategory with a space of objects $X$, is equivalent to the 
notion of Segal space with the space of objects $X$. This 
means that properly defined enriched categories should take into 
account a version of the completeness condition. This issue is already addressed in~\cite{GH}. We use the Yoneda lemma to present an alternative construction of the completion functor \cite{GH}, 5.6.

Choosing $\cB=P_\cM(\cA)$ and $X$ the subspace of representable 
$\cM$-presheaves on $\cA$, we get a universal arrow $\cA\to\cE$
from an $\cM$-enriched precategory to a complete $\cM$-
enriched precategory, see \ref{ss:localization}.

\subsection{Correspondences} In Section~\ref{sec:corr} we apply the
developed technique to studying correspondences between
enriched $\infty$-categories. Let $\cM$ be a monoidal 
category with colimits, $\cC$ and $\cD$ two $\cM$-enriched
categories. A correspondence from $\cC$ to $\cD$ can be 
defined as an $\cM$-functor $\cD\to P_\cM(\cC)$. We prove 
that, in the case where $\cM$ is a prototopos endowed with a 
cartesian monoidal structure (for instance, if $\cM$ is the 
category of $(\infty,n)$-categories), then the category of 
$\cM$-correspondences is equivalent to $\Cat(\cM)_{/[1]}$.
This result, for $(\infty,1)$-categories, is mentioned in 
\cite{L.T} without proof; a complete proof for $(\infty,1)$-categories is given in~\cite{AF}, 4.1.

\subsection{Acknowledgment}We are grateful to Nick 
Rozenblyum
who informed us about the work \cite{GH} at very early 
stages of the work. Our work does not rely on~\cite{GH}, 
and can be read completely independently. 
This led to a certain overlap, which the author has not 
found a way to avoid.

A part of this work was done during the author's stay
at Radboud University at Nijmegen, Utrecht University, MPIM 
and UC Berkeley. The author is very grateful to these 
institutions for the excellent working conditions.  
Numerous discussions with Ieke Moerdijk were very helpful.
John Francis' advice for a proof of 
Proposition~\ref{sss:algfolding} is appreciated.
We are very grateful to R.~Haugseng who pointed out an 
error in the first version of the manuscript.
We are also very grateful to the anonymous referee 
who has made an incredible job trying to make the manuscript better. 
The work was supported by ISF 446/15 grant.

\section{The language of $\infty$-categories}

\subsection{Introduction}
Throughout this paper a language of $(\infty,1)$-categories is used. 
Nowadays there are a number of more or less equivalent approaches to the notion
of $(\infty,1)$-category. All of them present a model category whose fibrant-cofibrant
objects can be though of as $(\infty,1)$-categories. These model categories are
proven to be Quillen equivalent, which allows one, in principle, to pass from one 
language to another.  However, this not at all an obvious procedure.

\

\subsection{Generalities}
We will try to make our usage of $(\infty,1)$-categories 
as independent of the model as possible --- working mostly in the 
$\infty$-category of $\infty$-categories
which we will simply denote by $\Cat$. This means that we will not 
use the usual categorical notions of fiber product, coproduct, or 
more general limits and colimits
--- but replace them with corresponding $\infty$-categorical notions. 
Since a Quillen equivalence of model categories induces an
equivalence of the underlying $\infty$-categories,
one can use any existing model for $\infty$-categories to prove claims
formulated in this model-independent language.

In particular, we will use the notions of left, right, cartesian 
or cocartesian fibration in a sense slightly different from the 
one defined in \cite{L.T}. 
For instance, our left fibrations are arrows in $\Cat$ 
representable by a left fibration in $\sSet$ in the sense of 
\cite{L.T}, Ch. 2. The notion of cocartesian fibration in $\Cat$
has a similar meaning: this is an arrow in $\Cat$ which is 
equivalent to a cocartesian fibration in the sense of \cite{L.T}, 
Ch.~2.

\ 

We will use the following standard notation throughout the 
paper.  $\cS$ is the $\infty$-category of spaces,
$\Cat$ is the $\infty$-category of $\infty$-categories.

In what follows we will use the word ``category'' instead of 
``$\infty$-category'', and ``conventional category'' instead of 
``category''.

\subsubsection{Basic vocabulary}

In what follows we will use the following notation.

The category $\Cat$  has products, and is closed, that 
is, it has internal mapping object denoted $\Fun(C,D)$ or $D^C$. 
The spaces form a full subcategory  $\cS$ of $\Cat$. 
The embedding $\cS\to\Cat$ has a right adjoint functor 
(of maximal subspace) and a left adjoint functor (of total 
localization). For a category 
$\cC$ and $x,y\in\cC$ a space $\Map_\cC(x,y)$ of maps
from $x$ to $y$ is defined,
canonically ``up to a contractible space of choices''.

Any category $\cC$ defines a conventional category
$\pi_0(\cC)$ and a canonical functor $\cC\to\pi_0(\cC)$.
The conventional category $\pi_0(\cC)$ has the same objects
as $\cC$ and the set $\Hom_{\pi_0(\cC)}(x,y)$
defined as $\pi_0(\Map_\cC(x,y)$. An arrow $f\in\Map_\cC(x,y)$ is called equivalence if its image in $\pi_0(\cC)$ is an isomorphism.

We will write sometimes $x=y$ for a natural equivalence
between two objects of a category.

For a category $\cC$ the maximal subspace of $\cC$ is denoted $\cC^\eq$. It is obtained from $\cC$ by ``discarding all arrows which are not equivalences''.

\subsubsection{Subcategory}

Note from the very beginning that our notion of subcategory does not generalize the conventional notion --- the latter
is not invariant under equivalences.

For $X\in \cS$ a {\sl subspace} of $X$ is a morphism 
$Y\to X$ which is an equivalence to a union of connected components of $X$.

A subcategory $\cD$ of $\cC\in\Cat$ is a morphism $f:\cD\to\cC$ in 
$\Cat$ defining, for any $\cA\in\Cat$, the space $\Map(\cA,\cD)$
as a subspace of $\Map(\cA,\cC)$. In particular, the maximal 
subspace $\cD^\eq$ of $\cD$ is a subspace of $\cC^\eq$ and for 
each pair of objects $x,y\in\cD$ the space $\Map_\cD(x,y)$
is a subspace of $\Map_\cC(f(x),f(y))$. This implies that
a subcategory $\cD\subset\cC$ is uniquely defined by a 
(conventional) subcategory
$D$ in the category $\pi_0(\cC)$ as the fiber product 
\begin{equation}
\label{eq:subcategory}
\begin{CD}
\cD @>>> \cC \\
@VVV @VVV \\
D @>>> \pi_0(\cC)
\end{CD},
\end{equation}
where the embedding $D\to\pi_0(\cC)$ satisfies the additional 
property saying that it induces an injective map on the sets 
of isomorphism classes of objects.
In this case one has $D=\pi_0(\cD)$. 

Vice versa, any subcategory $D$ as above of $\pi_0(\cC)$
defines a subcategory $\cD$ such that the diagram 
(\ref{eq:subcategory})
is cartesian. 

\subsubsection{Left fibrations}

A map $f:\cC\to B$ in $\Cat$ is called a left fibration
if the   map
$$ \cC^{[1]}\to B^{[1]}\times_B\cC$$
induced by the embedding $[0]=\{0\}\to[1]$ is an 
equivalence. Note that the above definition is invariant 
under equivalences in $\Cat$. 

Grothendieck construction for left fibrations yields a
canonical equivalence $\Lt(B)\to\Fun(B,\cS)$, 
where $\Lt(B)$ denotes the full subcategory of 
$\Cat_{/B}$ spanned by the left fibrations over $B$.

\subsubsection{The Yoneda lemma}
\label{sss:ooyoneda}

Let $\cC\in\Cat$. The assignment 
$(x,y)\mapsto\Map_\cC(x,y)$ is functorial. This means
one has a canonical functor $Y:\cC^\op\times\cC\to\cS$.
It can be otherwise presented by a left fibration
\begin{equation}\label{eq:Tw}
\Tw(\cC)\to\cC^\op\times\cC
\end{equation}
corresponding, via Grothendieck construction, to $Y$,
see~\cite{L.HA}, 5.2.1~\footnote{Note that we use the
different convention, as in \cite{L.HA}, 5.2.1, the functor $Y$
is encoded by a right fibration.}.
Later in this paper we will use the opposite right 
fibration
\begin{equation}\label{eq:Tw-right}
\Tw(\cC)^\op\to\cC\times\cC^\op
\end{equation}
classified by the same functor $Y$.

Yoneda embedding $Y:\Cat\to\Fun(\Cat^\op,\cS)$ restricted
to the subcategory $\Delta\subset\Cat$, yields a fully faithful embedding
\begin{equation}
\label{eq:Cat-as-sspace}
\Cat\to\Fun(\Delta^\op,\cS).
\end{equation}
In terms of this embedding the category $\Tw(\cC)$
is defined as the composition $\cC\circ\tau$ where $\cC$
in this formula is interpreted as a simplicial object in $\cS$, and $\tau:\Delta\to\Delta$ is the functor carrying
a finite totally ordered set $I$ to the join $I^\op\star I$.

In general, for $\cC\in\Cat$, we define $P(\cC)=\Fun(\cC^\op,\cS)$. This is the category of presheaves (of spaces) on $\cC$. The functor $Y:\cC^\op\times\cC\to\cS$ can be rewritten as a funtor $Y:\cC\to P(\cC)$ which is fully faithful. This is the $\infty$-categorical version of the classical Yoneda lemma. The category $P(\cC)$ can be otherwise interpreted as the category of right fibrations
$\cX\to\cC$. The right fibration corresponding to $Y(x)$,
$x\in\cC$, is the forgetful functor $\cC_{/x}\to\cC$,
where the overcategory $\cC_{/x}$ is defined as the fiber product $\cC^{[1]}\times_{\cC}\{x\}$, with
the map $\cC^{[1]}\to\cC$ defined by $\{1\}\to[1]$.

\subsubsection{Cocartesian fibrations} The left fibration
(\ref{eq:Tw}) is functorial in $X$. This implies that,
given $f:X\to B$ in $\Cat$, any arrow $a:[1]\to X$ with $a(0)=x,\ a(1)=y$ gives rise to a commutative diagram
\begin{equation}
\label{eq:XXBB}
\xymatrix{
&{X_{y/}}\ar[r]\ar[d] &{X_{x/}}\ar[d]\\
&{B_{f(y)/}}\ar[r]&{B_{f(x)/}}
}
\end{equation}

An arrow $a$ in $X$ is called $f$-cocartesian if the 
diagram (\ref{eq:XXBB}) is cartesian. The map $f:X\to B$
is called {\sl a cocartesian fibration} if for any $x\in X$
and for any $\bar a:f(x)\to b$ there exists an $f$-cocartesian arrow $a:x\to y$ lifting $\bar a$.

The category $\Coc(B)$ of cocartesian fibrations over $B$
is defined as follows. This is a subcategory of 
$\Cat_{/B}$. It is spanned by the cocartesian fibrations
over $B$. A morphism of cocartesian fibrations over $B$ is in $\Coc(B)$ iff it carries cocartesian arrows to
cocartesian arrows.

Grothendieck construction for left fibrations 
extends to cocartesian fibrations. It yields a
canonical equivalence 
\begin{equation}
\label{eq:Grothendieck-coc}
\Coc(B)\to\Fun(B,\Cat).
\end{equation}
The adjoint pair $i:\cS\rlarrows\Cat:K$, with $i$ the
obvious embedding and $K$ given by the formula 
$K(\cC)=\cC^\eq$, extends to the adjoint pair 
$i:\Lt(B)\rlarrows\Coc(B):K $, with $i$ the embedding and 
$K(\cC)$ defined as the subcategory of $\cC$ spanned by the
cocartesian arrows.

\subsubsection{Cartesian fibrations and bifibrations}
An arrow $f:X\to B$ in $\Cat$ is called a cartesian fibration if the arrow
$f^\op:X^\op\to B^\op$ is a cocartesian fibration. Via Grothendieck construction
cartesian fibrations over $B$ correspond to functors $B^\op\to\Cat$. Sometimes a ``mixed'' Grothendieck construction is more appropriate: a functor $B^\op\times C\to\Cat$ can be converted into $f:X\to B\times C$ which is cartesian over $B$ and cocartesian over $C$. We call such maps {\sl bifibrations}
~\footnote{Lurie~\cite{L.T} defines bifibration as the maps
$f:X\to B\times C$ corresponding to functors $B^\op\times C\to\cS$.}.

Here are the details. A functor $f:B^\op\times C\to\Cat$ is the same as a 
functor $B^\op\to\Fun(C,\Cat)=\Coc(C)$. Composing this with the forgetful functor
$\Coc(C)\to\Cat$, we get a contravariant functor from $B$ to $\Cat$ which,
by Grothendieck construction, converts to a cartesian fibration $p:X\to B$.
The constant functor $B^\op\to\Cat$ with value $C$ converts by the Grothendieck construction to the projection $B\times C\to C$. This yields a decomposition
of $p$ as $X\to B\times C\to B$.

In the opposite direction, given a map $f:X\to B\times C$ such that its composition with the projection to $B$ is a cartesian fibration, we get a map $B^\op\to\Cat_{/C}$. If the image of this map belongs to the subcategory $\Coc(C)$
of $\Cat_{/C}$, this defines a functor $B^\op\times C\to\Cat$.

\subsubsection{Correspondences. Adjoint functors} 
Given a pair of categories $\cC,\cD$, a correspondence
from $\cC$ to $\cD$ is a left fibration $p:\cE\to\cC^\op\times\cD$.

A correspondence is called left-representable if for each $x\in\cC$
the base change of $p$ with respect to $\cD \to\cC^\op\times\cD$
determined by $x$, defines a representable presheaf on $\cD^\op$.

A correspondence $p:\cE\to\cC^\op\times\cD$ is called 
right-representable if for each $y\in\cD$ the base change of $p$
with respect to the morphism $\cC^\op\to\cC^\op\times\cD$, determined by $y$,
corresponds to a representable presheaf on $\cC$.

A left-representable correspondence comes from a unique
(up to usual ambiguity) functor $\cC\to\cD$. A right-representable correspondence comes from a unique functor
$\cD\to\cC$.

A left fibration $\pi:\cE\to\cC^\op\times\cD$  that is both left and right representable,
determines a pair of functors, $F:\cC\to\cD$ and
$G:\cD\to\cC$. In this case $F$ is called left adjoint 
to $G$ and $G$ is called right adjoint to $F$.
Since $F$ and $G$ are defined, uniquely up to equivalence,  by $\phi$, and, vice versa, $\phi$ is defined
by any one of $F$, $G$, an adjoint functor, if it exists,
is unique.

\subsubsection{}
A functor $f:\cC\to\cD$ defines an adjoint pair
$$f_!:P(\cC)\rlarrows P(\cD):f^*,$$
with $f_!$ being the colimit-preserving functor defined
by the composition of $f$ with the Yoneda embedding
$Y:\cD\to P(\cD)$, and $f^*$ is defined by restriction
of presheaf to $\cC$. The functor $f^*$ is also colimit-preserving.

\subsection{Flat morphisms in $\Cat$}
\label{ss:flat}

A morphism $f:\cC\to\cD$ in $\Cat$ is called {\sl flat}
if the pullback $f^*:\Cat_{/\cD}\to\Cat_{/\cC}$ admits
a right adjoint. 

If $f:\cC\to\cD$ is flat, $f_*:\Cat_{/\cC}\to\Cat_{/\cD}$ denotes the functor right adjoint to $f^*$. For $X\in\Cat_{/\cC}$,
the fiber of $f_*(X)$ at $d\in\cD$ is given by the formula
\begin{equation}
f_*(X)_d=\Fun_\cC(\cC_d,X).
\end{equation}

In this form the definition belongs to \cite{AFR}, A.3,
where the term {\sl exponentiable} is used instead of {\sl flat}.

Note that, if $f:\cC\to\cD$ is flat, an internal mapping
object $\Fun_{\Cat_{/\cD}}(\cC,X)$ is defined for any 
$X\in\Cat_{/\cD}$ by the formula 
$\Fun_{\Cat_{/\cD}}(\cC,X)=f^*(f_*(X))$.

Flat morphisms were introduced by J.~Lurie in \cite{L.HA},
B.3, and were called there {\sl flat categorical fibrations}. The following characterizations of flat morphisms can be found in \cite{L.HA}, B.3.2 and 
\cite{AFR}, A.16.

\begin{prp}
\label{prp:flat-properties}
\hspace{1cm}
\begin{itemize}
\item[1.] A map $f:\cC\to\cD$ is flat iff for any $s:[2]\to\cD$ the base change $\cC\times_{\cD}[2]\to[2]$ is flat.
\item[2.]A functor $p:\cC\to[2]$ is  flat iff the natural map
$$ \cC_{\{0,1\}}\sqcup^{\cC_1}\cC_{\{1,2\}}\to\cC$$
where $\cC_i$, resp., $\cC_{\{i,j\}}$, are defined  by
base change of $\cC$ with respect to $\{i\}\to[2]$,
resp., $\{i,j\}\to[2]$, is an equivalence. 
\item[3.]A functor $p:\cC\to[2]$ is  flat iff  
for any arrow $f:A\to C$ in $\cC$, with $p(A)=\{0\},\
p(C)=\{2\}$, the full subcategory of $\Fun_{[2]}([2],\cC)$
spanned by $s:[2]\to\cC$ such that $d_1(s)=f$, is weakly contractible.
\end{itemize}
\end{prp}\qed

In this paper we will use an interpretation of flatness
in terms of correspondences.

Let $f:\cC\to[1]$ be a functor with fibers $\cC_0, \cC_1$
at $0$ and $1$. We denote by $\phi_i:\cC_i\to \cC$ the embeddings of the fibers.

 The composition
$P(\cC_1)\stackrel{\phi_{1!}}{\to}P(\cC)
\stackrel{\phi_0^*}{\to} P(\cC_0)$
is a colimit-preserving functor defining a correspondence
from $\cC_0$ to $\cC_1$.

Given $f:\cC\to\cD$, the base change along any arrow 
$\alpha:d\to d'$ 
in $\cD$ gives rise, therefore, to a colimit preserving 
functor $f^\alpha:P(\cC_{d'})\to P(\cC_d)$ between the presheaves on the fibers. 
This assignment, however, is not necessarily functorial.
We will show that it is functorial precisely when $f$ is flat.

Let $f:\cC\to [2]$. The maps $\phi_i:\cC_i\to\cC$ and
$\phi_{i,j}:\cC_{\{i,j\}}\to\cC$ are fully faithful, so that the
unit of adjunction $\id\to \phi^*\circ\phi_!$ is an equivalence for $\phi=\phi_i,\phi_{i,j}$.
The functor $f$ defines correspondences $f_{12}:P(\cC_2)\to
P(\cC_1)$ and $f_{01}:P(\cC_1)\to P(\cC_0)$ whose composition
is given as
$$ P(\cC_2)\stackrel{\phi_{2!}}{\to}
P(\cC)\stackrel{\phi_1^*}{\to} 
P(\cC_1)\stackrel{\phi_{1!}}{\to}
P(\cC)\stackrel{\phi_0^*}{\to} P(\cC_0).
$$
Applying the counit of adjunction $\phi_{1!}\circ\phi_1^*\to\id$, we get a morphism of functors 
\begin{equation}
\label{eq:comp-corr}
f_{01}\circ f_{12}\to f_{02}.
\end{equation}
We have
\begin{prp}
\label{prp:flat-corr}
A map $f:\cC\to[2]$ is  flat iff
the morphism of the functors (\ref{eq:comp-corr})
from $P(\cC_2)$ to $P(\cC_0)$
is an equivalence.
\end{prp}
\begin{proof}We choose $A\in\cC_0$, $C\in\cC_2$, and evaluate both functors $P(\cC_2)\to P(\cC_0)$ at $C$ 
(getting a morphism of presheaves on $\cC_0$) and at $A$
(getting a morphism of spaces). We have to verify when
this morphism of spaces is an equivalence. The target is
just $\Map_\cC(A,C)$. Let us calculate the source.
Denote by $Y_C$ the presheaf in $P(\cC)$ represented by $C\in\cC_2\subset\cC$. It is defined by the right fibration
$p_C:\cC_{/C}\to\cC$, so $\phi_1^*(Y_C)$ is defined by the
right fibration $(\cC_1)_{/C}\to \cC_1$ defined as the base change of $p_C$ with respect to the embedding $\cC_1\to\cC$. Thus, $\phi_1^*(Y_C)=\colim((\cC_1)_{/C}\to\cC_1\to P(\cC_1))$. Applying $\phi_{1!}$, we get
$$\phi_{1!}\phi_1^*(Y_C)=\colim((\cC_1)_{/C}\to\cC_1\to \cC\to P(\cC)).$$
The morphism (\ref{eq:comp-corr}) evaluated at $C$
is the canonical map $\phi_{1!}\phi^*_1(Y_C)\to Y_C$,
and it is given by the canonical map of colimits
$$
\colim((\cC_1)_{/C}\to \cC\to P(\cC))\to
\colim(\cC_{/C}\to\cC\to P(\cC)).
$$
To find out when is it an equivalence, we will evaluate the colimits at an arbitrary $A\in\cC_0\subset\cC$. The results are the colimits of the compositions 
$(\cC_1)_{/C}\to\cC_1\stackrel{e_A}{\to}\cS$
and 
$\cC_{/C}\to\cC\stackrel{e_A}{\to}\cS$ respectively,
where $e_A(B)=\Map_\cC(A,B)$. 
Note that the functor $e_A:\cC\to\cS$ is the left Kan extension of the terminal functor $t:\cC_{A/}\to\cS$,
with respect to the left fibration $\pi:\cC_{A/}\to\cC$.

Therefore, evaluating  (\ref{eq:comp-corr}) at $A$ and $C$,
we get a morphism of spaces
\begin{equation}
|(\cC_1)_{A/}\times_{\cC_1}(\cC_1)_{/C}|\to
|\cC_{A/}\times_{\cC}\cC_{/C}|=
\Map_\cC(A,C).
\end{equation}
It is an equivalence iff all its fibers are equivalences.
This, by Proposition~\ref{prp:flat-properties}(3), is equivalent to $f$ being flat.

\end{proof}
 
\

\subsection{Localization}

The embedding $\cS\to\Cat$ admits a left adjoint
$L:\Cat\to\cS$ called {\sl the total Dwyer-Kan localization}.

More generally, for $f:\cC^\circ\to\cC$ the Dwyer-Kan localization $L(f)=L(\cC,\cC^\circ)$ is defined as the colimit $L(\cC,\cC^\circ)=L(\cC^\circ)\sqcup^{\cC^\circ}\cC$.

This is the most general notion of localization available in the 
$(\infty,1)$-world.

\subsubsection{Marked categories}
\label{sss:marked}
One usually localizes along $\cC^\circ$ which is a subcategory of
$\cC$ having the same objects as $\cC$. It is also assumed that $\cC^\circ\supset\cC^\eq$. Such a pair $(\cC,\cC^\circ)$
is called {\sl a marked category}.

The subcategory $\cC^\circ$ determines for each pair of 
objects $x,y$ a set of {\sl marked} connected components in 
$\Map(x,y)$. The trivially marked category $\cC^\flat$ is defined as the pair $(\cC,\cC^\eq)$; the maximally marked
category $\cC^\#$ is $(\cC,\cC)$. Sometimes a category $\cC$ has a ``standard'' subcategory $\cC^\circ$ \footnote{For instance, the category of finite pointed sets $\Fin_*$ has the subcategory $\Fin_*^\circ$ spanned by the inert arrows.}. Then we denote $\cC^\natural=(\cC,\cC^\circ)$.

The category of marked categories is denoted by $\Cat^+$. It is defined as a full
subcategory of  $\Fun([1],\Cat)$ spanned by the embeddings $\cC^\circ\to\cC$. 
\subsubsection{}
A typical example of Dwyer-Kan localization is the functor from a model category to its underlying $\infty$-category.  

Let $M$ be a model category with the collection of 
weak equivalences $W\subset M$. Dwyer-Kan localization
$L(M,W)$ was the main object of study in the original 
series of papers \cite{DK1}--\cite{DK3} by Dwyer and Kan. We call this localization {\sl the $\infty$-category underlying $M$}.

\subsubsection{Bousfield localization} A functor
$L:\cC\to\cL$ is called a (Dwyer-Kan) localization if it 
induces an equivalence $L(\cC,\cC^\circ)\to\cL$ where 
$\cC^\circ=f^{-1}(\cL^\eq)$. A localization $L$ is called a 
Bouslfield localization if it admits a fully faithful right 
adjoint.

A typical example is the total DK localization $L:\Cat\to\cS$ whose right adjoint is the standard embedding
of the category of spaces $\cS$ to $\Cat$.

Another example is the functor $\Fun([1],\Cat)\to\Cat^+$
carrying an arrow $\cW\to\cC$ to the marked category $(\cC,\cC^\circ)$ where $\cC^\circ$ is the subcategory of $\cC$
generated by $\cC^\eq$ and the image of $\cW$.

Finally, the DK localization functor $L:\Cat^+\to\Cat$
is itself a Bousfield localization: its right adjoint $\cC\mapsto\cC^\flat$ is fully faithful.

\subsubsection{Complete Segal spaces} The embedding
$\Cat\to\Fun(\Delta^\op,\cS)$ described in (\ref{eq:Cat-as-sspace}) has a left
adjoint which is another example of Bousfield localization.
This adjoint pair of $\infty$-categories is usually deduced 
from the Rezk (complete Segal) model structure on bisimplicial sets 
obtained from the Reedy model structure by
(what is classically called) Bousfield localization
of a model category.

\subsection{Marked categories}

\subsubsection{}
\label{sss:marked-base-change}
A functor $f:\cC^\natural\to\cD^\natural$ between marked categories defines an adjoint pair of functors
\begin{equation}
f_!:\Cat^+_{/\cC^\natural}\rlarrows
\Cat^+_{/\cD^\natural}:f^*,
\end{equation}
with $f^*$ defined by pullback along $f:\cC^\natural\to\cD^\natural$ and $f_!$ given by composition with $f$.

\subsubsection{}
\label{sss:marked-base-change-flat}

In the case where $f$ is flat, the functor 
$f^*:\Cat^+_{/\cD^\natural}\to\Cat^+_{/\cC^\natural}$ has a right adjoint $\bar f_*$ whose explicit description is
presented below.

For $X^\natural=(X,X^\circ)\in\Cat^+_{/\cC^\natural}$ 
the marked category $\bar f_*(X^\natural)$ over $\cD^\natural$  is defined as $(Y,Y^\circ)$
where $Y$ is the full subcategory of $f_*(X)$ spanned by the objects $y:\cC_d\to X$ over $d\in\cD$ carrying
$\cC_d\cap\cC^\circ$ to $X^\circ$. An arrow in $Y$
over $a:d\to d'$ in $\cD$ is given by a functor 
$ \alpha:\cC\times_{\cD}\{a\}\to X$ over $\cC$; it belongs to $Y^\circ$ iff $a$ is marked in $\cD$ and $\alpha$ preserves markings.

Let $Z^\natural\in\Cat^+_{/Y^\natural}$. The space
$\Map_{\Cat^+_{/\cD^\natural}}(Z^\natural,\bar f_*(X^\natural))$ is, by definition, 
the subspace of $\Map_{\Cat_{/\cD}}(Z,f_*(X))=\Map_{\Cat_{/\cC}}(Z\times_\cD\cC,X)$ spanned by the maps
$\phi:Z\times_\cD\cC\to X$ such that
for any $\alpha:[1]\to Z^\circ$ over $a:[1]\to\cD$
the composition $[1]^\sharp\times_\cD\cC\to X$ preserves markings. This subspace is precisely 
$\Map_{\Cat^+_{/\cC^\natural}}
(f^*(Z^\natural),X^\natural)$,
so $\bar f_*$ is right adjoint to $f^*$.

\subsection{Categories with decomposition}
The following definition is an  adaptation
  of (a special case of) Lurie's notion
of categorical pattern, see~\cite{L.HA}, App.~B.

Categories with decomposition are an important technical 
tool in working with operads and operad-like objects.

\begin{dfn}
\label{dfn:catdec}
A category with decomposition is a triple $(\cC,\cC^\circ, D)$ where $(\cC,\cC^\circ)$ is a marked category
and $D$ is a collection of ``decomposition diagrams''
$\{\rho^{d,i}:C^d\to C^d_i\}_{d\in D}$ in $\cC^\circ$.
\end{dfn}
We will use the notation $\cC^\deco=
(\cC,\cC^\circ, D)$ for a category with decomposition.
Some special cases: for a marked category $\cC^\natural=(\cC,\cC^\circ)$ we denote by $\cC^{\natural,\emptyset}$
the category with decomposition $(\cC,\cC^\circ,\emptyset)$. In particular, $\cC^{\flat,\emptyset}$ and $\cC^{\sharp,\emptyset}$ are defined by the minimal and the maximal marking.

The first important example of a category with decomposition is  the
category of finite pointed sets $\Fin_*$.
\begin{exm} The decomposition structure on $\Fin_*$ is defined as follows. We define $\Fin_*^\circ$ to be the category
spanned by the inert arrows. For any $I_*\in\Fin_*$
any presentation $I=\sqcup J^i$ defines a decomposition
diagram consisting of the inert arrows $I_*\to J^i_*$. 
\end{exm}

Let $\cC^\deco=(\cC,\cC^\circ,D)$ be a category with 
decomposition.
 
 We will now define a subcategory $\Fib(\cC^\deco)$
of the category $\Cat_{/\cC}$.

\begin{dfn}(see~\cite{L.HA}, 2.3.3.28) An object $p:X\to\cC$ of $\Cat_{/\cC}$ is called
{\sl fibrous} if the following conditions are satisfied.
\begin{itemize}
\item[(Fib1)] For any $x\in X$ any marked 
$\alpha:p(x)\to C$
has a cocartesian lifting. In particular, any marked arrow
$\alpha:C\to C'$ in $\cC$ defines a functor $\alpha_!:X_C\to X_{C'}$.
\item[(Fib2)] For any $d\in D$ the collection of maps 
$\rho^{d,i}_!:X_{C^d}\to X_{C^d_i}$ defines an equivalence of categories $X_{C^d}\to\prod X_{C^d_i}$.
\item[(Fib3)] For any $d\in D$ and any $x\in X_{C^d}$ the diagram
of cocartesian liftings of $\rho^{d,i}$, $\{x\to x_i\}$,
is a $p$-product diagram, \cite{L.T}, 4.3.1.1.
\end{itemize}
\end{dfn}

Given a category with decomposition 
$\cC^\deco=(\cC,\cC^\circ,D)$, 
we define $\Fib(\cC^\deco)$ as the subcategory of 
$\Cat_{/\cC}$
spanned by the fibrous arrows $X\to\cC$, with morphisms
preserving the cocartesian liftings of arrows in 
$\cC^\circ$.

For example, $\Fib(\cC^{\flat,\emptyset})=\Cat_{/\cC}$
and $\Fib(\cC^{\sharp,\emptyset})=\Coc(\cC)$.  

\subsubsection{}
\label{sss:deco-bousfield-loc}
Let $\cC^\deco=(\cC^\natural,D)$ be a category with decomposition. The functor $i:\Fib(\cC^\deco)\to\Cat^+_{/\cC^\natural}$ carrying $p:X\to\cC$ to the marked  category
$X^\natural$ over $\cC^\natural$, with the marking defined
by the cocartesian lifting of the marked arrows in $\cC$, is fully faithful. 

\begin{Lem}
The functor $i$ has a left adjoint presenting $\Fib(\cC^\deco)$ as Bousfield localization of $\Cat^+_{/\cC^\natural}$.
\end{Lem}
\begin{proof}
This is an $\infty$-categorical version of Lurie's
Theorem B.0.20 \cite{L.HA} where $\Fib(\cC^\deco)$
is described as the $\infty$-category underlying a model structure on marked simplicial sets over $\cC^\natural$.

An arrow $f:X\to Y$ in $\Cat^+_{/\cC^\natural}$ will be called $\cC^\deco$-equivalence if, for any $Z\in\Fib(\cC^\deco)$, it induces and equivalence $\Map(Y,Z)\to\Map(X,Z)$. We have to verify that, for any $X\in\Cat^+_{/\cC^\natural}$, there exists a $\cC^\deco$-equivalence
$X\to X'$ with $X'\in\Fib(\cC^\deco)$.

We will deduce this fact from \cite{L.HA}, B.0.20.
Let $X^\natural$ be a marked category over $\cC^\natural$.
We will represent $X^\natural$ and $\cC^\natural$ by pairs of quasicategories $(X,X^\circ)$ and $(\cC,\cC^\circ)$.

According to \cite{L.HA}, B.0.20, there is a trivial cofibration $X^\natural\to X^{\prime\natural}$ to a fibrant
object $X^{\prime\natural}$. This means that for any fibrant $Z$ a homotopy equivalence 
$$\Map^\#_{\cC^\natural}(X^{\prime\natural},Z^\natural)
\to\Map^\#_{\cC^\natural}(X^\natural,Z^\natural)$$
is induced, where $\Map^\#$ defines the simplicial
structure on the category $\sSet^+_{/\cC^\natural}$ It remains to verify that $\Map^\#$ in our cases represents correctly the  mapping space in $\Cat^+$.
This is actually true for any $Z\in\Fib(\cC^{\natural,\emptyset})$. In fact, it is easy to see that
$\Map_{\Cat^+_{/\cC^\natural}}(X^\natural,Z^\natural)$
is represented by the maximal Kan simplicial subset of
$\Map^\#_{\cC^\natural}(X^\natural,Z^\natural)$. In the
case where
$Z^\natural$ is fibrant, 
$\Map^\#_{\cC^\natural}(X^\natural,Z^\natural)$ is Kan.
\end{proof}

Arrows of $\Cat^+_{/\cC^\natural}$ whose image is an
equivalence in $\Fib(\cC^\deco)$, will be called {\sl 
$\cC^\deco$-equivalences}. 
\subsubsection{} 
Let $\cC^\deco$ be a decomposition category. A collection of arrows $\rho^i:C\to C^i$ in $\cC^\circ$ is called
{\sl a weak decomposition diagram} if any 
 $X\in\Fib(\cC^\deco)$ satisfies  the properties (Fib2) 
and (Fib3) with respect to the collection 
$\{\rho^i:C\to C^i\}$.

\subsubsection{}
\label{sss:cd-base-change}
A {\sl functor between decomposition categories} 
$f:\cC^\deco\to\cD^\deco$ is defined as a 
functor $f:\cC\to\cD$ preserving the markings and  carrying decomposition diagrams to weak decomposition diagrams. 

A functor $f$ preserving the markings defines an adjoint pair $(f_!,f^*)$ of functors, see~(\ref{sss:marked-base-change}). 
 
If $f$ is a functor between decomposition categoires, 
$f^*$ preserves the fibrous objects. This implies that the functor
$f_!:\Cat^+_{/\cC^\natural}\to\Cat^+_{/\cD^\natural}$
carries $\cC^\deco$-equivalences to 
$\cD^\deco$-equivalences, see~\cite{L.HA}, B.2.9. In particular,   an adjoint pair of functors
\begin{equation}
\label{eq:functor-adj}
f_!:\Fib(\cC^\deco)\rlarrows\Fib(\cD^\deco):f^*
\end{equation}
is defined, where $f^*$ is just the restriction of the base change functor~(\ref{sss:marked-base-change}), and $f_!$ commutes with the 
localizations $\Cat^+_{/\cC^\natural}\to\Fib(\cC^\deco)$
and $\Cat^+_{/\cD^\natural}\to\Fib(\cD^\deco)$.

In particular, applying the above adjunction to the
functor $\id:\cC^{\flat,\emptyset}\to\cC^\deco$, we deduce
the following.
\begin{crl}
\label{crl:fibrous-limits}
The forgetful functor $\Fib(\cC^\deco)\to\Cat_{/\cC}$
preserves limits.
\end{crl}\qed

\subsubsection{}
\label{sss:cofunctor}

Let $\cC^\deco,\ \cD^\deco$ be categories with decomposition, and let $f:\cC^\natural\to\cD^\natural$
be a map of the corresponding marked categories. Assume that $f$ is flat.

We call $f$ {\sl a cofunctor between the decomposition
categories} if $f^*:\Cat^+_{/\cD^\natural}\to\Cat^+_{/\cC^\natural}$ carries $\cD^\deco$-equivalences to $\cC^\deco$-equivalences. A cofunctor
between decomposition categories gives rise therefore to an adjoint pair
\begin{equation}
\label{eq:cofunctor-adj}
f^*:\Fib(\cD^\deco)\rlarrows\Fib(\cC^\deco):\bar f_*
\end{equation}
where $\bar f_*$ is just the restriction of the functor
defined in \ref{sss:marked-base-change-flat}
to the subcategory $\Fib(\cC^\deco)$, and $f^*$ commutes with the localizations.
\footnote{Note that a cofunctor between decomposition categories
does not necessarily carry decomposition diagrams in $\cC^\deco$ to decomposition dieagrams in $\cD^\deco$ (even though in the cases we have in mind this does happen).
So the functor $f^*$ in the formula (\ref{eq:cofunctor-adj}) does not necessarily coincide with the one in 
(\ref{eq:functor-adj}).}

Lurie \cite{L.HA}, B.4.1, provides sufficient conditions
for a marked map $f$ to be a cofunctor. We will use them  
in~\ref{prp:op-lcc} in the study of flat operads.
We present below his result specialized to the context
of decomposition structures.

\begin{prp}(see~\cite{L.HA}, B.4.1).
\label{prp:cofunctor}
Let $\cC^\deco$, $\cD^\deco$ be categories with decomposition, and let $f:\cC^\natural\to\cD^\natural$
be a flat marked functor. Then $f$ is a cofunctor between the decomposition categories, provided the following conditions are fulfilled.
\begin{itemize}
\item[1.]For any marked arrow  $\alpha:[1]\to\cD^\circ$
the fiber products $f_\alpha:\cC_\alpha=[1]\times_\cD\cC\to[1]$ and $f^\circ_\alpha:\cC^\circ_\alpha=[1]
\times_{\cD^\circ}\cC^\circ\to[1]$ are cartesian fibrations, and the embedding $\cC^\circ_\alpha\to\cC_\alpha$ is a map of cartesian fibrations.
\item[2.]For any marked arrow $\alpha:[1]\to\cD^\circ$
appearing in a decomposition diagram in $\cD$ the maps
$f_\alpha$, $f^\circ_\alpha$ are also cocartesian fibrations, and the embedding $\cC^\circ_\alpha\to\cC_\alpha$ is a map of cocartesian fibrations.
\item[3.] A cocartesian lifting of a decomposition diagram
in $\cD$ is a weak decomposition diagram in $\cC$.
\end{itemize}
\end{prp}
\begin{proof}
Lurie~\cite{L.HA}, B.4.1, lists nine conditions. His conditions (ii), (iii), (iv) are part of our requirements. Condition (vi) is void in the context of decomposition categories. Conditions (i) and (v)
are equivalent to our condition (1),  conditions (vii) and (viii) are equivalent to our (2), and (ix) is just our (3).
\end{proof}

\subsubsection{}
\label{sss:decoproduct}
Given two decomposition categories $\cC^\deco=(\cC,\cC^\circ,D)$, $\cD^\deco=(\cD,\cD^\circ,D')$, we define
their product as the category $\cE=\cC\times\cD$,
with $\cE^\circ=\cC^\circ\times\cD^\circ$, and decompositions $(C,D)\to(C^i,D^j))$ where $(C\to C^i)$ and $(D\to D^j)$are weak decomposition diagrams~\footnote{Note that we have not defined the category of decomposition categories.}.

The product of categories over $\cC$ and $\cD$ defines a 
functor 
$$\Cat_{/\cC}\times\Cat_{/\cD}\to\Cat_{/\cC\times\cD}.$$
This functor carries pairs of fibrous categories to a fibrous category, so defining a functor
\begin{equation}\label{eq:deco-product}
\Fib(\cC^\deco)\times\Fib(\cD^\deco)\to
\Fib(\cC^\deco\times\cD^\deco).
\end{equation}

One has
\begin{prp}\label{prp:mu-functorial}
\hspace{1cm}
\begin{itemize}
\item[1.] The product map (\ref{eq:deco-product})
preserves colimits in each of the two arguments.
\item[2.]
It is functorial: a pair of maps 
$f:\cC^\deco\to\cC^{\prime\deco}$
and $g:\cD^\deco\to\cD^{\prime\deco}$ gives rise to a commutative diagram
\begin{equation}\label{eq:mu-functorial}
\xymatrix{
&{\Fib(\cC^\deco)\times\Fib(\cD^\deco)}\ar[r]\ar^{f_!\times g_!}[d]
&{\Fib(\cC^\deco\times\cD^\deco)}\ar^{(f\times g)_!}[d]\\
&{\Fib(\cC^{\prime\deco})\times\Fib(\cD^{\prime\deco})}
\ar[r]&{\Fib(\cC^{\prime\deco}\times\cD^{\prime\deco}).}
}
\end{equation}.
\end{itemize}
\end{prp}
\begin{proof}
This follows from \cite{L.HA}, B.2.5 and B.2.9. The product
map (\ref{eq:deco-product}) can be presented by a left 
Quillen bifunctor, and $f_!$ is also presented by left Quillen functor. The corresponding functors commute on the
level of model categories, so this proves the claim.
\end{proof}

\subsection{Operads} 
The category of operads $\Op$ is defined as $\Fib(\Fin_*^\deco)$.

Marked arrows in $\Fin_*^\deco$ are inert arrows 
of $\Fin_*$. Cocartesian liftings in $\cO$ of inert arrows 
in $\Fin_*$ are called {\sl inert arrows in} $\cO$.

Thus,  $\Op$ is the subcategory of $\Cat_{/Fin_*}$ spanned by 
the operads, with the arrows preserving the inerts.

Our definition is essentially equivalent to the one given 
in~\cite{L.HA}, Section 2. In fact,  
$p:\cO\to\Fin_*$ is an operad if an only if it 
satisfies the conditions of definition \cite{L.HA}, 
2.1.1.10 when presented by a categorical fibration.

An arrow $f:I_*\to J_*$ is called {\sl active} if the 
preimage of $*\in J_*$ consists of $*$ only. An arrow
in $\cO$ over an active arrow in $\Fin_*$ is also called an 
active arrow.

The fiber of $p:\cO\to\Fin_*$ at $\langle 1\rangle$, denoted in what follows by $\cO_1$, is the category of colors
of $\cO$.

\subsubsection{Strong approximation of operads}
This is a version of Lurie's {\sl approximation of operads} described in \cite{L.HA}, 2.3.3.

Let $\cO\in\Cat_{/\Fin_*}$ be an operad.

\begin{Dfn}
\label{dfn:strongapprx}
A map $f:\cC\to\cO\stackrel{p}{\to}\Fin_*$ in 
$\Cat_{/\Fin_*}$
is called {\sl a strong approximation} if it satisfies the following conditions. 
\begin{itemize}
\item[(1)] Let $C\in\cC$ have image $\langle n\rangle$ in $\Fin_*$; then there exist $p\circ f$-locally cocartesian liftings 
$a_i:C\to C_i$ of the standard inerts $\rho^i:\langle n\rangle\to\langle 1\rangle$, and $f(a_i)$ are inert in $\cO$.
\item[(2)] Let $a:X\to f(C)$ be active in $\cO$ (that is,
its image in $\Fin_*$ is active). Then there exists an
$f$-cartesian lifting $\wt X\to C$ of $a$.
\item[(3)] The map $f$ induces an equivalence 
$$ \cC_1:=\cC\times_{\Fin_*}\{\langle 1\rangle\}\to
\cO_1
$$
of the corresponding categories of colors.

\end{itemize}
\end{Dfn}
 
Strong approximation of operads allows one to have an
 alternative description of $\Op_{/\cO}$.

Let $f:\cC\to\cO$ be a strong approximation. We endow 
$\cC$ with an induced decomposition structure as follows.
$\cC^\circ$ is spanned by the arrows whose image in $\cO$ is inert. Decompositions are given by cocartesian liftings
of the standard inerts $\rho^i:\langle n\rangle\to\langle 1\rangle$.  As with the operads, the arrows of $\cC^\circ$
are called inert. Active arrows in $\cC$ are defined
as the cartesian liftings of the active arrows in $\cO$. 
Also as with operads, strong approximation inherits a 
factorization system from its operad: any arrow in $\cC$ 
decomposes as an inert followed by an active arrow,
see\cite{L.HA}, 2.3.3.8 and 2.1.2.5.

Fibrous objects over $\cC^\deco$ will also be called 
$\cC$-operads. We denote $\Op_\cC=\Fib(\cC^\deco)$.

 One has the following
\begin{prp}\label{prp:equivalence}
Let $f:\cC\to\cO$ be a strong approximation of an operad 
$\cO$.
Then the base change with respect to $f$ induces an equivalence $\Op_{/\cO}\to\Op_\cC$.
\end{prp} 
\begin{proof}See \cite{L.HA}, 2.3.3.26.
\end{proof}

The above equivalence is functorial: given a commutative
diagram 
\begin{equation}
\label{eq:twoapprx}
\xymatrix{
&{\cP}\ar^{g'}[r]\ar^{f'}[d]&{\cC}\ar^{f}[d]\\
&{\cO'}\ar^{g}[r]&{\cO}
}
\end{equation}
where $g$ is a morphism of operads and $f,f'$ are strong approximations, the base change with respect to $g'$ gives
rise to a commutative diagram 
\begin{equation}
\nonumber
\xymatrix{
&{\Op_{/\cO}}\ar^{}[r]\ar^\sim[d]&{\Op_{/\cO'}}\ar^{\sim}[d]\\
&{\Op_{\cC}}\ar^{ }[r]&{\Op_\cP}
}.
\end{equation}

Strong approximations of operads are preserved by the base change. 

\begin{lem}(see~\cite{L.HA}, 2.3.3.9.)
Let $f:\cC\to\cO$ be a strong approximation and let $p:\cO'\to\cO$ be a map of operads. The the map
$f':\cC'\to\cO'$ obtained by base change from $f$, is
also a strong approximation.
\end{lem}
\qed

\subsubsection{$\cO$-monoidal categories}
\label{sss:deco-mono}
Let $f:\cC\to\cO$ be a strong approximation. We denote by
$\Mon_{\cO}$ (resp., $\Mon_\cC$) the subcategory of 
$\Op_{/\cO}$ (resp., $\Op_\cC$) whose objects are 
cocartesian fibrous objects  and whose morphisms preserve 
the cocartesian arrows. The equivalence \ref{prp:equivalence} induces an equivalence
$\Mon_{\cO}\to\Mon_\cC$. One can describe $\Mon_\cC$
as $\Fib(\cC^\deco)$ for the following decomposition
structure on $\cC$:
\begin{itemize}
\item[(Mon1)] $\cC^\circ=\cC$.
\item[(Mon2)] Decomposition diagrams are the same
as in $\cC^\deco$.
\end{itemize}

Since cocartesian fibrations $X\to\cO$ can be described by
functors $\cC\to\Cat$, $\Mon_\cC$ identifies with the full
subcategory $\Fun^\lax(\cC,\Cat)\subset\Fun(\cC,\Cat)$ 
spanned by the functors $F$ satisfying the following version of Segal condition: each decomposition diagram 
$\{C\to C_i\}$ gives rise to
an equivalence $F(C)\to\prod_i F(C_i)$,
see \cite{L.HA}, 2.4.7.1, where such functors are called
lax cartesian structures.

\subsubsection{$\Cat$-enrichment}
\label{sss:alg}
Let $\cO',\cQ\in\Op_{/\cO}$. The category 
$\Alg_{\cO'/\cO}(\cQ)$ is defined as the full subcategory of $\Fun_\cO(\cO',\cQ)$ spanned by the $\cO$-operad maps. 

Given a pair of strong approximations as in (\ref{eq:twoapprx}),
and $\cQ\in\Op_\cC$, one defines  
$\Alg_{\cP/\cC}(\cQ)$ as the full subcategory of
$\Fun_\cC(\cP,\cQ)$ spanned by the maps $\cP\to\cQ$ over 
$\cC$ preserving the inerts~\footnote{Note that $\cP$ is
not necessarily a $\cC$-operad.}. Note that for 
$\cP,\cQ\in\Op_\cC$ one has $\Alg_{\cP/\cC}(\cQ)^\eq=\Map_{\Op_\cC}(\cP,\cQ)$, so that the assignment
$\cP,\cQ\mapsto\Alg_{\cP/\cC}(\cQ)$ provides (a sort of)
$\Cat$-enrichment on $\Op_\cC$. This enrichment is
independent of the approximation, as the following lemma
shows.

\begin{Lem}
Let $\cQ=\cC\times_\cO\cQ'$ for $\cQ'\in\Op_{/\cO}$.
Then the base change along $\cC\to\cO$ provides a natural
equivalence
$\Alg_{\cO'/\cO}(\cQ')\to\Alg_{\cP/\cC}(\cQ)$.
\end{Lem}
\begin{proof}One can describe $\Alg_{\cO'/\cO}(\cQ')$
as the fiber of the map $\Alg_{\cO'}(\cQ')\to
\Alg_{\cO'}(\cO)$ at $g\in\Alg_{\cO'}(\cO)$. Similarly, 
$\Alg_{\cP/\cO}(\cQ')$
is the fiber of the map $\Alg_\cP(\cQ')\to\Alg_\cP(\cO)$
at $f\circ g'$.
These two categories identify by \cite{L.HA}, 2.3.3.23.
Finally, the fiber at $f\circ g'$ can be calculated in two steps, starting with the base change with respect to
$\Alg_\cP(\cC)\to\Alg_\cP(\cP)$. This identifies
$\Alg_{\cP/\cO}(\cQ')$ with $\Alg_{\cP/\cC}(\cQ)$.
\end{proof}

Similarly, for $\cP,\cQ\in\Mon_\cC$ we define
$\Fun^\otimes_\cC(\cP,\cQ)$ as the full subcategory 
of $\Fun_\cC(\cP,\cQ)$ spanned by the maps preserving cocartesian arrows.

\subsection{Internal mapping object in operads}
\label{prp:op-lcc}

The category $\Op$ is not cartesian closed as direct
product of operads does not commute, in general, with colimits. We impose an extra condition on operads to
correct the problem.

\begin{dfn}
\label{dfn:flatoperad}
Let $\cC$ be a strong approximation of an operad. A $\cC$-operad $p:\cP\to\cC$ is called {\sl flat}
if any of the  equivalent conditions of Lemma~\ref{lem:flatoperad0} below is satisfied. 
\end{dfn}

\begin{lem}
\label{lem:flatoperad0}
Let $\cC$ be a strong approximation of an operad and let
$p:\cP\to\cC$ be a $\cC$-operad. The following conditions
are equivalent.
\begin{itemize}
\item[(1)] For any pair of composable arrows in $\cC$
given by $\sigma:[2]\to\cC$, with $d_0\sigma$ active,
the base change $\cP\times_\cC[2]\to[2]$ is  flat. 
\item[(2)] For any pair of composable active arrows in $\cC$
given by $\sigma:[2]\to\cC$,  
the base change $\cP\times_\cC[2]\to[2]$ is  flat.
\item[(3)] For any pair of composable active arrows in $\cC$
given by $\sigma:[2]\to\cC$, with $\sigma(2)\in\cC_1$,
the base change $\cP\times_\cC[2]\to[2]$ is flat.
\end{itemize}
\end{lem}
\begin{proof}
The implications $(1)\Rightarrow(2)\Rightarrow(3)$ are clear.

Let $\sigma:[2]\to\cC$ have $\gamma=d_0\sigma$ active.
We decompose $d_2\sigma=\beta\circ\alpha$ with $\beta$ active and $\alpha$ inert. If now $\tau:[3]\to\cC$
is composed of $\alpha,\beta$ and $\gamma$, one has a
commutative diagram of functors between the categories
of presheaves
\begin{equation}
\xymatrix{
&{(p^\alpha\circ p^\beta)\circ p^\gamma}\ar^{\theta_{\alpha,\beta}}[r]\ar@{=}[dr] 
&{p^{\beta\circ\alpha}\circ p^\gamma}\ar[r]
&{p^{\gamma\circ\beta\circ\alpha}}\\
& &{p^\alpha\circ(p^\beta\circ p^\gamma)}\ar[r]
&{p^\alpha\circ p^{\gamma\circ\beta}}\ar_{\theta_{\alpha,\gamma\circ\beta}}[u]
}
\end{equation}
Since $\alpha$ is inert, $\theta_{\alpha,\beta}$ and
$\theta_{\alpha,\gamma\circ\beta}$ are equivalences.
This proves that (2) implies  (1). 

Any base change with respect to 
a pair of active arrows decomposes as a  product of
base changes as in (3). Thus, $(3)\Rightarrow(2)$.
\end{proof}
\begin{Exm}
$\cC$-monoidal categories are flat $\cC$-operads as cocartesian fibrations are flat. Symmetric promonoidal categories in the sense of \cite{BGS}, 1.4,
are flat operads. The operads $\BM_X,\LM_X$ and $\Ass_X$ 
constructed in Section~\ref{sec:quivers}, are flat,
see \ref{prp:BMX-flat}.
\end{Exm}

The flatness of an operad 
$\cP\in\Op_{/\cO}$ is independent of
the choice of strong approximation. In fact, 
let $p:\cC\to\cO$ be a strong approximation.
Any $\sigma:[2]\to\cO$ defined by a pair of active arrows
and satisfying the condition $\sigma(2)\in\cO_1$, as in condition (3) above, factors through $p:\cC\to\cO$.
This means that $\cP$ is flat as $\cO$-operad iff its base change to $\cC$ is flat as $\cC$-operad.  

\

Proposition~\ref{prp:flatoperad} below is a version of
Lurie's \cite{L.HA}, 2.2.6,  where a Day
convolution monoidal structure is presented as a special 
case of internal mapping object in operads. This result also generalizes a part of \cite{BGS}, 1.6.
\begin{prp}
\label{prp:flatoperad}
Let $\cC$ be a strong approximation of an operad and let 
$\cP\in\Op_\cC$ be flat. Then the product functor 
$\times\cP:\Op_\cC\to\Op_\cC$
has a right adjoint.  
\end{prp}

The functor right adjoint to $\_\times\cP$ will be denoted
as $\Funop_\cC(\cP,\_)$.

Lurie in {\sl loc. cit.} requires $\cP$ to be $\cO$-monoidal. We were able to weaken the requirement, using
an observation proven in Lemma~\ref{lem:flatoperad0}. The
rest of the proof is just a variation of the proof of
\cite{L.HA}, 2.2.6.20.
\begin{proof}
The claim is independent of the choice of a strong approximation, so we will
assume $\cC=\cO$. 

We will present a category with decomposition $\cP'$
endowed with a marked functor $\pi:\cP'\to\cP$, such that the composition $p'=p\circ\pi:\cP'\to\cO$ is 
a cofunctor in the sense of ~\ref{sss:cofunctor}. 
Moreover, the composition $\pi_!\circ\pi^*$  will be equivalent to identity on $\Op_\cP$, so the functor
$\times\cP:\Op_\cO\to\Op_\cO$, being the composition
$p_!\circ p^*$, is equivalent to $p_!\circ\pi_!\circ\pi^*\circ p^*=p^{\prime *}\circ p'_!$.
 
Since  the factors $p^{\prime*}$ and $p'_!$ admit  right adjoints, this will prove the claim.

{\sl Definition of $\cP'$.}

We define $\cP'\in\Cat^+_{/\cP^\natural}$ as follows. 
Denote by $\Fun^{in}([1],\cO)$ the category of inert arrows
in $\cO$, with two maps, $s,t:\Fun^{in}([1],\cO)\to\cO$
assigning to an arrow its source and target.

As a category
over $\cP$,
$$\cP'=\Fun^{in}([1],\cO)\times_\cO\cP,$$
where we use the target map $t$ in the fiber product.
Two maps, $s,t:\cP'\to\cO$ are defined as compositions
with the projection $\cP'\to\Fun^{in}([1],\cO)$.
 
An arrow in $\cP'$ is marked if and only if its images in 
$\cP$ and in $\cO$ under $s$ are inert.

For any inert arrow $u:[1]\to\cO$ the base change
$\cP'_u=[1]\times_\cO\cP'\to[1]$ is a cocartesian fibration, see \cite{L.HA}, proof of Proposition 2.2.6.20, 
property (5). This defines a decomposition structure
on $\cP'$: for $x\in\cP'$ a decomposition diagram for $x$
is obtained by a cocartesian lifting of the decomposition
diagram for $s(x)\in\cO$.

{\sl An equivalence $\id\to\pi_!\circ\pi^*$.}

The diagonal embedding $\delta:\cO\to\Fun^{in}([1],\cO)$ 
induces $\delta:\cP\to\cP'$ which gives, for any $\cP$-operad 
$\cR$, a map
\begin{equation}
\label{eq:timesPprime}
\cR\to\cR\times_\cP\cP'
\end{equation}
in $\Cat^+_{/\cP^\natural}$. Thus, we have a morphism
of endofunctors $\id\to \pi_!\circ\pi^*$ on $\Cat^+_{/\cP^\natural}$. The map (\ref{eq:timesPprime}) is $\cP^\deco$-equivalence; that is, it induces an equivalence of endofunctors on $\Op_\cP$. In fact, 
for $\cQ\in\Op_\cP$ the marked maps 
$\cR\times_\cO\cP'\to\cQ$ are precisely the
right Kan extensions of the marked maps 
$\cR\to\cQ$. This proves the map  (\ref{eq:timesPprime})
induces an equivalence $\id\to\pi_!\circ\pi^*$ of endofunctors on $\Op_\cP$.

It remains to verify that the composition $p'=p\circ\pi:
\cP'\to\cP\to\cO$ is a cofunctor of categories with decomposition. We will verify that the
the conditions of Proposition~\ref{prp:cofunctor} hold.
Flatness of $p'$ is proven in Lemma~\ref{lem:flatoperad} below. Condition (3) holds by definition of decomposition
structure on $\cP'$. Conditions (1) and (2) of \ref{prp:cofunctor} are verified
 in~\cite{L.HA}, Proposition 2.2.6.20,
where they correspond to conditions (4), (7) and (5), (8) respectively.

\end{proof}

\begin{lem}
\label{lem:flatoperad}
Let $p:\cP\to\cO$ be an $\cO$-operad and let
$\cP'=\Fun^{in}([1],\cO)\times_\cO\cP$ be defined as above. 
If $\cP$ is a flat $\cO$-operad, then $s:\cP'\to\cO$ is
flat. 
\end{lem}
\begin{proof}We will use the flatness criterion \ref{prp:flat-properties}(3). Fix $\sigma:[2]\to\cO$ with the edges 
$\alpha,\beta$ and $\gamma=\beta\circ\alpha$. Let
$f':[1]\to\cP'$ be an arrow in $\cP'$ over
$\gamma$. Let $f:[1]\to\cP$ be the image of $f'$ in $\cP$.
Its image in $\cO$ is $\gamma'=t(f')$. These data define a commutative diagram in $\cO$ described below.
\begin{equation}\label{eq:2toFunin}
\xymatrix{
& & & u'\ar^{i''}@/^1pc/[rrrd]
\ar[dddrrr]|{\gamma'}\ar^{\alpha'}[ddrrr]\\
&u\ar_\gamma@/_1pc/[ddrr]\ar^{i'}@/^1pc/[urr]\ar^\alpha[dr]\ar^i[rrrrr]& & & & &u''\ar[d] \\
& &v\ar^\beta[dr]\ar^j[rrrr] & & & &v' \ar^{\beta'}[d] \\
& & &w\ar^k[rrr] &&&w'
}.
\end{equation}
The commutative square formed by the inerts $i',k$ and 
the arrows $\gamma,\gamma'$ is defined by
the image of $f'$ in $\Fun^{in}([1],\cO)$.

We decompose $k\circ\beta$ into an inert map 
$j$ followed by an active map $\beta'$. Then we decompose
$j\circ\alpha$ into an inert map $i$ followed by an active
map $u''\to v'$. 

We decompose $\gamma'$ into an inert $i''$ followed by an active map. Essential uniqueness of decomposition
of the map $u\to w'$ ensures that $u''$ can be choosen to  be a target of $i''$.
We finally denote by $\alpha'$ the composition $u'\to v'$.
We denote by $\sigma':[2]\to\cO$ the 2-simplex given by the 
commutative triangle with the edges 
$\alpha',\beta'\gamma'$.

We have to prove that the full subcategory 
$\cD\subset\Fun_{\cO}([2]_\sigma,\cP')$ spanned by the sections
$\tau:[2]\to\cP'$ satisfying $d_1\tau=f'$, is weakly contractible~\footnote{We denote by $[2]_\sigma$ the category over $\cO$ defined by $\sigma:[2]\to\cO$.}.

We will compare $\cD$ to the full subcategory
$\cD'\subset\Fun_\cO([2]_{\sigma'},\cP)$ spanned by the sections $\tau':[2]\to\cP$ satisfying $d_1\tau'=f$.
The category $\cD'$  is weakly contractible as $\cP$ is a flat operad (and $d_0\sigma'=\beta'$ is active).

The commutative diagram (\ref{eq:2toFunin}) yields
a map $\Sigma:[2]\to\Fun^{in}([1],\cO)$ 
such that $s\circ\Sigma=\sigma,\ t\circ\Sigma=\sigma'$. 
This defines a functor 
$\epsilon:\cD'\to\cD$ carrying $\tau'$ to the pair 
$(\Sigma,\tau')$. We will now show that $\epsilon$ has left 
adjoint; this will imply that $\cD$ is also weakly 
contractible.

Let $\tau:[2]\to\cP'$  be in $\cD$ and
$\tau':[2]\to\cP$ be in $\cD'$. One has 
$\tau=(\Sigma'',\tau'')$ where 
$\Sigma'':[2]\to\Fun^{in}([1],\cO)$ 
satisfies $s\circ\Sigma''=\sigma,\ t\circ\Sigma''=\sigma'',
d_1\sigma''=\gamma'$, and $\tau'':[2]\to\cP$ with $p\circ
\tau''=\sigma''$. The first and the last vertex of 
$\sigma''$ are $u'$ and $w'$; the remaining vertex will be 
denoted by $v''$.

Let $\cF$ be the full subcategory of 
$\Fun([2],\Fun^{in}([1],\cO))$ spanned by the functors
$\Phi$ satisfying the conditions $s\circ\Phi=\sigma,
d_1(t\circ\Phi)=\gamma'$. It is easy to see that 
$\Sigma\in\cF$ is a terminal object, so that we have a 
unique morphism $\Sigma''\to\Sigma$ in $\cF$; moreover, the 
corresponding map $v''\to v'$ is inert.

A map $\tau\to\epsilon(\tau')$ consists, by definition, of a map of triangles $\Sigma''\to \Sigma$ in $\cF$ and a compatible morphism $\tau''\to\tau'$ of triangles in $\cP$
over the unique map $\sigma''\to\sigma'$.

In Lemma~\ref{Lem:coc} we  verify that the map
$\theta:\sigma''\to\sigma'$ in $\Fun([2],\cO)$ admits a 
locally  cocartesian lifting to $\Fun([2],\cP)$. This means 
that a map $\tau''\to\tau'$ over $\theta$ can be
rewritten as
a map $\theta_!(\tau'')\to\tau'$ over $\sigma'$.

Therefore, the assignment 
$\tau=(\Sigma'',\tau'')\mapsto \theta_!(\tau')$ defines 
a functor left adjoint to $\epsilon$.
\end{proof}

\begin{lem}
\label{Lem:coc}
Let $p:\cP\to\cO$ be a functor, $\sigma:[n]\to\cO$
be a simplex 
$$ x_0\stackrel{f_1}{\to}x_1\stackrel{f_2}{\to}\ldots
\stackrel{f_n}{\to}x_n.$$
Assume that for a certain $k$ $f_k:[1]\to\cO$ admits a cocartesian lifting
to $\cP$. The simplex $\sigma$ determines
an arrow $\tilde\sigma:d_k\sigma\to d_{k-1}\sigma$ in $\Fun([n-1],\cO)$,
having the component $f_k:x_{k-1}\to x_k$ at place $k-1$ and identity elsewhere.
Then $\tilde\sigma$ admits a locally cocartesian lifting to $\Fun([n-1],\cP)$.
\end{lem}
\begin{proof}We have to verify that the base change map
$$\cQ:=[1]\times_{\Fun([n-1],\cO)}\Fun([n-1],\cP)\to[1]$$
is a cocartesian fibration. The category $\cQ$ over $[1]$
defines a presheaf on $\Delta_{/[1]}$ 
carrying $\tau:[k]\to[1]$ to the space of sections
$\cQ_\tau=\Map_{[1]}([k],\cQ)$ which is the fiber of the map
$$
\Fun([k],\Fun([n-1],\cP))\to\Fun([k],\Fun([n-1],\cO))
$$
at $\bar\sigma\circ\tau$.

This space is identified
with the space  
$$ \Map_{\cO}([k]\times[n-1],\cP),$$
where the map $[k]\times[n-1]\to\cO$ is defined by 
$\tilde\sigma\circ\tau$.
In the proof below we denote $k$-simplices $\tau$ in $[1]$ by monotone sequences of $0$ and $1$ of length $k+1$.
Let $x\in\cQ_0, y\in\cQ_1$ be the source and the target of
$f\in\cQ_{01}$. The arrow $f$ in $\cQ_{01}$ is cocartesian iff
the composition
$$
\{y\}\times_{\cQ_1}\cQ_{11}\stackrel{f}{\to}
\{x\}\times_{\cQ_0}\cQ_{01}\times_{\cQ_1}\cQ_{11}
\stackrel{\sim}{\leftarrow}
\{x\}\times_{\cQ_0}\cQ_{011}\stackrel{d_1}{\to}
\{x\}\times_{\cQ_0}\cQ_{01}
$$
is an equivalence (see Proposition~\ref{sss:coc} for a more general statement). Applying this criterion to the description of $\cQ_\tau$, we deduce that an arrow $f\in\cQ_{01}$ defined by the diagram in $\cP$ 
\begin{equation}
\xymatrix{
&y_0\ar^{g_0}[d]\ar[r]&\ldots\ar[r]&y_{k-1}\ar^{g_{k-1}}[d]\ar[r]&y_{k+1}\ar^{g_k}[d]\ar[r]
&\ldots\ar[r]&y_n\ar^{g_{n-1}}[d]\\
&y'_0\ar[r]&\ldots\ar[r]&y'_{k-1}\ar[r]&y'_{k+1}\ar[r]
&\ldots\ar[r]&y'_n
}
\end{equation}
is cocartesian if and only if the arrow $g_{k-1}$ is 
a cocartesian lifting of $f_k:x_{k-1}\to x_k$ and the rest of $g_i$ are
equivalences.
\end{proof}

\

We will now deduce some easy consequences of the existence 
of $\Funop_\cC(\cP,\cQ)$.

\begin{crl}For $\cP,\cQ,\cR\in\Op_{\cC}$ with $\cP$ flat,
one has a natural equivalence
\begin{equation}\label{eq:cat-internal}
 \Alg_{\cR/\cC}(\Funop_\cC(\cP,\cQ))=\Alg_{\cR\times_{\cC}\cP/\cC}(\cQ).
\end{equation}
\end{crl}
\begin{proof}Immediately follows from the definition of
$\Alg_{\cP/\cQ}$, see \ref{sss:alg}.
\end{proof}

For $\cO\in\Op$ we denote by $\cO_1$ the fiber of the projection $\cO\to Fin_*$ at $\langle 1\rangle$.
This notation extends to the objects of $\Op(\cC)$:
$\cP_1$ is the fiber of the composition $\cP\to\cC\to\cO\to
\Fin_*$ at $\langle 1\rangle$. In the case where $\cP,\cQ$
are $\cC$-operads and $\cP\to\cC$ factors through
$\cC^\circ$, one has
$\Alg_{\cP/\cC}(\cQ)=\Fun_{\cC_1}(\cP_1,\cQ_1)$.

Applying formula (\ref{eq:cat-internal}) to $\cR=\cC^\circ$, we get 
 
\begin{crl}
\label{crl:alg-c0c}
Assume $\cC_1$ is a space.
For $\cP,\cQ\in\Op_{\cC}$ with $\cP$ flat, one has
$$ \Alg_{\cC^\circ/\cC}(\Funop_\cC(\cP,\cQ))=
\Alg_{\cC^\circ\times_\cC\cP/\cC}(\cQ)=\Fun_{\cC_1}(\cP_1,\cQ_1).
$$

\end{crl}
 
\subsubsection{Base change}
\label{sss:basechange}

Let $\cC\to\cO$ be a strong approximation, 
$q:\cD\to\cC$ be in $\Op(\cC)$. The category $\cD$ inherits from $\cC$ a 
decomposition structure $(\cD,\cD^\circ,D)$, with 
$\cD^\circ$ being determined by inerts in $\cD$ and $D$ by 
the locally cocartesian lifting of the standard inerts 
$\rho^i:\langle n\rangle\to\langle 1\rangle$.

The base change functor 
$$ \Cat_{/\cC}\to\Cat_{/\cD}$$
restricts to a functor $q^*:\Op_\cC\to\Op_\cD$
having a left adjoint $q_!:\Op_\cD\to\Op_\cC$.
This is a special case of the functor described 
in~\ref{sss:cd-base-change}.

One has
\begin{Lem}
For $\cP,\cQ\in\Op_\cC$ with $\cP$ flat, one has an equivalence
\begin{equation}
q^*\Funop_\cC(\cP,\cQ)\to \Funop_\cD(q^*\cP,q^*\cQ).
\end{equation}
\end{Lem}
\begin{proof}

The adjoint pair  
\begin{equation}
q_!:\Cat^+_{/\cD^\natural}\rlarrows:\Cat^+_{/\cC^\natural}:q^*
\end{equation}
satisfies the projection formula: 
for $X\in\Cat^+_{/\cD^\natural}$ and for 
$\cP\in\Op_\cC\subset\Cat^+_{/\cC^\natural}$,  
one has a 
natural equivalence
\begin{equation}
\label{eq:proj}
q_!(X\times q^*(\cP))\to q_!(X)\times\cP,
\end{equation}
with the first projection induced by 
$X\times q^*(\cP)\to X$
and the second projection by $X\times q^*(\cP)\to q^*(\cP)$ and
by the counit of the adjunction $q_!\circ q^*\to\id$.

Since $\cP$ is flat, the functor $\times\cP$ preserves
$\cC^\deco$-equivalences, so the equivalence (\ref{eq:proj})
induces the projection formula for the localized adjunction
\begin{equation}
q_!:\Fib(\cD^\deco)\rlarrows:\Fib(\cC^\deco):q^*.
\end{equation}
Now our claim immediately follows from this projection formula.
\begin{eqnarray}
\nonumber
\Map(X,q^*\Funop_\cC(\cP,\cQ))=\Map(q_!X,\Funop_\cC(\cP,\cQ))=\Map(q_!X\times\cP,\cQ)\\
=\Map(q_!(X\times q^*\cP),\cQ)=\Map(X\times q^*\cP,q^*\cQ)\\
\nonumber=\Map(X,\Funop_\cD(q^*\cP,q^*\cQ).
\end{eqnarray}

\end{proof}

\subsubsection{Alternative description of $\Funop(\cP,\cQ)$}
\label{sss:funop} 
Let $\cC\to\cO$ be a strong approximation and let
$\cP,\cQ$ be $\cC$-monoidal categories.  

We present below a very explicit description of the 
$\cC$-operad $\Funop_\cC(\cP,\cQ)$.
Let $p:\Ar\to\Cat\times\Cat$ be the bicartesian fibration corresponding,
via Grothendieck construction, to the functor
$$
\Cat^\op\times\Cat\to\Cat,\ (\cP,\cQ)\mapsto\Fun(\cP,\cQ).
$$
The category $\Ar$ can be realized as the
full subcategory of $\Cat_{/[1]}$ spanned by the 
cartesian fibrations over $[1]$. 
The objects of $\Ar$ are defined by the arrows $A\to B$ in $\Cat$; an 
arrow in $\Ar$ from $f:A\to B$ to $f':A'\to B'$ is given by a 2-diagram
$$
\xymatrix{
&A\ar[d]\ar^f[r] &B\ar[d]\ar@{=>}[dl]\\
&A'\ar^{f'}[r] &B'
}.
$$ 
The category $\Ar$ is closed under products in 
$\Cat_{/[1]}$, so the embedding $\Ar\to\Cat_{/[1]}$ 
preserves the products. The restriction to the ends of 
$[1]$ defines a functor 
$p=(p_1,p_0):\Ar\to\Cat\times\Cat$.
We  denote by $\Ar^\times$ the corresponding SM category 
(presented as a cocartesian fibration over 
$\Com=\Fin_*$). The functor $p$ preserves products, so
it induces a functor 
$$
p^\times:\Ar^\times\to(\Cat\times\Cat)^\times=\Cat^\times\times_\Com
\Cat^\times.
$$
The functor $p^\times$ is obviously fibrous. 
We define an operad $\cF_\cC^{\cP,\cQ}$ as the fibre product
\begin{equation}
\label{eq:F}
\xymatrix{
&\cF_\cC^{\cP,\cQ}\ar[d]\ar[rr]&&\Ar^\times \ar^{p^\times}
[d]\\
&\cC\ar^{(\cP,\cQ)}[rr]&&\Cat^\times\times_\Com\Cat^\times,
}
\end{equation}
where the lower horizontal arrow is defined by the pair of 
$\cC$-monoidal categories $\cP$ and $\cQ$.

We claim
\begin{Prp}
The operad $\cF_\cC^{\cP,\cQ}$ defined by (\ref{eq:F}) is naturally equivalent to $\Funop_\cC(\cP,\cQ)$.
\end{Prp}

We delay the proof of this result till the end of 
Section~\ref{sec:yoneda}
when we develop means to compare $\cF_\cC^{\cP,\cQ}$ to 
$\Funop_\cC(\cP,\cQ)$. The proof is given 
in~\ref{sss:proof-funop}.

\subsubsection{}
\label{sss:Dcart}
We now assume $\cC=\Com$. We have an even nicer description of 
$\Funop(\cP,\cQ)$
if $\cP$ has a cocartesian SM structure.
Define an operad $\cF^{\cP,-}$ by the cartesian diagram
\begin{equation}
\label{eq:F_D}
\xymatrix{
&\cF_\Com^{\cP,\cQ}\ar[d]\ar[rr]&&\cF^{\cP,-}\ar^r[d]\ar[rr]&&\Ar^\times \ar^{p^\times}[d]\\
&\Com\ar^{\{\cQ\}}[rr]&&\Cat^\times\ar^{\{\cP\}\times\id}[rr]&&\Cat^\times\times_\Com\Cat^\times,
}
\end{equation}

We have the following.
\begin{Prp}
\begin{itemize}
\item[1.] Let $\cP$ have a cocartesian SM structure. Then the the 
morphism $r:\cF^{\cP,-}\to\Cat^\times$ is a cocartesian fibration.
In particular, $\cF^{\cP,\cQ}$ is a SM category for any $\cQ$. 
\item[2.] If, in addition, $\cQ$ is cartesian, $\cF^{\cP,\cQ}$ is also cartesian.
\end{itemize}
\end{Prp}
\begin{proof}
1. Given $f=\{f_i:\cP\to\cC_i\}, i\in I$, and a functor $a:\prod \cC_i\to\cD$, a locally cocartesian lifting of $a$ is given by
a universal 2-diagram
\begin{equation}
\xymatrix{
&{\prod_i\cP}\ar[r]\ar^\sqcup[d] &{\prod_i\cC_i}\ar^a[d]\ar_\eta@{=>}[dl]\\
&\cP\ar^g[r]&\cD.
}
\end{equation}
One can easily see that the choice
\begin{itemize}
\item $g=a\circ f\circ\delta$, $\delta:\cP\to\prod_{i\in I}\cP$ being the diagonal.
\item $\eta$ is induced by the morphism $h:\id\to\delta\circ\sqcup$
having the components $x_i\to\sqcup_{i\in I}x_i$ 
\end{itemize}
yields a locally cocartesian lifting of $a$.

 The locally cocartesian arrows defined as above are closed under 
composition, so the map $r:\cF^{\cP,-}\to\Cat^\times$ is a cocartesian fibration.

2. Let now  $\cQ$ be cartesian. The unit of 
$\cF^{\cP,\cQ}$ is obviously the terminal object $t:\cP\to\cQ$, the one that factors through the terminal object 
$[0]\to\cQ$. It remains to verify that for any pair of objects $f_i:\cP\to\cQ, i=0,1$, the maps
$$ f_0\otimes f_1\to f_0\otimes t=f_0
\textrm{ and }
f_0\otimes f_1\to t\otimes f_1=f_1$$
form a product diagram.
This is an easy straightforward check.
\end{proof}

\subsection{Examples}

We will now present some operads and their approximations
appearing in this paper. Two of them,
the operads governing associative algebras and left 
modules, are described in \cite{L.HA}. We will also need
 an approximation for the operad governing bimodules. 
The approximations for associative algebras and left
(right) modules can be realized as its full subcategories. 

\subsubsection{}
\label{sss:delta-over}

For $\cC\in\Cat$ we denote as $\Delta_{/\cC}$ the fiber product
\begin{equation}
\Delta\times_{\Cat}\Cat^{[1]}\times_{\Cat}\{\cC\},
\end{equation} 
that is, the category, whose objects are functors 
$a:[n]\to\cC$, with a morphism from $a$ to $b:[m]\to\cC$
given by a commutative triangle
$[n]\to [m]\to\cC$.

\subsubsection{} 
\label{sss:bm}
We define $\BM=(\Delta_{/[1]})^\op$. 
The objects of $\BM$ are functors $\sigma:[n]\to[1]$, that 
is length $n+1$ monotone sequences of $0$ and $1$.

The operad for bimodules
(which we denote as $\BM^\otimes$ to distinguish from $\BM$)
governs triples $(A,M,B)$ where $A$ and $B$ are associative algebras and $M$ has compatible left $A$-module and a right $B$-module structures.
An object of $\BM^\otimes$ over $I_*\in\Fin_*$ is a
map $f:I\to\{a,m,b\}$. An arrow from $f:I\to\{a,m,b\}$
to $g:J\to\{a,m,b\}$ over $\alpha:I_*\to J_*$ is a
collection of total orders at each $\alpha^{-1}(j)$, $j\in J$, such that 
\begin{itemize}
\item if $g(j)=a$ then $f(i)=a$ for all 
$i\in\alpha^{-1}(j)$.
\item if $g(j)=b$ then $f(i)=b$ for all 
$i\in\alpha^{-1}(j)$.
\item if $g(j)=m$ then there exists a unique $i\in\alpha^{-1}(j)$ with $f(i)=m$; moreover, $f(i')=a$ for $i'<i$ and $b$ for $i'>i$.  
\end{itemize}

The map $\iota:\BM\to\BM^\otimes$ converts a sequence
$\sigma$ of $0$ and $1$ of length $n+1$ into a sequence 
of $a,m,b$ of length $n$, each encoding a pair of 
consecutive numbers, $a$ for $00$, $m$ for $01$ and $b$ for 
$11$. This sequence of letters defines a  map
$I:=\{1,\ldots,n\}\to\{a,m,b\}$, that is an object of $\BM^\otimes$. Given an arrow $\sigma\to\tau$ in $\BM$, with 
$\iota(\sigma)=(I,f)$ and $\iota(\tau)=(J,g)$, the arrow
$(I,f)\to (J,g)$ in $\BM^\otimes$ is defined by the arrow
$I\to J$ and the orders on the preimages of $j\in J$, 
induced from the total order on $I$.

 Degeneracies in $\BM$ correspond to inserting a 
letter $a$ or $b$, inner faces correspond to 
``multiplications'' $aa\to a$, $am\to m$, $mb\to m$, 
$bb\to b$, and outer faces erase the leftmost or the 
rightmost letter. 
\begin{Lem}
The map $\iota:\BM\to\BM^\otimes$ is a strong approximation.
\end{Lem}
\begin{proof}
We will verify the properties of Definition~\ref{dfn:strongapprx}. The property (3) is obvious:
one has $\BM_1=\{a,m,b\}=\BM^\otimes_1$. To verify (1),
let $\sigma:[n]\to[1]$ represent an object $C\in\BM$ and let $\rho^i:\langle n\rangle\to\langle 1\rangle$ be the standard inert. The locally cocartesian lifting of $\rho^i$
is given by the inert map $[1]\to[n]$ carrying $\{0\}$ to $\{i-1\}$ and $\{1\}$ to $\{i\}$. Finally, to verify (2),
let $a:(J,g)\to\iota(C)$ be an active arrow in 
$\BM^\otimes$. If $\iota(C)=(I,f)$, $I$ acquires a total ordering. The active arrow $a$ defines total orderings on
$J$. Overall, this allows one to uniquely define a
total ordering on $J$ compatible with the ordering on the fibers of $a$, such that $a$ is monotone. This is equivalent to lifting of $a$ to $\BM$. This lifting is obviously cartesian. 
 
\end{proof}

Outer faces are inerts in $\BM$, and inner faces (as well as the degeneracies) are active.

\subsubsection{}
Strong approximations to the operads $\Ass^\otimes,
\LM^\otimes,\RM^\otimes$ can be found among
full subcategories of $\BM$. Here they are.
 
\begin{itemize}
\item The subcategory $\Ass$ is spanned by constant maps
$[n]\to[1]$ with value $0$. It is isomorphic to 
$\Delta^\op$; the object of $\Ass$ corresponding to 
$[n]^\op\in\Delta^\op$ will be denoted as $\langle n\rangle$. The map from $\Ass$ to the operad for associative algebras
$\Ass^\otimes$ is compatible with the ``left'' embedding $\Ass^\otimes\to\BM^\otimes$. It coincides 
with the map $\Cut:\Delta^\op\to\Ass^\otimes$ defined by Lurie in \cite{L.HA}, 4.1.2.9. 
\item The category $\LM$ is the full subcategory of $\BM$ 
spanned by $\sigma:[n]\to[1]$ satisfying $\sigma(0)=0$ and having at most one value equal to $1$. This category is isomorphic to $\Delta^\op\times[1]$ and the map
$\LM\to\LM^\otimes$ coincides with 
the map $\gamma:\Delta^\op\times[1]\to\LM^\otimes$ defined by Lurie in \cite{L.HA}, 4.2.2.
The isomorphism $\LM\to\Delta^\op\times[1]$ carries
$a^n$ to $([n],1)$ and $a^nm$ to $([n],0)$. The arrow
$([n],0)\to([n],1)$ corresponds to the inert $a^nm\to a^n$
forgetting $m$.
\item Similarly, $\RM$ is the full subcategory spanned
by  $\sigma:[n]\to[1]$ satisfying $\sigma(n)=1$ and
having at most one value $0$.
\end{itemize}
\subsubsection{}
\label{sss:opmaps}
One has embeddings 
\begin{equation}
\Ass\to\LM\to\BM\leftarrow\RM\leftarrow\Ass.
\end{equation}
We denote the two copies of $\Ass$ inside $\BM$  as
$\Ass_-$ (the left copy) and $\Ass_+$ (the right copy).

\subsubsection{}
The category $\Op_\Ass$ is called the category of planar operads. We do not have special names for the categories
$\Op_\BM$, $\Op_\LM$, etc.
 
\begin{Dfn}
For a $\BM$-operad $\cO$ the planar operads 
$\Ass_\pm\times_\BM\cO$ are called the 
$\Ass_\pm$-components of $\cO$, denoted as $\cO_a$ and $\cO_b$, and the
category $\{m\}\times_\BM\cO$ is called the 
$m$-component of $\cO$.
\end{Dfn}

One has a projection $\pi:\BM\to\Ass$ induced by 
$[1]\to[0]$.
This defines a base change functor $\pi^*:\Op_{\Ass}\to\Op_{\BM}$.

\subsubsection{}
\label{sss:Qn}
Here is one more approximation. Let $C_n$ 
be the free planar operad generated by one $n$-ary operation. It has $n+1$ colors $0,1,\ldots,n$, and one
$n$-ary operation with the inputs of colors $1,\ldots,n$,
and the output of color $0$.
It has the strong approximation $Q_n\to C_n$ defined
as follows. $Q_n$ has one object $\{1,\ldots,n\}$ over $\langle n\rangle$, and the objects $0,\ldots,n$ over $\langle 1\rangle$. One has inert arrows $\{1,\ldots,n\}\to i$ for $i=1,\ldots,n$, and an active arrow $\{1,\ldots,n\}\to 0$.

The properties (1)--(3) of Definition~\ref{dfn:strongapprx}
are verified immediately.

Note that $Q_n$ is also a strong approximation of
the free (non-planar) operad $\bC_n\in\Op$ generated by one
$n$-ary operation.

\begin{rem}
\label{sss:rem}
The projection $\pi$, as well as the maps
(\ref{sss:opmaps}), are not fibrous, so $\BM$ is not a planar operad and, for instance, $\Ass$ is not a $\BM$-operad. 
\end{rem}


\subsection{Bilinear maps of operads. Tensor product}
\label{ss:tensoroperad}

In this subsection we present a slightly generalized
version of Lurie's \cite{L.HA}, 3.2.4. In this form it also 
includes a presentation of $\BM$ as the tensor 
product of $\LM$ and $\RM$, as well as a presentation
 of the operad $\cC^\sqcup$, defined
by a category $\cC$ as in~\cite{L.HA}, 2.4.3,
as a tensor product of $\Com$ with $\cC$.

Let $\cP^\deco,\cQ^\deco,\cR^\deco$ be categories with decomposition (in practice they are usually approximations of operads). 
A bilinear map $\mu:\cP^\deco\times\cQ^\deco
\to\cR^\deco$ is a map of categories with decomposition , where the decomposition structure of the product is defined
as in \ref{sss:decoproduct}.

Given a bilinear map $\mu:\cP^\deco\times\cQ^\deco\to\cR^\deco$ and  $\cX\in\Fib(\cR^\deco)$, one defines a category
$p:\Alg^\mu_{\cQ/\cR}(\cX)\to\cP$ over $\cP$ by the formula
\begin{equation}\label{eq:Algmu}
\Map_{\Cat_{/\cP}}(K,\Alg^\mu_{\cQ/\cR}(\cX))=
\Map_{\Cat^+_{/\cR^\natural}}(K^\flat\times\cQ^\natural,\cX^\natural),
\end{equation}
where $\cQ^\natural,\cR^\natural$ and $\cX^\natural$ denote the corresponding marked categories. One has
\begin{prp}
\label{prp:alg-mu}(Compare to \cite{L.HA}, 3.2.4.3).
\begin{itemize}
\item[1.] The formula (\ref{eq:Algmu}) defines a category
$\Alg^\mu_{\cQ/\cR}(\cX)$ over $\cP$.
\item[2.] $\Alg^\mu_{\cQ/\cR}(\cX)$ is  $\cP$-fibrous.
\item[3.] The equivalence (\ref{eq:Algmu}) extends to an
equivalence 
\begin{equation}\label{eq:Algmu-2}
\Map_{\Cat^+_{/\cP}}(K^\natural,\Alg^\mu_{\cQ/\cR}(\cX))=
\Map_{\Cat^+_{/\cR^\natural}}(K^\natural\times\cQ^\natural,\cX^\natural).
\end{equation}
\item[4.] If $\pi:\cX\to\cR$ is a 
cocartesian fibration, $p$ is also a cocartesian fibration.
\end{itemize}
\end{prp}
\begin{proof}
1. The right-hand side of (\ref{eq:Algmu}) preserves
limits as a functor of $K\in\Cat_{/\cP}^\op$. Since
$\Cat_{/\cP}$ is presentable,  
$\Alg^\mu$ is correctly defined, see~\cite{L.T}, 5.5.2.2.

2.  We apply \ref{prp:mu-functorial} to $f:\cP^{\flat,\emptyset}\to\cP^\deco$ and $g=\id_{\cQ^\deco}$.The commutative diagram (\ref{eq:mu-functorial}) restricted to
$\cQ\in\Fib(\cQ^\deco)$ yields a commutative diagram
of colimit preserving functors
\begin{equation}
\xymatrix{
&{\Fib(\cP^{\flat,\emptyset})}
\ar^{m_\flat\quad}[r]\ar^{f_!}[d]
&{\Fib(\cP^{\flat,\emptyset}\times\cQ^\deco)}
\ar^{(f\times\id)_!}[d]
&{} \\
&{\Fib(\cP^\deco)}
\ar^{m\quad}[r]
&{\Fib(\cP^\deco\times\cQ^\deco)}
\ar^{\mu_!}[r]
&{\Fib(\cR^\deco),}
}
\end{equation}
$m$ and $m_\flat$ being evaluations of the product maps
(\ref{eq:deco-product}) at $\cQ\in\Fib(\cQ^\deco)$.
All categories are presentable, so this defines a commutative diagram of the corresponding right adjoints,
which we define $\mu^*,\ m^*$ and $m_\flat^*$ respectively.
By definition
$$\Alg^\mu_{\cQ/\cR}(\cX)=m_\flat^*\circ(f\times\id)^*\circ\mu^*(\cX).$$
Therefore,
$$\Alg^\mu_{\cQ/\cR}(\cX)=f^*\circ m^*\circ\mu^*(\cX),$$
which means that this category over $\cP$ is 
$\cP$-fibrous.

3. Note that we simultaneously proved that the functor
$\cX\mapsto\Alg^\mu_{\cQ/\cR}(\cX)$ is right adjoint to
the composition $\mu_!\circ m$. This yields the 
equivalence~(\ref{eq:Algmu-2}).

4. We can replace the decomposition structures on
$\cP$ and on $\cR$, retaining the decomposition diagrams
and assuming $\cP^\circ=\cP$, $\cR^\circ=\cR$. The functor $\Alg^\mu_{\cQ/\cR}$ does not change after this replacement. This implies the required claim.
\end{proof}

\subsubsection{}
\label{sss:bifun}
Given $\mu:\cP^\deco\times\cQ^\deco\to\cR^\deco$ as above
and $\cX\in\Fib(\cR^\deco)$, we define
$\BiFun_\cR(\cP,\cQ;\cX)$ as the full subcategory of
$\Fun_{\cR}(\cP\times\cQ,\cX)$ spanned by the maps preserving the marked arrows. According to ~\ref{prp:alg-mu} (3),
$\BiFun_\cR(\cP,\cQ;\cX)=\Alg_\cP(\Alg^\mu_{\cQ/\cR}(\cX))$.
\subsubsection{Examples}

\begin{itemize}
\item[1.] If $\cP,\cQ,\cR$ are operads and $\mu$ is a 
bilinear map of operads, our definition of $\Alg_{\cQ/\cR}$ is just Lurie's \cite{L.HA}, 3.2.4.

\item[2.] For $\cP=[0]$ a bilinear map $\mu$ is just a map 
$\cQ\to\cR$ of decomposition categories. In the case when
$\cQ$ and $\cR$ are strong approximations of operads,
we recover the definition of $\Alg_{\cQ/\cR}(\cX)$
given in \ref{sss:alg}.
\end{itemize}

\subsubsection{Tensor product}
Let  
\begin{equation}
\mu:\cC^\deco\times\cC^{\prime\deco}\to\cC^{\prime\prime\deco}
\end{equation}
be a bilinear map of  categories with decomposition.
We define a $\mu$-tensor product 
$$\Fib(\cC^\deco)\times\Fib(\cC^{\prime\deco})\to
\Fib(\cC^{\prime\prime\deco})$$
as follows. Given $\cP\in\Fib(\cC^\deco)$, $\cQ\in\Fib(\cC^{\prime\deco})$,
define
$\cP\otimes^\mu\cQ\in\Cat^+_{/\cC''}$ as the object representing the functor
\begin{equation}
\cR\mapsto \BiFun_{\cC''}(\cP,\cQ;\cR),
\end{equation}
where the decomposition structure on $\cP$ and $\cQ$
is induced from $\cC$ and $\cC'$ and the bilinear map
$\cP\times\cQ\to\cC''$ is defined 
as the composition 
$\cP\times\cQ\to\cC\times\cC'\stackrel{\mu}{\to}\cC''$. 

The tensor product can be calculated as the image of
$\cP\times\cQ$ under the localization functor $\Cat^+_{/\cC''}\to
\Fib(\cC^{\prime\prime\deco})$.

\subsubsection{}
\label{sss:tensorproducts}
We will mention a few instances of the 
tensor product defined above.
\begin{itemize}
\item[0.] $\cC=\cC'=\cC''=[0]$. In this case our definition gives a product in $\Cat$.
\item[1.] $\cC=\cC'=\cC''=\Fin_*$.
We endow $\Fin_*$ with the smash product $\mu(I_*,J_*)=
(I\times J)_*$. The corresponding tensor product is the standard tensor product of operads defined in \cite{L.HA},
3.2.4.
\item[2.] $\cC=\cC'=\cC''=\CM$ where $\cM$
is the operad governing pairs $(A,M)$
where $A$ is a commutative monoid and $M$ is an $A$-module.
Its objects over $I_*\in\Fin_*$ are subsets $I_0\subset I$
and a morphism from $(I,I_0)$ to $(J,J_0)$
is a map $\phi:I_*\to J_*$ such that $\phi^{-1}(j)\cap I_0$
is empty if $j\not\in J_0$ and is a singleton otherwise.
The functor $\mu:\CM\times\CM\to\CM$ is defined by the
formula $\mu((I,I_0),(J,J_0))=(K,K_0)$ where
\begin{equation}
K=(I\times J_0)\coprod^{I_0\times J_0}(I_0\times J),\ 
K_0=I_0\times J_0.
\end{equation}
\item[3.] We choose as the bilinear map $\mu$ the identity map $[0]\times\cC\to\cC$ where $\cC$ is a strong approximation of an operad. 
 
One has $\Fib([0])=\Cat$, so this type of tensor product
assigns to a pair $(K,\cP)\in\Cat\times\Op_\cC$ a
$\cC$-operad which we will denote by $\cP_K$.  
The $\cC$-operad $\cP_K$ is independent of $\cC$ in the following sense.  
If $f:\cC\to\cC'$ is a morphism of strong 
approximations as described by the diagram 
(\ref{eq:twoapprx}), the (derived) functor $f_!$, see
\ref{sss:cd-base-change}, carries the $\cC$-operad 
$\cC_K$ to the $\cC'$-operad with the same name.
\end{itemize}

Here is the main property of the tensor product defined above.

\begin{prp}
\label{prp:AlgAlg}
Let  
$\mu:\cC^\deco\times\cC^{\prime\deco}\to\cC^{\prime\prime\deco}
$ be a bilinear map, and let $\cP\in\Fib(\cC^\deco)$, 
$\cQ\in\Fib(\cC^{\prime\deco})$ with 
$\cR=\cP\otimes^\mu\cQ\in\Fib(\cC^{\prime\prime\deco})$. Then for any $\cX\in\Fib(\cR^\deco)$
one has a natural equivalence
\begin{equation}
\label{eq:prp-tensorproduct}
\Alg_\cP(\Alg^\mu_{\cQ/\cR}(\cX))=\Alg_\cR(\cX).
\end{equation}
Conversely, a bilinear map $\cP\times\cQ\to\cR$
lifting $\mu$ and inducing an equivalence
\ref{eq:prp-tensorproduct}, induces an equivalence
$\cP\otimes^\mu\cQ\to\cR$.
\end{prp}
\begin{proof}
Immediately follows from \ref{sss:bifun}.
\end{proof}

\subsubsection{}
The operads $\LM^\otimes$ and $\RM^\otimes$ are operads 
over $\CM$. The bilinear structure on $\CM$ described in
\ref{sss:tensorproducts} (2) lifts to a map 
$\bPr:\LM^\otimes\times\RM^\otimes\to\BM^\otimes$,
see~\cite{L.HA}, 4.3.2.1.
 
Theorem 4.3.2.7 of \cite{L.HA} asserts that the map $\bPr$
induces the equivalence
\begin{equation}\label{eq:BM=RMLM}
\Alg_{\RM}(\Alg^\mu_{\LM/\BM}(\cX))=\Alg_\BM(\cX).
\end{equation}

In other words, this means that $\bPr$ induces an equivalence
\begin{equation}
\LM^\otimes\otimes^\mu\RM^\otimes=\BM^\otimes.
\end{equation}

Another connection between bimodules 
and left modules is described in \ref{ss:folding}.

\subsubsection{} 
\label{sss:mon-family}
We are back to Example 3 from \ref{sss:tensorproducts}. 
If $\cX$ is a $\cP_K$-operad, $\Alg_{\cP/\cP_K}(\cX)$
is a category over $K$. In case $\cX=\pi^*(\cY)$
for a $\cP$-operad $\cY$ and $\pi:\cP_K\to\cP$ the 
natural projection, $\Alg_{\cP/\cP_K}(\cX)=K\times
\Alg_\cP(\cY)$, so Proposition~\ref{prp:AlgAlg}
yields $\Alg_{\cP_K}(\cY)=\Fun(K,\Alg_\cP(\cY))$.

Let now $\cX$ be a $\cP_K$-monoidal category. 
Then, by 
Proposition~\ref{prp:alg-mu}, (4), the map
$p:\Alg_{\cP/\cP_K}(\cX)\to K$ is a cocartesian fibration.  

Note that a $\cP_K$-monoidal category is given by
a $\cP_K$-algebra in $\Cat$, that is, by a functor
$\chi:K\to\Mon_\cP$ to $\cP$-monoidal categories.
The composition
$K\stackrel{\chi}{\to}\Mon_\cP\stackrel{\Alg_\cP}{\to}\Cat$ classifies the cocartesian fibration
$\Alg_{\cP/\cP_K}(\cX)\to K$, see \ref{sss:fam-op-mon}
and~\ref{rem:PK-mon-family} for a more general claim.

\subsection{Operad families}

\subsubsection{Definition}
\label{sss:opfam}
Let $\cC\to\cO$ be a strong approximation. Let $K$ be a
category. We define on $K\times\cC$ the structure of a
category with decomposition \ref{dfn:catdec} as follows.
\begin{itemize}
\item
The inerts of $K\times\cC$ are defined as
$ (K\times\cC)^\circ=K^\eq\times\cC^\circ.$
\item
Decomposition diagrams are defined by pairs $(x,d)\in K\times D$, $D$ being the set of decomposition diagrams
$d=\{\rho^{d,i}:C^d\to C^d_i\}$ in $\cC$, as the 
collection of maps
$$ (\id_x,\rho^{d,i}):(x,C^d)\to (x,C^d_i).$$
\end{itemize}

\begin{Dfn} A $\cC$-operad family indexed by $K$
is an object of $\Fib(K\times\cC)^\deco$. 
\end{Dfn}

\begin{Rems}
\begin{itemize}
\item[1.] In the case $\cC=\Fin_*$ our notion is equivalent to  Lurie's notion of a family of operads, \cite{L.HA}, 2.3.2.10.
\item[2.] Lurie's notion of a family of operads is equivalent to the notion of generalized operad, see
\cite{L.HA}, 2.3.2.11. This is not so for general $\cC$
\footnote{basically since $\cC$ may not have an initial object.}.
For instance, a family of planar operads is not the same
as a generalized planar operad introduced in the work of
Gepner-Haugseng~\cite{GH}, 2.4.1.
\end{itemize}
\end{Rems}
A $\cC$-operad family $\cO\to K\times\cC$ is 
{\sl cartesian} if the projection $\cO\to K$ is a cartesian fibration. The Grothendieck construction
converts a cartesian $\cC$-operad family over $K$ into a functor $K^\op\to\Cat_{/\cC}$ 
whose essential image belongs to
$\Op_\cC$. Thus, the notion of cartesian $\cC$-operad family is equivalent to a (contravariant) functor to 
$\Op_\cC$. Cocartesian families of operads are defined in the same way.

Similarly, a $\cC$-operad family 
$\cO\to K\times L\times\cC$ is {\sl bifibered}
if the projection $\cO\to K\times L$ is a bifibration. 
Bifibered $\cC$-operad family over $K\times L$ is 
equivalent to a functor $K^\op\times L\to\Op_\cC$.

\subsubsection{Cartesian families of monoidal categories}
\label{sss:cartesianfamilies}

Let $\cO$ be an operad and $\cC\to\cO$ a strong approximation. A cartesian $\cC$-operad family 
$\cM\to K\times\cC$ is called {\sl a cartesian family of 
$\cC$-monoidal categories} if for each $x\in K$ the fiber \
$\cM_x\to\cC$ is a $\cC$-monoidal 
category~\footnote{Note that the cartesian liftings $\cM_{x'}\to\cM_x$ of arrows $x\to x'$
in $B$ are lax monoidal.}. Cartesian families of
$\cC$-monoidal categories over $K$ form a category
$\Fam\Mon^\cart_\cC(K)$,
with arrows $\cM\to\cM'$ inducing $\cC$-monoidal functors
$\cM_x\to\cM'_x$ for each $x\in K$. This category has a very simple description. Let $\Cat^\cart_{/K}$ denote
the full subcategory of $\Cat_{/K}$ spanned by the cartesian fibrations. 

\begin{prp}
There is a natural equivalence 
$$ \Fam\Mon^\cart_\cC(K)=\Alg_\cC(\Cat^\cart_{/K}).$$
\end{prp}
\begin{proof}
The category $\Fam\Mon^\cart_\cC(K)$ is, by definition, a 
subcategory of $\Cat_{/K\times\cC}$. Its objects are 
bifibered families with the Segal condition. Morphisms are 
morphisms of families preserving cocartesian liftings of 
arrows in $\cC$. 

The right-hand side can be described as the category
of lax functors, see~\ref{sss:deco-mono}, 
$\Fun^\lax(\cC,\Cat^\cart_{/K})$. This
is also a subcategory of $\Cat_{/K\times\cC}$ with the
same description of objects and arrows. 
\end{proof}

\subsubsection{Families of operads and of monoidal 
categories}
\label{sss:fam-op-mon}
There is a category $\Fam\Op_\cC$ of $\cC$-operad
families together with a cartesian fibration 
$\Fam\Op_\cC\to\Cat$, whose fiber at $X$ is 
$\Fib(X\times\cC)^\natural$. Any $\cC$-operad $\cO$ 
is a trivial $\cC$-operad family ($X$ is a point)  
and a morphism from it to a $K$-indexed family $\cP$ consists in a choice of $x\in K$ and a morphism 
$\cO\to\cP_x$ of $\cC$-operads. In this context
the category $\Alg_\cO(\cP)$ of $\cO$-algebras
in the $K$-indexed family $\cP$ of $\cC$-operads
is defined as the object of $\Cat_{/K}$ representing the 
functor
\begin{equation}
\label{eq:repalg}
X\mapsto \Map_{\Cat^+_{/K^\flat\times\cC^\natural}}(X^\flat
\times\cO^\natural,\cP^\natural)~\footnote{This is the 
special case of the formula~(\ref{eq:Algmu}) applied to the 
identity functor $\mu=\id_{K\times\cC}$.}.
\end{equation}
Let $\cP$ be a cocartesian family of $\cC$-operads classified 
by a functor $f:K\to\Op_\cC$. In this case the functor
(\ref{eq:repalg}) is represented by the cocartesian fibration
$p:\cA\to K$ classified by the composition
$$K\stackrel{f}{\to}\Op_\cC\stackrel{\Alg_\cO}{\to}
\Cat.
$$
In fact, for a $\cC$-operad $\cQ$ one has 
$\Alg_\cO(\cQ)$ is a full subcategory of $\Fun_\cC(\cO,\cQ)$,
so $\cA$ is a full subcategory of 
$$\Fun^K_\cC(\cO,\cP):=\Fun_\cC(\cO,\cP)\times_{\Fun(\cO,K)}
K,$$
so $\Map_{\Cat_{/K}}(X,\Fun^K_\cC(\cO,\cP))$ is a full subcategory of
$\Map_{\Cat_{/K\times\cC}}(X\times\cO,\cP)$. One easily verifies  that $\cA$ represents (\ref{eq:repalg}).
 
Thus, in the case when $\cP$ is a cocartesian family of $\cC$-operad classified by $f:K\to\Op_\cC$, the category
$\Alg_\cO(\cP)$ is classified by the composition
$\Alg_\cO\circ f$.

The category $\Fam\Mon_\cC$ is defined as a subcategory 
of $\Fam\Op_\cC$. Its objects are the families $\cO\to X\times\cC$ whose fibers at any $x\in X$ are $\cC$-monoidal categories. The morphism are those inducing $\cC$-monoidal functors on the fibers. 

Similarly to the above, we extend the notation  
$\Fun^\otimes_\cC(\cP,\cQ)$ to families  
of monoidal categories.

\begin{rem}
\label{rem:PK-mon-family}
Let $\cP$  be a $\cC_K$-operad. Then
$\cP'=\cP\times_{\cC_K}(K\times\cC)$ is a family of $\cC$-operads. The category
$\Alg_{\cC}(\cP')$ of $\cC$-algebras in the family $\cP'$ identifies by definition with $\Alg_{\cC/\cC_K}(\cP)$.  
\end{rem}

\subsection{Operadic sieves}
\label{ss:sieves}

Recall~\cite{L.T}, 6.2.2.1, that a sieve $\cC_0\to\cC$
is a full subcategory such that, if $y\in\cC_0$ and 
$f:x\to y$ in $\cC$, then $x\in\cC_0$. Equivalently,
$\cC_0$ is a sieve in $\cC$ if there is a functor $p:\cC\to[1]$ such that $\cC_0=p^{-1}(0)$.

We define operadic sieves in a similar way. Recall
that for an 
operad $\cP$ and a category $K$, one assigns  a new operad $\cP_K$ 
governing $K$-diagrams of $\cP$-algebras. Then the operad 
$\Com_{[1]}$ governing maps of commutative algebras will 
play, for operads, the role of $[1]$ in the definition of a 
sieve in a category.

\begin{dfn} An operadic sieve on $\cP$ is a suboperad $\cQ$ 
presentable as the fiber of a map $\cP\to\Com_{[1]}$
with respect to the embedding $\Com\to\Com_{[1]}$ defined 
by $\{0\}\in[1]$. 
\end{dfn}
An operadic sieve $j:\cQ\to\cP$ is uniquely determined by
a sieve $\cQ_1$ in $\cP_1$.

In this subsection we will prove the following.
\begin{prp}
\label{prp:cartesian-algebras}
Let $\cQ$ be an operadic sieve in $\cP$. Then
\begin{itemize}
\item[1.] For any 
$\cP$-operad $\cC$ the restriction functor 
$$j^*:\Alg_\cP(\cC)\to\Alg_{\cQ}(\cC)$$
is a cartesian fibration. 
\item[2.] For any map $f:\cC\to\cC'$ of $\cP$-operads the 
induced map $f_!:\Alg_\cP(\cC)\to\Alg_\cP(\cC')$ preserves
$j^*$-cartesian arrows.
\end{itemize}
\end{prp}
\begin{exm}
\label{exm:cartesian-algebras}
The suboperad $\Ass\subset\LM$ is an operadic sieve.
Similarly, $\Ass_-\sqcup\Ass_+\subset\BM$ is an operadic sieve.
\end{exm}

The claim of Proposition~\ref{prp:cartesian-algebras}
can be slightly strengthened.
\begin{crl}
\label{crl:cartesian-algebras-fam}
Let $\cQ$ be an operadic sieve in $\cP$ and let $\cC$ be
a cartesian family of $\cP$-operads. Then the restriction
functor 
$$j^*:\Alg_\cP(\cC)\to\Alg_{\cQ}(\cC)$$
is a cartesian fibration. 
\end{crl}
\begin{proof}
If $\cC$ is a cartesian family of $\cP$-operads with base 
$K\in\Cat$, $\Alg_\cP(\cC)$ and $\Alg_\cQ(\cC)$ are
cartesian fibrations over $K$, so that $j^*$ is a map
of cartesian fibrations. According 
to~\ref{prp:cartesian-algebras}(1), the map 
$j^*_x:\Alg_\cP(\cC_x)\to\Alg_{\cQ}(\cC_x)$
is a cartesian fibration for any $x\in K$. Then 
\cite{L.T}, 2.4.2.11 implies that $j^*$ is a locally cartesian fibration. Moreover, an arrow $e:A\to A'$
in $\Alg_\cP(\cC)$ is $j^*$-locally cartesian iff it is equivalent to a composition $e=e''\circ e'$
where $e'$ id $j_x$-cartesian and $e''$ is a cartesian
lifting of an arrow in $K$. Now the second part of
Proposition~\ref{prp:cartesian-algebras} implies that 
the collection of locally cartesian arrows is closed under
composition. This proves the claim.

\end{proof}

\subsubsection{}
\label{sss:cocartesianoperad}
Jacob Lurie in \cite{L.HA}, 2.4.3,
provides a very concrete description of the operad $\cP_K$.
Note that, by definition, $\cP_K$ is a $\cP$-operad such that, for any $\cP$-operad $\cC$, the category of
$\cP_K$-algebras in $\cC$ is canonically equivalent to
$\Fun(K,\Alg_\cP(\cC))$.

 We present below the digest of {\sl loc. cit.}

\begin{Prp}
\begin{itemize}
\item[1.] $\cP_K=\cP\times\Com_K$.
\item[2.] $\Com_K$ is the operad $K^\sqcup$ defined by 
J.~Lurie in~\cite{L.HA}, 2.4.3.1. As a category over 
$\Com=\Fin_*$, it represents the functor
$$
B\mapsto\Map(B\times_{\Fin_*}\Gamma^*,K),
$$
where $\Gamma^*$ is the conventional category of pairs 
$(I_*,i)$ with $I_*\in\Fin_*$ and $i\in I$, with the arrows
$(I_*,i)\to(J_*,j)$ given by arrows $I_*\to J_*$ carrying $i$ to $j$, and the functor $\Gamma^*\to\Fin_*$ carries
$(I_*,i)$ to $I_*$.
\item[3.] $\Com_K$ is a flat operad. In particular, the assignment $\cP\mapsto\cP_K$ preserves colimits.
 
\end{itemize}
\end{Prp}
\begin{proof}
The first two claims are in \cite{L.HA}, Theorem 2.4.3.18.
The third claim follows from the explicit description
of $\Com_B$ and the criterion~\ref{prp:flat-properties}, 
(3).  
\end{proof}

Proposition~\ref{prp:cartesian-algebras}
will be proven in ~\ref{sss:proof-cartesian-algebras},
after a certain preparation.

\begin{lem}
\label{lem:Coma}
\begin{itemize}
\item[1.] The functor $B\mapsto\Com_B$ preserves limits.
\item[2.] Let $q:A\to B$ be a cartesian fibration. The induced 
map of operads $\Com_A\to\Com_B$ defines $\Com_A$ as a flat 
$\Com_B$-operad.
\end{itemize}
\end{lem}
\begin{proof}
The first claim is a direct consequence of Proposition
~\ref{sss:cocartesianoperad}(2). To verify the 
second claim, we will show that any active arrow in 
$\Com_B$ admits a cartesian lifting. In fact, an active map
in $\Com_B$ over an active arrow $\alpha:I_*\to J_*$ in $\Com$ is given by a collection of arrows $b_i\to b'_{\alpha(i)}$ with $b_i,b'_j\in B$. Its cartesian lifting
is just a collection of cartesian liftings of the separate
maps $b_i\to b'_{\alpha(i)}$.
\end{proof}

\subsubsection{}
Let $p:\cP\to\Com_{[1]}$ be an operadic sieve and let 
$q:A\to[1]$ be a cartesian fibration. The fiber product
$\cP_q:=\cP\times_{\Com_{[1]}}\Com_A$ has a beautiful 
presentation as a colimit. Let $j:\cQ\subset\cP$ be defined by $p$.
Let, furthermore, $A_0$ and $A_1$ be the fibers of $q$ at $0$, $1$, and let $\phi:A_1\to A_0$ be the functor classifying the cartesian fibration $q$. We define a map
\begin{equation}
\label{eq:eta-sieve}
\eta:\cQ_{A_0}\sqcup^{\cQ_{A_1}}\cP_{A_1}\to
\cP_q
\end{equation} 
as the one induced by the presentation $A=A_0\sqcup^{A_1}(A_1\times[1])$.
Proposition~\ref{prp:eta-sieve} below claims
that (\ref{eq:eta-sieve}) is an equivalence of operads.  

\begin{prp}$  $
\label{prp:eta-sieve}
\begin{itemize}
\item[1.] Let $q:A\to[1]$, $A=\colim A^\alpha$ in 
$\Cat_{/[1]}$, so that q, as well as $q^\alpha:A^\alpha
\to A\to[1]$, are cartesian fibrations. 
Then the maps $\cP_{q^\alpha}\to\cP_q$
form a colimit diagram of operads over $\Com_{[1]}$.
\item[2.]In particular, the map $\eta$ (\ref{eq:eta-sieve}) is an equivalence of operads. 
\end{itemize}
\end{prp}
The proof of the proposition is given 
in~\ref{sss:proof-prp-eta-sieve}.

\subsubsection{Proof of \ref{prp:cartesian-algebras}}
\label{sss:proof-cartesian-algebras}
We denote by $j:\cQ\to\cP$ the embedding.
The idea is to express the the cartesian lifting for
the functor $\Alg_\cP(\cC)\to\Alg_\cQ(\cC)$ via algebras
over operads of type $\cP_q$ for certain cartesian fibrations $q:A\to[1]$.
We will have a plethora of operads of such form.
\begin{itemize}
\item The first operad of this form, $\cP_1$, is defined
by $\sigma^0:A=[2]\to[1]$. Here we have $A_0=[1], 
A_1=[0], \phi(0)=1$. Thus, by~\ref{prp:eta-sieve}(2),
 $\cP_1=\cQ_{[1]}\sqcup^\cQ\cP$,
so that $\cP_1$-algebras are triples $(C,B,f)$ with
$C\in\Alg_\cQ$, $B\in\Alg_\cP$, $f:C\to j^*(B)$.
\item The second operad, $\cP_2$, is defined by
$\sigma^1:A=[2]\to[1]$. Here we have $A_0=[0], A_1=[1]$.
By~\ref{prp:eta-sieve}(2),
$\cP_2=\cQ\sqcup^{\cQ_{[1]}}\cP_{[1]}$,  so that 
$\cP_2$-algebras are morphisms $g:B\to B'$ of $\cP$-algebras such that $j^*(g)$ is an equivalence.
\item The operad $\cP_3$ is defined by 
$A=[1]\times[1]\to[1]$ given by the projection to the first 
factor. Obviously, $\cP_3=\cP_{[1]}$.
\end{itemize}
The map $\sigma^0:[2]\to[1]$ has two sections $\partial^0,\partial^1:[1]\to[2]$. The map $\partial^0$ induces an
obvious map $\cP\to\cP_1$ determined by decomposition
of $\cP_1$; the map $\partial^1$ defines a more interesting
map $u:\cP\to\cP_1$. The map $u^*:\Alg_{\cP_1}\to\Alg_\cP$
assigns a $\cP$-algebra $C'$ to a  map $C\to j^*(B)$
such that $C=j^*(C')$.
\begin{Exm}
Let $j$ be the embedding $\Ass^\otimes\to\LM^\otimes$. 
Then $u^*:\Alg_{\cP_1}\to\Alg_\cP$ assigns to a 
pair $(C\to B,_BM)$ the $C$-module $M$. 
\end{Exm}
The map $\partial^1:[1]\to[2]$ defines as well a map
$v:\cP\to\cP_2$. In the above example, 
$v^*:\Alg_{\cP_2}\to\Alg_\cP$ assigns to a map of 
$B$-modules $M\to N$ the $B$-module $N$.

Presenting the square $[1]\times[1]$ as  glued from two 
triangles along the common hypotenuse, we get, 
by~\ref{prp:eta-sieve}(1), an equivalence  
$\cP_3=\cP_1\sqcup^\cP\cP_2$. In our example this
means that a morphism from an $A$-module $M$ to a
$B$-module $N$ is given by a pair $(f,g)$ where $f:A\to B$ 
is a morphism of algebras and $g:M\to f^*(N)$ is a morphism 
of $A$-modules. The projection $\cP_2\to\cP$ induced by
$\sigma^1:[2]\to[1]$ yields, together with the decomposition $\cP_3=\cP_1\sqcup^\cP\cP_2$, a projection
$\pi:\cP_3\to\cP_1$. The induced map of algebras defines a (contravariant) lifting of arrows in $\Alg_\cQ$ with respect to the functor $j^*:\Alg_\cP\to\Alg_\cQ$. The same
decomposition of $\cP_3=\cP_{[1]}$ can be now interpreted
as a proof that this lifting is locally cartesian.

Note that at this point we have already verified that 
$j^*$ are locally cocartesian fibrations and that for any
morphism of $\cP$-operads $f:\cC\to\cC'$ the functor
$f_!:\Alg_\cP(\cC)\to\Alg_\cP(\cC')$ preserves the locally
cocartesian liftings.

To prove that $j^*$ is actually a cartesian fibration,
we have to consider a slightly more complicated diagram
of the same type. We need two more operads.
\begin{itemize}
\item $\cP_4=\cQ_{[2]}\sqcup^{\cQ_{[1]}}\cP_{[1]}$
where the map $\cQ_{[1]}\to\cQ_{[2]}$ is induced by
$\partial^1:[1]\to[2]$. Algebras over $\cP_4$ are given
by a morphism $f$ of $\cP$-algebras, together with a 
presentation of $j^*(f)$ into a composition of two
morphisms of $\cQ$-algebras.
\item $\cP_5=\cP_1\sqcup^\cP\cP_{[1]}$, with
$u:\cP\to\cP_1$ and the map $\cP\to\cP_{[1]}$ induced by
$\{1\}\to[1]$. $\cP_5$-algebras are defined by 
a pair of $\cP$-algebras $B'',B$, a map of $\cQ$-algebras
$a:C\to j^*(B)$, and a map of $\cP$-algebras $B''\to B'$
where $B'$ is obtained by the local cartesian lifting of
$a$.
\end{itemize}
Using Proposition~\ref{prp:eta-sieve}, one easily deduces 
that both $\cP_4$ and $\cP_5$ are naturally equivalent to 
$\cP_q$ where $q:A\to[1]$ is the cartesian fibration 
classified by the map $\partial^1:[1]\to[2]$.

This proves the assertion. \qed

\subsubsection{}
Let $\cC$ be a $\Com_{[1]}$-monoidal category classified by a 
SM functor $\cC_0\to\cC_1$ and let 
$q:A\to[1]$ be a cartesian fibration classified by the map 
$\phi:A_1\to A_0$. The category $\Fun_{[1]}(A,\cC)$
is a cocartesian fibration over $[1]$ classified
by the SM functor $\phi^*:\Fun(A_0,\cC_0)\to\Fun(A_1,\cC_1)$.
Therefore, it is a $\Com_{[1]}$-monoidal category.
 
We will now show that this $\Com_{[1]}$-monoidal category
is canonically equivalent to $\Funop_{\Com_{[1]}}(\Com_A,\cC)$. This fact will easily lead to the proof 
of~\ref{prp:eta-sieve}.

Note that $\Fun_{[1]}(A,\cC)$ is a cocartesian 
fibration over $[1]$ and, therefore, it has a standard
presentation as a fiber product, see~\cite{H.L}, 9.8.8, formula (72), as
\begin{equation}
\label{eq:fun-fiberproduct}
\Fun_{[1]}(A,\cC)=\Fun_{[1]}(A_1\times[1],\cC)\times_{\Fun_{[1]}(A_1^\triangleleft,\cC)}
\Fun_{[1]}(A_0^\triangleleft,\cC),
\end{equation}
where  $K^\triangleleft=[0]\sqcup^K(K\times[1])$.
Recall that $\Com_A$ is flat over $\Com_{[1]}$, see 
\ref{lem:Coma}(2).
We can now define a canonical map
\begin{equation}
\label{eq:funopcoma}
\theta_A:\Funop_{\Com_{[1]}}(\Com_A,\cC)\to
\Fun_{[1]}(A,\cC)
\end{equation}
using the decomposition $A=A_0\sqcup^{A_1}(A_1\times[1])$
and ~(\ref{eq:fun-fiberproduct}). 

We have
\begin{prp}
\label{prp:thetaA-eq}
The map $\theta_A$ (\ref{eq:funopcoma}) is an equivalence of $\Com_{[1]}$-operads.
\end{prp}
\begin{proof}
Let us first calculate the fibers of both operads at 
$\langle1\rangle\in\Com$. For the right-hand side 
we have, obviously, the cocartesian fibration 
$\Fun_{[1]}(A,\cC)$ over $[1]$.
The fiber at $\langle1\rangle$ of the left-hand side is
\begin{eqnarray}
\nonumber\Alg_{\iota([1])/\Com_{[1]}}
(\Funop_{\Com_{[1]}}(\Com_A,\cC))=
\Alg_{\iota([1])\times_{\Com_{[1]}}\Com_A/\Com_{[1]}}
(\cC)=\\
\nonumber\Alg_{\iota(A)/\Com_{[1]}}(\cC)=
\Fun_{[1]}(A,\cC),
\end{eqnarray}
the same cocartesian fibration over $[1]$. The morphism
is constructed to be identity for the special cases $A=A_1\times[1]$, $A=K^\triangleleft$, so it is an equivalence in general.

In order to prove that the map (\ref{eq:funopcoma}) is an
equivalence, it is therefore sufficient to verify that
it induces an equivalence of spaces of active maps.
These can be expressed via the category of algebras
$\Alg_{\bC_n}$ (see~\ref{sss:Qn}) with values in these operads.

Let $s:[1]\to\Com_{[1]}$ be an active arrow over the active 
arrow $\langle n\rangle\to\langle 1\rangle$ in $\Com$.
We denote the colors of $\Com_{[1]}$ as $0$, $1$, so
$s$ is either $(0,\ldots,0)\to 0$, or $(0^k1^{n-k})\to 1$
with $k\in\{0\ldots,n\}$. We will study in detail
the second case as it is less obvious. Let $f_i:A_0\to\cC_0$,
$i=1,\ldots, k$, $g_j:A_1\to\cC_1$, $j=1,\ldots,n-k$,
$g:A_1\to\cC_1$ be the colors over $0$ or $1\in\Com_{[1]}$ 
respectively. We will calculate the space
\begin{equation}
\label{eq:mapsfigjg}
\Map(f_1\oplus\ldots\oplus g_{n-k},g)
\end{equation}
in $\Funop_{\Com_{[1]}}(\Com_A,\cC)$. The map $s$ extends (essentially) uniquely
to a map of operads which we will denote by $\tilde s:\bC_n\to
\Com_{[1]}$. The collection of colors $f_i, g_j, g$ determine a map $\bC_n^\circ\to\Funop_{\Com_{[1]}}(\Com_A,\cC\times\Com_{[1]})$, and the space (\ref{eq:mapsfigjg})
identifies with the fiber of the functor
\begin{eqnarray}
\label{eq:mapsfigjg2}
\Alg_{\bC_n/\Com_{[1]}}(\Funop_{\Com_{[1]}}(\Com_A,\cC))\to 
\Alg_{\bC^\circ_n/\Com_{[1]}}(\Funop_{\Com_{[1]}}(\Com_A,\cC)) 
\end{eqnarray}
at $(f_i,g_j,g)$. The categories of algebras in question
can be easily described using the strong approximation 
$Q_n\to\bC_n$. The result is described as follows.

Denote as $\phi_s:A_1\to A_0^k\times 
A_1^{n-k}$ the map with components $\phi$ and $\id_{A_1}$ 
respectively, depending on the target. Let $A_s$ be the
cartesian fibration over $[1]$ classifying the map $\phi_s$. Then the map (\ref{eq:mapsfigjg2}) identifies with
\begin{equation}
\label{eq:mapsfigjg3}
\Fun_{[1]}(A_s,\cC)\to
\Fun(A_0,\cC_0)^{1+k}\times\Fun(A_1,\cC_1)^{n-k}.
\end{equation}
Description of the space of maps in the $\Com_{[1]}$-monoidal category 
$\Fun_{[1]}(A,\cC)$ gives obviously the same result.

\end{proof}

\subsubsection{Proof of Proposition~\ref{prp:eta-sieve}}
\label{sss:proof-prp-eta-sieve}

A map $\phi:\cP\to\cP'$ of operads is an equivalence iff
for any SM category $\cC$ the induced map
$$\phi^*:\Alg_{\cP'}(\cC)\to\Alg_\cP(\cC)$$
is an equivalence. This easily follows from~\cite{L.HA}, 2.2.4.10, where any operad $\cP$ is proven to be equivalent
to a full subcategory of a SM category.

The operad $\Com_A$ is flat over $\Com_{[1]}$, therefore,
for any SM category $\cC$ 
$$\Alg_{\cP_q}(\cC)=
\Alg_{\cP/\Com_{[1]}}(\Funop_{\Com_{[1]}}
(\Com_A,\cC\times\Com_{[1]}))=
\Alg_{\cP/\Com_{[1]}}(\Fun_{[1]}(A,\cC\times[1]))$$ 
by Proposition~\ref{prp:thetaA-eq}.
The first part of Proposition~\ref{prp:eta-sieve} 
is a direct consequence of this formula.
In fact, if $A=\colim A_\alpha$, then 
$\Fun_{[1]}(A,\cC\times[1])=\lim\Fun_{[1]}(A_\alpha,\cC
\times[1])$, and, therefore, $\Alg_{\cP_q}(\cC)=\lim
\Alg_{\cP_{q_\alpha}}(\cC)$. This proves the assertion.
 
To verify the second claim of  
Proposition~\ref{prp:eta-sieve}, we apply the first claim 
to the decomposition $A=A_0\sqcup^{A_1}(A_1\times[1])$.
This is a presentation of $A$ as a colimit of cartesian
fibrations over $[1]$, so it gives a presentation
of $\cP_q$ as a colimit. It remains to add that
$\cP\times_{\Com_{[1]}}\Com(A_1\times[1])=\cP_{A_1}$
and $\cP\times_{\Com_{[1]}}\Com_{A_i}=\cQ_{A_i}$
where $A_i\to[1]$ factors through $\{0\}\to[1]$.
\qed

\subsubsection{}
Let $\pi:\cP_3=\cP_{[1]}\to\cP_{\sigma^0}=\cP_1$ be the map 
defined in~\ref{sss:proof-cartesian-algebras}, 
corresponding to a projection of a square to one of 
its halves. Let $\cY$ be a $\cP_1$-monoidal category
and $\cX=\pi^*(\cY)$.  The $\cP_{[1]}$-monoidal 
category $\cX$ is presented by a $\cP$-monoidal functor $\cX_0\to\cX_1$. We have

\begin{prp}
\label{prp:same-modules}
The diagram
\begin{equation}
\label{eq:algcart}
\xymatrix{
&\Alg_\cP(\cX_0)\ar[d]\ar[r]&\Alg_\cP(\cX_0)\ar[d]\\
&\Alg_\cQ(\cX_0)\ar[r] &\Alg_\cQ(\cX_1)
}
\end{equation}
is cartesian.
\end{prp}
 
Let us show the meaning of this claim for $\cP=\LM$.
The $\LM_{[1]}$-monoidal category $\cX$ is defined 
by an $\LM$-monoidal functor 
$f:(\cC_0,\cM_0)\to(\cC_1,\cM_1)$ between 
$\LM$-monoidal categories, such that
the functor $\cM_0\to\cM_1$ is an equivalence. Then 
our claim amounts to saying that, for any associative 
algebra $A$ in $\cC_0$ the category of $A$-modules
in $\cM_0$ is equivalent to the category of 
$f(A)$-modules in $\cM_1$.
\begin{proof}
Look at the pair of bilinear maps
$$ \mu:[1]\times\cQ\to \cQ_{[1]}\to\cP_1 
$$
and
$$ \nu:[1]\times\cP\to \cP_{[1]}\stackrel{\pi}{\to}\cP_1.
$$
The $\cP_1$-monoidal category $\cY$ defines 
cocartesian fibrations $p:\Alg^\nu_{\cP/\cP_1}(\cY)
\to[1]$ and $q:\Alg^\nu_{\cQ/\cP_1}(\cY)
\to[1]$ by \ref{prp:alg-mu}(4). These cocartesian
fibrations are classified by the functors 
$$\Alg_\cP(\cX_0)\to\Alg_\cP(\cX_1)
\textrm{ and }
\Alg_\cQ(\cX_0)\to\Alg_\cQ(\cX_1)
$$
respectively.

We have $\cP_1=\cQ_{[1]}\sqcup^\cQ\cP$, so
$$\Alg_{\cP_1}(\cY)=\Alg_{\cQ_{[1]}}(\cY)\times
_{\Alg_\cQ(\cY)}\Alg_\cP(\cY).$$
Now, $\Alg_{\cP_1}(\cY)=\Fun_{[1]}([1],
\Alg^\nu_{\cP/\cP_1}(\cY))$, and we denote by
$\Alg^\coc_{\cP_1}(\cY)$ the subcategory of cocartesian section of $\Alg^\nu_{\cP/\cP_1}(\cY)$.
We define in the similar way
$\Alg^\coc_{\cQ_{[1]}}(\cY)$. 

$\Alg^\coc_{\cP_1}(\cY)$ is the full subcategory
of $\Alg_{\cP_1}(\cY)$ spanned by the
$\cP_1$-sections carrying the 
arrow 
$$[1]\stackrel{\partial^2}{\to}[2]\to\cP\times_{\Com_{[1]}}\Com_{[2]}=\cP_1$$
to a cocartesian arrow in $\cY$.

This implies that
$$\Alg^\coc_{\cP_1}(\cY)=\Alg^\coc_{\cQ_{[1]}}(\cY)\times
_{\Alg_\cQ(\cY)}\Alg_\cP(\cY).$$
Thus, the fiber of the map
$\Alg^\coc_{\cP_1}(\cY)\to\Alg^\coc_{\cQ_{[1]}}(\cY)$ 
at a point is equivalent to the fiber of
$\Alg_\cP(\cX_1)\to\Alg_\cQ(\cX_1)$ at the image of 
this point. This proves the claim.

\end{proof}

\subsection{Opposite monoids and opposite algebras}

All imaginable meanings of the notion ``opposite'' in 
category theory can be expressed in terms of the functor 
$I\mapsto I^\op$ assigning to a totally ordered finite set 
the same set with the opposite order.
We denote this functor by $\op:\Delta\to\Delta$. 
We denote by the same letter the endofunctor on 
$\Ass=\Delta^\op$~\footnote{Note that $\Ass$ stands for a very concrete approximation of the operad for associative
algebras. Were we to mean
the operad for associative algebras, we would denote it 
by $\Ass^\otimes$.}

\subsubsection{}Let $\cC$ be a category with products.
If $A:\Ass\to\cC$ is an associative monoid (that is a
functor satisfying Segal condition and such that $A(\langle0\rangle)$ 
is terminal), the composition $A\circ\op$ is also a monoid
called {\sl the opposite monoid} and denoted by $A^\op$.

If $A:\Fin_*\to\cC$ is a commutative monoid, one has a canonical equivalence $A=A^\op$.

This construction applies, in particular, to $\cC=\Cat$,
which gives the notion of the opposite monoidal category.
In order to avoid confusion with the notion of opposite category, we will denote the monoidal category opposite
to $\cM$ as $\cM^\rev$ and will call it {\sl reversed monoidal category}.

\subsubsection{}
The notion of opposite category can be also extracted 
from the functor $\op:\Ass\to\Ass$. Categories,
according to our favorite description, are
complete Segal spaces, that is functors
$$ \cC:\Ass\to\cS$$
satisfying completeness and Segal conditions. Composing
$\cC$ with $\op$, we get another complete Segal functor
$\Ass\to\cS$, that is a new category denoted by $\cC^\op$.
Note that, in this description, a space $X$ coincides 
``identically'' with its opposite.

\subsubsection{}
The endofunctor $\op:\Ass\to\Ass$ extends to 
$\op:\BM\to\BM$ carrying $\LM$ to $\RM$ and vice versa.
This identifies left modules over $A$ with right modules 
over $A^\op$ in a category with products.

In particular, the categories left-tensored over $\cM$ identify with the categories right-tensored over 
$\cM^\rev$.

\subsubsection{}Let now $A$ be an associative algebra
in a monoidal category $\cM$. By definition, $A$ is a section of the canonical projection $p:\cM\to\Ass$.
Composing $A$ with $\op:\Ass\to\Ass$, we get a section 
of the base change of $\cM$ with respect to $\op$, which is
precisely $\cM^\rev$. Thus, for $A\in\Alg_\Ass(\cM)$ one has
$A^\op\in\Alg_\Ass(\cM^\rev)$.

If $\cM$ is symmetric monoidal, $\cM^\rev=\cM$, and we get
$A^\op$ as an algebra in $\cM$.

 In the case where $\cM$ is cartesian
both constructions of opposite algebra in $\cM$ coincide.

If $\cM$ is a monoidal category and $A$ is an associative algebra in $\cM$, one has a canonical equivalence 
$\LMod_A(\cM)=\RMod_{A^\op}(\cM^\rev)$ \footnote{$\cM$ and 
$\cM^\rev$ have the same underlying category!}.

\subsubsection{} The notion of reversed monoidal category
extends to planar operads. If fact, if $p:\cP\to\Ass$
is a planar operad, the composition $\op\circ p:\cP\to\Ass$
is also a planar operad called {\sl the reversed operad}.
Note that the functor $\op:\BM\to\BM$ provides an 
isomorphism of $\BM$ with $\BM^\rev$.

If $\cP$ is flat, $\cP^\rev$ is also flat.
One has $(\cP\times\cQ)^\rev=\cP^\rev\times\cQ^\rev$, so
$\Funop(\cP,\cQ)^\rev=\Funop(\cP^\rev,\cQ^\rev)$.

\section{Enriched quivers}
\label{sec:quivers}

\subsection{Introduction}

In this section we present a construction that assigns,
to each $X\in\Cat$ and $\cM\in\Op_\Ass$, a new planar operad $\Quiv_X(\cM)$ called the planar operad of $\cM$-enriched $X$-quivers.

With the aim of describing a universal property of this construction later on, we present two more versions of
this construction: $\Quiv^\LM_X(\cM)\in\Op_\LM$ of an $\LM$-operad $\cM$ and $\Quiv^\BM_X(\cM)$ for any $\cM\in\Op_{\BM}$.

 The constructions are functorial in $X$ and in $\cM$, as explained in~\ref{sss:dependence-M},
\ref{sss:dependence-X} and \ref{ss:functoriality-quiv} below.

The category of colors \ $\Quiv_X(\cM)_1$ \ of the planar operad \ $\Quiv_X(\cM)$ \ is $\Fun(X^\op\times X,\cM_1)$.
These are functors from $X^\op\times X$ to the category of colors of $\cM$; we interpret them as quivers,
with the category of objects $X$, and with values in $\cM_1$. 

The most important for us is the case when $\cM$ is a 
monoidal category having colimits
(precise requirements are given below); in this case 
$\Quiv_X(\cM)$
is also a monoidal category. We weaken the requirements
on $X$ and $\cM$ in order to better understand the 
functoriality of the construction.

The ultimate goal of the paper are $\cM$-enriched 
$\infty$-categories. Similarly to the conventional case where categories can be defined as associative algebra objects
in the appropriate monoidal category of quivers, we define
$\cM$-enriched categories as associative algebras in 
$\Quiv_X(\cM)$ satisfying some extra (completeness)
properties. The full meaning of the following definitions \footnote{presented here as an advertisement.}  will become clear later.
\begin{dfn}
\label{dfn:Mprecat}Let $X$ be a category and $\cM$ be a planar operad. An $\cM$-enriched precategory with the category of objects $X$ is an associative algebra object in $\Quiv_X(\cM)$.
\end{dfn}

\begin{dfn}
\label{dfn:Mcat}
Let $X$ be a space, $\cM$ a monoidal category with colimits. An $\cM$-enriched category $A$ with the space of objects $X$ is an $\cM$-enriched precategory 
satisfying a completeness condition, see 
Definition~\ref{dfn:enrichedcat}.
\end{dfn}

\subsubsection{Dependence on $\cM$}
\label{sss:dependence-M}
The $\BM$-operad $\Quiv^\BM_X(\cM)$ is defined, using
the internal mapping object in $\Op_\BM$, as
\begin{equation}
\Quiv^\BM_X(\cM)=\Funop_\BM(\BM_X,\cM),
\end{equation}
where $\cM$ is a $\BM$-operad and  $\BM_X$ is a flat 
$\BM$-operad depending of $X\in\Cat$, defined below. 

If $\cM$ is in $\Op_\Ass$,
we will sometimes  write $\Quiv^\BM_X(\cM)$ instead of
$\Quiv^\BM_X(\pi^*\cM)$, where $\pi:\BM\to\Ass$ is the natural projection, see \ref{sss:opmaps}, 

The construction of $\BM_X$ is presented below, after
a certain preparation in \ref{sss:cat-via-functor}.

We will only mention now that the $\Ass_-$-component
of $\BM_X$ is equivalent to the planar operad $\cO_X$
defined in \cite{GH}, 4.2.4, when $X$ is presented by a 
simplicial category.

\subsubsection{Relative categories}
\label{sss:cat-via-functor}
Recall that the embedding $\Delta\to\Cat$, together with the Yoneda embedding, 
identifies $\Cat$ with the full subcategory of $\Fun(\Delta^\op,\cS)$ spanned by the 
complete Segal objects. We will now present a similar description for categories over a fixed category $\cC\in\Cat$.

Recall from~\ref{sss:delta-over} that $\Delta_{/\cC}$ denotes 
the full subcategory of $\Cat_{/\cC}$ spanned by the objects $\sigma:[n]\to\cC$.

An object $\cX\in\Cat_{/\cC}$ yields a Yoneda map
$$ Y_\cX:(\Cat_{/\cC})^\op\to\cS,$$
whose restriction to $\Delta_{/\cC}$ defines a presheaf
$$F_\cX:(\Delta_{/\cC})^\op\to\cS.$$
Note the following
\begin{Lem}
The presheaf $F_\cX\in P(\Delta_{/\cC})$ is presented by the right fibration $\Delta_{/\cX}\to\Delta_{/\cC}$
induced by the map $\cX\to\cC$.
\end{Lem}
\begin{proof}
The Yoneda lemma implies that $Y_\cX\in P(\Cat_{/\cC})$ is presented by the right fibration 
$$(\Cat_{/\cC})_{/\cX}\to\Cat_{/\cC}.$$
Proposition 2.1.2.5 in \cite{L.T} asserts that 
$(\Cat_{/\cC})_{/\cX}$ is canonically equivalent to 
$\Cat_{/\cX}$. The required result follows from this by
base change along $\Delta\to\Cat$.
\end{proof}
In practice we will mostly be interested in categories
over a conventional category $B$ having no nontrivial isomorphisms.
In this case
$\Delta_{/B}$ is  a conventional category whose objects are the functors 
$a:[n]\to B$, with a morphism from $a$ to $b:[m]\to B$
given by a commutative triangle $[n]\to [m]\to B$.

Given $\cX\in\Cat_{/B}$, the corresponding functor 
$F_\cX:(\Delta_{/B})^\op\to\cS$ carries $\sigma:[n]\to B$
to the space of sections of the base change
\begin{equation}
\cX\times_B [n]\to[n].
\end{equation}

Conversely, a functor $F:(\Delta_{/B})^\op\to\cS$
determines a simplicial space $\cX$ over $B$ 
which is  a category over $B$ iff it satisfies a version
of completeness and Segal conditions. We will be only interested in the special case when $B$ is a conventional
category with no nontrivial isomorphisms. In this case
for each $n$ the space $B_n$ is discrete and $\cX_n=\sqcup_{\sigma\in B_n}F(\sigma)$.

Therefore, for $B$ a conventional category with no nontrivial isomorphisms $\cX$ is a category over $B$
iff $F$ satisfies the following properties.

\begin{itemize}
\item[(CS1)] Let $\sigma:[n]\to B$ be given by a sequence of arrows $f_i:x_{i-1}\to x_i$, $i=1,\ldots,n$, in $B$. Then the canonical map 
 
$$F(\sigma)\to F(f_1)\times_{F(x_1)}\ldots
\times_{F(x_{n-1})}F(f_n)$$ 
is an equivalence.
\item[(CS2)] For any $x\in B$ the Segal
space obtained by the composition 
$$\Delta^\op=(\Delta_{/\{x\}})^\op\to(\Delta_{/B})^\op\stackrel{F}{\to}\cS,$$
is complete.
\end{itemize}
The property (CS1) implies that $\cX$ is Segal; the property (CS2) ensures it is complete.

\subsubsection{Dependence on $X$}
\label{sss:dependence-X}

The $\BM$-operad $\BM_X$ will be defined by a functor 
$$
\cF^\BM_X:(\Delta_{/\BM})^\op\to\cS,
$$
whose dependence on $X$ is seen from the formula
\begin{equation}
\label{eq:omegaplus} 
\cF^\BM_X(\sigma)=\Map(\cF^\BM(\sigma),X),
\end{equation}
where $\cF^\BM:\Delta_{/\BM}\to\Cat$ is independent of 
$X$; it is defined below.

\subsection{Functor $\cF^\BM$}\label{ss:cXplus}

The category  $\Delta_{/\BM}$ has objects
$\sigma:[n]\to\BM$, and a map from $\sigma$ to $\tau:[m]\to\BM$ is given by the commutative diagram
$$
\xymatrix{
&{[n]}\ar[rr]\ar[rd]^\sigma&{}&{[m]}\ar[ld]_\tau\\
&{}&{\BM}.
}
$$

\subsubsection{ A general description}

The functor $\cF^\BM$ will have values in conventional categories. It will also satisfy the following
property which will ensure that $\cF^\BM_X$ satisfies Segal condition.

Let $\sigma:[n]\to\BM$ be given by a collection of
arrows $f_1,\ldots,f_n$, $ f_i:v_{i-1}\to v_i$ in $\BM$.
 Then the natural map
\begin{equation}
\label{eq:cosegal}
\cF^\BM(f_1)\sqcup^{\cF^\BM(v_1)}\cF^\BM(f_2)
\sqcup^{\cF^\BM(v_2)}\sqcup\ldots\sqcup^{\cF^\BM
(v_{n-1})}\cF^\BM(f_n)\to\cF^\BM(\sigma)
\end{equation}
is an equivalence.

Having this in mind, we will describe $\cF^\BM$ first
on $\sigma:[n]\to\BM$ with $n=0$, then for $\sigma$ with 
$n=1$, and then use the formula (\ref{eq:cosegal}) to 
define $\cF^\BM(\sigma)$
for general $\sigma$ as the left-hand side of the 
formula.
 
Finally, we will have to define the functors between 
the $\cF^\BM(\sigma)$ corresponding to the inner faces.

In what follows we will write $|\sigma|=n$ for
$\sigma:[n]\to\BM$.

The functor $\cF^\BM$ will assign discrete categories
to $0$-dimensional simplices, and some disjoint unions of 
$[1]$ to one-dimensional simplices.

\subsubsection{$|\sigma|=0$}
\label{sss:dim=0}

An object $\sigma:[0]\to\BM$ is given by a length $n+1$ monotone 
sequence $w$ of $0$'s and $1$'s. Let us,
following \ref{sss:bm}, translate $w$ into a length $n$ 
sequence of letters $a,m,b$ --- it should have the form
$a^k m^\alpha b^l$, where $k,l\geq 0$, $\alpha\in\{0,1\}$,
and $\alpha=1$ if both $k,l$ are nonzero, with 
$n=k+l+\alpha$.

The functor $\cF^\BM$ assigns to $\sigma$ the discrete category with
$2k+\alpha$ objects, denoted by $x_r$ and $y_r$ with $r=1,\ldots,k$, and $y$ if $\alpha=1$.
Note that there are two objects of $\BM$ of length $n=0$
which have an empty presentation as an $amb$-word. The value of $\cF^\BM$ at these two objects is the empty category.
In the pictures below we draw the objects in the following order:
\begin{equation}\label{eq:a-k}
\xymatrix{
&&\stackrel{x_k}{\circ} &\stackrel{y_k}{\bullet}&\ldots & 
&\stackrel{x_1}{\circ} &\stackrel{y_1}{\bullet}
&(\alpha=0)\\
&\stackrel{y}{\bullet}&\stackrel{x_k}{\circ} &\stackrel{y_k}{\bullet}&\ldots & 
&\stackrel{x_1}{\circ} &\stackrel{y_1}{\bullet}
&(\alpha=1)
}
\end{equation}

Note that $x$-type objects are denoted by $\circ$, whereas 
$y$-type objects are denoted by $\bullet$.

\subsubsection{$|\sigma|=1$}
\label{sss:dim=1}
 An object 
$\sigma:[1]\to\BM$ is defined by an arrow $f:w\to w'$ in 
$\BM$. 

We assume $w=a^km^\alpha b^l$ and $w'=a^{k'}m^{\alpha'} 
b^{l'}$, 
using the standard notation.
The arrow $f:w\to w'$ in $\BM=(\Delta_{/[1]})^\op$ defines
(and is defined by) an arrow $\pi(f):\langle n\rangle\to\langle n'\rangle$ in $\Ass$, so by an arrow
$\phi:[n']\to[n]$ in $\Delta$. The object $w$ is a functor
$w:[n]\to[1]$ and $w'$ is the composition $w\circ \phi$.

\

The objects of $\cF^\BM(f)$  are the disjoint union of 
objects of $\cF^\BM(w)$ and $\cF^\BM(w')$.

The nontrivial arrows of $\cF^\BM(f)$ are all disjoint,
and they are specified below. 

We will denote the objects of $\cF^\BM(w)$, as in \ref{sss:dim=0}, by $x_r,y_r$ ($r=1,\ldots,k$)
or $y$, and the objects of $\cF^\BM(w')$ as
$x'_r,y'_r$ ($r=1,\ldots,k'$) or $y'$.

Each segment $\{i-1,i\}$ of $[n']$, with $i=1,\ldots,k'$
defines $\phi(i)-\phi(i-1)+1$ nontrivial disjoint arrows
in $\cF^\BM(f)$ according to the following rule.

\begin{itemize}
\item If $\phi(i-1)=\phi(i)$, one has an arrow from 
$x'_i$ to $y'_i$.
\item If $\phi(i-1)<\phi(i)$, one has arrows 
$x'_i\to x_{\phi(i)}$, $y_{\phi(i-1)+1}\to y'_i$, as well as $\phi(i)-\phi(i-1)-1$
arrows from $y_j$ to $x_{j-1}$, for $j=\phi(i-1)+2,\ldots,\phi(i)$. 
\end{itemize}

Furthermore, in the case $\alpha'=1$ (and also, since 
$f:w\to w'$ exists, $\alpha=1$), there are 
$\phi(k'+1)-\phi(k')$ more nontrivial arrows in 
$\cF^\BM(f)$; these are the arrow $y\to x_k,\ y_j\to x_{j-1}$ for $j=k,\ldots,\phi(k')+2$ ($j$ decreases by $1$), as well as $y_{\phi(k')+1}\to y'$.

As a result, the above description yields a category 
$\cF^\BM(f)$ having $2(k+k')+\alpha+\alpha'$ objects
(note that $\alpha=\alpha'$ if $f:w\to w'$ exists) and
$k'+\phi(k'+\alpha')-\phi(0)$ disjoint nontrivial arrows.

\subsubsection{$|\sigma|=1$, in pictures}
\label{sss:dim=1}
In the pictures below we denote by $\circ$ the $x$-objects and $\bullet$ the $y$-objects. The upper row describes $w$,
and the lower row describes $w'$.

First of all, the picture (\ref{eq:inertarrow}) below
presents $\cF^\BM(f)$ for a typical inert arrow corresponding to $\phi:[2]\to [4]$ carrying $i=0,1,2$ to 
$i+1$
(here $w=a^4,\ w'=a^2$).
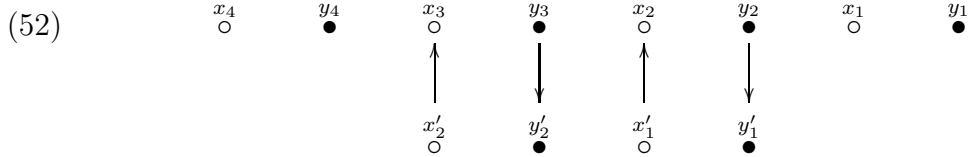
\begin{equation}\label{eq:inertarrow}
\xymatrix{
&\stackrel{x_4}{\circ} &\stackrel{y_4}{\bullet}&
\stackrel{x_3}{\circ} &\stackrel{y_3}{\bullet}\ar[d]&
\stackrel{x_2}{\circ} &\stackrel{y_2}{\bullet}\ar[d]&
\stackrel{x_1}{\circ} &\stackrel{y_1}{\bullet}\\
& & &\stackrel{x'_2}{\circ}\ar[u] &\stackrel{y'_2}{\bullet}&\stackrel{x'_1}{\circ}\ar[u] &\stackrel{y'_1}{\bullet}
}
\end{equation}

The next picture is of a typical active arrow $[1]\to[3]$
(here $w=a^3,\ w'=a^1$).
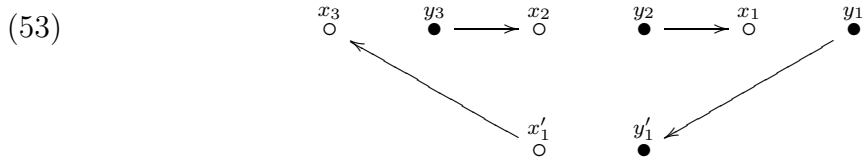
\begin{equation}\label{eq:activearrow}
\xymatrix{
&\stackrel{x_3}{\circ} &\stackrel{y_3}{\bullet}\ar[r]&
\stackrel{x_2}{\circ} &\stackrel{y_2}{\bullet}\ar[r]&
\stackrel{x_1}{\circ} &\stackrel{y_1}{\bullet}\ar[lld]\\
& & &\stackrel{x'_1}{\circ}\ar[ull] &\stackrel{y'_1}{\bullet} 
}
\end{equation}

The following picture describes $\cF^\BM(f)$ in a case
where the corresponding $\phi:[n']\to[n]$ is neither inert, nor active or injective. Here $n=4,\ n'=2$ ($w=a^4,\ w'=a^2$), and the map $\phi$ carries $0$ to $1$, $1$ to $3$ and $2$ to $3$.

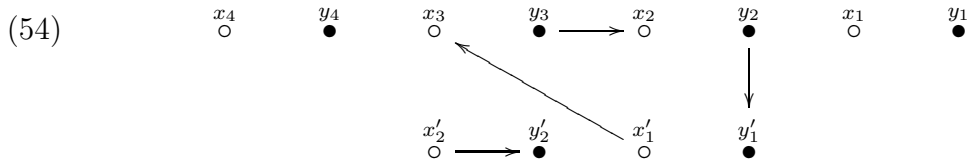
\begin{equation}\label{eq:arrow}
\xymatrix{
&\stackrel{x_4}{\circ} &\stackrel{y_4}{\bullet}&
\stackrel{x_3}{\circ} &\stackrel{y_3}{\bullet}\ar[r]&
\stackrel{x_2}{\circ} &\stackrel{y_2}{\bullet}\ar[d]&\stackrel{x_1}{\circ} &\stackrel{y_1}{\bullet}\\
& & &\stackrel{x'_2}{\circ}\ar[r] &\stackrel{y'_2}{\bullet}&\stackrel{x'_1}{\circ}\ar[llu] &\stackrel{y'_1}{\bullet}
}
\end{equation}

The last example demonstrates $\cF^\BM(f)$ for  the active arrow $f$ and $\alpha=1$. Here $w=a^2m$ and $w'=am$.
The corresponding map $\phi:[2]\to[3]$ carries $0$ to $0$,
$1$ to $1$ and $2$ to $3$.

\begin{equation}\label{eq:LMarrow}
\xymatrix{
&\stackrel{y}{\bullet}\ar[r]&\stackrel{x_2}{\circ} 
&\stackrel{y_2}{\bullet}\ar[ld]&\stackrel{x_1}{\circ}
&\stackrel{y_1}{\bullet}\ar[ld] \\
& &\stackrel{y'}{\bullet} &\stackrel{x'_1}{\circ}\ar[ur]  & 
\stackrel{y'_1}{\bullet} 
}
\end{equation}

\subsubsection{Description of $\cF^\BM(\sigma)$}

Given an object $\sigma$ in $\Delta_{/\BM}$ presented by a sequence of arrows $w_0\stackrel{f_1}{\to}w_1\to\ldots
\stackrel{f_n}{\to}w_n$, one defines $\cF^\BM(\sigma)$ 
as the colimit  
\begin{equation}
\label{eq:Xplussigma}
\cF^\BM(f_1)\sqcup^{\cF^\BM(w_1)}\cF^\BM(f_2)\sqcup
\ldots\sqcup^{\cF^\BM(w_{n-1})}\cF^\BM(f_n).
\end{equation}
This is the category freely generated by the union of $n$ 
diagrams corresponding to the arrows $f_i$ as shown in
the pictures  above. It is convenient to present it as a 
multi-story graph, with the $i$-th story describing  
$\cF^\BM(f_i)$.

\subsubsection{Examples}
\label{sss:connected-path}

Let  
$\sigma=(w\stackrel{f}{\to}w'\stackrel{g}{\to}w'')$ be given by $w:[n]\to[1]$ and a simplex
$[n'']\stackrel{\psi}{\to}[n']\stackrel{\phi}{\to}[n]$.
One has $w'=w\circ\phi$, $w''=w\circ\phi\circ\psi$. 

1. We assume $\psi(i)>\psi(i-1)$ but $\phi(\psi(i))=\phi(\psi(i-1))$. The corresponding fragment of the graph
of $\cF^\BM(\sigma)$ ( ``a cap'') is connected.

\begin{equation}\label{eq:connected-graph}
\xymatrix{
&\stackrel{x'_{\psi(i)}}{\circ}\ar[r] &\stackrel{y'_{\psi(i)}}
{\bullet}\ar@{.}[r]&\stackrel{x'_{\psi(i-1)+1}}{\circ}\ar[r]&\stackrel{y_{\psi(i-1)+1}}{\bullet} \ar[dl]
 \\
& &\stackrel{x''_i}{\circ}\ar[ul]  & \stackrel{y''_i}{\bullet} 
}
\end{equation}

2. Let now assume $\phi(i)>\phi(i-1)=\ldots=\phi(i-k)
>\phi(i-k-1)$,
and assume that there exists $r$ such that 
$\psi(r-1)\leq i-k-1,\ \psi(r)\geq i$. 

In this case we get another connected fragment of the graph
of $\cF^\BM(\sigma)$ ( ``a cup'').

\begin{equation}\label{eq:connected-graph-2}
\xymatrix{
& &\stackrel{y_{\phi(i-1)+1}}{\bullet}\ar[dl]  & \stackrel{x_{\phi(i-k)}}{\circ} \\
&\stackrel{y'_i}{\bullet}\ar[r] &\stackrel{x'_{i-1}}
{\circ}\ar@{.}[r]&\stackrel{y'_{i-k+1}}{\bullet}\ar[r]&\stackrel{x'_{i-k}}{\bullet} \ar[ul]
}
\end{equation}
Note that in this case the vertices $y_{\phi(i-1)+1}$ and $x_{\phi(i-k)}$ are neighbors and 
$$\phi\psi(r-1)+2\leq \phi(i-1)+1\leq\phi\psi(r).$$
In fact, $\phi(i-1)+1\leq\phi(i)\leq\phi\psi(r)$ as $i\leq\phi(r)$ and $\phi\psi(r-1)\leq\phi(i-k-1)\leq\phi(i-k)-1
=\phi(i-1)-1$ as $\psi(r-1)\leq i-k-1$.

\subsubsection{}
In general, $\cF^\BM(\sigma)$ is a free category defined
by this multi-story graph. This graph has very pleasant properties. Let us start with a few remarks.

\begin{itemize}
\item We consider the graph of  $\cF^\BM(\sigma)$ as embedded
into the plane as described above. This embedding makes it a planar graph (its arrows have no intersection).
\item The horizontal arrows are directed rightward.
\item The ends of the upward arrows are $x$-objects, and 
the ends of the downward arrows are $y$-objects. 
All vertical arrows come in pairs $(\alpha,\beta)$
so that $\alpha$ is upward, $\beta$ is downward, and 
the source of $\alpha$ is a neighbor of the target of $\beta$. Such pairs of arrows will be called adjoint.
\end{itemize}

\begin{lem}$ $
\label{lem:graph-F}
\begin{itemize}
\item[1.] No vertex in the graph $\cF^\BM(\sigma)$ has
 two incoming or two outgoing arrows.
\item[2.]Let a connected path in the graph of $\cF^\BM(\sigma)$, 
having two ends at the same level, starts with an upward 
arrow $\alpha$ and ends with a downward arrow $\beta$. Then 
the ends of the path are neighbors and $\alpha$ and $\beta$ are adjoint.
\item[2$^\bis$.]
Let a connected path in the graph of $\cF^\BM(\sigma)$, 
having two ends at the same level, starts with an downward 
arrow $\beta$ and ends with an upward arrow $\alpha$.   
Then the ends of the path are neighbors.
\item[3.] The graph $\cF^\BM(\sigma)$  has no closed cycles.
\end{itemize}
\end{lem}
\begin{proof}
1. An arrow in the graph $\cF^\BM(\sigma)$ cannot go up or 
down for more than one floor.
Therefore, we can assume the simplex $\sigma$ consists of two consecutive arrows,  $f$ and $g$, and the vertex
in question belongs to $\cF^\BM(w)$ where $w$ is the target of $f$ and the source of $g$. If there are two incoming arrows 
for a vertex, one of them should originate from 
$\cF^\BM(f)$, and the other from $\cF^\BM(g)$. But the one originated from $\cF^\BM(f)$ should be
$y$-vertex, whereas the one originated from $\cF^\BM(g)$ should be $x$-vertex. A similar reasoning shows that a vertex cannot have two outgoing arrows.

2. A simplest example of a connected path in question is 
given in~\ref{sss:connected-path}. The claim is proven 
by induction in the number of  $u$ upward arrows (which is 
the same as the number of downward arrows). If $u=1$, our 
path consists of $\alpha$, followed by a number of horizontal 
arrows, followed by $\beta$. This is only possible when
$\beta$ and $\alpha$ are adjoint. In case $u>1$,
let us write down, moving along the path, all ups and downs. 
Choose a pair of neighboring $\alpha',\beta'$, with the 
downward arrow 
$\beta'$ following the upward arrow $\alpha'$. The 
corresponsing segment of the path satisfies the condition
of \ref{lem:graph-F}(2), and so, its ends are a pair of 
neighboring $x$-and $y$-objects. Our segment looks just
as in the Example~\ref{sss:connected-path}. Replacing,
in the notation of {\sl loc. cit}, the value of the map 
$\psi:[n'']\to[n']$ at $i$, declaring $\psi(i):=\psi(i-1)$,
we replace our segment with an arrow
going from $x$ to $y$. We get a shorter path for which the 
claim is assumed to be already proven.

2$^\bis$. The proof is identical to the proof of claim 2.

3. Assume there is a  closed loop in the graph of 
$\cF^\BM(\sigma)$. Look at the vertices at the lowest row.
There should be a upward arrow from (at least) one of them,
and a downward arrow to one of them. According to part 2
of the lemma, these two are adjoint. On the other hand,
the source of the upward adjoint arrow cannot be connected
to thetarget of the downward adjoint arrow.
Contradiction.
\end{proof}
The lemma above implies that $\cF^\BM(\sigma)$, considered as a category, is a disjoint union of a finite number
of finite totally ordered sets.

\subsubsection{Two examples}
For later reference, we present two examples of 
$\cF^\BM(\sigma)$. 
In both cases $\sigma:[k]\to\BM$ is the composition 
of a map $h:[k]\to[1]$ with an active map 
$\alpha:[1]\to\BM$ given 
by an arrow in $\BM$.
The map $h$ is given by a sequence of $k+1$ zeros and ones,
say, $h=0^{i+1}1^j$ with $i+j=k$. We will present two
examples of $\alpha$: the first one, is 
$\alpha_1:a^n\to a$, and the second is 
$\alpha_2:a^{n-1}m\to m$.  
The pictures of
$\cF^\BM(\alpha_1\circ h)$ and 
$\cF^\BM(\alpha_2\circ h)$ are presented below.
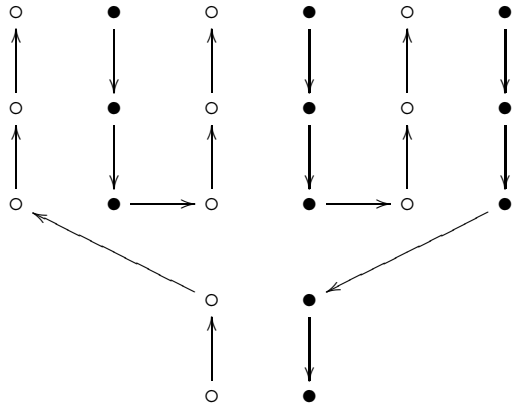
\begin{equation}\label{eq:activesimplex-an}
\xymatrix{
&\circ &\bullet\ar[d]&\circ &\bullet\ar[d]&\circ &\bullet\ar[d]\\
&\circ\ar[u] &\bullet\ar[d]&\circ\ar[u] &\bullet\ar[d]&\circ\ar[u] &\bullet\ar[d]\\
&\circ \ar[u]&\bullet\ar[r]&\circ\ar[u] &\bullet\ar[r]&\circ \ar[u]&\bullet\ar[lld]\\
& & &\circ\ar[ull] &\bullet\ar[d] \\
& & &\circ\ar[u] &\bullet
}
\end{equation}
--- for $\cF^\BM(\sigma_1)$, and 

\begin{equation}\label{eq:LMsimplex}
\xymatrix{
&\bullet\ar[d]&\circ &\bullet\ar[d]&\circ &\bullet\ar[d]\\
&\bullet\ar[d]&\circ\ar[u] &\bullet\ar[d]&\circ\ar[u] &\bullet\ar[d]\\
&\bullet\ar[r]&\circ\ar[u] &\bullet\ar[r]&\circ\ar[u] &\bullet\ar[dllll] \\
&\bullet\ar[d]& & &  & \\
&\bullet& & &  & 
}
\end{equation}
for $\cF^\BM(\sigma_2)$ (compare to the pictures
(\ref{eq:activearrow}) and (\ref{eq:LMarrow})).

\subsubsection{$\cF^\BM$ is a functor}
\label{sss:cF-functor}

To define the functor $\cF^\BM$, we have to describe the 
operation of ``erasing a row of objects'' corresponding to 
an inner face $d_i:[n-1]\to [n]$, $i\ne 0,n$.

Since $\cF^\BM(\sigma)$ are posets, the inner face maps 
$\cF^\BM(d_i\sigma)\to\cF^\BM(\sigma)$ are uniquely defined 
once we verify the following.

\begin{Lem}$ $
\begin{itemize}
\item[1.]The embedding 
$\Ob(\cF^\BM(d_i\sigma))\to\Ob(\cF^\BM(\sigma))$
is a map of posets.
\item[2.] If $i\ne 0,|\sigma|$, this is an
embedding of posets (that is, two elements of 
$\Ob(\cF^\BM(d_i\sigma))$ are comparable iff their images are).
\end{itemize}
\end{Lem}
\begin{proof} 
The first property is obvious for $|\sigma|=1$; The decomposition formula (\ref{eq:Xplussigma}) reduces
both claims to the case $|\sigma|=2$ and $i=1$. 

From now on we assume $|\sigma|=2$ and $i=1$. 

We have to prove that for $v,v'\in\cF^\BM(gf)$ $v$ is less than $v'$ in $\cF^\BM(gf)$ if and only if it is less as
an element of $\cF^\BM(\sigma)$.

We will verify this in four possible cases separately 
(the cases A,B,C and D below). 

We use the following notation.
The simplex $\sigma=(w\stackrel{f}{\to}w'\stackrel{g}{\to}w'')$ is given by $w:[n]\to[1]$ and a simplex
$[n'']\stackrel{\psi}{\to}[n']\stackrel{\phi}{\to}[n]$.
One has $w'=w\circ\phi$, $w''=w\circ\phi\circ\psi$. 

A. Assume $v=x''_i$ and $v'$ is of $y$-type. 
The condition $v<v'$ in $\cF^\BM(gf)$ means that $\phi\psi(i)
=\phi\psi(i-1)$. The same condition in $\cF^\BM(\sigma)$
means that there is a path connecting $v$ and $v'$. This is 
only possible if the path is as described 
by~\ref{eq:connected-graph} and this happens precisely
when $\phi\psi(i)=\phi\psi(i-1)$.

B. Assume $v=y_j$, and $v'$ is of $x$-type. The condition $v<v'$ in $\cF^\BM(gf)$ means that there is $r$ such that 
$\phi\psi(r-1)+2\leq j\leq\phi\psi(r)$.
The same condition in $\cF^\BM(\sigma)$
means that there is a path connecting $v$ and $v'$. This is 
only possible if the path is as described 
by~\ref{eq:connected-graph-2} and we will now show that this 
happens precisely when there is $r$ such that 
$\phi\psi(r-1)+2\leq j\leq\phi\psi(r)$.
 
Choose $i$ such that $\phi(i-1)+1\leq j\leq\phi(i)$. If $j\geq\phi(i-1)+2$,
this means that our horizontal arrow belongs to 
$\cF^\BM(g)$.
Otherwise $j=\phi(i-1)+1$; we choose a maximal $k$
such that $j=\phi(i-k)+1$. Then we have a downward arrow
$y_j\to y'_i$, an upward arrow $x'_{i-k}\to x_{j-1}$,
and all horizontal arrows connecting $y'_i$ with $x'_{i-k}$.

C. Assume $\cF^\BM(gf)$ has an upward arrow, say, going
from $v=x''_i$ to $v'=x_{\phi\psi(i)}$. This means that 
$\phi\psi(i-1)<\phi\psi(i)$ which implies 
$\psi(i-1)<\psi(i)$. This gives a vertical arrow
$x''_i\to x'_{\psi(i)}$ in $\cF^\BM(\sigma)$. 
Choose $k\in[\psi(i-1),\psi(i)-1]$, maximal among those
satisfying the condition $\phi(k)<\phi\psi(i)$.
Then there is a vertical arrow $x'_{k+1}\to 
x_{\phi\psi(i)}$. We also have arrows from $x'_j$ to $y'_j$
for $j=k,\ldots,\psi(i)$ as in this range $\phi(j-1)=
\phi(j)$, and also arrows from $y'_j$ to $x'_{j-1}$ for
$j=k+2,\ldots,\psi(i)$. This gives a path from $v$ to $v'$ in $\cF^\BM(\sigma)$. In the opposite direction,
if there is a path in $\cF^\BM(\sigma)$ connecting 
$v=x''_i$ to $v'=x_j$, it should be necessarily weakly monotone, as otherwise it would contain a subpath having 
both ends in $w''$ (downstairs) which is impossible by the description of such paths in part A of the proof. This implies that the path starts with an upward arrow, continues horizontally, and ends with another upward arrow. This is
only possible if $j=\phi\psi(i)$ and $\phi\psi(i)>
\phi\psi(i-1)$.

D. A similar reasoning works for downward arrows.
Assume $\cF^\BM(gf)$ has a downward arrow, say, going
from $v=y_{\phi\psi(i-1)+1}$ to $v'=y''_i$. This means that 
$\phi\psi(i-1)<\phi\psi(i)$ which implies 
$\psi(i-1)<\psi(i)$. This means one has a vertical arrow
$y'_{\psi(i-1)+1}\to y''_i$. 
Choose $k$ minimal among those satisfying the condition 
$\phi(k-1)=\phi\psi(i-1)$.
Then there is a vertical arrow $y_{\phi\psi(i-1)+1}\to 
y'_k$. We also have arrows from $x'_j$ to $y'_j$
for $j=\psi(i-1)+1,\ldots,k-1$ as in this range $\phi(j-1)=
\phi(j)$, and also arrows from $y'_j$ to $x'_{j-1}$ for
$j=\psi(i-1)+2,\ldots,k$. This gives a path from $v$ to $v'$ in $\cF^\BM(\sigma)$. In the opposite direction,
if there is a path in $\cF^\BM(\sigma)$, connecting
$v=y_j$ to $v'=y''_i$, it should be necessarily weakly 
monotone. This implies that the path starts with a downward 
arrow, continues horizontally, and ends with another downward 
arrow. This is
only possible if $j=\phi\psi(i-1)+1$ and $\phi\psi(i)>
\phi\psi(i-1)$.

\end{proof}

\subsubsection{$\BM_X$ and variants}

We now define, for $X\in\Cat$, a functor $\cF^\BM_X: (\Delta_{/\BM})^\op\to\cS$ by the formula
\begin{equation}
\label{eq:BMX}
\cF^\BM_X(\sigma)=\Map(\cF^\BM(\sigma),X).
\end{equation}

We will prove in the next subsection that the functor 
$\cF^\BM_X$ represents a $\BM$-operad which we will denote
by $\BM_X$.

The object that mostly interests us is its $\Ass_-$-
component which we will denote by $\Ass_X$. This is 
a planar operad represented by the functor 
\begin{equation}
\cF_X(\sigma)=\Map(\cF(\sigma),X),
\end{equation}
where $\cF$ is the restriction of 
$\cF^\BM:\Delta_{/\BM}\to\Cat$ to $\Delta_{/\Ass}$ (embedded as the left copy). It also makes sense to 
consider the $\LM$-component of $\BM_X$ which we will denote by $\LM_X$. This will be an $\LM$-operad
represented by the functor
\begin{equation}
\cF^\LM_X(\sigma)=\Map(\cF^\LM(\sigma),X),
\end{equation}
where $\cF^\LM$ is the restriction of $\cF^\BM$
to $\Delta_{/\LM}$.
\begin{rem}
The operad $\Ass_X$ is an
$\infty$-categorical version of
the construction of Gepner-Hauseng ~\cite{GH}, 4.2.4, 
presented in the context of simplicial categories. Our 
picture~(\ref{eq:activearrow}) should be compared to the formula in {\sl loc. cit.} A detailed comparison of two
definitions is given in the recent paper by A. Macpherson,
see~\cite{M}.
\end{rem}

Note the following connection between the different variants of the construction. In the following lemma, as usual, $\pi:\BM\to\Ass$ is the natural projection.

\begin{lem}\label{lem:piass}
One has a natural equivalence
$$\pi^*(\Ass_X)=\BM_X\times_\BM(\BM_{X^\op})^\rev.$$
\end{lem}
\begin{proof}
Both constructions are represented by a functor 
$\Delta_{/\BM}\to\Cat$, with values in posets.
 The left-hand side is described by the composition
$$ \Delta_{/\BM}\stackrel{\pi}\to\Delta_{/\Ass}\stackrel{\cF}{\to}\Cat,$$
whereas the right-hand side is described by 
$\cG:\Delta_{/\BM}\to\Cat$ given by the formula
$$\cG(\sigma)=\cF^\BM(\sigma)\sqcup \cF^\BM(\sigma^\op)^\op.
$$
We will construct the equivalence of $\cG$ with $\cF\circ\pi$ 
as follows. We use that both functors take values in posets
and both satisfy the dual Segal conditions. 
We will first present a one-to-one correspondence
between the sets $\Ob\cG(\sigma)$ and $\Ob\cF\circ\pi(\sigma)$ 
functorial in $\sigma$. Then, in order to verify that the
presented correspondence is compatible with the partial orders, we will verify this in the case where $|\sigma|=1$ and then
we will use the dual Segal condition which ensures that
the partial order on $\Ob\cF\circ\pi(\sigma)$ and  
$\Ob\cG(\sigma)$, with $\sigma=(w_0\stackrel{f_1}{\to}
w_1\to\ldots\stackrel{f_n}{\to}w_n)$,
is uniquely determined by the partial orders on $\cF\circ\pi(f_i)$
and $\cG(f_i)$, respectively.

\

So, it is sufficient
to compare $\Ob\cF\circ\pi(\sigma)$ with $\Ob\cG(\sigma)$  and verify, for $|\sigma|=1$, the isomorphism of the corresponding posets.  

Let us first establish the bijection for $\sigma:[0]\to\BM$.
We will use the notation of ~\ref{sss:dim=0}. For $\sigma$
presented by the word $a^km^\alpha b^l$, the poset 
$\cF\circ\pi(\sigma)$
is a (discrete) set of $2(k+\alpha+l)$ black and white points,
each letter among the $k+\alpha+l$ defining one black and one white point (drawn from right to left), see the first line of (\ref{eq:a-k}). For the same $\sigma$, the poset 
$\cG(\sigma)$
is also a discrete set of $2(k+\alpha+l)$ black and white
points, as $\cF^\BM(\sigma)$ gives $2k$ points corresponding to
$a$-letters and one black point to the $m$ letter (if $\alpha=1$), see the second line of (\ref{eq:a-k}), whereas $\cF^\BM(\sigma^\op)^\op$ gives a white point for $m$ and pair of points for each $b$-letter. Note that
passage $\sigma\mapsto\sigma^\op$ interchanges $a$'s with $b$'s, whereas the passage $\cF^\BM\mapsto(\cF^\BM)^\op$ interchanges black and white points. 

Since $\Ob\cG(\sigma))=\sqcup\Ob\cG(w_i)$ and
$\Ob\cF\circ\pi(\sigma))=\sqcup\Ob\cF\circ\pi(w_i)$,
the correspondence for zero-dimensional simplices extends
functorially to all $\sigma$.

Now, the pictures presented in \ref{sss:dim=1} show that for $\sigma:[1]\to\BM$ our one-to-one correspondence is compatible with the partial order.
\end{proof}

\subsection{$\BM_X$ is a flat operad}
 
Let $X\in\Cat$. In this subsection we prove that $\BM_X$ 
defined in (\ref{eq:BMX}), is a flat $\BM$-operad. This will also imply that $\Ass_X$ is a flat planar operad.

First of all, one has the following.
\begin{lem}\label{lem:assx-cat}
$\BM_X$ is a category over $\BM$.
\end{lem}
\begin{proof}
We have to verify that $\BM_X$ is Segal and complete.
It is Segal because of the dual Segal condition of 
$\cF^\BM$, see(\ref{eq:cosegal}). To verify completeness, one has to calculate the fibers. Fix $w=a^km^\alpha b^l$. The fiber of 
$\cF^\BM$ at $w$ carries $[n]$ to $([n]^\op\sqcup[n])^k\sqcup ([n]^\op)^\alpha$~\footnote{$[n]^\op$ is, of course, isomorphic to $[n]$. But it has a different functoriality in $n$.}. Thus, the fiber of $\cF^\BM_X$
at $w$ is $(X^\op\times X)^k\times (X^\op)^\alpha$. This is
obviously a complete Segal space. 
\end{proof}

\subsubsection{Cocartesian arrows over conventional categories}
\label{sss:coc}
Let $B$ be a conventional category with no nontrivial 
isomorphisms,
and let a category $X$ over $B$ be presented by a functor 
$F:(\Delta_{/B})^\op\to\cS$. We wish to formulate a condition for an arrow
$a\in F(\alpha)$, $\alpha:[1]\to B$, to be cocartesian. 
Let $a:x\to y$ and $\alpha:\bar x\to \bar y$.
Recall that $a$ is cocartesian  if the
diagram
\begin{equation}
\label{eq:cocart-arrow}
\begin{CD}
X_{y/}@>>> X_{x/} \\
@VVV @VVV \\
B_{\bar y/} @>>> B_{\bar x/}
\end{CD}
\end{equation} 
induces an equivalence $X_{y/}\to B_{\bar y/}\times_{B_{\bar x/}}X_{x/}$ of left fibrations over $X$. 

In what follows we use the presentation of categories by simplicial spaces, so that $X_n=\Map([n],X)$. A map of left
fibrations over $X$ is an equivalence if and only if it induces an equivalence of the fibers.

This means that it is sufficient to verify that the diagram
(\ref{eq:cocart-arrow}) induces a cartesian square of the
corresponding maximal subspaces
\begin{equation}
\label{eq:cocart-arrow0}
\begin{CD}
\{y\}\times_{X_0}X_1@>>> \{x\}\times_{X_0}X_1 \\
@VVV @VVV \\
\{\bar y\}\times_{B_0}B_1@>>> \{\bar x\}\times_{B_0}B_1
\end{CD}.
\end{equation}

$B$ is discrete. Therefore, to verify that(\ref{eq:cocart-arrow}) is cartesian, it is sufficient to verify the 
equivalence of the fibers of the vertical arrows. To formulate the criterion we will use the following notation. 

Given an arrow $\beta:\bar y\to \bar z$ in $B$ we 
denote by $(\alpha,\beta)$ the $2$-simplex 
$\bar x\stackrel{\alpha}{\to}
\bar  y\stackrel{\beta}{\to}\bar z$ in $B$.

\begin{Prp}Let $B$ be a conventional category with no nontrivial isomorphisms. Let a map $f:X\to B$ in $\Cat$
be given by a functor $F:(\Delta_{/B})^\op\to\cS$. Then
an arrow $a:x\to y$ is $f$-cocartesian 
iff for any $\beta:\bar y\to \bar z$ in $B$ the map 
$\{y\}\times_{F(\bar y)}F(\beta)\to
\{x\}\times_{F(\bar x)}F(\beta\circ\alpha)$ defined as a
composition
\begin{equation}\nonumber
\{y\}\times_{F(\bar y)}F(\beta)\stackrel{a}{\to}
\{x\}\times_{F(\bar x)}F(\alpha)\times_{F(\bar y)}F(\beta)\stackrel{\sim}{\leftarrow}
\{x\}\times_{F(\bar x)}F(\alpha,\beta)\stackrel{d_1}{\to}
\{x\}\times_{F(\bar x)}F(\beta\circ\alpha)),
\end{equation}
is an equivalence.
\end{Prp}
\qed

\begin{prp}The map $p:\BM_X\to \BM$ presents $\BM_X$ as a 
$\BM$-operad.
\end{prp}
\begin{proof}
We have to prove that $\BM_X$ is fibrous over $\BM$.

1. Let us verify that inert arrows in $\BM$ have cocartesian liftings. Let $f:w'\to w$
be inert, with $w'=a^{k'}m^{\alpha'} b^{l'}$ and $w=a^km^\alpha b^l$. The space 
$\cF^\BM_X(f)$ is the product
of $2k+\alpha$ copies of $X_1$ (as in (\ref{eq:inertarrow})). The description of the functor 
$\cF^\BM_X$ presenting $\BM_X$ (\ref{eq:BMX}), (\ref{eq:Xplussigma}) together with Proposition~\ref{sss:coc} assert that
an object in the product presents a cocartesian arrow iff
all its components are equivalences. This gives a recipe
to construct a cocartesian lifting: an object in 
$\cF_X(w')$ is given by a collection of $2k'+\alpha'$ 
elements in $X_0$, and its cocartesian lifting is obtained
by replacing those corresponding to the $2k+\alpha$
positions retained by $w$, with their degeneracies in $X_1$.
 
2. Let $w=a^km^\alpha b^l$. The next step is to verify  
that $k+\alpha+l$ cocartesian liftings of the inerts 
$w\to x$, $x\in\{a,m,b\}$, give rise to the equivalence
$$(\BM_X)_w\to(\BM_X)_a^k\times(\BM_X)_m^\alpha\times(\BM_X)_b^l.$$

This also follows from the description of cocartesian liftings as presented by collections of equivalences.

3. The last condition means that, given $f:w\to u$ in $\BM$
and objects $x\in(\BM_X)_w$, $y\in(\BM_X)_u$, with a choice
of cocartesian arrows $\rho_!^i:y\to y_i$ over the inerts 
$\rho^i:u\to u_i$ decomposing $u$, one has an equivalence
\begin{equation}
\label{eq:BMX-operad-3}
\Map_f(x,y)\to\prod_i\Map_{\rho^i\circ f}(x,y_i).
\end{equation}
 
The space $\Map_f(x,y)$ is, by definition, the fiber of the restriction
$$ \Map_\BM([1],\BM_X)\to\Map_\BM(\partial[1],\BM_X)$$
at $(x,y)$, with the map $[1]\to\BM$ given by $f$. In other words, $\Map_f(x,y)$
is the fiber of the restriction
map 
\begin{equation}
\label{eq:productmaps}
\Map(\cF^\BM(f),X)\to\Map(\cF^\BM(w)\sqcup\cF^\BM(u),X).
\end{equation}

Now recall that, for $f:w\to u$ in $\BM$, $\cF^\BM(f)$
is presented by a graph, whose set of vertices is precisely
$\cF^\BM(w)\sqcup\cF^\BM(u)$. Thus, the pair $(x,y)$ is given by a collection of objects in $X$ numbered
by the set $\cF^\BM(w)\sqcup\cF^\BM(u)$, and the fiber of
(\ref{eq:productmaps}) is given by the product of the mapping spaces in $X$ between all pairs of vertices in the graph, connected by an arrow in $\cF^\BM(f)$.

The same calculation can be  applied to $\rho^i\circ f$.
The space $\Map_{\rho^i\circ f}(x,y_i)$ is similarly 
defined by the arrows of the graph 
$\cF^\BM(\rho^i\circ f)$. According to the description of
the arrows given in \ref{sss:dim=1}, each arrow from
$\cF^\BM(f)$ appears in precisely one of 
$\cF^\BM(\rho^i\circ f)$.  This proves that 
(\ref{eq:BMX-operad-3}) is an equivalence.
\end{proof}

\

Recall \ref{ss:flat} that a functor $f:\cC\to\cD$ assigns
to each $\alpha:d\to d'$ in $\cD$  a colimit preserving map $f^\alpha:P(\cC_{d'})\to P(\cC_d)$. Flatness of $f$,
according to \ref{prp:flat-corr}, means that this assignment is compatible with the compositions.  We will use this approach to prove that $p:\BM_X\to\BM$ is a flat 
$\BM$-operad. First of all, we
will give an explicit description of the map 
$p^\alpha:P((\BM_X)_w)\to P((\BM_X)_v)$ for any 
$\alpha:v\to w$ in $\BM$. 

\subsubsection{}
\label{sss:bm-fiber}
Fix $\alpha:v\to w$ in $\BM$. We will describe 
$p^\alpha:P((\BM_X)_w)\to P((\BM_X)_v)$ in terms of the
graph presenting $\cF^\BM(\alpha)$, see, for instance, 
(\ref{eq:arrow}) as illustration.

The categories $\cF^\BM(w)$ and $\cF^\BM(v)$ are discrete, 
consisting of points of $x$-type (denoted as $\circ$) and of 
$y$-type (denoted as $\bullet$). 
We assign to each $x$-type point the category $P(X)$ and to 
any $y$-type point  the category $P(X^\op)$. To a disjoint union
of points we assign the tensor product of the corresponding
$P(X)$ and $P(X^\op)$~\footnote{The tensor product of 
categories with colimits, see~\cite{L.HA}, 5.8.1.}. So we get 
the categories 
$P((\BM_X)_w)$ and $P((\BM_X)_v)$. The functor $p^\alpha$
will be defined as the tensor product of functors numbered 
by the connected components of $\cF^\BM(\alpha)$.

All components consist of one vertex or one arrow. Taking into account the type of the vertices, one has six different components.
\begin{itemize}
\item[(1)] An $x$-type single vertex. 
\item[(2)] An $y$-type single vertex.  
\item[(3)] An upward arrow.  
\item[(4)] A downward arrow.  
\item[(5)] A lower horizontal arrow.
\item[(6)] An upper horizontal arrow.
\end{itemize}
To each component $\cF$ of one of the types above, we assign an arrow $\phi_\cF$ as follows.
\begin{itemize}
\item[(1)] For $\cF$ of type (1) 
$\phi_\cF:\cS\to P(X)$ is the colimit preserving map 
carrying $*$ to the final object. 
\item[(2)] For $\cF$ of type (2) 
$\phi_\cF:\cS\to P(X^\op)$ is the colimit preserving map 
carrying $*$ to the final object.
\item[(3)] For $\cF$ of type (3)
$\phi_\cF:P(X)\to P(X)$ is the identity.
\item[(4)] For $\cF$ of type (4)
$\phi_\cF:P(X^\op)\to P(X^\op)$ is the identity.
\item[(5)] For $\cF$ of type (5)
$\phi_\cF:P(X^\op\times X)\to\cS$ is the colimit preserving
functor extending the Yoneda $X^\op\times X\to\cS$.
\item[(6)] For $\cF$ of type (6)
$\phi_\cF:\cS\to P(X\times X^\op)$ is the colimit preserving map carrying $*$ to the final object. 
\end{itemize}

We are now ready to formulate the result.
\begin{lem}
\label{lem:palpha}
Let $\alpha:v\to w$ be an arrow in $\BM$.
The functor $p^\alpha:P((\BM_X)_w)\to P((\BM_X)_v)$
is naturally equivalent to the tensor product of $\phi_\cF$
described above, over the connected components $\cF$ of the graph $\cF^\BM(\alpha)$.
\end{lem}
\begin{proof}

Denote by $\cF^\alpha$  the composition 
$$\cF^\alpha:\Delta_{/[1]}\stackrel{\alpha}{\to}
\Delta_{/\BM}\stackrel{\cF^\BM}{\to}\Cat,$$
where the functor $\cF^\BM$ is described in 
\ref{ss:cXplus}.

Note that $\cF^\alpha$ is uniquely defined by its value
$\cF^\BM(\alpha)$ at $\id_{[1]}$.

The functor
$\cF^\alpha_X:\Delta_{/[1]}^\op\to\Cat$ given by the 
formula $\cF^\alpha_X(s)=\Map(\cF^\alpha(s),X)$ defines
the base change $\BM_X^\alpha=\BM_X\times_\BM[1]$ as a 
category over $[1]$.

\

Recall that $\cF^\alpha$ has the following properties.
\begin{itemize}
\item[(D1)] It has values in disjoint unions of totally ordered posets.
\item[(D2)] For $s\in\Delta_{/[1]}$ with $|s|=0$, 
$\cF^\alpha(s)$
is discrete.
\item[(D3)] $\cF^\alpha$ satisfies a dual Segal condition.
\end{itemize}

A functor $\cF:\Delta_{/[1]}\to\Cat$ satisfying the properties (D1)--(D3) listed above will be called {\sl distinguished}. Any distinguished functor $\cF$ determines
a category $T(\cF)$ over $[1]$ represented by the
simplicial space over $[1]$ given by the formula $s\mapsto
\Map(\cF_s,X)$. Thus, it defines a colimit preserving functor 
$\phi_\cF:P(T(\cF)_1)\to P(T(\cF)_0)$, so that in our new notation $p^\alpha=\phi_{\cF^\alpha}$. Disjoint
union $\cF\sqcup\cG$ of distinguished functors gives rise to the product
of spaces over $[1]$, so to the tensor product $\phi_\cF\otimes\phi_\cG$ of the corresponding colimit preserving functors. Thus, $\phi_{\cF^\alpha}$ is the tensor product
of $\phi_\cF$ where $\cF$ runs through indecomposable
subfunctors of $\cF^\alpha$. These are determined by the connected components of the category 
$\cF^\BM(\alpha)=\cF^\alpha(\id:[1]\to[1])$.

The six possible types of components of
$\cF^\BM(\alpha)$ are described in~\ref{sss:bm-fiber}.
It remains to verify the formulas for $\phi_\cF$ in each case presented there.

In all cases apart from (5) the corresponding category over
$[1]$ is a cocartesian fibration; this can be verified, for instance, using Proposition~\ref{sss:coc}.
This gives the formulas for $\phi_\cF$ in the cases 
(1)--(4) and (6).
 
The case (5) is more interesting. The fibers at $0$ and at
$1$ of the total category $Z:=T(\cF)$ are $*$ and 
$X\times X^\op$ respectively. We have to describe the colimit preserving functor $\phi_\cF:P(X^\op\times X)\to\cS$. It is uniquely
defined by its restriction $\phi:X^\op\times X\to\cS$.

It remains to show that $\phi$ is just the canonical Yoneda 
map carrying $(x,y)$ to $\Map(x,y)$.
This is verified as follows. For any category $Z$ over 
$[1]$ with fibers $Z_0$ and $Z_1$ over $\{0\}$ and $\{1\}$ respectively the classifying map $Z_0^\op\times Z_1\to\cS$
is defined by the bifibration
$$ \Fun_{[1]}([1],Z)\to Z_0\times Z_1,$$
see~\cite{L.T}, 2.4.7.10.
In our case $Z_0$ is a point and $Z_1=X^\op\times X$ so the above bifibration becomes a left fibration
$\Fun_{[1]}([1],Z)\to X^\op\times X$. It remains to identify $\Fun_{[1]}([1],Z)$ with $\Tw(X)$. Looking at these categories as simiplicial spaces, we see that 
$\Tw(X)_n=X_{2n-1}$ and $\Fun_{[1]}([1],Z)_n=
\Map_{[1]}([1]\times[n],Z)$. Let us calculate this space.
Let $s_m:[k]\to[1]$ be given by $k-m$ values of zeros and $m+1$ ones. Denote by $[k]_m$ the $k$-simplex over $[1]$ defined by $s_m$. Then  $\Map_{[1]}([k]_m,Z)=X_{2m+1}$. The standard
decomposition of the product of simplices
$$
[1]\times[n]=[n+1]_n\sqcup^{[n]_{n-1}}
[n+1]_{n-1}\sqcup^{[n]_{n-2}}\ldots\sqcup[n+1]_0,
$$
see, for instance, \cite{GZ}, II.5.5, is 
an equivalence of categories over $[1]$.
which yields 
$$\Fun_{[1]}([1],Z)_n=
X_{2n+1}\times_{X_{2n-1}}X_{2n-1}\times\ldots\times_{X_1}X_1
=X_{2n+1}.$$

It is easy to see that these identifications are compatible with the
faces and the degeneracies in $\Delta$, which finally
identifies $\Fun_{[1]}([1],Z)$ with $\Tw(X)$.

\end{proof}

\begin{prp}
\label{prp:BMX-flat}
The map $p:\BM_X\to\BM$ is a flat $\BM$-operad.
\end{prp}
\begin{proof}
 Let  $\sigma:[2]\to\BM$ be given by a pair of composable 
arrows $\alpha$ and $\beta$, with $\beta$ active.
We can use Lemma~\ref{lem:palpha} to calculate the maps 
$p^\alpha$, $p^\beta$ and $p^{\beta\circ\alpha}$.
The map $p^{\beta\circ\alpha}$ is determined by the 
connected components of $\cF^\BM(\beta\circ\alpha)$, 
whereas the composition $p^\alpha\circ p^\beta$ is 
determined by the components of $\cF^\BM(\sigma)$.
The components of $\cF^\BM(\beta\circ\alpha)$ are determined 
by the components of $\cF^\BM(\sigma)$ in 
Lemma~\ref{sss:cF-functor}.
The only component of $\cF^\BM(\sigma)$ making an impact
different from the corresponding component of
$\cF^\BM(\beta\circ\alpha)$, is a horizontal segment
$x'_1\to y'_1$ defined by $[0]\to[1]\to[0]$.
Since $\beta$ is active, such component does not appear. 
\end{proof}

\subsection{Quivers and their action}

\subsubsection{}
\label{sss:quivs0}
For a planar operad $\cM$ we define the planar {\sl
quiver operad}
$\Quiv_X(\cM)=\Funop_\Ass(\Ass_X,\cM)$.

We will also use $\LM$ and $\BM$-versions of quiver operads.
For an $\LM$-operad $\cM$ we define 
$\Quiv^\LM_X(\cM)=\Funop_\LM(\LM_X,\cM)$.

Similarly, for a $\BM$-operad $\cM$ we define
$\Quiv^\BM_X(\cM)=\Funop_\BM(\BM_X,\cM)$.

\subsubsection{}
\label{sss:quivs}

Let $\cM=(\cM_a,\cM_m)$ be an $\LM$-operad,
with an $\Ass$-component $\cM_a\in\Op_\Ass$ and an 
$m$-component $\cM_m\in\Cat$. Then, by 
(\ref{sss:basechange}), the $\Ass$-component of the 
$\LM$-operad $\Quiv_X^\LM(\cM)$ is $\Quiv_X(\cM_a)$ and its 
$m$-component is $\Fun(X,\cM_m)$.

Similarly, for a $\BM$-operad $\cM=(\cM_a,\cM_m,\cM_b)$,
the components of $\Quiv_X^\BM(\cM)$ are
$(\Quiv_X(\cM_a),\Fun(X,\cM_m),\cM_b)$.

\begin{rem}
As our notation suggests, we consider $\Quiv_X(\cM)$ as
a primary object, and introduce $\Quiv^\LM$ and $\Quiv^\BM$ 
only to describe the action of $\Quiv_X(\cM)$ on various objects.
\end{rem}

\begin{prp}
Let  $X$  be a space. Then the category of associative algebras $\Alg_\Ass(\Quiv_X(\cM))$
is equivalent to the category $\Alg_{\Delta^\op_X}(\cM)$ of Gepner-Haugseng \cite{GH}.
\end{prp}
\begin{proof}
By definition, $\Alg_\Ass(\Quiv_X(\cM))=\Fun_{\Ass}(\Ass_X,\cM)$. D.~Gepner and R.~Haugseng define $\cM$-enriched
precategories with the space of objects $X$ as 
$\Delta_X$-algebras in $\cM$, where $\Delta_X$ is a certain generalized planar operad. They prove in~\cite{GH}, 4.2.7,
that the image of $\Delta_X$ under the localization functor
to planar operads is naturally equivalent to 
$\cO_X$, defined in \cite{GH}, 4.2.4, which is just a version of our $\Ass_X$, described in terms of simplicial categories.
For a detailed comparison of two
definitions see  A. Macpherson~\cite{M}.
\end{proof}

\subsection{Dependence on $X$ and $\cM$}
\label{ss:functoriality-quiv}

We have just assigned to any pair $(X,\cM)$, where 
$X\in\Cat$ and $\cM\in\Op_\Ass$, a planar operad 
$\Quiv_X(\cM)$, such that associative algebras in it are
$\cM$-enriched precategories. We would like to be able to
define an enriched precategory without specifying
the space of objects or even the planar operad $\cM$.

This can be done using the formalism of operad families.

\subsubsection{The $\BM$-operad family $\Quiv^\BM$}
\label{sss:quivbm}

We have a functor 
\begin{equation}
\label{eq:quiv}
\Cat^\op\times\Op_\BM\to\Op_\BM
\end{equation}
carrying $(X,\cM)$ to $\Quiv^\BM_X(\cM)$.

This functor is a composition of
\begin{equation}
\label{eq:bm}
\BM:\Cat\to\Op_{\BM}
\end{equation}
carrying $X$ to $\BM_X$,
and the bifunctor
\begin{equation}
\label{eq:funop}
\Funop:\Op_{\BM}^\op\times\Op_{\BM}\to\Op_{\BM}
\end{equation}
adjoint to the product in $\Op_{\BM}$.

The functor (\ref{eq:quiv}) defines a bifibered
$\BM$-operad family over
$\Cat\times\Op_\BM$. We denote it as $\Quiv^\BM$.

A map $j:X\to Y$ in $\Cat$ defines $j^!:\Quiv^\BM_Y(\cM)\to\Quiv^\BM_X(\cM)$. A map $f:\cM\to\cN$ of $\BM$-operads defines $f_!:\Quiv^\BM_X(\cM)\to\Quiv^\BM_X(\cN)$.

Given $j:X\to Y$ and $f:\cM\to\cN$ as above, and $P\in\Quiv^\BM_X(\cM),\ Q\in\Quiv^\BM_Y(\cN)$ the space $\Map^{j,f}_{\Quiv^\BM}(P,Q)$ of maps over $(j,f)$
is equivalent to $\Map_{\Quiv^\BM_X(\cN)}(f_!P,j^!Q)$. 

Similarly to the above, one defines the families 
$\Quiv^\LM$ and $\Quiv$.

\subsubsection{Multiplicativity}
\label{sss:mult}
The functor (\ref{eq:bm}) preserves products as it is corepresentable, when considered as a functor 
$(\Delta_{/\BM})^\op\to\cS$ and as the embedding
$\Op_{\BM}\to\Fun((\Delta_{/\BM})^\op,\cS)$ preserves 
limits by \ref{crl:fibrous-limits}.
The functor (\ref{eq:funop}) is lax symmetric monoidal,
with respect to the cartesian symmetric monoidal structure
on $\Op_{\BM}$ as it is right adjoint to the direct product
functor which is symmetric monoidal, see~\cite{H.R}, A.5.3.

Therefore, the functor 
$\Quiv^\BM:\Cat^\op\times\Op_\BM\to\Op_\BM$,
defined in~(\ref{eq:quiv}), is lax symmetric monoidal.

The latter implies, in particular, that one has a canonical  operad map
\begin{equation}\label{eq:gamma}
\mu:\Quiv^\BM_X(\cM)\times\Quiv^\BM_Y(\cN)\to
\Quiv^\BM_{X\times Y}(\cM\times\cN).
\end{equation}

In the same way, one defines lax symmetric monoidal functors
$\Quiv^\LM:\Cat^\op\times\Op_\LM\to\Op_\LM$
and
$\Quiv:\Cat^\op\times\Op_\Ass\to\Op_\Ass$.

\

One can reformulate the above multiplicativity property in terms of families of operads.

The lax symmetric monoidal functor~(\ref{eq:quiv})
defines, by adjunction, a map of operads
\begin{equation}
\quiv^\BM:\Op_\BM\to\Funop(\Cat^\op,\Op_\BM),
\end{equation} 
where $\Op_\BM$ is endowed with its cartesian SM structure. By Proposition~\ref{sss:Dcart}, the right-hand side identifies with the category $\Fam\Op_\BM$ endowed
with the cartesian SM structure. This yields a
lax SM functor
$$\quiv^\BM:\Op_\BM\to\Fam\Op_\BM$$
carrying a $\BM$-operad $\cM$ to the family 
$\Quiv^\BM(\cM)$ of $\BM$-operads. In the same way, one has lax symmetric monoidal functors
$\quiv^\LM:\Op_\LM\to\Fam\Op_\LM$
and
$\quiv:\Op_\Ass\to\Fam\Op_\Ass$.

This implies the following.
\begin{crl}
Let $\cM\in\Alg_\cO(\Op_\Ass)$. Then the category
of $\cM$-precategories $\PCat(\cM)$ has a canonical 
$\cO$-monoidal structure. In particular, if $\cM$
is an $\cE_n$-monoidal category, $\PCat(\cM)$
is a $\cE_{n-1}$-monoidal.
\end{crl}
\begin{proof}
By the additivity theorem~\cite{L.HA}, 5.1.2.2, $\cM$ 
can be seen as an $\cE_{n-1}$-algebra object in 
$\Op_\Ass$.
Since the functor $\quiv:\Op_\Ass\to\Fam\Op_\Ass$ is lax 
symmetric monoidal, $\Quiv(\cM)$ is an 
$\cE_{n-1}$-algebra in $\Fam\Op_\Ass$.

Now, the functor
$$\Alg:\Fam\Op_\Ass\to\Cat$$
assigning to a family of planar operads $\cM$ the 
category $\Alg_\Ass(\cM)$, preserves limits, 
so it carries $\cE_{n-1}$-algebras to $\cE_{n-1}$-
algebras.
\end{proof}

\

Fix $\cM\in\Op_{\BM}$.

\begin{lem}\label{lem:kappa}
The contravariant functor
$$X\mapsto\Quiv_X^\BM(\cM)$$
from $\Cat$ to $\Op_{\BM}$
carries $\kappa$-filtered colimits to limits, for a certain regular cardinal $\kappa$.
\end{lem}
\begin{proof}
The functor is a composition of several functors.
\begin{itemize}
\item The functor $X\mapsto\BM_X$ considered as a functor
$\Cat\to P(\Delta_{\BM})$ preserves $\kappa$-filtered colimits
as the corepresenting objects $\cF^\BM(\sigma)$ are $\kappa$-compact categories (for any $\kappa$).
\item The embedding $\Op_{\BM}\to P(\Delta_{\BM})$ reflects equivalences. It is a
composition 
$$\Op_{\BM}\stackrel{g}{\to}\Cat^+_{\BM^\natural}\stackrel{j}{\to}\Cat_{\BM}\stackrel{h}{\to} P(\Delta_{\BM}),$$ 
with $g,h$ accessible full embeddings and $j$ colimit preserving. Thus, this embedding preserves $\kappa$-filtered colimits for some $\kappa$.
\item Finally, we compose the above with the functor $\Funop(\_,\cM)$ carrying colimits to limits.
\end{itemize}
\end{proof}

\subsection{Folding of $\BM$}
\label{ss:folding}

\setcounter{subsubsection}{-1}
\subsubsection{}
\label{sss:foldingintro}
The aim of this subsection is to find a proper place for 
the following well-known fact:
 
\ 

{\sl $A$-$B$ bimodules are the same
as left $A\otimes B^\op$-modules.}

\

A $\BM$-monoidal category $\cC$ consists of a pair of monoidal categories $\cC_a$ and $\cC_b$ acting from the left and from the right on a category $\cC_m$. The same
action can alternatively be presented by a left 
$\cC_a\times\cC_b^\rev$ action on $\cC_m$, defining, 
therefore, an $\LM$-monoidal category. This fact,
very well-known for the conventional categories,
requires a justification in the context of $\infty$-
categories.

In this subsection we generalize and prove the above 
observation to operads, constructing a {\sl folding functor}
$$\phi:\Op_\BM\to\Op_\LM.$$
Furthermore, $\phi$ induces an equivalence of the corresponding
categories of algebras, in particular, an equivalence
between $A$-$B$ bimodules and left 
$A\boxtimes B^\op$-modules.

 This construction assigns to a $\BM$-operad $\cP$
with components $(\cP_a,\cP_m,\cP_b)$ an $\LM$-operad
with components $(\cP_a\times\cP^\rev_b,\cP_m)$.

This functor ``folds'' $\BM$ into $\LM$, similarly
to folding the simply-laced Dynkin diagram $A_{2n-1}$ into $B_n$. This subsection is not formally connected to the rest of this section. We will use it in 
Section~\ref{sec:yoneda} to construct the Yoneda embedding.

\subsubsection{Folding}
The functor $\phi$ will be expressed via the functor
\begin{equation}\label{eq:Phi}
\psi:\Delta_{/\LM}\to P(\Delta_{/\BM}) 
\end{equation}
defined in~\ref{sss:Psi}. The functor $\psi$ induces
a functor 
\begin{equation}\label{eq:phi}
\phi:P(\Delta_{/\BM})\to P(\Delta_{/\LM}),
\end{equation}
right adjoint to the colimit preserving extension 
$\hat\psi:P(\Delta_{/\LM})\to P(\Delta_{/\BM})$. This
functor carries $F\in P(\Delta_{/\BM})$ to $\phi(F)$ 
defined by the formula 
$$\phi(F)(\sigma)=\Map_{P(\Delta_{/\BM})}(\psi(\sigma),F).
$$

We present below the definition of (\ref{eq:Phi}) and,
after that, we verify that the functor (\ref{eq:phi}) 
induced by $\psi$, carries $\Op_{\BM}$ to $\Op_{\LM}$.

\subsubsection{Functor $\psi$}
\label{sss:Psi}

Recall that $\BM=(\Delta_{/[1]})^\op$
and $\LM$ is the full subcategory of $\BM$ spanned by the
$s:[n]\to[1]$ satisfying $s(0)=0$ and having at most one 
value $1$.

The order-inverting functor $\op:\Delta\to\Delta$ induces 
$\op:\BM\to\BM$
(carrying $s:[n]\to[1]$ to $s^\op:[n]^\op\to[1]^\op\stackrel{\sim}{\to}[1]$) interchanging 
the subcategories $\Ass_-$ with $\Ass_+$ and $\LM$ with $\RM$.

$\Ass_-$ is a full subcategory of $\LM$; denote by $\LM^-$ the full subcategory of $\LM$ spanned by the objects that are not in $\Ass_-$. There are no arrows in $\LM$ from an object of $\Ass_-$ to an object of $\LM^-$. 

The category $\LM^-$ is isomorphic to the subcategory
of $\BM$ consisting of the objects and the arrows invariant 
with respect to $\op$. Any object $w\in\BM$ satisfying
$w=w^\op$ is given by $s:[2n-1]\to[1]$ with $s(n-1)=0,\ s(n)=1$; the corresponding object $v$ of $\LM_-$ is given
by the composition $[n]\stackrel{i}{\to}[2n-1]\stackrel{s}
{\to}[1]$. We write $w=v^*$ in this case; the same * 
notation is used for the arrows of $\LM^-$.

Let us describe $\psi(\sigma)$ for $\sigma:[n]\to\LM$. 
Let 
$$\sigma:v_0\stackrel{f_1}{\to}\ldots\stackrel{f_n}{\to}v_n.$$
Assume that the objects $v_0,\ldots,v_m$, $m\geq -1$, are in $\LM^-$, and the rest
of $v_i$ are in $\Ass_-$~\footnote{there are no arrows from $\Ass_-$ to $\LM^-$}. We will define $\psi(\sigma)$ by the cocartesian
square
\begin{equation}
\label{eq:psisigma}
\xymatrix{
&  &{\psi(\sigma^{\leq m})}\ar[ld]\ar[rd]& & \\
& \psi_-(\sigma)\ar[dr] & & \psi_+(\sigma)\ar[dl] & \\
& & {\psi(\sigma)}& & 
}
\end{equation}
of representable presheaves in $P(\Delta_{/\BM})$ where 
\begin{itemize}
\item[]$\psi(\sigma^{\leq m}):\ v_0^*\to\ldots\to v_m^*$,
\item[]$\psi_-(\sigma):\ v_0^*\to\ldots\to v_m^*\to v_{m+1}\to\ldots\to v_n$,
\item[]$\psi_+(\sigma):\ v_0^*\to\ldots\to v_m^*\to v_{m+1}^\op\to\ldots\to v_n^\op$,
\end{itemize}
and the map $v^*_m\to v_{m+1}$ 
(resp., $v^*_m\to v^\op_{m+1}$)
is given as the composition of $f_{m+1}$ 
(resp., $f^\op_{m+1}$) with the inert $v^*_m\to v_m$ 
(resp., $v^*_m\to v^\op_m$).
In particular, if $m=n$, that is, if $\sigma$ is a simplex in $\LM^-$, $\psi(\sigma)$ is obtained by applying the functor $^*$ to $\sigma$. On the contrary,
if $m=-1$, $\psi(\sigma)$ is the coproduct of two representables, $\sigma$ and
$\sigma^\op$.

Note the following {\sl dual Segal condition} for $\psi$.

\begin{Lem}
Let $\sigma:[n]\to\LM$ be as above, $0<k<n$. Denote by
$\sigma^{\leq k}$ and $\sigma^{\geq k}$ the two halves of 
$\sigma$ of dimensions $k$ and $n-k$ respectively. Then
$$
\psi(\sigma)=\psi(\sigma^{\leq k})\sqcup^{\psi(v_k)}
\psi(\sigma^{\geq k}).
$$
\end{Lem}\qed
 
\subsubsection{}

First of all we will verify that $\phi$ carries categories
over $\BM$ to categories over $\LM$. A presheaf $F\in P(\Delta_{/\BM})$ represents a category over $\BM$ iff
it satisfies the completeness and the Segal condition.
Let us verify that $\phi(F)$ satisfies the Segal condition.
The latter means that $\phi(F)(\sigma)\to 
\phi(F)(\sigma^{\leq k})\times_{\phi(F)(v_k)}\phi(F)(\sigma^{\geq k})$ is an equivalence for all $k$, $0<k<n$. This immediately follows
from the dual Segal condition for $\psi$.

Completeness of $\phi(F)$ can be verified pointwise
 (see~\ref{sss:cat-via-functor}). 
If $v\in\Ass^-$, the corresponding fiber is $F(v)\times F(v^\op)$. Otherwise, the fiber is $F(v^*)$.  In any case, this is a complete Segal space.

\subsubsection{$\Cat$-enrichment}
Let $B$ be a category. For $X,Y\in\Cat_{/B}$ we denote
as $\Fun_B(X,Y)$ the category representing the functor
$$ K\mapsto\Map_{\Cat_{/B}}(X\times K,Y).$$

Let  $K\in\Cat$ be defined
by a simplicial space $\cK$. Then $\BM\times K$
as an object of $\Cat_{/\BM}$ is defined by the 
presheaf $(\Delta_{/\BM})^\op\to\Delta^\op
\stackrel{\cK}{\to}\cS$. Therefore, $\phi(\BM\times K)$ 
is defined by the functor carrying $\sigma\in\Delta_{/\LM}$ as in (\ref{eq:psisigma}) to $\cK_n\times_{\cK_m}\cK_n$. 

This yields a canonical morphism $\phi(X)\times K\to
\phi(X)\times\phi(\BM\times K)=\phi(X\times K)$.

This implies that $\phi:\Cat_{/\BM}\to\Cat_{/\LM}$
preserves this $\Cat$-enrichment, that is, induces
a map $\Fun_\BM(X,Y)\to\Fun_\LM(\phi(X),\phi(Y))$ for any
pair $X,Y\in\Cat_{/\BM}$.   

\subsubsection{$\phi$ carries operads to operads}
It remains to verify  that
$\phi(F)$ is fibrous if $F$ is fibrous.
The first condition is the existence of cocartesian 
liftings of the inerts. To verify it, we will use 
Proposition~\ref{sss:coc}.

Let $\alpha:u\to v$ be an inert arrow in $\LM$. 
In the case $u,v\in\Ass_-$ $\phi(F)(\alpha)=F(\alpha)\times
F(\alpha)^\op$, so we choose a pair of cocartesian liftings
in $F(\alpha)$ and  $F(\alpha^\op)$ separately. In the case
$u,v\in\LM_-$, $\alpha^*:u^*\to v^*$ is also inert and
we choose its cocartesian lifting in $F(\alpha^*)$.

 In the remaining case, $\alpha:a^nm\to a^k$,
one has
$$\phi(F)(\alpha)=F(a^nmb^n\to a^n)\times_{F(a^nmb^n)}
F(a^nmb^n\to a^n),$$
and we choose a pair of inerts in $F(a^nmb^n\to a^k)$
and in $F(a^nmb^n\to b^k)$ having
the same source in $\phi(F)(u)=F(a^nmb^n)$. The liftings chosen are cocartesian by Proposition~\ref{sss:coc}.

Segal condition for $\phi(F)$ is pretty clear.

It remains to verify the property (Fib3): given an arrow
$a:v\to w$ in $\LM$, $x\in\phi(F)(v)$ and $y\in\phi(F)(w)$,
the map
$$
\Map^a(x,y)\to\prod_i\Map^{\rho^i\circ a}(x,y_i)
$$
defined by cocartesian liftings $y\to y_i$ of $\rho^i:w\to w_i$ decomposing $w$, is an equivalence. The claim directly follows from the definition of $\phi(F)$ and from the 
property (Fib3) for the $\BM$-operad defined by the functor $F$.

\subsubsection{$\phi$ carries $\BM$-monoidal categories to $\LM$-monoidal categories}

Let $F\in P(\Delta_{/\BM})$ represent a $\BM$-monoidal category. Then  $\phi(F)\in\Op_\LM$ is an $\LM$-monoidal category. In fact, we have to verify that any arrow 
$\alpha:u\to v$ in $\LM$ has a cocartesian lifting in $\phi(F)$.
We already know this for $\alpha$ inert. For $\alpha$ active,
either $\alpha$ belongs to $\Ass_-$, or to $\LM^-$. 
In the first case $\psi(\alpha)=\alpha\sqcup\alpha^\op$,
whereas in the second case $\psi(\alpha)=\alpha^*$.
In any case
the candidate for a cocartesian lifting of $\alpha$ comes from a cocartesian lifting of $\alpha,\alpha^\op$ or 
$\alpha^*$ in $F$. Proposition~\ref{sss:coc} allows one to verify that the candidate is in fact a cocartesian lifting.

\begin{Rem}
\label{rem:phimon}
Note that the functor $\phi$ restricted to $\BM$-monoidal 
categories has an alternative (simpler) description.
A monoidal $\BM$-category is given by a functor $F:\BM\to\Cat$ satisfying Segal condition. The functor $\phi(F):\LM\to\Cat$ can be defined as the composition
$$\LM\stackrel{\Psi}{\to} P(\BM)\stackrel{F'}{\to}\Cat,$$
where $F'$ is the colimit preserving extension of $F$ and
$\Psi$  carries $v\in\Ass_-$ to $v\sqcup v^\op$ and $v\in\LM^-$ to $v^*$.
\end{Rem}
\subsubsection{}\label{sss:algfolding}
The functor $\phi:\Op_{\BM}\to\Op_{\LM}$ preserves
$\Cat$-enrichment and carries $\BM$ to 
$\LM$. Therefore, it defines
a canonical map
\begin{equation}
\label{eq:algfolding}
\Phi:\Alg_\BM(\cC)\to\Alg_{\LM}(\phi(\cC)).
\end{equation}

We have
\begin{Prp}
$\Phi$ is an equivalence.
\end{Prp}

\begin{proof}
 
Both categories of algebras are cartesian fibrations over 
$\Alg_\Ass(\cC_a)\times\Alg_\Ass(\cC_b)$ by 
\ref{prp:cartesian-algebras}. Thus,
to prove $\Phi$ is an equivalence, it is sufficient,
for any choice of $A\in\Alg_\Ass(\cC_a),\ B\in\Alg_\Ass(\cC_b)$,  
to prove that the map
$$ \Phi_{A,B}:_A\BMod_B(\cC_m)\to
\LMod_{A\boxtimes B^\op}(\cC_m)$$
is an equivalence. Look at the commutative diagram
\begin{equation}
\xymatrix{
&{_A\BMod_B(\cC_m)} \ar^{\Phi_{A,B}}[rr]
\ar^{G_\BM}[rd]&{} &{\LMod_{A\boxtimes B^\op}(\cC_m)}\ar^{G_\LM}@<1ex>[ld] \\
&{} &{\cC_m}\ar^{F_\LM}@{.>}[ru]\ar^{F_\BM}@{.>}@<1ex>[lu] &{}
}
\end{equation}
of solid arrows (the arrows $G_\BM$ and $G_\LM$ are the 
forgetful functors).

We will first verify the claim in the case $\cC$ is a
$\BM$-monoidal category.
In this case the functors $G_\BM$ and $G_\LM$ admit left
adjoint functors $F_\BM$ and $F_\LM$ of free 
$(A,B)$-bimodule and of free left $A\boxtimes B^\op$-module 
respectively. Commutativity of the solid diagram provides
a map $F_\LM\to\Phi_{A,B}\circ F_\BM$, and, therefore,
$G_\LM\circ F_\LM\to G_\BM\circ F_\BM$. This is an 
equivalence: 
if $V\in\cC_m$ and $u:V\to F_\BM(V)$ is the 
canonical map in $\cC_m$ presenting $F_\BM(V)$ as a free
$(A,B)$-bimodule, the same map will present $F_\BM(V)$ as a 
free left $A\boxtimes B^\op$-module. Finally, according 
to~\cite{L.HA}, 4.7.3.16, this implies that $\Phi_{A,B}$
is an equivalence. 

To prove the claim for general $\cC\in\Op_\BM$, we will
find a fully faithful $\BM$-operad map $\cC\to\cD$ into a 
$\BM$-monoidal category.  We proceed as follows. 
Recall that $\BM\to\BM^\otimes$ is a strong approximation.
Let $\cC'\in\Op_{\BM^\otimes}$ correspond to $\cC$ under the equivalence \ref{prp:equivalence}. Let 
$\cD'=\Env_{\BM^\otimes}(\cC')$ be the $\BM^\otimes$-monoidal 
envelope of $\cC'$. According to ~\cite{L.HA}, 2.2.4.10, 
one has a fully faithful morphism $\cC'\to\cD'$ of 
$\BM^\otimes$-operads. The equivalence 
$\Op_{\BM^\otimes}\to\Op_\BM$ is given by a base change, so 
it carries the fully faithful map $\cC'\to\cD'$ to a fully 
faithful map $\cC\to\cD$ in $\Op_\BM$.

Now, the fully faithful embedding $\cC\to\cD$ gives rise
to a cartesian square 
\begin{equation}
\nonumber
\xymatrix{
&{_A\BMod_B(\cC_m)} \ar[r]\ar^{G_\BM}[d]
&{_A\BMod_B(\cD_m)}\ar^{G_\BM}[d]\\
&{\cC_m}\ar[r]&{\cD_m} }
\end{equation}
for $_A\BMod_B(\cC_m)$, and a similar cartesian square 
for $\LMod_{A\boxtimes B^\op}(\cC_m)$. This reduces the 
assertion for general $\cC\in\Op_\BM$ to that for $\cD$.
\end{proof}

\subsubsection{} By \ref{sss:algfolding}, for $A\in\Alg_\Ass(\cC_a)$ and for $B\in\Alg_\Ass(\cC_b),$  one has an equivalence
$_A\BMod_B(\cC_m)=\LMod_{A\boxtimes B^\op}(\cC_m).$ 
Let us now assume $\cC$ is a SM category. Then $\cC_a=\cC_b=\cC_m=\cC$ and the action of $\cC\times\cC$ on $\cC$ factors through the bifunctor
$\mu:\cC\times\cC\to\cC$ defined by the SM structure on $\cC$. This implies that, according to~
\ref{prp:same-modules},
one has an equivalence
$\LMod_{A\boxtimes B^\op}(\cC)=
\LMod_{A\otimes B^\op}(\cC)$. We have verifies the following.
\begin{Crl}
\label{crl:BM=LM}
Let $A,B$ be associative algebras in a symmetric monoidal category $\cC$. Then one has a natural equivalence
$$_A\BMod_B(\cC)=\LMod_{A\otimes B^\op}(\cC).$$ 
\end{Crl}

In particular, for $\cC=\Cat$ we justify the 
$\infty$-categorical version of the claim made in~\ref{sss:foldingintro}.

\subsubsection{}

The (symmetric) operads $\LM^\otimes$ and 
$\RM^\otimes$ are equivalent: there is an obvious
equivalence of $\LM$ and $\RM$-algebras with values
in any $\cC\in\Op$, carrying an $\LM$-algebra
$(A,M)$ to the $\RM$-algebra $(M,A^\op)$. Similarly,
the operad $\BM^\otimes$ governing triples $(A,M,B)$
consisting of two algebras and a bimodule, is 
equivalent to the operad $\LLM$ governing a pair
of algebras $A,C$ acting on the left on $M$, so that 
the actions commute with each other. The equivalence 
carries a triple $(A,M,B)$ the triple $(A,B^\op,M)$.

Using these equivalences, we can deduce from
(\ref{eq:BM=RMLM}) and Corollary~\ref{crl:BM=LM}
the following.
\begin{Crl}
\label{crl:LMod=LModLMod}
Let $\cM$ be left-tensored over monoidal categories
$\cA$ and $\cB$, so that the two actions commute.
Let $A$ be an algebra in $\cA$, $B$ an algebra in 
$\cB$, and denote $A\boxtimes B$ the corresponding algebra in $\cA\times\cB$. Then there is a
canonical equivalence of the categories of modules
$$\LMod_{A\boxtimes B}(\cM)=\LMod_A(\LMod_B(\cM)).$$

\end{Crl}

\subsubsection{}
\label{sss:phir} Let $X$ be a fixed category and let
an operad $\cP\in\Op_\BM$ be given by the functor
$\cP(\sigma)=\Map(\cY(\sigma),X)$ for a certain functor
$\cY:\Delta_{/\BM}\to\Cat$ (the operads $\BM_X$
and the relatives constructed above have this form). Then $\phi(\cP)$
is defined by the formula
$\phi(\cP)(\sigma)=\Map(\cZ(\sigma),X)$ where $\cZ$ is the 
composition
$$ \Delta_{/\LM}\stackrel{\psi}{\to}P(\Delta_{/\BM}) 
\stackrel{\cY}{\to}\Cat,$$
where $\cY$ is extended from $\Delta_{/\BM}$ to preserve colimits.

\begin{prp}\label{prp:phi}
\vbox{\ }
\begin{itemize}
\item[1.] The functor $\phi:\Op_\BM\to\Op_\LM$ preserves limits.
\item[2.] For $\cP\in\Op_\BM$ one has 
$\phi(\cP)=\phi(\cP^\rev)$.
\item[3.] $\phi(\BM_X)=\LM_X$.
\end{itemize}
\end{prp}
\begin{proof}
The functor $\phi:\Cat_{/\BM}\to\Cat_{/\LM}$ has a left 
adjoint, and so preserves limits. The embedding $\Op_\LM\to
\Cat_{/\LM}$ creates limits as it is conservative and 
is a composition of two right adjoint functors,
$$\Op_\LM\to\Cat^+_{/\LM^\natural}\to\Cat_{/\LM},$$
the second being the functor forgetting the markings.
The second claim is an immediate consequence of the 
definition. To prove the third claim, we use Remark
\ref{sss:phir}. The functor $\phi(\BM_X)$ is represented
by the composition
$$ \Delta_{/\LM}\stackrel{\psi}{\to}P(\Delta_{/\BM})
\stackrel{\cF^\BM}{\to}\Cat,
$$
which is easily seen to coincide with $\cF^\LM$.
\end{proof}

\begin{crl}\label{crl:phib}
\vbox{\ }
\begin{itemize}
\item[1.] $\phi(\BM_X\times_\BM(\BM_Y)^\rev)=\LM_{X\times Y}$.
\item[2.] In particular, $\phi(\pi^*(\Ass_X))=
\LM_{X\times X^\op}$.
\end{itemize}
\end{crl}
\begin{proof}
The first claim directly follows from  Proposition~\ref{prp:phi}. The second claim is a special case
of the first, with $Y=X^\op$, joined with Lemma~\ref{lem:piass}.
\end{proof}

\subsubsection{}
\label{sss:phifunop}
Since $\phi$ preserves products, one has
a canonical arrow
\begin{equation}
\phi(\Funop_\BM(\cP,\cQ))\to\Funop_\LM(\phi(\cP),\phi(\cQ)).
\end{equation}
Applying this to $\cP=\pi^*(\Ass_X)$ and $\cQ=\pi^*(\cM)$, and taking into account \ref{crl:phib} (2), we get a canonical map
\begin{equation}
\label{eq:phiquiv}
\phi\pi^*(\Quiv_X(\cM))\to\Quiv^\LM_{X\times X^\op}(\phi\pi^*(\cM))
\end{equation}
of $\LM$-operads.
Note that the map \ref{eq:phiquiv}  is a morphism of functors from 
$\Cat^\op\times\Op_{\Ass}$ to $\Op_{\LM}$.

\section{$\Quiv_X(\cM)$ when $\cM$ is a monoidal category}
\label{sec:quivers-mon}

In reasonably good cases $\Quiv_X^\BM(\cM)$ 
is a $\BM$-monoidal category. In this section (Theorem~\ref{thm:Quiv-monoidal}) we will show
that this happens when $\cM$ itself is a $\BM$-monoidal category, 
with the monoidal structure behaving well with respect to 
certain colimits.  A similar result holds for the $\LM$-version. Furthermore, in 
Section~\ref{ss:Quiv-end} we identify, for a monoidal category $\cM$ with colimits, $\Quiv_X(\cM)$ with the monoidal category of 
endomorphisms of $\Fun(X,\cM)$ considered as a right 
$\cM$-module.

We will proceed as follows. First of all, we describe the category of colors of $\Quiv^\BM_X(\cM)$, that is the 
fibers of $p:\Quiv^\BM_X(\cM)\to\BM$ at the objects of 
$\BM_1=\{a,m,b\}$.

In order to prove the theorem, we describe local 
cocartesian liftings of the active arrows in $\BM$
for $q:\Quiv^\BM_X(\cM)\to\BM$. 
This will allow us to see that, under certain conditions on 
$\cM$, such local cocartesian
liftings exist and commute with the compositions.

\subsection{Colors of $\Quiv^\BM_X(\cM)$}
 
Let us  describe the fibers of $\Quiv^\BM_X(\cM)$ for 
$\cM\in\Op_\BM$ at $a,m,b\in\BM$. 

The fibers of $\BM_X$ at $a,m$ and $b$ in $\BM$ respectively are $X^\op\times X$, $X$, 
and $[0]$.

By definition, $\Quiv^\BM_X(\cM)_a$ is
\begin{equation}
\nonumber\Alg_{\Ass_-^\circ/\BM}(\Quiv^\BM_X(\cM))=
\Alg_{\Ass_-^\circ\times_{\BM}\BM_X/\Ass}(\Ass_-\times_\BM\cM)=
\Alg_{\Ass_X^\circ/\Ass}(\Ass_-\times_\BM\cM),
\end{equation}
see~\ref{crl:alg-c0c}.
The latter easily yields
\begin{lem}\label{lem:quiva}
 $\Quiv^\BM_X(\cM)_a=\Fun(X^\op\times X,\cM_a).$
\end{lem}
\qed

In the same way one obtains
\begin{lem}\label{lem:quivmb}
$\Quiv^\BM_X(\cM)_m=\Fun(X,\cM_m)$. \ 
$\Quiv^\BM_X(\cM)_b=\cM_b.$
\end{lem}
\qed

The same formulas describe the colors of $\Quiv_X^\LM(\cM)$ 
for $\cM\in\Op_\LM$.
\begin{lem}\label{lem:quivLM}
$\Quiv^\LM_X(\cM)_a=\Fun(X^\op\times X,\cM_a)$. \ 
$\Quiv^\LM_X(\cM)_m=\Fun(X,\cM_m).$
\end{lem}
\qed

\subsection{Spaces of active maps}
Our next step is to describe the mapping spaces in 
$\Quiv^\BM_X(\cM)$ over active arrows of $\BM$. The 
description is given in Proposition~\ref{prp:active}. 

We proceed as follows. We assume that the active arrow $\alpha:w\to u$ in $\BM$ lies over 
$\langle n\rangle\to\langle 1\rangle$ in $\Ass$, as in general our mapping spaces will be products
of mapping spaces over such $\alpha$. 

Fix an object $g$ in $\Quiv^\BM_X(\cM)_u$ and an object 
$f=(f_1,\ldots,f_n)$ in $\Quiv^\BM_X(\cM)_w$. We wish 
to describe the space
\begin{equation}\label{eq:mapalpha} 
\Map^\alpha_{\Quiv^\BM_X(\cM)}(f,g)
\end{equation}
of arrows from $f$ to $g$ over $\alpha$. 

Denote $A=\alpha^*(\BM_X)$, $M=\alpha^*(\cM)$. These are
categories over $[1]$. Our first (quite straightforward) step will be 
to identify (\ref{eq:mapalpha}) with the fiber of the 
restriction map
\begin{equation}
\label{eq:fibermapAM}
\Fun_{[1]}(A,M)\to\Fun(A_0,M_0)\times\Fun(A_1,M_1),
\end{equation}
at $(f,g)$. Here $A_i,\ M_i, i=0,1,$ are the fibers of $A,\ M$ at the 
ends of $[1]$.

In order to calculate this fiber, we will find a special presentation of the category $A$ over $[1]$ as a certain colimit. This is done as follows. For each
$\alpha$  we find a  category $C$ with a pair of maps 
$p:C\to A_0$, $q:C\to A_1$, so that one has 
an equivalence
\begin{equation}\label{eq:Aeq}
\Theta:A_0\sqcup^C(C\times[1])\sqcup^CA_1\to A.
\end{equation}
Presentation of $A$ as colimit (\ref{eq:Aeq}) is given in~\ref{ss:A-eq-colim}.

Proposition~\ref{prp:active} below easily follows from this presentation.

\begin{prp}
\label{prp:active}
Let $\alpha:w\to u$ be an active arrow in $\BM$
over $\langle n\rangle\to\langle 1\rangle$. Let
$f\in\Quiv^\BM_X(\cM)_w$ and 
$g\in\Quiv^\BM_X(\cM)_u$. Then
\begin{equation}
\label{eq:prp-active}
\Map^\alpha_{\Quiv^\BM_X(\cM)}(f,g)=
\Map_{\Fun(C,M)}(p\circ f,q\circ g),
\end{equation}
where, as above,  $M=\alpha^*(\cM)$, and $(C,p,q)$ are
described in~\ref{sss:alpha-cases}.
The same description holds for the space of active arrows
in $\Quiv^\LM_X(\cM)$ and $\cM\in\Op_\LM$.
\end{prp}
\begin{proof}
Taking into account the presentation (\ref{eq:Aeq}), the formula (\ref{eq:fibermapAM})
can be rewritten as the fiber of the restriction
$$
\Fun_{[1]}(C\times[1],M)\to\Fun(C,M_0)\times
\Fun(C,M_1)
$$
at $(f\circ p,f\circ q)$. This fiber is easily identified with
the right-hand side of the formula (\ref{eq:prp-active}).
\end{proof}

\subsubsection{}
\label{sss:active-1}
Recall \ref{sss:Qn} that $C_n$ denotes the 
free planar operad generated by one $n$-ary operation.
This means that the colors of $C_n$ are numbered by $1,\ldots,n,0$, with the operations generated by the only
element of $C_n((1,\ldots,n),0)$. 
The map $\alpha$ in $\BM$ defines a unique map $C_n\to\BM$
which we denote as $\wt\alpha$. The operadic map $C_n^\circ\to\Quiv^\BM_X(\cM)$ is given by a choice of
a pair of objects in $\Quiv^\BM_X(\cM)$, one over $\langle n\rangle$ and another over $\langle 1\rangle$. Thus, our
mapping space (\ref{eq:mapalpha}) can be described as the fiber
of the restriction map
\begin{equation}
\label{eq:mapalpha2}
\Alg_{C_n}(\Quiv^\BM_X(\cM))\to\Alg_{C^\circ_n}(\Quiv^\BM_X(\cM))
\end{equation}
at $(f,g)$.

We can replace $C_n$ in the above formula with its strong approximation 
$Q_n\to C_n$, and $C_n^\circ$ with $Q_n^\circ$, see~\ref{sss:Qn}.
Since $\Quiv_X^\BM(\cM)=\Funop_{\BM}(\BM_X,\cM)$, this allows one to rewrite (\ref{eq:mapalpha2}) as the fiber
of
\begin{equation}\label{eq:map-of-alg}
\Alg_{Q_n\times_\BM\BM_X/Q_n}(Q_n\times_\BM\cM)\to
\Alg_{Q^\circ_n\times_\BM\BM_X/Q^\circ_n}(Q^\circ_n\times_\BM\cM)
\end{equation}
 at $(f,g)$.
 
\subsubsection{}
Let $\mu:[1]\to Q_n$  be the only active arrow in $Q_n$. 
A $Q_n$-operad is uniquely described by its base change $X\to [1]$
with respect to $\mu$, together with a decomposition of $X_0$, the fiber at $\{0\}\in[1]$, into a product $X_0=\prod_{i=1}^n X_{0,i}$. This implies that,
given two $Q_n$-operads presented by categories $A$ and $M$ over $[1]$, the
space of active maps from $f:A_0\to M_0$ to $g:A_1\to M_1$ is given by the fiber
of the map
$$ 
\Fun_{[1]}(A,M)\to
\Fun(A_0,M_0)\times\Fun(A_1,M_1)
$$
at $(f,g)$.
 
We now apply the above reasoning to $A=\alpha^*(\BM_X)$ and 
$M=\alpha^*(\cM)$. We deduce the description of (\ref{eq:mapalpha}) as the fiber at $(f,g)$ of the map
(\ref{eq:fibermapAM}).

\subsection{$A$ as a colimit}
\label{ss:A-eq-colim}
In this subsection we construct an equivalence
(\ref{eq:Aeq}) for all values of $\alpha$.
We distinguish four different cases for $\alpha$ listed
in  \ref{sss:alpha-cases}. In each one of the cases 
we provide formulas for $C$ and for the maps $p:C\to A_0$
and $q:C\to A_1$.

To get the equivalence (\ref{eq:Aeq}), we use the fact that
both sides are described by formulas  ``universal in $X$''.
The latter means the following. Both sides of the 
equivalence are categories over $[1]$ and so can be 
presented by a functor $F:(\Delta_{/[1]})^\op\to\cS$.
It is of the form 
 $F(\sigma)=\Map(\cY(\sigma),X)$ where 
$\cY:\Delta_{/[1]}\to\Cat$ has values in conventional
categories (presented by finite posets), and is independent 
of $X$. So, the task of construction of equivalence
(\ref{eq:Aeq}) reduces to a comparison of finite posets.

\subsubsection{}
\label{sss:alpha-cases}
The active arrow $\alpha:w\to u$ appearing in the definition
of $A$, is uniquely defined by its source which is an object 
of $\BM_{\langle n\rangle}$. We distinguish below the following cases.
\begin{itemize}
\item[(w0)] $w$ is presented by $\sigma:[n]\to[1]$,$n>0$,
 having the constant value $0$.
\item[(w00)] $w$ is presented by $\sigma:[0]\to[1]$,
$\sigma(0)=0$.
\item[(w1)] $w$ is presented by $\sigma:[n]\to[1]$ having the constant value $1$.
\item[(w2)] $w$  is presented by $\sigma:[n]\to[1]$ such that $\sigma(0)=0,\ \sigma(n)=1$. In this case let $k$ 
be such that  $\sigma(i)=0$ for  $i\leq k$ and $\sigma(i)=1$ for
$i>k$ ($k<n$).
\end{itemize}

We will now present formulas for $C$, $p:C\to A_0$ and $q:C\to A_1$ for each of the types of $\alpha$ separately.
 
\subsubsection*{The case {\rm(w0)}}

In this case $w=a^n\in\BM_{\langle n\rangle}$ and $u=a\in\BM_{\langle 1\rangle}$.

We have $A_0=(X^\op\times X)^n,\ A_1=X^\op\times X$.

We define the category $C=X^\op\times(\Tw(X)^\op)^{n-1}\times X$
and a pair of arrows $p:C\to A_0,\ q:C\to A_1$, as follows.

The map $p$ is induced by the $n-1$ projections $\Tw(X)^\op\to X\times X^\op$, see~\ref{eq:Tw-right},  whereas $q$ is the projection to the first and the last factors.

\subsubsection*{The case {\rm(w00)}}
Here we have $A_0=[0]$, $A_1=X^\op\times X$, $C=\Tw(X)$.
The map $q:C\to A_1$ is the canonical map (\ref{eq:Tw})
defining Yoneda embedding.

\subsubsection*{The case {\rm(w1)}}

Here $A_0=A_1=[0]$ and we put $C=[0]$.

\subsubsection*{The case {\rm(w2)}}
In this case $w=a^kmb^{n-k-1}\in\BM_{\langle n\rangle}$ and $u=m\in\BM_{\langle 1\rangle}$.
We have $A_0=X\times(X^\op\times X)^k$ and $A_1=X$.
We define $C= (\Tw(X)^\op)^k\times X$.

The map $p$ is induced by the $k$ projections 
$\Tw(X)^\op\to X\times X^\op$,   whereas 
$q$ is the projection to the last factor.

\subsubsection{}

The category $A$ is described by a functor $F_A:(\Delta_{/[1]})^\op\to\cS$ which is a restriction of $\BM_X$, 
with $F_A(\sigma)=\Map(\cF^\BM(\sigma),X)$ where
$\cF^\BM(\sigma)$ are presented by the diagrams
(\ref{eq:activesimplex-an}) and (\ref{eq:LMsimplex}).

We will calculate the functor $F:(\Delta_{/[1]})^\op\to\cS$ describing the colimit
$A_0\sqcup^C(C\times[1])\sqcup^CA_1$ and compare it to 
$F_A$. 

In all cases appearing in \ref{sss:alpha-cases}, apart from 
(w00), the map $p:C\to A_0$ is a right fibration. 
In the case (w00), the map $q:C\to A_1$ is a left 
fibration. This is what makes the calculation
easy.

\subsubsection{The calculation of $F$}
 Let
\begin{equation}
\label{eq:A0CA1}
A_0\stackrel{p}{\longleftarrow}C\stackrel{q}{\longrightarrow}A_1 
\end{equation}
be a diagram with $p$ a right fibration, and let 
$B=A_0\sqcup^C(C\times[1])\sqcup^CA_1$. 
Denote $D=(C\times[1])\sqcup^CA_1$, so that 
$B=A_0\sqcup^CD$.
We denote by $A_{0m},A_{1m}, C_m$ the $m$-components of
the presentation of $A_0,A_1,C$ as simplicial spaces.
The map $D\to[1]$ is a cocartesian fibration, so it is easy
to describe its representative $F_D$ in $\Fun((\Delta_{/[1]})^\op,\cS)$. For $\sigma:[m]\to[1]$ the pullback $[m]\times_{[1]}D$ is the iterated cylinder corresponding to the sequence
$$ C\stackrel{\id}{\to} C\to\ldots\to C\stackrel{q}{\to}  A_1\to\ldots  \stackrel{\id}{\to} A_1,$$
so that 
\begin{equation}
F_D(\sigma)=\begin{cases}
C_m, & \sigma(i)=0\textrm{ for all }i,\\
A_{1m},& \sigma(i)=1\textrm{ for all }i,\\
C_a\times_{A_{10}}A_{1b},
&\sigma=\{0^{a+1}1^b\},m=a+b,
\end{cases}
\end{equation}
see~\cite{H.L}, 9.8.6.

Let us now describe $B=A_0\sqcup^CD$. In general, a colimit in $\Cat_{/[1]}$ can be expressed as a colimit in presheaves on $\Delta_{/[1]}$, followed by the localization functor $L:\Fun((\Delta_{/[1]})^\op,\cS)\to\Cat_{/[1]}$.

Therefore, $B=L(F_{B'})$ where $F_{B'}=A_{0\bullet}
\sqcup^{C_\bullet}F_D$ is the
colimit in $\Fun((\Delta_{/[1]})^\op,\cS)$. It is very easy to calculate  $F_{B'}$. One has
\begin{equation}
\label{eq:B=}
F_{B'}(\sigma)=\begin{cases}
A_{0m}, & \sigma(i)=0\textrm{ for all }i,\\
A_{1m},& \sigma(i)=1\textrm{ for all }i,\\
C_a\times_{A_{10}}A_{1b},
&\sigma=\{0^{a+1}1^b\},m=a+b.
\end{cases}
\end{equation}
Fortunately, in the case where $p:C\to A_0$ is a right 
fibration, $F_{B'}$ satisfies completeness and Segal 
conditions. Therefore, no localization is needed and 
$B=B'$.  

\

In the case of diagram~(\ref{eq:A0CA1}) with $q$  a left fibration, we proceed dually, presenting the colimit in question as $B=D\sqcup^CA_1$, where 
$D=A_0\sqcup^C(C\times[1])$. In this case $D\to[1]$ is a
cartesian fibration and we can repeat the above calculation. As a result, we get
\begin{equation}
\label{eq:B=dual}
F_{B'}(\sigma)=\begin{cases}
A_{0m}, & \sigma(i)=0\textrm{ for all }i,\\
A_{1m},& \sigma(i)=1\textrm{ for all }i,\\
A_{0a}\times_{A_{00}}C_b,
&\sigma=\{0^a1^{b+1}\},m=a+b.
\end{cases}
\end{equation}

\subsubsection{} 
In all cases described in \ref{sss:alpha-cases},
 apart of (w00), the map
$p:C\to A_0$ is a right fibration. In the case (w00) the map $q$ is a left fibration.  Thus, we can apply formulas (\ref{eq:B=}) and (\ref{eq:B=dual})
to the calculation of the colimits.

We will make a calculation separately
for different values of $\alpha$.

\subsubsection{The case {\rm(w0)}}
\label{sss:alpha-1}
Here $A_0=(X^\op\times X)^n$, $A_1=X^\op\times X$,
$C=X^\op\times\Tw(X)^{n-1}\times X$. The formula
(\ref{eq:B=}) yields
\begin{equation}
\label{eq:B-an}
F(\sigma)=\begin{cases}
(X_m^\op\times X_m)^n, & \sigma(i)=0\textrm{ for all }i,\\
X_m^\op\times X_m,& \sigma(i)=1\textrm{ for all }i,\\
(X_m)^\op\times\Tw(X)_a^{n-1}\times X_m,
&\sigma=\{0^{a+1}1^b\},m=a+b.
\end{cases}
\end{equation}
\subsubsection{The case {\rm(w00)}}
\label{sss:alpha-1-0}
Here $A_0=[0]$, $A_1=X^\op\times X$,
$C=\Tw(X)$. The formula
(\ref{eq:B=dual}) yields
\begin{equation}
\label{eq:B-an-0}
F(\sigma)=\begin{cases}
[0], & \sigma(i)=0\textrm{ for all }i,\\
X_m^\op\times X_m,& \sigma(i)=1\textrm{ for all }i,\\
\Tw(X)_b,&\sigma=\{0^a1^{b+1}\},m=a+b.
\end{cases}
\end{equation}
\subsubsection{The case {\rm(w1)}}
Here obviously $F(\sigma)=[0]$.
\subsubsection{The case {\rm(w2)}}
\label{sss:alpha-3}
Here $A_0=(X^\op\times X)^k\times X^\op$, $A_1=X^\op$,
$C=X^\op\times\Tw(X)^k$. The formula
(\ref{eq:B=}) yields
\begin{equation}
\label{eq:B-akmbl}
F(\sigma)=\begin{cases}
(X_m^\op\times X_m)^k\times X^\op_m, & \sigma(i)=0\textrm{ for all }i,\\
X_m^\op ,& \sigma(i)=1\textrm{ for all }i,\\
(X_m)^\op\times\Tw(X)_a^k,
&\sigma=\{0^{a+1}1^b\},k=a+b.
\end{cases}
\end{equation}

\subsubsection{Conclusion}
The formulas (\ref{eq:B-an}), (\ref{eq:B-an-0}), (\ref{eq:B-akmbl}) 
have form $F(\sigma)=\Map(\cY(\sigma),X)$ where $\cY:\Delta_{/[1]}\to
\Cat$ takes values in finite posets.Comparing these formulas  with the pictures  (\ref{eq:activesimplex-an}) and (\ref{eq:LMsimplex}), we see that $\cY(\sigma)=
\cF^\BM(\sigma)$ and, therefore, $F=F_A$.

\subsection{Categories with colimits}
In what follows we will need a few basic facts about 
categories with colimits which can be found in \cite{L.T}
and \cite{L.HA}. We present them below.

\subsubsection{}
We fix a set of categories $\cK$. The category 
$\Cat^\cK$ is defined as the subcategory of $\Cat$ whose 
objects are categories having $\cK$-indexed colimits, and 
morphism preserving these colimits. The category $\Cat^\cK$  
has a symmetric monoidal structure induced from the cartesian 
structure on $\Cat$, see~\cite{L.HA}, 4.8.1.3, 4.8.1.4. The latter means that tensor product $\cC=\otimes_{i=1}^n\cC_i$ is defined by a universal map  
$$ F:\cC_1\times\ldots\times\cC_n\to\cC$$
preserving $\cK$-indexed colimits
along each argument. 

Let $\cC$ be a small category. The category $P^\cK(\cC)$ 
is defined as the smallest full subcategory of $P(\cC)$ 
containing the image of $\cC$ and closed under $\cK$-
indexed colimits, see~\cite{L.T}, 5.3.6.2 (our 
$\cP^\cK(\cC)$ is $\cP^\cK_\cR(\cC)$ with $\cR=\emptyset$).
 The embedding $Y:\cC\to P^\cK(\cC)$
is fully faithful and it is universal among functors
from $\cC$ to categories having $\cK$-indexed colimits.

The unit object in $\Cat^\cK$ is $\cS^\cK=P^\cK([0])$.
This is the smallest full subcategory of spaces
containing $[0]$ and closed under $\cK$-colimits.

In what follows we need a certain smallness property
of categories.

\begin{dfn}
\label{dfn:small}
Let $\cK$ be a set of categories.
A category $\cC$ is strongly $\cK$-small if
\begin{itemize}
\item $\Map_\cC(x,y)\in\cS^\cK$ for all $x,y\in\cC$.
\item For any $F\in\Fun(\cC^\op,\cS^\cK)$ the total category
of $F$ considered as right fibration over $\cC$, is in 
$\cK$.
\item For any $n$ the category $\Tw(\cC)^n$ is in $\cK$.
\end{itemize}
\end{dfn}

\begin{lem}
Let $\cC$ be strongly $\cK$-small. Then 
 $P^\cK(\cC)=\Fun(\cC^\op,\cS^\cK)$. 
\end{lem}
\begin{proof}
The Yoneda lemma yields a functor 
$Y:\cC^\op\times\cC\to\cS$ whose essential image is in $\cS^\cK$. This can be interpreted as a functor
$\cC\to\Fun(\cC^\op,\cS^\cK)$. The universal property
of $P^\cK$ yields a canonical functor $j:P^\cK(\cC)\to\Fun(\cC^\op,\cS^\cK)$. The composition of this with the obvious embedding to $P(\cC)$ is a full embedding. Therefore,
$j$ is a full embedding. It remains to verify that $j$ is
essentially surjective. Let $F:\cC^\op\to\cS^\cK$.
Grothendieck construction converts $F$ into a right 
fibration $\cF$ over $\cC$ which belongs to $\cK$. $F$ is 
the colimit of the composition $\cF\to \cC\to P^\cK(\cC)$,
which proves $j$ is essentially surjective.
\end{proof}

The following result is standard; it is a minor modification of \cite{L.HA}, 4.6.1.6.

\begin{lem}
\label{lem:dualobject}
Let $\cC$ be a closed monoidal category and let 
$B\in\cC$. Put $C=\hom(B,\one)$ and let $e:B\otimes C\to\one$ be the canonical evaluation map. Then the following properties
are equivalent.
\begin{itemize}
\item[1.] $e$ is a counit of the  duality.
\item[2.] For any $A\in\cC$  the map
$
\label{eq:dualobject}
A\otimes B\to
\hom(C,A)$
adjoint to $A\otimes B\otimes C\stackrel{e}{\to} A\otimes\one=A$,
is an equivalence.
\item[3.] Property 2 is valid for $A=C$.
\end{itemize}
\end{lem}
\begin{proof}This is a direct consequence of the proof
of \cite{L.HA}, 4.6.1.6. Note that Property 2 for $A=C$
provides the unit of the duality $c:\one\to C\otimes B$.
\end{proof}

The symmetric monoidal category $\Cat^\cK$ is closed,
with the internal mapping object assigning to
 $(\cA,\cB)$ the full subcategory $\Fun^\cK(\cA,\cB)$  of $\Fun(\cA,\cB)$ spanned by the functors
preserving $\cK$-indexed colimits, see \cite{L.HA}, 4.8.1.6. In particular, for $B=P^\cK(\cC)$ one has
$C:=\Fun^\cK(B,\cS^\cK)=\Fun(\cC,\cS^\cK)=
P^\cK(\cC^\op)$.

\begin{lem}
Let $\cC$ be strongly $\cK$-small. Then $P^\cK(\cC)$ is a 
dualizable object in $\Cat^\cK$ whose dual is 
$P^\cK(\cC^\op)$.
\end{lem}
\begin{proof}
We will verify that the map $e:P^\cK(\cC)\otimes P^\cK(\cC^\op)\to\cS^\cK$ defined above satisfies \ref{lem:dualobject}. It is sufficient to verify the condition 2 of of the lemma for $A=P^\cK(\cD)$.

We have to verify that the map 
$$P^\cK(\cD)\otimes 
P^\cK(\cC)\to\Fun^\cK(P^\cK(\cC^\op),P^\cK(\cD))
=\Fun(\cC^\op,P^\cK(\cD))=P^\cK(\cD\times\cC)$$
is an equivalence.

This is a colimit-preserving map and its restriction to
$\cD\times\cC$ easily identifies with the Yoneda embedding.
This proves the assertion.
\end{proof}

The counit of the duality
\begin{equation}
\label{eq:counit}
e:P^\cK(\cC)\otimes P^\cK(\cC^\op)\to\cS^\cK
\end{equation}
can be described as the colimit-preserving functor
induced by the Yoneda map $\cC\times\cC^\op\to\cS^\cK$
for $\cC^\op$. The unit of the duality is the map
$c:\cS^\cK\to P^\cK(\cC^\op)\otimes P^\cK(\cC)$ preserving colimits and the terminal object.

\subsubsection{Examples}
\begin{itemize}
\item Let $\alpha$ be a regular uncountable
cardinal. We say that a space has size $\alpha$ if it can be presented by a simplicial set with less than $\alpha$ simplices. A category $X$ has size $\alpha$ iff the spaces $X_n$ have size $\alpha$ for all $n$. Let
$\cK$ consist  of all categories of size $\alpha$. 
In this case $X$ is strongly $\cK$-small iff 
$X\in\cK$.

\item The same is true if $\cK$ is the collection of 
all spaces of size $\alpha$.
\end{itemize}

\subsubsection{}
Let $\cO$ be an operad and $\cC\to\cO$ a strong approximation. A $\cC$-algebra object in $\Cat^\cK$ will be 
called {\sl a $\cC$-monoidal category with $\cK$-indexed colimits.}

\

Here is our first important result.

\begin{thm}\label{thm:Quiv-monoidal}
Let $\cM$ be a $\BM$-monoidal category with $\cK$-indexed colimits
and let  $X$ be strongly $\cK$-small as defined in \ref{dfn:small}. 
Then $\Quiv^\BM_X(\cM)$ is a $\BM$-monoidal category.
The same claim holds with  $\BM$ replaced by $\LM$ or 
$\Ass$.
\end{thm}
\begin{proof}
We will verify that any active arrow in $\BM$ has a locally 
cocartesian lifting, and that locally cocartesian liftings 
are closed under composition.

It is sufficient to study the active arrows $\alpha:w
\to u$ over
$\langle n\rangle\to\langle 1\rangle$ in $\Ass$. 
Given 
$f\in\Quiv^\BM_X(\cM)_w,\  
g\in\Quiv^\BM_X(\cM)_u$, the space  
$\Map^\alpha(f,g)$ is given by the formula 
(\ref{eq:prp-active}) of Proposition~\ref{prp:active}.  
This formula implies that the locally cocartesian lifting 
of $\alpha$ is given by a left Kan extension
of $p\circ f$  with respect to the projection $q:C\to A_1$.
 
If $\cM$ has $\cK$-indexed colimits, this Kan extension 
exists. If the monoidal structure on $\cM$ preserves $\cK$-indexed colimits in each argument, the composition of
locally cocartesian liftings will be locally cocartesian.
This proves the theorem.
\end{proof}

Note that  $\Quiv^\BM_X(\cM)$ is a $\BM$-monoidal category
with $\cK$-colimits.

\begin{crl}
Let $\cM\in\Alg_\BM(\Cat^\cK)$ and let $X$ be strongly 
$\cK$-small. Then $\Quiv_X^\BM(\cM)\in\Alg_{\BM}(\Cat^\cK)$.
\end{crl}
\begin{proof}
The fibers $\Quiv^\BM_X(\cM)_w$, $w\in\BM$, obviously 
have $\cK$-colimits. The monoidal structure is defined
via left Kan extensions which preserve colimits as they 
have a right adjoint.
\end{proof}

\subsubsection{More functoriality}
\label{sss:moref}

The assignment $(X,\cM)\mapsto\Quiv^\BM_X(\cM)$ is functorial. This means that   maps $j:X'\to X$ and
$f:\cM\to\cM'$ yield a map of operads $f_!j^!:\Quiv^\BM_X(\cM)\to\Quiv^\BM_{X'}(\cM')$. Let us now assume the conditions of Theorem
~\ref{thm:Quiv-monoidal} are fulfilled, so that $\Quiv^\BM_X(\cM)$ and $\Quiv^\BM_X(\cM')$ are monoidal. We claim that, if $f:\cM\to\cM'$ is monoidal and preserves $\cK$-indexed colimits, the induced functor
$f_!:\Quiv^\BM_X(f):\Quiv^\BM_X(\cM)\to\Quiv^\BM_X(\cM')$ is also monoidal.
In fact, we have to verify that $f_!$ preserves cocartesian liftings of active arrows in $\BM$. The description of cocartesian liftings as left Kan extension
proves the claim.

\

In the following subsection we will describe the
monoidal structure on $\Quiv_X(\cM)$ by a universal property.

\subsection{$\Quiv_X(\cM)$ as the category of endomorphisms}
\label{ss:Quiv-end}

Assume that $\cM$ is a monoidal category with  
$\cK$-indexed colimits. Recall that 
$\pi:\BM\to\Ass$ is the standard projection. The  $\BM$-operad $\Quiv^\BM_X(\pi^*\cM)$
is a $\BM$-monoidal category.  Its $\Ass_-$-component is
$\Quiv_X(\cM)$, $\Fun(X,\cM)$ is its
$m$-component, and $\cM$ is the $\Ass_+$-
component of $\Quiv^\BM_X(\pi^*\cM)$. 
This means that $\Quiv_X(\cM)$ acts on the left
on the right $\cM$-module $\Fun(X,\cM)$. In this subsection 
we will prove that $\Quiv_X(\cM)$ is the endomorphism 
object of $\Fun(X,\cM)$ considered as 
right $\cM$-module. Moreover, the family of monoidal
categories $\Quiv(\cM)$ will turn out to be the endomorphism object of the family of right $\cM$-modules
$\Fun(\_,\cM)$, see \ref{prp:quiv-end} and 
\ref{prp:quiv-end-fam}.

\subsubsection{}

Recall the setup of the endomorphism objects 
presented in \cite{L.HA}, 4.7.2.

Let $\cC$ be a monoidal category and let $\cA$ be left tensored over $\cC$. For $a\in\cC$ the endomorphism
object $\End_\cC(a)$ in $\cC$ is defined as the one representing the functor
\begin{equation}\label{eq:mor}
\Map_\cC(c,\End_\cC(a))=\Map_\cA(c\otimes a,a).
\end{equation}
Endomorphism object does not necessarily exist; but if it does, it is uniquely defined. If it exists, it 
automatically acquires a structure of associative algebra object in $\cC$ (see~\cite{L.HA}, 4.7.2.40). Moreover, for any 
associative algebra $c$ in $\cC$ there is a canonical
equivalence
\begin{equation}
\Map_{\Alg_\Ass(\cC)}(c,\End_\cC(a))\to
\{c\}\times_{\Alg_\Ass(\cC)}\Alg_\LM(\cC,\cA)\times_{\cA}\{a\}
\end{equation}
from the space of algebra maps $c\to\End_\cC(a)$ to the space of left $c$-module structures on $a$ (here we denote
as $(\cC,\cA)$ the $\LM$-monoidal category defined by 
the left $\cC$-module $\cA$).

Here is a short description of the construction. By definition,  $\End_\cC(a)$ is the terminal object
in the category $\cC[a]$ whose objects are pairs
$(c,f:c\otimes a\to a)$. It turns out that $\cC[a]$ has
a canonical structure of monoidal category such that the
forgetful functor $\cC[a]\to\cC$ is monoidal. Then by
a general result 3.2.2.5, \cite{L.HA}, the terminal object
of a monoidal category acquires automatically an algebra 
structure.

\subsubsection{}\label{sss:end-setup}
We will apply the above construction as follows.
We assume $\cM$ is a monoidal category with $\cK$-colimits
and $X$ is strongly $\cK$-small.

We set $\cC=\Cat^\cK$ and $\cA=\RM_{\cM}(\cC)$.
We choose $a=\Fun(X,\cM)$ (considered as a right $\cM$-module).
The $\BM$-monoidal category $\Quiv^\BM_X(\pi^*\cM)$ defines
an action of the monoidal category $\Quiv_X(\cM)$ on the right $\cM$-module $a=\Fun(X,\cM)$.   
Here is our second main result  of this section.
 
\begin{prp}\label{prp:quiv-end}
The action of  $\Quiv_X(\cM)$ on $\Fun(X,\cM)\in\RM_\cM(\Cat^\cK)$
presents  $\Quiv_X(\cM)$ as the endomorphism object.
\end{prp}
\begin{proof}
In short, universality of the action of $\Quiv_X(\cM)$
on $\Fun(X,\cM)$ follows from duality between $P^\cK(X)$ and $P^\cK(X^\op)$. The details are below.

1. The action of $\Quiv_X(\cM)$ on $\Fun(X,\cM)$ is described 
in Theorem~\ref{thm:Quiv-monoidal} as the composition

\begin{eqnarray}\label{eq:action}
\Quiv_X(\cM)\otimes\Fun(X,\cM)=
\Fun(X^\op\times X,\cM)\otimes\Fun(X,\cM)\to\\
\nonumber\Fun(X\times X^\op\times X,\cM\otimes\cM)
\stackrel{\mu}{\to}
\Fun(X\times X^\op\times X,\cM)\to\\
\nonumber
\Fun(X\times\Tw(X)^\op,\cM)\stackrel{\kappa}{\to}\Fun(X,\cM),
\end{eqnarray}
with $\mu$ defined by the monoidal structure of $\cM$
and $\kappa$ defined by the left Kan extension along
the projection $X\times\Tw(X)^\op\to X$.

Let us rewrite the composition (\ref{eq:action}) replacing all functor categories with presheaves. We get
\begin{eqnarray}\label{eq:action2}
P^\cK(X^\op\times X\times X^\op)\otimes\cM\otimes\cM
\stackrel{\mu}{\to}
P^\cK(X^\op\times X\times X^\op)\otimes\cM\to\\
\nonumber
P^\cK(X^\op\times\Tw(X))\otimes\cM\stackrel{\kappa}{\to}
P^\cK(X^\op)\otimes\cM,
\end{eqnarray}
The composition of the last two arrows comes from the composition
$$
P^\cK(X\times X^\op)\to P^\cK(\Tw(X))\stackrel{\colim}
{\longrightarrow}\cS^K
$$
which in turn is equivalent to the counit map 
(\ref{eq:counit}) for $X:=\cC^\op$.

2. Let $C\in\Cat^\cK$. 
 A map of right $\cM$-modules 
$C\otimes P^\cK(X^\op)\otimes\cM\to P^\cK(X^\op)\otimes\cM$
is uniquely defined by its restriction
$$C\otimes P^\cK(X^\op)\to P^\cK(X^\op)\otimes\cM.$$

Using duality between $P^\cK(X)$ and $P^\cK(X^\op)$,
we can rewrite the latter as a map
\begin{equation}
\label{eq:CtoUniv}
C\to P^\cK(X)\otimes P^\cK(X^\op)\otimes\cM=\Fun(X^\op\times X,\cM).
\end{equation}
In the case $C=\Quiv_X(\cM)$, the map~(\ref{eq:CtoUniv}) is 
an equivalence. 
This means that $\Quiv_X(\cM)$ is a terminal object
of $\cC[a]$, in the notation of\ref{sss:end-setup}.

\end{proof}

\

The identification of $\Quiv_X(\cM)$ with the endomorphism
object of $\Fun(X,\cM)$ is functorial in $X$. To show this,
we will describe the whole family $\Quiv(\cM)$
of monoidal categories, based on $\cat_\cK$, the category 
of strongly $\cK$-small categories
~\footnote{Note the difference between $\cat_\cK$ and
$\Cat^\cK$, the category of categories with $\cK$-indexed colimits.}, as the endomorphisms
of the family $X\mapsto\Fun(X,\cM)$. For this one needs to
properly define the categories involved.

\subsubsection{Endomorphisms of $\Fun(\_,\cM)$}

We replace the category $\Cat^\cK$ from 
\ref{sss:end-setup} with the category of families
$\cC:=\Cat^{\cK,\cart}_{/B}$, with $B=\cat_\cK$. 
This is the subcategory of $\Cat^\cart_{/B}$, 
see~\ref{sss:cartesianfamilies},
spanned by the families $\cX\to B$ classified by the functors
$B^\op\to\Cat^\cK$, with morphisms $\cX\to \cX'$ over $B$
inducing colimit preserving functors $\cX_X\to \cX'_X$ for each 
$X\in B$. The category $\cC$ has a symmetric monoidal structure induced from the cartesian structure on
$\Cat^\cart_{/B}$: if the families $\cX$ and $\cY$ over $B$
are classified by the functors $\tilde \cX,\tilde \cY:B^\op
\to\Cat^\cK$, the tensor product of $\cX$ and $\cY$ is 
classified by the functor 
$X\mapsto\tilde \cX(X)\otimes\tilde \cY(X)$.

We define $\cA=\RM_\cM(\cC)$, and, as before, we have a left $\cC$-action on $\cA$. The canonical cartesian family
of categories $X\mapsto \Fun(X,\cM)$ gives an object
$a\in\cA$. The family $\Quiv(\cM)$ is a cartesian family of monoidal categories, and  $\Quiv^\LM(\cM)$ defines a left 
action of $\Quiv(\cM)$ on $a$.

In order to prove universality of the action of 
$\Quiv(\cM)$ on $a=\Fun(\_,\cM)$, we will repeat the second part of the proof of Proposition~\ref{prp:quiv-end},
working with families over $\cat_\cK$. 

The functor $\cat_\cK^\op\to\Cat^\cK$ carrying $X$ to 
$P^\cK(X)$ and $f:X\to Y$ to $f^*:P^\cK(Y)\to P^\cK(X)$,
defines a cartesian family which will be denoted as $P^\cK(\_)$. Similarly, the assignment $X\mapsto P^\cK(X^\op)$
defines a cartesian family $P^\cK(\_^\op)$. The canonical map $P^\cK(\_^\op)\otimes\cM\to\Fun(\_,\cM)$ is a map of cartesian families, inducing equivalence of the fibers.
Therefore, it is an equivalence. In order to repeat the
argument of the proof of \ref{prp:quiv-end}, Step 2,
it is enough to construct the duality data for the families 
$P^\cK(\_)$ and $P^\cK(\_^\op)$. The counit of the duality 
$$e:P^\cK(\_)\otimes P^\cK(\_^\op)\to\cS^\cK\times\cat_\cK$$
is defined by the Yoneda map, exactly as for a single $X$,
see (\ref{eq:counit}). 
Let us construct the unit of the duality.

Denote by
$p:\cP\to\cat_\cK$ the cartesian fibration classified by 
the functor $X\mapsto P^\cK(X\times X^\op)$. Let $\cQ$ be 
the full subcategory of $\cP$ spanned by essentially 
constant presheaves. Since for any $f:X\to Y$ in $\cat_\cK$ 
the  map
$f^*:P^\cK(Y\times Y^\op)\to P(X\times X^\op)$ preserves 
essentially constant presheaves, $\cQ$ is also a cartesian 
fibration. Moreover, it is obviously equivalent to 
$\cS^\cK\times\cat_\cK$. This gives the required unit map
$c:\cS^\cK\times\cat_\cK\to\cP$.  
The compositions $(\id\otimes e)(c\otimes\id)$ and
$(e\otimes\id)(\id\otimes c)$ are maps of cartesian families, whose fibers at  any $X\in\cat_\cK$ come from the duality 
data for $P^\cK(X)$. So they are equivalences.

We have verified  the following result.

\begin{prp}
\label{prp:quiv-end-fam}
The action of $\Quiv(\cM)$ on $\Fun(\_,\cM)\in\RM_\cM(\Cat^{\cK,\cart}_{/\cat_\cK})$ presents $\Quiv(\cM)$
as the endomorphism object. 
\end{prp}

\subsection{Multiplicativity}
We will now reconsider once more the multiplicativity 
property~\ref{sss:mult}. Recall that 
for a pair of $\BM$-operads $\cM$, $\cN$
and a pair of categories $X,Y$ one has a canonical  map 
(\ref{eq:gamma}) of $\BM$-operads
$$
\mu:\Quiv^\BM_X(\cM)\times\Quiv^\BM_Y(\cN)\to
\Quiv^\BM_{X\times Y}(\cM\times\cN).
$$
We now assume that, for a collection of categories $\cK$,
$\cM,\cN\in\Alg_\BM(\Cat^\cK)$ and $X,Y$ and $X\times Y$ 
are  strongly $\cK$-small. 
The $\BM$-operads of quivers,
under this assumption, are $\BM$-monoidal categories with 
$\cK$-indexed colimits. In this case we can strengthen our claim.
\begin{prp}\label{prp:mult-monoidal}
\begin{itemize}
\item[1.] The map $\mu$ ~(\ref{eq:gamma}) is a $\BM$-monoidal functor.
\item[2.] The composition
\begin{equation}\label{eq:gamma1}
\Quiv^\BM_X(\cM)\times\Quiv^\BM_Y(\cN)\to
\Quiv^\BM_{X\times Y}(\cM\otimes\cN)
\end{equation}
is also a $\BM$-monoidal functor.
\item[3.] The functor (\ref{eq:gamma1}) preserves 
$\cK$-indexed colimits in each component and it induces
an equivalence of  $\BM$-monoidal categories
\begin{equation}\label{eq:gamma2}
\mu:\Quiv^\BM_X(\cM)\otimes\Quiv^\BM_Y(\cN)
\stackrel{\simeq}{\to}
\Quiv^\BM_{X\times Y}(\cM\otimes\cN).
\end{equation}
\end{itemize}
\end{prp}
\begin{proof}
1. We have to verify that the functor (\ref{eq:gamma})
preserves cocartesian liftings of the active arrows.
These are described in Theorem~\ref{thm:Quiv-monoidal}
as left Kan extensions.  Let $\alpha:w\to u$ be an active arrow in $\BM$ over the active map 
$\langle n\rangle\to\langle 1\rangle$.
Following the calculation from \ref{ss:A-eq-colim}, we
denote as
$(\BM_X)_w\stackrel{p_X}{\leftarrow}C_X\stackrel{q_X}{\to}(\BM_X)_u$,

$(\BM_Y)_w\stackrel{p_Y}{\leftarrow}C_Y\stackrel{q_Y}{\to}(\BM_Y)_u$
and
$(\BM_{X\times Y})_w\stackrel{p_{X\times Y}}{\leftarrow}C_{X\times Y}\stackrel{q_{X\times Y}}{\to}(\BM_{X\times Y})_u$
the diagrams described in~\ref{sss:alpha-cases}
for the arrow $\alpha$ and the category $X$, $Y$, and $X\times Y$ respectively.
One has $\alpha^*(\BM_X)\times\alpha^*(\BM_Y)=\alpha^*(\BM_{X\times Y})$ and
$\alpha^*(\cM)\times\alpha^*(\cN)=\alpha^*(\cM\times\cN)$.
Furthermore, the formulas \ref{sss:alpha-1} --- \ref{sss:alpha-3} show that
the the diagram $(p_{X\times Y},q_{X\times Y})$ is equivalent to the product of $(p_X,q_X)$ with 
$(p_Y,q_Y)$. 
Now the claim follows as  Kan extensions commute with (this 
type of ) products.

2. Note that the second claim does not immediately follow 
from \ref{sss:moref} as the functor $\cM\times\cN\to
\cM\otimes\cN$ does not preserve colimits. However,
the left Kan extension of a functor
$C_X\times C_Y\to\alpha^*(\cM)\times\alpha^*(\cN) $ along 
$(q_X,q_Y):C_X\times C_Y\to(\BM_X)_u\times(\BM_Y)_u$ 
can be calculated as a composition of two left Kan 
extensions preserved by the functor
$\cM\times\cN\to\cM\otimes\cN$, see Lemma~\ref{lem:leftKan-product}. This implies the claim.

3. To verify the claim, we can forget about the monoidal
structure of the categories involved. Then the claim
follows from Lemma~\ref{lem:bilinearhom} below.
\end{proof}
\begin{lem}
\label{lem:bilinearhom}
Let $X,Y$ be in $\cat_\cK$, $\cM,\cN\in\Cat^\cK$.
Then the composition
\begin{equation}
\label{eq:bilinearhom}
\Fun(X,\cM)\times\Fun(Y,\cN)\to
\Fun(X\times Y,\cM\times\cN)\to
\Fun(X\times Y,\cM\otimes\cN)
\end{equation}
is universal among the maps preserving $\cK$-indexed
colimits in each argument.
\end{lem}
\begin{proof} It is straightforward to see that
\ref{eq:bilinearhom}) preserves $\cK$-indexed colimits in each variable. We will now prove the universality.
First of all, let $\cM=\cN=\cS^\cK$.
Yoneda embedding $X^\op\to\Fun(X,\cS^\cK)$ is determined by 
the functor $X^\op\times X\to\cS^\cK$ that classifies
the left fibration $\Tw(X)\to X^\op\times X$, see (\ref{eq:Tw}).
The diagram (\ref{eq:Tw}) is functorial in $X$ and, moreover, preserves products. Therefore, Yoneda embedding for $X^\op\times Y^\op$ is equivalent to the composition
$$X^\op\times Y^\op\to\Fun(X,\cS^\cK)\times\Fun(Y,\cS^\cK)\to
\Fun(X\times Y,\cS^\cK\times\cS^\cK)\to\Fun(X\times Y,\cS^\cK),$$
where the last map is induced by the product in $\cS^\cK$.
This is precisely the claim of the lemma for 
$\cM=\cN=\cS^\cK$.

We will now verify the claim for arbitrary 
$\cM,\cN\in\Cat^\cK$.  We will first make a number of simple observations.
\begin{itemize}
\item[1.] Let a functor $f:X\times Y\to Z$ be adjoint to
$\hat f:X\to\Fun(Y,Z)$ and let $f':X'\times Y'\to Z'$
be adjoint to $\hat f':X'\to\Fun(Y',Z')$. Then the 
map $\hat h:X\times X'\to\Fun(Y\times Y',Z\times Z')$ adjoint to the product $f\times f'$ is the composition
of $\hat f\times\hat g$ with the product map
$$\Fun(Y,Z)\times\Fun(Y',Z')\to\Fun(Y\times Y',Z\times Z').$$
\item[2.]  For arbitrary $Z\in\cat_\cK$, $\cC\in\Cat^\cK$, the map
$Z^\op\times\cC\to\Fun(Z,\cC)$,
universal among the maps to an object in $\Cat^\cK$ 
preserving the $\cK$-colimits in $\cC$, is determined by the composition
$$Z^\op\times Z\times\cC\to\cS^K\times\cC\to\cS^\cK\otimes\cC=\cC,$$
where the first arrow is induced by the Yoneda map.
\item[3.]Let $X\in\cat_\cK$, $\cM,\cN\in\Cat^\cK$. The 
composition 
$$X^\op\times\cM\times\cN\to X^\op\times\cM\otimes\cN\to
\Fun(X,\cM\otimes\cN)$$
is universal among the maps to an object in $\Cat^\cK$ 
preserving the $\cK$-colimits in $\cM$ and in $\cN$.

\end{itemize}

We now apply Observation 2 in three instances, for
$(Z,\cC)=(X,\cM), (Y,\cN)$ and $(X\times Y,\cM\otimes\cN)$.
The maps $y_X:X^\op\times X\times\cM\to\cM$ and 
$y_Y:Y^\op\times Y\times\cN\to\cN$ induced by the Yoneda maps, are adjoint to the universal maps
$e_X:X^\op\times\cM\to\Fun(X,\cM)$ and 
$e_Y:Y^\op\times\cN\to\Fun(Y,\cN)$ preserving colimits in $\cM$ and $\cN$ respectively. According to Observation 1,
the map adjoint to their product 
\begin{equation}
\label{eq:adj1}
X^\op\times Y^\op\times X\times Y\times\cM\times\cN\to\cM\times\cN
\end{equation}
is given by the composition of the product $e_X\times e_Y$
with the map 
\begin{equation}
\label{eq:adj2}
\Fun(X,\cM)\times\Fun(Y,\cN)\to\Fun(X\times Y,\cM\times\cN).
\end{equation}
The last claim will remain true if we compose both
(\ref{eq:adj1}) and (\ref{eq:adj2}) with $\cM\times\cN\to
\cM\otimes\cN$. This precisely means that the composition
of (\ref{eq:bilinearhom}) with 
$$X^\op\times Y^\op\times\cM\times\cN\to
\Fun(X,\cM)\times\Fun(Y,\cN)$$
is a universal map to an object of $\Cat^\cK$ preserving colimits on $\cM$ and $\cN$.

\end{proof}

\begin{lem}
\label{lem:leftKan-product}
Let $F_i:\cD_i\to\cM_i$ be left Kan extensions of 
$f_i:\cC_i\to\cM_i$ along $u_i:\cC_i\to\cD_i$ for $i=1,2$.
Assume $\cM_i$ have $\cK$-colimits and $\cC_i,\cD_i$ 
and $\cC_1\times\cC_2$, $\cD_1\times\cD_2$ are strongly
$\cK$-small.
Then $F_1\otimes F_2:\cD_1\times\cD_2\to\cM_1\otimes\cM_2$ defined as the composition of $(F_1,F_2)$ with the canonical map $\cM_1\times\cM_2\to\cM_1\otimes\cM_2$, is a left Kan extension of $f_1\otimes f_2$.
\end{lem}
\begin{proof}
We will prove that all compositions from the upper left corner to the lower right corner of the diagram below
are equivalent.
\begin{equation*}
\xymatrix{
&{\Fun(\cC_1,\cM_1)\times\Fun(\cC_2,\cM_2)}\ar^m[d]&{} \\
&{\Fun(\cC_1\times\cC_2,\cM_1\times\cM_2)}\ar^\simeq[d]\ar[r]
&{\Fun(\cC_1\times\cC_2,\cM_1\otimes\cM_2)}\ar^\simeq[d] \\
&{\Fun(\cC_1,\Fun(\cC_2,\cM_1\times\cM_2))}\ar[d]\ar[r]
&{\Fun(\cC_1,\Fun(\cC_2,\cM_1\otimes\cM_2))}\ar[d] \\
&{\Fun(\cC_1,\Fun(\cD_2,\cM_1\times\cM_2))}\ar[d]\ar[r]
&{\Fun(\cC_1,\Fun(\cD_2,\cM_1\otimes\cM_2))}\ar[d] \\
&{\Fun(\cD_1,\Fun(\cD_2,\cM_1\times\cM_2))}\ar^\simeq[d]\ar[r]
&{\Fun(\cD_1,\Fun(\cD_2,\cM_1\otimes\cM_2))}\ar^\simeq[d] \\
&{\Fun(\cD_1\times\cD_2,\cM_1\times\cM_2)}\ar[r]
&{\Fun(\cD_1\times\cD_2,\cM_1\otimes\cM_2)} 
}.
\end{equation*}
Here $m$ is the product and the rest of nontrivial vertical arrows are left Kan extensions. Two middle squares are not commutative as the functor $\cM_1\times\cM_2\to\cM_1\otimes
\cM_2$ do not preserve colimits. However, they preserve colimits in each of the arguments, so these squares commute
on the image of $\Fun(\cC_1,\cM_1)\times\Fun(\cC_2,\cM_2)$.

\end{proof}

\subsection{Example: $\cS$-enrichment of a category}

In this subsection $\cM=\cS$ is the category of spaces.
We are working with the category $\Cat^L$ of categories
with small colimits (it is big and even locally big).

We will show that any category $\cC$ gives rise to an 
$\cS$-enriched (pre)category.

\subsubsection{Unit in $\Quiv_X(\cS)$}
Let $X$ be a category. Since the monoidal category
$\Quiv_X(\cS)$ is identified with the category of colimit preserving endomorphisms of $P(X)$, the unit in 
$\Quiv_X(\cS)$ is given by the identity $\id:P(X)\to P(X)$.

The identification of endomorphisms of $P(X)$ with
$\Fun(X^\op\times X,\cS)$ identifies $\id$ with the Yoneda
map  
$$ Y:X^\op\times X\to \cS$$
assigning $\Map(x,y)$ to the pair $(x,y)\in X^\op\times X$.

\subsubsection{}
\label{sss:from-cat-to-senriched}

 Let $\cC$ be a category. Denote by
$j:\cC^\eq\to\cC$ the maximal subspace of $\cC$.

One has a lax monoidal functor 
$$j^!:\Quiv_\cC(\cS)\to\Quiv_{\cC^\eq}(\cS).$$

The functor $j^!$ carries algebras in $\Quiv_\cC(\cS)$
to algebras in $\Quiv_{\cC^\eq}(\cS)$. The image of the 
unit described above is called the $\cS$-enriched (pre)category corresponding to $\cC$. As expected, 
it has $\cC^\eq$ as its space of objects, and it assigns 
the space $\Map(x,y)$ to a pair $(x,y)\in (\cC^\eq)^\op\times(\cC^\eq)$.

\subsubsection{Unit in $\Quiv_X(\cM)$}
\label{sss:unit-general}
Let now $\cM\in\Alg_\Ass(\Cat^L)$ be arbitrary.
The canonical functor $i:\cS\to\cM$ (the unit of $\cM$
considered as an associative algebra in $\Cat^L$) is 
monoidal and preserves colimits, therefore 
(see~\ref{sss:moref}) it induces a monoidal functor 
$$i_!:\Quiv_X(\cS)\to\Quiv_X(\cM).$$

In particular, the unit in $\Quiv_X(\cM)$ is described by the composition
$$ 
\begin{CD}
Y_\one:X^\op\times X@>h>> \cS @>i>>\cM
\end{CD}
$$ 
of the Yoneda map for $X$ and the embedding $i:\cS\to\cM$. 

\subsection{Completeness: advertisement}

Associative algebras in $\Quiv_X(\cM)$ are not called enriched categories ---
but precategories --- for two reasons.

The first is that we want the category $X$ of objects in an 
enriched category to be a space. The second reason is more 
important and it has 
to do with the completeness property.

We will see in Section~\ref{sec:cart} that, if  $\cM=\cS$ 
is the category of spaces  and if $X$ is a space, 
then the category of associative algebras in 
$\Quiv_X(\cS)$ is equivalent to that of Segal spaces. 

Let now $\cC\in\Quiv_X(\cM)$ be an $\cM$-enriched 
precategory.
Let $j:\cM\to\cS$ be the functor right adjoint to the 
canonical embedding $\cS\to\cM$. The functor $j$ is lax 
monoidal, so $j^!(\cC)$ is an $\cS$-enriched
precategory, that is, corresponds to a Segal space.

We will say that $\cC$ is an $\cM$-enriched category if $X$ is a space and if
 $j^!(\cC)$, considered as a Segal space, is complete.

We will show later that  $\cM$-enriched categories
form a localization of the category of $\cM$-enriched precategories.
This means that the full embedding of $\cM$-enriched categories into
$\cM$-enriched precategories admits a left adjoint functor.

\section{The case $\cM$ is cartesian}
\label{sec:cart}

Let $\cM$ be a category with colimits and finite limits satisfying certain weakened topos conditions, see Definition~\ref{dfn:prototopos} below. We endow $\cM$ with the cartesian monoidal structure.

In this section we study precategories enriched over $\cM$ in the sense of \ref{dfn:Mprecat} having a space of objects $X$. We prove that such precategories can be equivalently described as  simplicial objects $A_\bullet$ in $\cM$ satisfying the Segal condition, such that $A_0=X$. 
The comparison of these two notions of enriched category proceeds in two steps.
  
\begin{itemize}
\item[{\sl Step 1.}]
We show that for any $X\in\cM$ the category $\cM_{/X\times X}$ 
has a canonical structure of monoidal category. We show further 
that the associative algebras
in this monoidal category are precisely simplicial objects 
$A_\bullet$ in $\cM$
satisfying the  condition $A_0=X$ together with the Segal 
condition.
\item[{\sl Step 2.}] 
After that, assuming $\cM$ satisfies the properties of 
Definition~\ref{dfn:prototopos}, we identify the family of 
monoidal categories $\cM_{/X\times X}$ for varying 
$X\in\cS\subset\cM$ with the family of monoidal categories 
$\Quiv_X(\cM)$ defined in Section~\ref{sec:quivers}.
\end{itemize}

\subsection{Non-symmetric cartesian structures}
 
Recall that $\Ass=\Delta^\op$ ; we denote $\langle n\rangle=[n]^\op$. If $\cM\in\Op_{\Ass}$, the fiber
$\cM_0$ is contractible; we will denote by $*$ an object
of $\cM_0$.

The following definition is a non-symmetric version of \cite{L.HA}, 2.4.1.

\begin{dfn}
Let $q:\cM\to\Ass$ be an object of $\Fib(\Ass^{\natural,\emptyset})$; that is, a map admitting cocartesian lifting of the inerts~\footnote{This includes planar operads 
$p:\cM\to\Ass$, as well as, for instance, $\BM$-operads,
see Remark in~\ref{sss:rem}.}.
\begin{itemize}
\item[(NC1)] A functor $F:\cM\to \cN$ to a category with fiber products $\cN$ is called a lax NC \footnote{NC is short for ``non-symmetric cartesian''.} 
structure if for each commutative diagram in $\Ass$ of form
\begin{equation}
\xymatrix{
&{\langle k+l\rangle}\ar[r]^{a}\ar[d]^{b} &{\langle k\rangle}
\ar[d]^t\\
&{\langle l\rangle}\ar[r]^s &{\langle 0\rangle}
}
\end{equation}
with $t$ and $s$ given by $\{k\}\in[k]$ and $\{0\}\in[l]$
respectively, $a$ corresponding to the embedding $[k]\to
[k+l]$ as the initial segment, and $b$ corresponding to the embedding $[l]\to[k+l]$ as the terminal segment,
and for each $x\in\cM_{k+l}$, the diagram

\begin{equation}
\begin{CD}
F(x)@>>> F(a_!(x))\\
@VVV @VVV \\
F(b_!(x))@> >>F(t_!b_!(x)).
\end{CD} 
\end{equation}
is cartesian.
 
\item[(NC2)] $f$ is a weak NC structure if it is a lax NC structure, $q:\cM\to\Ass$ is a monoidal category, and $f$ carries cocartesian liftings of the active arrows to equivalences.
\item[(NC3)] $f$ is an NC structure if it is a weak NC structure and
the natural functor 
\begin{equation}
\label{eq:m-one}
\cM_1\to \cN_{/F(*)\times F(*)}
\end{equation}
induced by $F$ and by the cocartesian liftings of two maps $\langle 1\rangle\to\langle 0\rangle$ in
$\Ass$, is an equivalence. 
\end{itemize}
\end{dfn}

\begin{rems}
\begin{enumerate}
\item[1.] The original version of the theory developed by J.~Lurie
in \cite{L.HA}, 2.4.1, aims to prove that an $\infty$-category $\cM$ with products has an essentially unique structure of symmetric monoidal $\infty$-category with monoidal structure defined by the product; moreover,
algebras with values in this symmetric monoidal category can be described by
monoids in the original $\infty$-category $\cM$.
\item[2.] In the original (symmetric) setup a cartesian structure exists and is essentially unique for a SM $\infty$-category $\cM$ satisfying the following properties.
\begin{itemize}
\item The unit of $\cM$ is a terminal object.
\item For any $x,y\in\cM$ the pair of maps 
\begin{equation*}
x\otimes y\to x\otimes\one = x,\quad
x\otimes y\to\one\otimes y=y,
\end{equation*}
yields a cartesian diagram. The following NC version
of the above characterisation seems very plausible.

{\bf A plausible claim.} {\it Let $\cM$ be a SM category
with a terminal object $*$. Then $\cM$ has an NC structure 
iff for any $x,y\in\cM$ the diagram}
\begin{equation*}
\xymatrix{
&{x\otimes y}\ar[r]\ar[d]&{x\otimes *}\ar[d] \\
&{*\otimes y}\ar[r] &{*\otimes *}
}
\end{equation*}
is cartesian.
\end{itemize}  
\end{enumerate}
\end{rems}

Our non-symmetric analog will prove the following.

First of all, for any $\infty$-category $\cM$ with fiber 
products and for any object $X\in\cM$ we will construct an 
explicit monoidal category $\cM_X^\Delta$ with an NC 
structure $F:\cM_X^\Delta\to\cM$ such that $F(*)=X$
\footnote{That is, we will construct a monoidal structure 
on $\cM_{/X\times X}$.}. 

Further, we will prove that for any monoidal category 
$q:\cC\to\Ass$ the composition
with $F$ establishes an equivalence between the category 
of monoidal functors $\Fun^\otimes_\Ass(\cC,\cM_X^\Delta)$ 
 and the category $\Fun^\weak(\cC,\cM)_X$ of weak NC 
structures $\cC\to\cM$
carrying $*\in \cC_0$ to $X\in\cM$. This proves, in 
particular, that any monoidal category $\cC$ with a NC 
structure $F:\cC\to\cM$
is equivalent to one of the $\cM_X^\Delta$.

The family of monoidal categories $\cM^\Delta_X$
appears as a small fragment of a two-parametric family
$\cM_{X,Y}^\odot$ of $\BM$-monoidal categories constructed
in \ref{ss:family-bm}; the subfamily  $\cM_{X,*}^\odot$
is also very important to us; it is denoted as $\cM^\og_X$.
The family of monoidal categories $\cM^\Delta_X$ is just
the $\Ass_-$-component of $\cM^\og_X$.
 
Finally, this will allow us to describe algebras over 
a $\BM$-operad $\cO$ in $\cM_X^\odot$ in
terms of lax functors  $\cO\to\cM$. As a consequence, 
associative algebras in $\cM_X^\Delta$ will be identified
with simplicial objects $Y_\bullet$ in $\cM$ satisfying the 
condition $Y_0=X$ and the Segal condition. 

\

The details are below.

\subsection{A canonical family of monoidal categories}
\label{ss:family-bm}

In this subsection we will construct, for any category 
$\cM$ with fiber products
and a terminal object, a family of  $\BM$-monoidal categories
$\cM^\odot\to\cM\times\cM\times\BM$ parametrized by $\cM\times\cM$, so that its fiber $\cM_{X,Y}^\odot$ at
$(X,Y)\in\cM\times\cM$ is the triple $(\cM_{/X\times X},\cM_{/X\times Y},\cM_{/Y\times Y})$ with a certain monoidal category structure on $\cM_{/X\times X}$ and 
$\cM_{/Y\times Y}$ 
so that $\cM_{/X\times Y}$ is left-tensored over $\cM_{/X\times X}$ and right-tensored over $\cM_{/Y\times Y}$. 

Especially important for us is the family $\cM^\og$
obtained from $\cM^\odot$ by base change along $\cM\to
\cM\times\cM$, carrying $X\in\cM$ to the pair $(X,*)$,
$*$ being a terminal object of $\cM$.

Assuming some extra properties for $\cM$, 
we will be able to identify $\cM_X^\og$ with 
$\Quiv_X^\BM(\cM)$, see~\ref{prp-quiv-bm-proto}.

In our construction of the family 
$\cM^\odot\to\cM\times\cM\times\BM$,
we follow the original construction of \cite{L.HA}, 2.4.1. First 
of all, we describe a functor $\cE:\BM^\op\to\Cat$ 
to (conventional) categories. This functor allows one to assign to $\cM$, in a way close to \cite{L.HA}, 2.4.1,
a category $\bar\cM^\odot$ over $\cM\times\cM\times\BM$; the family
 of $\BM$-monoidal categories $\cM^\odot$ will be defined as a full subcategory of $\bar\cM^\odot$.

\subsubsection{The functor $\cE:\BM^\op\to\Cat$}
\label{sss:E}

Here is the rationale for our choice of $\cE$. 
Recall that $\BM^\op=\Delta_{/[1]}$. Given an 
object $s:[n]\to[1]$ in $\BM^\op$, we expect the fiber of 
$\cM_X^\odot$ at $s$ to be equivalent to the product
\begin{equation}\label{eq:fiber-Modot}
(\cM_{/X\times X})^{\alpha}\times(\cM_{/X\times Y})^{\mu}\times
\cM_{/Y\times Y}^{\beta},
\end{equation}
where $s=a^\alpha m^\mu b^\beta$, $\alpha+\mu+\beta=n$, in the presentation of $s$ 
described in ~\ref{sss:bm}.

In order to have a functorial dependence on $s$ in
(\ref{eq:fiber-Modot}), we will describe this category as the category of functors $\cE(s)\to\cM$ 
satisfying certain properties. 

We will use the following notation. For $s:[n]\to [1]$
we denote $|s|=n$. $\BM_0$ has two objects, $\emptyset_L$
and $\emptyset_R$, corresponding to the maps $[0]\to[1]$
with the image $0$ and $1$ respectively.

The objects of $\cE(s)$ are the pairs $(i,j)$
such that $0\leq i\leq j\leq |s|$, as well as two extra
objects, $L$ and $R$.  Let $E_1$ be the poset
with $(i,j)\leq(i',j')$ iff $i\leq i'\leq j'\leq j$.
Define, in addition, a groupoid $E_0$ with the objects
$L,R,(0,0),\ldots,(n,n)$, and with a unique isomorphism
from $(i,i)$ to $L$ if $s(i)=0$ and to $R$  
if $s(i)=1$. The category $\cE(s)$ is defined as the 
colimit of the diagram
\begin{equation}
E_0\longleftarrow \{(0,0),\ldots,(n,n)\}\longrightarrow
E_1.
\end{equation}

A map $f:s\to s'$ in $\BM^\op$ is given by 
$[n]\to[n']\to[1]$. It defines $f:\cE(s)\to\cE(s')$
carrying $(i,j)$ to $(f(i),f(j))$, and preserving $L$ and $R$.

\begin{exm}
The category $\cE(\emptyset_L)$ looks as follows
\begin{equation}
\xymatrix{
&L &(0,0)\ar_\sim[l]&R}
\end{equation}
and $\cE(\emptyset_R)$ is similar, with $(0,0)$ isomorphic to $R$.

There are three non-isomorphic objects of $\BM$ 
lying over $\langle 1\rangle$. These are 
$s=(00),(01),(11)$. 

The category $\cE(00)$ looks as follows.

\begin{equation}
\xymatrix{
&{} &{(0,1)} \ar[dl]\ar[dr] \\
&{(0,0)}\ar[rd]^\sim &{} &{(1,1)}\ar[ld]_\sim &R \\
&{} &L
} 
\end{equation}
The category $\cE(11)$ looks the same, with replacement of
$R$ and $L$. 

Finally, the category $\cE(01)$ is as follows.
\begin{equation}
\xymatrix{
&{} &{(0,1)} \ar[dl]\ar[dr] \\
&{(0,0)}\ar[d]^\sim &{} &{(1,1)}\ar[d]_\sim \\
&{L} &{}  &{R}
} 
\end{equation}
\end{exm}

\subsubsection{Distinguished squares in $\cE(s)$}

The following commutative squares in $\cE(s)$ will be called {\sl distinguished
squares}. We will use them in order to define $\cM^\odot$ as a full subcategory 
of $\bar\cM^\odot$.
Let $s:[n]\to[1]$. For $0\leq i\leq j\leq k\leq n$
the commutative diagram in $\cE(s)$
\begin{equation}
\xymatrix{
&{} &{(i,k)} \ar[dl]\ar[dr] \\
&{(i,j)}\ar[rd]  &{} &{(j,k)}\ar[ld]   \\
&{} &{(j,j)}
} 
\end{equation}
will be called {\sl distinguished}.

For any map $f:s\to s'$ in $\BM^\op$ the corresponding 
functor $f:\cE(s)\to\cE(s')$ preserves distinguished 
arrows.

\subsubsection{}
\label{sss:mapstoM}
Fix $\cM\in\Cat$. The functor $\cE:\BM^\op\to\Cat$ 
described above defines a functor
\begin{equation}\label{eq:fun-em}
\Fun(\cE,\cM):\BM\to\Cat,
\end{equation}
carrying $s\in\BM^\op$ to $\Fun(\cE(s),\cM)$. There are morphism
of functors $\iota_L,\iota_R:[0]\to\cE$ from the constant functor with value $[0]$ defined by the objects 
$L,R\in\cE(s)$, respectively. This yields a morphism of
functors $\Fun(\cE,\cM)\to\Fun([0]\sqcup[0],\cM)=\cM\times\cM$. Let 
$\bar\cM^\odot$ be the cocartesian fibration classified by 
the functor~(\ref{eq:fun-em}). The morphisms $\iota_L,\iota_R$ induce  
a map $q=(q_L,q_R):\bar\cM^\odot\to\cM\times\cM$.

We have defined a map 
$(q,p):\bar\cM^\odot\to\cM\times\cM\times\BM$
which will be shown to be a family of cocartesian
fibrations based on $\cM\times\cM$. The fiber
of $q$ at $(X,Y)$ is denoted as $\bar\cM^\odot_{X,Y}$.

We will need one more map, $\bar m:\bar\cM^\odot\to\cM$,
defined as follows. The functor $\cE:\BM^\op\to\Cat$ is 
classified by a cartesian fibration $\wt\cE\to\BM$. 
The cocartesian fibration
$\bar\cM^\odot$ can be then described as an internal 
mapping object in the category $\Cat_{/\BM}$,
$\Fun_{\Cat_{/\BM}}(\wt\cE,\cM\times\BM)$,
see \cite{L.T}, Corollary 3.2.2.13 and \cite{GHN}, 7.3.

One has a canonical section $\top:\BM\to\wt\cE$ carrying
$s\in\BM$ to the object $(0,|s|)$ of $\cE(s)$. This map
induces, by functoriality of $\Fun_\BM$, a map
$\bar\cM^\odot\to\cM\times\BM$ over $\BM$. We denote
by $\bar m:\bar\cM^\odot\to\cM$ the projection to the 
first factor.

\subsubsection{}
\label{sss:full-BM-family}
The objects of $\bar\cM^\odot$ over 
$(X,Y,s)\in\cM\times\cM\times\BM$ are 
the functors $\phi: \cE(s)\to\cM$, endowed with a pair of equivalences
$\phi(L)\simeq X$, $\phi(R)\simeq Y$.

We now define $\cM^\odot$ as the full subcategory of 
$\bar\cM^\odot$
spanned by the objects $\phi:\cE(s)\to\cM$ carrying
the distinguished squares of $\cE(s)$ to cartesian squares of $\cM$. We denote by $\cM^\odot_{X,Y}$ the fiber of 
$q:\cM^\odot\to\cM\times\cM$ at $(X,Y)$.

\
\

\subsubsection{}

The fiber of $\cM^\odot$ at  $(X,Y,(00))$
is equivalent to $\cM_{/X\times X}$, the fiber
at $(X,Y,(01))$ is equivalent to $\cM_{/X\times Y}$ and the fiber at $(X,Y,(11))$ is equivalent to 
$\cM_{/Y\times Y}$.

The following theorem is a direct analog of Theorem 2.4.1.5, \cite{L.HA}. 
In the assertion (2) of the theorem
we denote by $\cE(\alpha):\cE(s')\to\cE(s)$
the functor corresponding to $\alpha:s\to s'$ in 
$\BM$.
\begin{thm}
\label{thm:NCmonoidal}
\begin{itemize}
\item[1.] The map $p:\bar\cM^\odot\to\BM$
is a cocartesian fibration.
\item[2.] A morphism $\phi\to\phi'$ over $\alpha:s\to s'$ is a cocartesian 
lifting iff for every $e\in \cE(s')$ the map 
$\phi(\cE(\alpha)(e))\to\phi'(e)$
is an equivalence in $\cM$.
\item[3.] For any $(X,Y)\in\cM\times\cM$ the restriction of $p:\bar\cM^\odot\to\BM$ to 
$\cM^\odot_{X,Y}$ defines a cocartesian fibration to $\BM$.
\item[4.] If $\cM$ has fiber products and a terminal object, 
the projection $\cM^\odot_{X,Y}\to\BM$ defines on the triple 
$(\cM_{/X\times X},\cM_{/X\times Y},\cM_{/Y\times Y})$ a structure of 
$\BM$-monoidal category.
\end{itemize}
\end{thm}
 \begin{proof}

The assertions (1) and (2) are immediate consequences of  the construction. (2) implies that $p:\cM^\odot\to\BM$
is a cocartesiain fibration. Also by (2), $p$-cocartesian 
arrows in $\cM^\odot$ are sent to equivalences by $q$.
According to Lemma~\ref{lem:cocfiber} below, this implies
that the fibers $\cM^\odot_{X,Y}\to\BM$ are cocartesian 
fibrations. This proves assertion (3).
 
It remains to verify that, if $\cM$ admits  fiber products 
and a terminal object, the map $p:\cM^\odot_{X,Y}\to\BM$ 
is fibrous. We already know that $p$ is a cocartesian fibration. Let $s\in\BM$ be of dimension $n$ and let
$\rho^i:s\to s_i,\ i=1,\ldots,n$ be the inerts decomposing $s$. We have to verify that the map 
\begin{equation}\label{eq:decomp-mxs}
\cM^\odot_{X,Y,s}\to\prod_i\cM^\odot_{X,Y,s_i}
\end{equation} 
is an equivalence.
Here $\cM^\odot_{X,Y,s_i}=\cM_{/X\times X},\cM_{/X\times Y},\cM_{/Y\times Y}$
for $s_i=a,m$ or $b$ respectively. By definition, 
$\cM^\odot_s$ is the subcategory of $\Fun(\cE(s),\cM)$
carrying distinguished squares of $\cE(s)$ to
cartesian squares. Let $\cE(s)^\circ$ denote the full subcategory
of $\cE(s)$ spanned by the objects $(i,j)$ with $j\leq i+1$, $L$ and $R$. The functors in $\cM^\odot_{X,Y,s}$ are right Kan extensions of functors $\cE(s)^\circ\to\cM$
carrying $L$ to $X$ and $R$ to $Y$. By~\cite{L.T}, 4.3.2.15, $\cM^\odot_{X,Y,s}$
is the subcategory of $\Fun(\cE(s)^\circ,\cM)$ carrying 
$L$ to $X$ and $R$ to $Y$. Now, $\cE(s)^\circ$ decomposes into a colimit
of diagrams corresponding to each $(i-1,i)$; this proves 
(\ref{eq:decomp-mxs}) is an equivalence.
The fact that the $\rho^i$ form a $p$-product diagram is automatic for cocartesian fibrations, see \cite{L.HA}, 
2.1.2.12.
\end{proof}

\begin{lem}
\label{lem:cocfiber}
Given $(q,p):M\to A\times B$ such that $p:M\to B$ is a cocartesian fibration. Assume that $q$ carries $p$-cocartesian arrows in $M$ to equivalences in $A$. Then the 
the map $(q,p)$ is also a cocartesian fibration.
\end{lem}
\begin{proof}
We will show that
any $p$-cocartesian arrow $f$ in $M$ is also $(q,p)$-cocartesian. The arrow $f:x\to y$
in $M$ is cocartesian iff the diagram
$$
\xymatrix{
&{M_{y/}}\ar[r]\ar[d] &{M_{x/}}\ar[d]\\
&{B_{p(y)/}}\ar[r]&{B_{p(x)/}}
}
$$
is cartesian. Since $q$ carries $i(f)$ to an equivalence,
the diagram 
$$
\xymatrix{
&{M_{y/}}\ar[r]\ar[d] &{M_{x/}}\ar[d]\\
&{A_{q(y)/}\times B_{p(y)/}}\ar[r]&{
A_{q(x)/}\times B_{p(x)/}}
}
$$
is also cartesian. 
\end{proof}
\subsubsection{Subfamilies}
\label{sss:subfamily}

The base change of $\cM^\odot$ with respect to 
the embedding $\cM\to\cM\times\cM$ carrying $X$ to $(X,*)$,
$*$ being the terminal object of $\cM$, gives a one-parametric family $p:\cM^\og\to\cM\times\BM$ of $\BM$-monoidal categories. Its fiber at $X$, $\cM^\og_X$,
is the triple $(\cM_{/X\times X},\cM_{/X},\cM)$.

The $\Ass_-$-component of $\cM^\og$ is the family of monoidal categories
\begin{equation}
\label{eq:res-family}
p:\cM^\Delta\to \cM\times\Ass,
\end{equation}
 whose fiber at $X\in\cM$ is a monoidal structure on 
$\cM_{/X\times X}$.

The functor $m:\cM^\odot\to\cM$ is the restriction of 
$\bar m:\bar\cM^\odot\to\cM$ defined
in~\ref{sss:mapstoM}. We denote by the same letter
the restriction of $m$ to $\cM^\Delta$.

The result below is an immediate consequence of the construction and of the assertion 4 of~\ref{thm:NCmonoidal}.
\begin{prp}Let $\cM$ admit products and fiber products over $X$. Then the map $m:\cM^\Delta\to\cM$ 
restricted to $\cM^\Delta_X$, yields a
NC structure.  
\end{prp}
\qed

\subsection{}

The following result is an analog of \cite{L.HA}, 2.4.1.7,
and our proof is very close to the original one.
A similar claim \ref{prp:alg-monoids-ass} for planar 
operads is deduced from it.

\begin{prp}\label{prp:alg-monoids}
Let $\cO\in\Op_{\BM}$ and let 
$p:\cM^\odot\to \cM\times\cM\times\BM$ be the family 
of $\BM$-monoidal categories constructed in 
\ref{sss:full-BM-family}.
Then the composition with $m:\cM^\odot\to\cM$ gives rise to an equivalence
\begin{equation}
\label{eq:theta}
\theta:\Alg_\cO(\cM^\odot)\to \Fun^\laxNC(\cO,\cM),
\end{equation} 
where $\Fun^\laxNC(\cO,\cM)\subset\Fun(\cO,\cM)$ denotes the 
full subcategory spanned by the lax NC structures on 
$\cO$ with values
in $\cM$.  
\end{prp}
\begin{proof}
First of all, we will describe the category of 
$\cO$-algebras in $\cM^\odot$ as a certain subcategory
of $\Fun(\bO,\cM)$, for a specially designed category 
$\bO$. An $\cO$-algebra in $\cM^\odot$ is a commutative
square
$$\xymatrix{
&{\cO}\ar[r]\ar[d] &{\cM^\odot}\ar[d]\\
&\BM\ar[r]&\cM\times\cM\times\BM
},
$$
where the lower horisontal map makes a choice of $(X,Y)\in\cM\times\cM$, and the upper map preserves inerts.
We keep the notation of\ref{sss:mapstoM}.
$\Fun_\BM(\cO,\cM^\odot)$ is a full subcategory of 
$\Fun_\BM(\cO,\bar\cM^\odot)=\Fun(\tilde\cO,\cM)$,
with $\tilde\cO:=\cO\times_\BM\tilde\cE$
(see~\cite{L.T}, 3.2.2.13 and \cite{GHN}, 7.3), spanned
by the functors carrying distinguished squares in $\cE(s)$
to cartesian squares.

We will denote by $\cO^\triangleright$ the colimit
$\cO\times[1]\leftarrow \cO\times\{1\}\to[0]$. For  $x\in\cO$ the canonical arrow from $x$ to the cone point 
will be denoted by $t_x$.

We define $\bO$ as the colimit of the following diagram.
\begin{equation}
\label{eq:coequalizer}
\xymatrixcolsep{4pc}\xymatrix{
&{\cO^\triangleright\sqcup\cO^\triangleright}&{\cO\sqcup\cO}\ar[l]\ar^-{(\iota_L,\iota_R)}[r] &\tilde\cO
}.
\end{equation}

By definition of $\bO$, $\Alg_\cO(\cM^\odot)$ identifies with the full subcategory of $\Fun(\bO,\cM)$ spanned by the  functors $F:\bO\to\cM$ satisfying the following properties.
 
\begin{itemize}
\item[(i)] For every object $x$ in $\cO$ over $s\in\BM$ 
the functor $F$ 
carries the distinguished squares in $\cE(s)$ to cartesian squares in $\cM$. 
\item[(ii)] For every inert morphism $a:x\to y$ in $\cO$ over 
$s\to t$ in $\BM$ and for any $e\in\cE(t))$ the map
$F(x,a^*(e))\to F(y,e)$ is an equivalence.
\item[(iii)] $F$ carries the arrows $t_x$ in both copies of
$\cO^\triangleright$ to equivalences.
\end{itemize}

The section $\top:\BM\to\tilde\cE$ induces a map
$\top:\cO\to\tilde\cO$ which, composed with the canonical map $\tilde\cO\to\bO$, defines a restriction map
\begin{equation}
\label{eq:totheta}
\Fun(\bO,\cM)\to\Fun(\cO,\cM).
\end{equation}
We wish to get the equivalence (\ref{eq:theta}) as the 
one induced from (\ref{eq:totheta}).

Let $\cO'$ be the full subcategory of $\tilde\cO$
spanned by the objects $(x,e)$ where $x\in\cO_s$ and $e\in\cE(s)$ is either $(0,|s|)$ or $(i,i),L$ or $R$.
Let us show that the embedding $\cO'\to\tilde\cO$ admits a left adjoint.
Given $(x,e)\in\tilde\cO_s$, $(y,f)\in\cO'_t$, a map
$(x,e)\to (y,f)$ is given by a collection of the following
data:
\begin{itemize}
\item $\alpha:s\to t$ in $\BM$.
\item $a:x\to y$ in $\cO$ over $\alpha$.
\item $u:e\to\alpha^*(f)$ in $\cE(s)$.
\end{itemize}
In the case $e=(i,j)$ with $i<j$, there is a unique inert 
map
$\beta:s\to s'$ with $s'\in\BM$ defined by $[j-i]\to[n]$, such that
$\beta^*((0,j-i))=e$. We choose $v=\id$ and $b=\beta_!$
and the triple $(\beta,b,v)$ defines  a universal map from $(x,e)$ to an object in $\cO'$.

In the case $e$ is $(i,i)$ or $L$ or $R$ the universal map is obviously the identity.

We define $\bO'$ as the colimit of the following
diagram similar to (\ref{eq:coequalizer})
\begin{equation}
\label{eq:coequalizer1}
\xymatrixcolsep{4pc}\xymatrix{
&{\cO^\triangleright\sqcup\cO^\triangleright}&{\cO\sqcup\cO}\ar[l]\ar^-{(\iota_L,\iota_R)}[r] &\cO'
}.
\end{equation}

A functor $F:\bO\to\cM$ is a right Kan extension
of its restriction to $\bO'$ if and only if $F$ carries
the universal maps $(x,(i,j))\to(\beta_!(x),(0,j-i))$
to equivalence
for all $x\in\cO_s$ and all $i<j\leq |s|$ (here $\beta:s\to 
s'$ is the inert map mentioned above).
Any $F$ satisfying the condition (ii) is therefore a right Kan extension. 
 
By~\cite{L.T}, 4.3.2.15, $\Fun(\bO',\cM)$ identifies with the full subcategory of $\Fun(\bO,\cM)$ spanned by the functors that are right Kan extensions of their restrictions to $\bO'$. Thus,
$\Alg_\cO(\cM^\odot)$ identifies with the full subcategory of $\Fun(\bO',\cM)$
spanned by the functors whose right Kan extension satisfies
(i)---(iii).

The composition $\cO\stackrel{\top}{\to}\cO'\to\bO'$ induces a map 
$\Fun(\bO',\cM)\to\Fun(\cO,\cM)$; this map establishes an equivalence of $\Fun^\laxNC(\cO,\cM)$
with the full subcategory of $\Fun(\bO',\cM)$ spanned by the functors $F$ satisfying the following conditions.
\begin{itemize}
\item[(1)] The composition $\cO\to\bO'\to\cM$ is a lax functor.
\item[(2)] $F$ carries the arrows $t_x$ in both copies 
of $\cO^\triangleright$ to equivalences.
\end{itemize}

Let us verify that the properties $(1),(2)$ of a functor
$F':\bO'\to\cM$ are equivalent to the properties 
(i)---(iii) of  $F:\bO\to\cM$ obtained from $F'$ by right Kan extension. Recall that $F(x,e)$ is defined by the equivalence $F(x,e)\to F'(x',e')$ where $(x,e)\to(x',e')$
is the universal map from $(x,e)$ to an object $(x',e')$
of $\cO'$. Taking this into account, (i) is equivalent to 
$(1)$, (ii) follows from the description of $F(x,e)$
in terms of $(x,e)\to(x',e')$, and condition (iii) is equivalent to $(2)$. 
\end{proof}

Here is a similar claim for planar operads. We deduce it
from Proposition~\ref{prp:alg-monoids}.

\begin{prp}
\label{prp:alg-monoids-ass}
Let $\cO$ be a planar operad and let 
$p:\cM^\Delta\to\cM\times\Ass$ be the family 
of monoidal categories constructed in \ref{sss:subfamily}.
Then the composition with $m:\cM^\Delta\to\cM$ gives rise to an equivalence
\begin{equation}
\label{eq:theta-ass}
\theta:\Alg_\cO(\cM^\Delta)\to \Fun^\laxNC(\cO,\cM),
\end{equation} 
where $\Fun^\laxNC(\cO,\cM)\subset\Fun(\cO,\cM)$ denotes the 
full subcategory spanned by the lax NC structures on 
$\cO$ with values
in $\cM$.  
Moreover, if $\cO$ is a monoidal category, 
$\theta$ restricts
to an equivalence
\begin{equation}
\label{eq:theta0}
\theta_0:\Fun^\otimes_\Ass(\cO,\cM^\Delta)\to 
\Fun^\weak(\cO,\cM).
\end{equation} 
\end{prp}
\begin{proof}
The planar operad $\cO$ is obtained by base change from 
an operad map $\cO^\otimes\to\Ass^\otimes$. Denote 
$\Ass'=\Ass^\otimes\times_{\BM^\otimes}\BM$. This is a 
strong approximation of $\Ass$ that is a $\BM$-operad. 
Denote $\cO'=\Ass'\times_{\Ass^\otimes}\cO^\otimes$. 
According to Proposition~\ref{prp:alg-monoids}, 
one has an equivalence
$$\theta':\Alg_{\cO'}(\cM^\odot)\to\Fun^\laxNC(\cO',\cM).$$

Note that $\Ass'=\Ass\sqcup\{\emptyset_R\}$
and $\cO'=\cO\sqcup\{\emptyset_R\}$.
Making the base change with respect to the map $\cM\to\cM\times\cM$ carrying $X$ to $(X,*)$, we get the required equivalence.

The second part of the theorem claiming the equivalence
(\ref{eq:theta0}), is straightforward.
\end{proof}

\begin{crl}
Let $p:\cC\to\Ass$ be a monoidal category and let 
$f:\cC\to\cM$ be an NC structure. Let $X\in\cM$ be the image of $*\in\cC_{\langle 0\rangle}$. Then the monoidal functor 
$\hat f:\cC\to\cM_X^\Delta$,
corresponding via (\ref{eq:theta0}) to $f$, is an equivalence.
\end{crl}
\begin{proof}
The $\langle 1\rangle$-component of the monoidal functor 
$\cC\to\cM_X^\Delta$ is, by definition, an equivalence.
A monoidal functor which is an equivalence of the underlying categories, is a monoidal equivalence.
\end{proof}

\subsubsection{Multiplicativity}
\label{sss:cart-mult}
The functor $\cM\mapsto\cM^\odot$ is corepresentable, therefore, it commutes with limits. In particular,
$$(\cM\times\cM')^\odot=\cM^\odot\times_\BM\cM^{\prime\odot}.$$

\subsection{Prototopoi}
\label{ss:prototopoi}
 
From now on we impose some extra conditions on a category $\cM$ which allow for an analog of Grothendieck construction interpreting functors $X\to \cM$ as objects of an overcategory $\cM_{/X}$.

\begin{dfn}
\label{dfn:prototopos}
An $\infty$-category $\cM\in\Cat^\cK$ is called {\sl $\cK$-prototopos}
if it satisfies the following properties.
\begin{itemize}
\item[(PT1)] $\cM$ has finite limits.
\item[(PT2)] Finite products in $\cM$ commute with $\cK$-indexed colimits.
\item[(PT3)] For any space $X\in\cS^\cK$ the functor $\colim:\Fun(X,\cM)\to\cM$ establishes an equivalence
$$\Fun(X,\cM)\to \cM_{/X}.$$ 
\end{itemize}
\end{dfn}

\subsubsection{}
The category $n$-$\Cat$ of $(\infty,n)$ categories as
defined in \cite{L.G} or \cite{Rz}, satisfies the above properties with $\cK$ the collection of all small categories. 
\begin{Lem}
The category $\cM=n$-$\Cat$ is a prototopos.
\end{Lem} 
\begin{proof}
Condition (PT1) is obvious and (PT2) follows from 
cartesian closedness of $n$-$\Cat$, see \cite{Rz}.

Let us check the condition (PT3). For $X$ a point there is nothing to check.
In general both $\Fun(X,\cM)$ and $\cM_{/X}$ considered as functors of $X$
carry colimits to limits. This is obvious for $\Fun(X,\cM)$ and follows
from property (4) of Definition 1.2.1, \cite{L.G}, for $X\mapsto\cM_{/X}$, as $n$-$\Cat$ is an absolute 
distributor, \cite{L.G}, 1.4.
\end{proof}

\subsubsection{Functoriality}
For a prototopos $\cM$ the equivalence (PT3) has good 
functorial properties.  An arrow $f:X\to Y$ in $\cS^\cK$
gives rise to an adjoint pair
$$ f_!:\cM_{/X}\rlarrows\cM_{/Y}:f^*$$
with $f_!$ defined by the composition and $f^*$ be the base change. Similarly, $f$ gives rise to adjoint pair
$$ f_!:\Fun(X,\cM)\rlarrows\Fun(Y,\cM):f^*$$
with $f_!$ defined by the left Kan extension and $f^*$
by the composition with $f$. The functors $f_!$ commute
with (PT3), so the adjoints also commute.

\subsubsection{Convolution}
Given three arrows
$ f_i:T\to U_i, i=1,2,$ and $g:T\to V$ in $\cS^\cK$,
one defines an operation 
$\cM_{/U_1}\times\cM_{/U_2}\to\cM_{/V}$
as the composition
\begin{equation}\label{eq:convolution0}
\cM_{/U_1}\times\cM_{/U_2}\stackrel{\times}{\to}
\cM_{/U_1\times U_2}\stackrel{(f_1\times f_2)^*}
{\longrightarrow}\cM_{/T}\stackrel{g_!}{\to}\cM_{/V}.
\end{equation}
The equivalence (PT3) allows one to rewrite this operation
in terms of the functors, as
\begin{eqnarray}\label{eq:convolution}
\Fun(U_1,\cM)\times\Fun(U_2,\cM)\to
\Fun(U_1\times U_2,\cM\times\cM)\to\\
\nonumber\Fun(U_1\times U_2,\cM)
\stackrel{(f_1\times f_2)^*}{\longrightarrow}\Fun(T,\cM)
\stackrel{g_!}{\to}\Fun(V,\cM).
\end{eqnarray}
\subsubsection{}

Applying (\ref{eq:convolution0}) and (\ref{eq:convolution}) 
to the case 
$U_1=X\times Y$, $U_2=Y\times Z$,
$V=X\times Z$, $T=X\times Y\times Z$, with the 
obvious choice of the arrows, we get maps
\begin{equation}\label{eq:convolution1}
\cM_{/X\times Y}\times\cM_{/Y\times Z}\to
\cM_{/X\times Z},
\end{equation}
\begin{equation}\label{eq:convolution2}
\Fun(X\times Y,\cM)\times\Fun(Y\times Z,\cM)\to
\Fun(X\times Z,\cM),
\end{equation}
generalizing the action of $\cM_{/X\times X}$ on $\cM_{/X}$
given by Theorem~\ref{thm:NCmonoidal}  and the action of
 $\Quiv_X(\cM)$ on $\Fun(X,\cM)$,
see (\ref{eq:action}) \footnote{Here $X$ is a space and so
$X=X^\op=\Tw(X)$.}.

\subsection{Identifying $\Quiv(\cM)$ with $\cM^\Delta$}
\label{quiv=seg}

Let $\cM$ be a $\cK$-prototopos. We consider $\cM$ as a category in $\Cat^\cK$ with the
cartesian monoidal structure.  We will now identify the 
family of monoidal categories $\Quiv_X(\cM)$, 
$X\in\cS^\cK$, with the family 
$\cM^\Delta=\{\cM^\Delta_X\}_{X\in\cM}$, restricted to 
$\cS^\cK\subset\cM$.

According to Proposition~\ref{prp:quiv-end-fam},
$\Quiv(\cM)$ is the endomorphism object of the cartesian family $\Fun(\_,\cM)$ over $\cS$ having the fiber 
$\Fun(X,\cM)$ over $X\in\cS$.
Since $\cM$ is a prototopos, the family
$\Fun(\_,\cM)$
is equivalent to the cartesian family of right $\cM$-modules $X\mapsto\cM_{/X}$. 
Since the family of monoidal categories $\cM^\Delta$
acts on the $\cM$-module $\cM_{/\_}$, this yields a 
canonical monoidal functor $\theta:\cM^\Delta\to\Quiv(\cM)$ 
of cartesian families. In order to verify that this functor 
is an equivalence, we can forget 
the monoidal structure. Comparison of the formulas
(\ref{eq:convolution1}) and (\ref{eq:convolution2})
shows that $\theta$ is an equivalence.

We have proved the following result.
\begin{prp}
\label{prp-quiv-bm-proto}
Let $\cM$ be a prototopos. We consider it as
a symmetric monoidal category with cartesian structure.
The family $\cM^\og$ of  $\BM$-monoidal categories is canonically equivalent to $\Quiv^\BM(\cM)$.
\end{prp}\qed

The equivalence above commutes with products. In fact,
given two prototopoi, $\cM$ and $\cM'$, the product
$\cM^\Delta\times\cM^{\prime\Delta}$ acts on the right
on the family of $\cM\times\cM'$-modules 
$\cM_{/\_}\times\cM'_{/\_}$. This
gives a canonical map $\cM^\Delta\times\cM^{\prime\Delta}
\to\Quiv^\BM(\cM\times\cM')$. It is easy to see this map is an equivalence. Thus, 
\begin{crl}
\label{crl:comm-prod}
The equivalence $\cM^\og\to\Quiv^\BM(\cM)$
commutes with products.
\end{crl}\qed

\subsection{$\cM$-enriched precategories, cartesian case}

Applying Proposition~\ref{prp:alg-monoids-ass} to $\cO=\Ass$,
and using the identification of $\Quiv_X(\cM)$ with 
$\cM^\Delta_X$, we immediately get the following.
\begin{crl}
\label{cor:precat=seg}
Let $\cM$ be a $\cK$-prototopos and $X\in\cS^\cK$. Then
the category of $\cM$-enriched precategories with the space of objects $X$ identifies with the category of simplicial
objects $A\in\Fun(\Delta^\op,\cM)$ satisfying the Segal condition and having $A_0=X$.
\end{crl}\qed

This result was previously obtained by R.~Haugseng
\cite{H1}, 7.5.

\section{Enriched presheaves and the Yoneda lemma}
\label{sec:yoneda}

Let $\cM$ be a monoidal category with colimits.
 In this section we construct, for an  $\cM$-enriched 
precategory $\cA$, a category of $\cM$-presheaves
$P_\cM(\cA)$; we construct a Yoneda embedding $\cA\to P_\cM(\cA)$ and prove it is fully faithful.

Note that $P_\cM(\cA)$ does not necessarily have a structure of $\cM$-enriched precategory. It has another
type of $\cM$-enrichment: it is just a left $\cM$-module.

It turns out that the Yoneda lemma is precisely about the interplay of these two types of enrichment. In this section
we  define and study functors from an $\cM$-enriched precategory to a category left-tensored over $\cM$.

This approach to the Yoneda lemma was described, for conventional enriched categories,  in our note~\cite{H.Y}.
As in \cite{H.Y}, the central notion here is the notion 
of a functor from an $\cM$-enriched precategory to a left 
$\cM$-module.

We fix a collection of categories $\cK$. 
Throughout this section $\cM$ is a monoidal category with 
$\cK$-colimits. A left $\cM$-module is, by definition, a 
category $\cB$ in $\Cat^\cK$ with the left $\cM$-action 
commuting with $\cK$-colimits in each argument. 

\subsection{Functors}

Let $\cM$ be a monoidal category in $\Cat^\cK$. Let $\cA$ be an $\cM$-precategory
with a strongly $\cK$-small category of objects $X$ and let $\cB$ be a left $\cM$-module.

Let us first recall the conventional setup. An $\cM$-functor from $\cA$ to 
$\cB$ is given by a map $f:\Ob(\cA)\to\Ob(\cB)$, and a compatible collection
of maps 
\begin{equation}
\label{eq:AtoB-conv}
\Hom_\cA(x,y)\otimes f(x)\to f(y),
\end{equation}
see \cite{H.Y}, 3.2.

In our context, a functor from $\cA$ to $\cB$ will be given by a map
$f:X\to\cB$, together with an extra structure which will correspond to
(\ref{eq:AtoB-conv}). We will now describe this extra structure.

\subsubsection{}
\label{sss:quiv-mod}
An $\cM$-module structure on $\cB$ yields an $\LM$-monoidal category which we denote by $(\cM,\cB)$. Applying
to it the functor $\Quiv^\LM_X$, see~\ref{sss:quivs},
we get a left $\Quiv_X(\cM)$-module structure on
the category $\Fun(X,\cB)$.

\

\subsubsection{}
\label{sss:explicit-action}
An explicit formula for the $\Quiv_X(\cM)$-action on
$\Fun(X,\cB)$ is given by
Proposition~\ref{prp:active} and the formulas
\ref{sss:alpha-cases} (our case is $w=am$). Let $\cA\in\Quiv_X(\cM)=\Fun(X^\op\times X,\cM)$ and 
$F\in\Fun(X,\cB)$. Then $\cA\otimes F$ is the colimit of the functor $\Tw(X)^\op\to\Fun(X,\cB)$ carrying 
$\phi:x\to y\in\Tw(X)$ to $\cA(y,\_)\otimes F(x)$~\footnote{this is, of course, the expected coend formula.}. 

\subsubsection{}
\label{sss:functor}
 
We are now ready to give our key definition.

\begin{Dfn}
Let $\cM$ be a monoidal category with $\cK$-indexed colimits, 
$\cA\in\Alg_\Ass(\Quiv_X(\cM))$ be an $\cM$-enriched precategory and let $\cB$ be a left $\cM$-module in $\Cat^\cK$.
An $\cM$-functor $F:\cA\to\cB$ is a left $\cA$-module in 
$\Fun(X,\cB)$.
\end{Dfn}

Taking into account the description of the 
$\Quiv_X(\cM)$-action given in~\ref{sss:explicit-action}, 
an $\cA$-module structure on $F\in\Fun(X,\cN)$
determines a compatible collection of arrows

\begin{equation}\label{eq:A-action}
\cA(x,y)\otimes F(x)\to F(y).
\end{equation}

$\cM$-functors from $\cA$ to a left $\cM$-module $\cB$ form a category 
denoted as $\Fun_\cM(\cA,\cB)$. This is always a category with $\cK$-indexed colimits (as it is a category of modules, see~\cite{L.HA}, 4.2.3.5).

\subsubsection{}
\label{sss:functoriality-fun-ab}
The category $\Fun_\cM(\cA,\cB)$ has the expected functoriality
in $\cA$ and in $\cB$. To see this, look at the bifibered family of $\LM$-operads $\Quiv^\LM$ over 
$\Cat\times\Op_\LM$, see~\ref{sss:quivbm}. For a fixed 
$\cM$, this gives a bifibered family $\Quiv^\LM(\cM,\_)$
of $\LM$-operads over $\Cat\times\LMod_\cM(\Cat^\cK)$ which we prefer to see as a cofibered family of 
$\LM$-operads. Applying~\ref{exm:cartesian-algebras}
and~\ref{crl:cartesian-algebras-fam}, we deduce
that the category of $\LM$-algebras in it is a bifibered 
family of the categories $\Fun_\cM(\cA,\cB)$ over 
$\PCat(\cM)\times\LMod_\cM$.

\begin{rem}
We know from Section~\ref{sec:cart} that an $\cS$-enriched precategory 
$\cA$ having a space of objects is nothing but a
Segal space. Furthermore, any category
$\cB$ with colimits is a left $\cS$-module. We will see 
in~\ref{sss:miss} that the notion of $\cM$-functor
$F:\cA\to\cB$  in this case coincides with that of a 
morphism of Segal spaces.
\end{rem}

\subsubsection{Digression: functors to operads, functors to algebras}
\label{sss:fun-op-alg}

Let $\cR^\deco$ be a category with decomposition, let $p:\cC\to\cR$ be in $\Fib(\cR^\deco)$,
and let $X$ be a category. We define $\Fun^\cR(X,\cC)$ as the fiber product
$\Fun(X,\cC)\times_{\Fun(X,\cR)}\cR$, with the diagonal map $\delta:\cR\to\Fun(X,\cR)$.
 
The object $\Fun^\cR(X,\cC)$ is fibrous over $\cR$, 
with the fiber 
$\Fun^\cR(X,\cC)_x=\Fun(X,\cC_x)$ at $x\in\cR$.

Let now $\mu:\cP^\deco\times\cQ^\deco\to\cR^\deco$
be a universal bilinear map, so that 
$\cR=\cP\otimes^\mu\cQ$.

One has the following.
\begin{Lem}
One has a canonical equivalence in $\Fib(\cP^\deco)$
\begin{equation}\label{eq:famalg}
\Alg^\mu_{\cQ/\cR}(\Fun^\cR(X,\cC))\stackrel{\sim}{\to}\Fun^\cP(X,\Alg^\mu_{\cQ/\cR}(\cC)).
\end{equation}
\end{Lem} 
\begin{proof}
Both objects represent the functor  
carrying $K\in\Cat_{/\cP}$ to the space of the maps 
$$\Map_{\Cat^+_{/\cP^\natural}}
(X^\flat\times K^\flat\times\cQ^\natural,\cC^\natural).$$
Here the map $X^\flat\times K^\flat\times\cQ^\natural
\to\cP^\natural$ is induced by $\mu$.
\end{proof}

\subsubsection{Multiplicative property}
\label{sss:multquiv}

Recall that, given monoidal categories $\cM,\cM'\in\Alg_\Ass(\Cat^\cK)$ and $\cK$-strongly small $X,X'$,
one has a canonical map (\ref{eq:gamma2})
\begin{equation}\label{eq:mu2}
\mu:\Quiv_X(\cM)\otimes\Quiv_{X'}(\cM')\to\Quiv_{X\times X'}(\cM\otimes\cM').
\end{equation}

For $\cA\in\Quiv_X(\cM)$ and $\cA'\in\Quiv_{X'}(\cM')$ 
we denote by
$\cA\boxtimes\cA'$
the image of the pair $(\cA,\cA')$ 
in $\Quiv_{X\times X'}(\cM\otimes\cM')$. 

Let now $\cB$ be a left $\cM\otimes\cM'$-module.
According to \ref{sss:algfolding}, we can think
of $\cB$ as an $\cM'$-$\cM^\rev$-bimodule. This makes
$\Fun(X',\cB)$ a $\Quiv_{X'}(\cM')$-$\cM^\rev$-bimodule,
or, equivalently, an 
$\cM$-$\Quiv_{X'}(\cM')^\rev$-bimodule. Applying the functor $\Quiv^\BM_X$, we get a 
$\Quiv_X(\cM)$-$\Quiv_{X'}(\cM')^\rev$-bimodule structure
on $\Fun(X,\Fun(X',\cB))$, which, in turn, can be 
equivalently described as a structure of a left
$\Quiv_X(\cM)\otimes\Quiv_{X'}(\cM')$-module on 
$\Fun(X,\Fun(X',\cB))$. We will see that this structure
factors through (\ref{eq:mu2}). Moreover, the following
result holds.

\begin{Prp}
There is a canonical equivalence
\begin{equation}\label{eq:mult}
\Fun_{\cM\otimes\cM'}(\cA\boxtimes\cA',\cB)=
\Fun_\cM(\cA,\Fun_{\cM'}(\cA',\cB)).
\end{equation}
\end{Prp}
\begin{proof}
The triple $\cT=(\cM,\cB,\cM'^\rev)$ is a $\BM$-monoidal category. We have
\begin{eqnarray}
\nonumber\Funop_\BM(\BM_X\times\BM_{X'}^\rev,\cT)=
\Funop_\BM(\BM_X,\Funop_\BM(\BM_{X'}^\rev,\cT))=\\
\Quiv_X^\BM(\Quiv_{X'}^\BM(\cT^\rev)^\rev).
\end{eqnarray}
This $\BM$-monoidal category describes the category
$\Fun(X,\Fun(X',\cB))$ as $\Quiv_X(\cM)$-$\Quiv_{X'}(\cM')^\rev$-bimodule. We apply the folding functor $\phi$ to get
a structure of left $\Quiv_X(\cM)\otimes\Quiv_{X'}(\cM')$-module on $\Fun(X,\Fun(X',\cB))$. Applying
Corollary~\ref{crl:phib} and \ref{sss:phifunop}, we get a canonical map 
$$\Fun(X,\Fun(X',\cB))\to\Fun(X\times X',\cB)$$
from the left 
$\Quiv_X(\cM)\otimes\Quiv_{X'}(\cM')$-module to the left
$\Quiv_{X\times X'}(\cM\otimes\cM')$-module. 

This is an equivalence of left modules over
the equivalence (\ref{eq:mu2}) of monoidal 
categories.

This implies, together with~\ref{crl:LMod=LModLMod}, the equivalence of the corresponding categories of modules
$$\LMod_{\cA\boxtimes\cA'}(\Fun(X\times X',\cB))=
\LMod_\cA(\LMod_{\cA'}(\Fun(X,\Fun(X',\cB)))).$$
Taking into account Lemma~\ref{sss:fun-op-alg},
the right-hand side of the equivalence can be rewritten
to yield
\begin{equation}\label{eq:mult1}
\LMod_{\cA\boxtimes\cA'}(\Fun(X\times X',\cB))=
\LMod_\cA(\Fun(X,\LMod_{\cA'}(\Fun(X',\cB)))).
\end{equation}
 This is precisely our claim.

\end{proof}

\subsubsection{Pre-enrichment of left $\cM$-module}
\label{sss:weak-enrichment}
Let, as above, $\cB$ be a left $\cM$-module.
For any pair of objects $b,c\in\cB$ one defines a functor
\begin{equation}
\label{eq:weak-enrichment}
\hom_\cB(b,c):\cM^\op\to\cS
\end{equation}
to spaces as a composition of $\_\otimes b:\cM\to\cB$
and the presheaf $\cB^\op\to\cS$ represented by $c\in\cB$.

We will refer to the collection of presheaves $\hom_\cB(b,c)$ on $\cM$ as a {\sl pre-enrichment of $\cB$}.

Given an $\cM$-functor $F:\cA\to\cB$, where $\cA$ is an 
$\cM$-precategory and $\cB$ a left
$\cM$-module, for any two objects $x,y\in\cA$ we have a map of presheaves on
$\cM$
\begin{equation}\label{eq:map-enrichments}
\cA(x,y)\to\hom_\cB(F(x),F(y))
\end{equation}
defined by the map 
$ \cA(x,y)\otimes F(x)\to F(y)$, see~(\ref{eq:A-action}).

The map (\ref{eq:map-enrichments}) is defined uniquely up to equivalence.
\begin{dfn}
An $\cM$-functor $F:\cA\to\cB$ is called {\sl  $\cM$-fully faithful} if
(\ref{eq:map-enrichments}) is an equivalence of presheaves for each pair $x,y\in\cA$.
\end{dfn}

The functor $i:\cS^\cK\to\cM$ preserving $\cK$-indexed colimits and carrying the terminal object to the unit in $\cM$, is the unit of $\cM$ considered as an algebra in $\Cat^\cK$. Thus, it is monoidal. 
It induces an $\LM$-monoidal functor 
$i^\LM:(\cS^\cK,\cB)\to(\cM,\cB)$, as a cartesian lifting of $i:\cS^\cK\to\cM$, 
see~\ref{exm:cartesian-algebras}. The functor $i^\LM$ admits a right adjoint
$j^\LM:(\cM,\cB)\to(\cS^\cK,\cB)$ which is, by the general
property~\cite{L.HA}, 7.3.2.7,  lax $\LM$-monoidal. It induces a functor
$j^\LM_!:\Quiv^\LM_X(\cM,\cB)\to\Quiv^\LM_X(\cS^\cK,\cB)$ that  carries $\LM$-algebras to $\LM$-algebras.

An $\cM$-functor $F:\cA\to\cB$ is just an $\LM$-algebra
$(\cA,F)$ in $\Quiv^\LM_X(\cM,\cB)$, so its image under
$j^\LM_!$ is an $\LM$-algebra  in 
$\Quiv^\LM_X(\cS^\cK,\cB)$. In other words, 
this is an $\cS^\cK$-functor $j_!(\cA)\to\cB$ which we 
denote as $j_!(F)$.

The following result is immediate.
\begin{lem}Let $F:\cA\to\cB$ be an $\cM$-fully faithful 
functor from an enriched precategory $\cA$ to a left $\cM$-module $\cB$. Then the $\cS^\cK$-functor 
$j_!F:j_!\cA\to\cB$
is $\cS^\cK$-fully faithful.
\end{lem}
\qed

\subsection{Enriched presheaves}

First of all, we will define the opposite of an
$\cM$-enriched precategory. Let $\cA\in\Alg_\Ass(\Quiv_X(\cM))$.
The opposite precategory $\cA^\op$ will have $X^\op$ as
the category of objects, and will be enriched over 
$\cM^\rev$.

Furthermore, the category $\cM$ has a right $\cM$-module structure which can be 
interpreted as a left $\cM^\rev$-module structure.
This allows one  to define the category of $\cM$-presheaves 
on $\cA$, $P_\cM(\cA)$, as $\Fun_{\cM^\rev}(\cA^\op,\cM)$. 

Finally, we will see that  $P_\cM(\cA)$ has a natural left 
$\cM$-module structure coming from the left $\cM$-module 
structure on $\cM$. 

This will allow us to define Yoneda map 
$Y:\cA\to P_\cM(\cA)$ as an $\cM$-functor. 
 
Details are presented below.

\subsubsection{Opposite enriched  category}
 
Since $\cA\in\Alg_\Ass(\Quiv_X(\cM))$, one has
an opposite algebra $\cA^\op\in\Alg_\Ass(\Quiv_X(\cM)^\rev)$.

One has a natural equivalence
$$
\Quiv_X(\cM)^\rev=\Funop(\Ass_X,\cM)^\rev=
\Funop(\Ass_X^\rev,\cM^\rev).$$

\begin{Lem}
The planar operad $\Ass_X^\rev$ is naturally equivalent to
$\Ass_{X^\op}$.
\end{Lem}
\begin{proof}
The planar operad $\Ass_X$ is defined by the functor
$\cF_X:\Delta_{/\Ass}^\op\to\cS$ given by the formula
$$\cF_X(\sigma)=\Map(\cF(\sigma),X)$$
where the functor $\cF:\Delta_{/\Ass}\to\Cat$ is described
in \ref{ss:cXplus}. The planar operad $\Ass_X^\rev$
is therefore defined by the functor carrying $\sigma$
to $\Map(\cF(\sigma\circ\op),X)$. The claim will follow
from the functorial identification
$$ \cF(\sigma\circ\op)=\cF(\sigma)^\op,$$
which, taking into account that $\cF(\sigma)$ are
posets, is enough to define on the objects, 
that is, on $\sigma:[0]\to\Ass$ with the image $\langle n\rangle$. In this case
the objects of $\cF(\sigma)$ are $x_i,y_i,\ i=1,\ldots,n$, 
and the equivalence is given by ``reading the sequence
of objects from left to right'', that is carrying $x_i$ to  $y_{n-i}$
and $y_i$ to $x_{n-i}$.
\end{proof}

The lemma identifies $\Quiv_X(\cM)^\rev$ with 
$\Quiv_{X^\op}(\cM^\rev)$. Thus, for any enriched 
precategory $\cA\in\Alg_\Ass(\Quiv_X(\cM))$ we can now assign 
its opposite $\cA^\op\in\Quiv_{X^\op}(\cM^\rev)$.

\subsubsection{Enriched presheaves}
\label{sss:enriched-presheaves}

Given an enriched precategory $\cA$, we define 
$P_\cM(\cA)$ as the category of $\cM^\rev$-functors
from $\cA^\op$ to $\cM$ (considered as a left $\cM^\rev$-module). 

Thus, 
\begin{equation}\label{eq:Mpresheaves}
P_\cM(\cA)=\LMod_{\cA^\op}(\Fun(X^\op,\cM)),
\end{equation}
where $\Fun(X^\op,\cM)$ has a canonical left 
$\Quiv_{X^\op}(\cM^\rev)$-module structure described in 
\ref{sss:quiv-mod}.
The category $P_\cM(\cA)$ has $\cK$-indexed colimits as 
$\Fun(X^\op,\cM)$ 
has $\cK$-indexed colimits and the action of 
$\Quiv_{X^\op}(\cM^\rev)$ respects them.

\subsubsection{$P_\cM(\cA)$ is a left $\cM$-module}

The left $\cM$-module structure on $\cM$
yields a left $\cM$-module structure on $P_\cM(\cA)$. This follows from \ref{sss:multquiv}: 
the category $\Fun(X^\op,\cM)$ is a
$\Quiv_{X^\op}(\cM^\rev)$-$\cM^\rev$-bimodule,
so $P_\cM(\cA)$, the category of $\cA$-modules in 
$\Fun(X^\op,\cM)$,
is a right $\cM^\rev$-module, which is the same as a
left $\cM$-module.

\subsubsection{Yoneda map}\label{sss:yoneda}

Yoneda map we present below is a special case of 
a very general phenomenon --- the structure of a left
$A\boxtimes A^\op$-module on $A$ for any associative 
algebra $A$ in a monoidal category $\cC$. Here is the construction.

 Let $\cC$ be a monoidal category and let $A\in\Alg_\Ass(\cC)$. The opposite algebra $A^\op$ 
is an algebra in $\cC^\rev$. The algebra $A$ defines 
an object in $\Alg_\BM(\pi^*(\cC))$ --- this is $A$ considered as an $A$-bimodule
in $\cC$ (considered as a $\cC$-bimodule category). 
We denote by
$A\boxtimes A^\op$ the algebra in the monoidal category
$\cC\times\cC^\rev$ defined by the pair $(A,A^\op)$. The 
category $\cC$ is a left $\cC\times\cC^\rev$-module and 
$A$ is a module over $A\boxtimes A^\op$, as shown 
in ~\ref{sss:algfolding}.

We will apply this to $\cC=\Quiv_X(\cM)$ and 
$\cA\in\Alg_\Ass(\Quiv_X(\cM))$.
We get a structure of a left $\cC\times\cC^\rev$-module 
on $\cC$, and a left $\cA\boxtimes\cA^\op$-module 
structure on $\cA$.

The folding map~(\ref{eq:phiquiv}) applied to the left 
$\cA\boxtimes\cA^\op$-module $\cA$, yields an 
$\LM$-algebra in 
$\Quiv^\LM_{X\times X^\op}(\phi\pi^*(\cM))$,
whose $\Ass$-component is $\cA\boxtimes\cA^\op$, an 
associative algebra object in 
$\Quiv_{X\times X^\op}(\cM\otimes\cM^\rev)$, and whose 
$m$-component is $\Fun(X\times X^\op,\cM)$, where $\cM$ 
is considered as a left $\cM\otimes\cM^\rev$-module. 
In other words, we have defined an 
$\cM\otimes\cM^\rev$-functor $\tilde Y$ from 
$\cA\boxtimes\cA^\op$ to $\cM$.

The Yoneda map is defined as the
$\cM$-functor $Y:\cA\to P_\cM(\cA)$ (from 
the $\cM$-precategory $\cA$ to the left $\cM$-module 
$P_\cM(\cA)$) corresponding to $\tilde Y$ via the equivalence
\begin{equation}
\Fun_{\cM\otimes\cM^\rev}(\cA\boxtimes\cA^\op,\cM)=
\Fun_\cM(\cA,P_\cM(\cA)),
\end{equation}
obtained as the special case of (\ref{eq:mult})
for  $\cM'=\cM^\rev$ and $\cA'=\cA^\op$.

\

The enriched presheaves of the form $Y(x)$, $x\in X$, are
called {\sl representable presheaves}. We will write 
$Y_\cA(x)$ when we have to explicitly mention the enriched precategory $\cA$.

\subsubsection{Free presheaves}

Since enriched presheaves on $\cA$ are just
$\cA^\op$-modules with values in the category 
$\Fun(X^\op,\cM)$, we have  
the notion of a free presheaf --- this is just a free
left $\cA^\op$-module.  We will now show that  
representable presheaves are free.

Let $\bar h:X\times X^\op\to\cS^\cK$ be the Yoneda map
described in \ref{sss:ooyoneda}. For $x\in X$
we define $h_x:X^\op\to\cS^\cK$ the presheaf 
represented by 
$x$. Applying the unit $i:\cS^\cK\to\cM$, we get a functor $i\circ h_x\in\Fun(X^\op,\cM)$. We have the
following result.

\begin{crl} The enriched presheaf $Y_\cA(x)$ on $\cA$, 
represented by $x\in X$,
is a free $\cA^\op$-module generated by 
$i\circ h_x$.
\end{crl}
\begin{proof}
Denote 
$$ F:\Fun(X^\op,\cM)\rlarrows P_\cM(\cA):G$$
the adjoint pair of functors, $G$ being the forgetful functor and $F$ being the free $\cA^\op$-algebra functor. The functor $G(Y_\cA(x))$
is just $\cA(\_,x)\in\Fun(X^\op,\cM)$.

Denote by $\one_X\in\Quiv_X(\cM)$ the unit of the 
monoidal structure. According to~\ref{sss:unit-general},
$G(Y_{\one_X}(x))$ identifies with $i\circ h_x$.
Let $i_\cA:\one_X\to\cA$ be the unit map. One has a canonical 
map 
$G(Y_{\one_X}(x))\to G(Y_\cA(x))$. This
gives a functor $i\circ h_x\to G(Y_\cA(x))$, or, by 
adjunction, a map $F(i\circ h_x)\to Y_\cA(x)$.

To verify this is an equivalence, we can apply once more
the forgetful functor $G$ and prove the equivalence
$\theta:\cA^\op\otimes(i\circ h_x)\to G(Y_\cA(x))$.  
The explicit formula \ref{sss:explicit-action}
for the tensor product expresses 
$\cA^\op\otimes(i\circ h_x)$
as the colimit
\begin{equation}
\label{eq:colimAh}
\colim_{z\to t}\{\cA(\_,z)\times\Map_X(t,x):
\Tw(X)^\op\to\Fun(X^\op,\cM)\}.
\end{equation}

One has a canonical map $j:\cA(\_,x)\to\cA^\op\otimes 
(i\circ h_x)$ defined by $\id_x\in\Tw(X)$ and 
$\id_x\in\Map_X(x,x)$. 
Its composition with $\theta$ is an equivalence, as 
$\theta$ can be presented is the composition
$$\cA^\op\otimes (i\circ h_x)\to
\cA^\op\otimes G(Y(x))\to G(Y(x)),$$
and the composition $\theta\circ j$ can be now factored
as 
$$\cA(\_,x)\to\cA(\_,x)\otimes\cA(x,x)\to\cA(\_,x),$$
where the first arrow is induced by the identity $\one\to\cA(x,x)$, and the second by composition in $\cA$.
Now it remains to verify that $j$ is an equivalence.
We can rewrite (\ref{eq:colimAh}) as the colimit
\begin{equation}
\label{eq:colimAh2}
\colim_{z\to t\to x}\{A(\_,z):\Tw(X)^\op\times_{X^\op}
(X_{/x})^\op\to\Fun(X^\op,\cM)\}
\end{equation}
of the functor assigning $\cA(\_,x)$ to the object
$z\to t\to x$ of $\Tw(X)^\op\times_{X^\op}(X_{/x})^\op$,
where the colimit of the composition 

$$ \Tw(X)^\op\times_{X^\op}(X_{/x})^\op
\stackrel{\pi}{\to}X_{/x}\stackrel{\alpha}{\to}
\Fun(X^\op,\cM),$$
with $\pi$ carrying $z\to t\to x$ to $z\to x$
and $\alpha$ carrying $z\to x$ to $\cA(\_,x)$.
Now our claim follows as $\pi$ is cofinal by the Quillen's Theorem A \cite{L.T}, 4.1.3.1 and $X_{/x}$
has a terminal object $\id_x$.
\end{proof}

\subsubsection{The enriched Yoneda lemma}
\label{sss:EYL}

Let $x\in\cA$ and let $F\in P_\cM(\cA)$.

\begin{Prp}
The presheaf 
$\hom_{P_\cM(\cA)}(Y(x),F)$
on $\cM$ is represented by  $F(x)$.
\end{Prp}  
\begin{proof}We will construct an equivalence of presheaves
\begin{equation}\label{eq:yoneda-1}
h_{F(x)}\to\hom_{P_\cM(\cA)}(Y(x),F),
\end{equation}
where $h_{F(x)}$ is the presheaf on $\cM$ represented by $F(x)$.

We will construct a canonical equivalence
\begin{equation}\label{eq:yoneda-2}
\Map_\cM(m,F(x))\to\Map_{P_\cM(\cA)}(m\otimes Y(x),F),
\end{equation}
which is, since $Y(x)$ is a free right $\cA$-module
generated by $i(h_x)$, the same as
\begin{equation}\label{eq:yoneda-3}
\Map_\cM(m,F(x))\to\Map_{\Fun(X^\op,\cM)}
(m\otimes i(h_x),F).
\end{equation}

An object $m\in\cM$ defines an adjoint pair
\begin{equation}
L_m:\cS\rlarrows\cM:R_m,
\end{equation}
where $L_m(S)=i(S)\otimes m$  and $R_m(M)=\Map(m,M)$.
Applying to both parts the functor $\Fun(X^\op,\_)$, we get an adjoint pair
\begin{equation}
L_m:P(X)\rlarrows\Fun(X^\op,\cM):R_m.
\end{equation}
The equivalence (\ref{eq:yoneda-3}) is obtained
from this adjunction and from the conventional Yoneda
applied to $\Map_\cM(m,F)\in P(X)$.
\end{proof}
As usual, we now have

\begin{Crl}
The Yoneda embedding $Y:\cA\to P_\cM(\cA)$
is $\cM$-fully faithful.
\end{Crl}
\qed

\subsection{Enrichment of a left-lensored category}

Let once more $\cM$ be a monoidal category with colimits
and let $\cB$ be a left $\cM$-module.

Assume that for all $b,c\in\cB$ the weak enrichment functor
$$
\hom_\cB(b,c):\cM^\op\to\cS
$$
defined by (\ref{eq:weak-enrichment}), is representable.
Then one would like to believe
that $\cB$ acquires a canonical $\cM$-enrichment. 

This is in fact so, as Proposition~\ref{prp:repr}  
below asserts.

Let $X$ be a category with a functor $F:X\to\cB$. Recall that 
the category $\Fun(X,\cB)$ is left-tensored over 
$\Quiv_X(\cM)$, so it makes sense to look for an endomorphism object of $F\in\Fun(X,\cB)$ in 
$\Quiv_X(\cM)$.

\begin{prp}\label{prp:repr}
Let $\cM$ be a monoidal category in $\Cat^\cK$ and $\cB$ be a 
left $\cM$-module.
Let $F:X\to\cB$ be a functor from a strongly $\cK$-small category $X$ to 
$\cB$ so 
that, for any 
$x,y\in X$ the functor $\hom_\cB(F(x),F(y)):\cM^\op\to\cS$ is 
representable.
Then the endomorphism object $\cA=\End_{\Quiv_X(\cM)}(F)$ 
exists; $\cA$ is an $\cM$-enriched precategory and
the corresponding $\cM$-functor $\wt F:\cA\to\cB$, extending 
$F:X\to\cB$, is $\cM$-fully faithful.
\end{prp}
\begin{proof}
Let us recall the construction of the endomorphism object
presented in~\cite{L.HA}, 4.7.2.
For a monoidal category $\cC$, a left $\cC$-module $\cF$, and an object
$F\in\cF$, a monoidal category $\cC[F]$, whose objects are $C\otimes F\to F,\ C\in\cC$, is constructed. If $\cC[F]$ has a terminal object, this is 
the endomorphism object of $F$ in $\cC$, and it automatically acquires an algebra structure. 

We apply this construction to $\cC=\Quiv_X(\cM)$ acting on 
$\cF=\Fun(X,\cB)$ and to the object $F\in\cF$.

Let us describe the terminal object of $\cC[F]$.
The formulas~\ref{sss:explicit-action}
for the action of $\Quiv_X(\cM)$ on $\Fun(X,\cB)$
identify $\Map_\cF(\cA\otimes F,F)$ with
\begin{multline}
\lim_{\psi:z\to z'\in\Tw(X)}\lim_{\phi:x\to y\in\Tw(X)}\Map(\cA(y,z)\otimes
F(x),F(z'))=\\
\lim_{\psi:z\to z'\in\Tw(X)}\lim_{\phi:x\to y\in\Tw(X)}
\Map(\cA(y,z),\hom_\cB(F(x),F(z'))).
\end{multline}
This proves that the terminal object 
$\cA:X^\op\times X\to\cM$
is given by the formula $\cA(x,y)=\hom_\cB(F(x),F(y))$
\footnote{More precisely, $\cA(x,y)$ is the object of 
$\cM$ representing
$\hom_\cB(F(x),F(y))$.}.

\end{proof}

Let us study the functoriality of the above construction with respect to change of $\cM$ and $X$.

Let $a:\cN\to\cM$ be a monoidal functor having right adjoint $b:\cM\to\cN$
($b$ is automatically lax monoidal). Given a left $\cM$-module $\cB$
and a functor $F:X\to \cB$ satisfying the conditions of the previous proposition,
we can construct two endomorphism algebras, $\cA=\End_{\Quiv_X(\cM)}(F)$
and $\cA_\cN=\End_{\Quiv_X(\cN)}(F)$. One has
\begin{lem}\label{lem:repr-comparison}
One has a canonical equivalence $\cA_\cN=b_!(\cA)$.
\end{lem}
\begin{proof}
The adjoint pair $a:\cN\rlarrows\cM:b$ induces an
adjoint pair 
$$a_!:\Quiv_X(\cN)\rlarrows\Quiv_X(\cM):b_!.$$

Thus, for any $\cN$-enriched precategory $\cA'$ one has
\begin{eqnarray}
\nonumber\Map(\cA',\cA_\cN)=\Map(\cA'\otimes F,F)=
\Map(a_!(\cA')\otimes F,F)=\\
\nonumber\Map(a_!(\cA'),\cA)=
\Map(\cA',b_!(\cA)).
\end{eqnarray}
This implies $\cA_\cN=b_!(\cA)$.
\end{proof}

Let now $f:X'\to X$ be a map of spaces and $F:X\to\cB$ be as above.
Assume $F$ satisfies the conditions of 
Proposition~\ref{prp:repr}, so that
$\cA=\End_{\Quiv_X(\cM)}(F)$ exists.
Denote $F'=f^!(F)=F\circ f$.
\begin{lem}\label{lem:repr-comparison2}
The functor $F'$ satisfies~\ref{prp:repr} and 
its endomorphism object is $\cA':=f^!(\cA)$.
\end{lem}
\begin{proof}
The condition of ~\ref{prp:repr} is clearly satisfied.
Let $\cA'=\End_{\Quiv_X'(\cM)}(F')$. Applying $f^!$ to the pair $(\cA,F)$, we get $(f^!\cA,F')$, so an algebra map
$f^!\cA\to\cA'$. It remains to prove this is an equivalence. For this one can forget the multiplicative structure and compare $f^!\cA(x,y)$ with $\cA'(x,y)$ 
for $x,y\in X'$. This proves the claim.
\end{proof}

\subsubsection{} 
\label{sss:prp-repr}
The smallness requirement in 
Proposition~\ref{prp:repr} is not really important.
One has
\begin{Crl}Let $\cM$ be a monoidal category on $\Cat^L$,
$\cB$ be a left $\cM$-module in $\Cat^L$ and let $f:X\to\cB$ be a functor, such that, for any $x,y\in X$ $\hom_\cB(F(x),F(y))$ is representable. Then $\cA=\End_{\Quiv_X(\cM)}(f)$ exists; $\cA$ is an $\cM$-enriched precategory and the $\cM$-functor $\tilde F:\cA\to\cB$  extending $F:X\to\cB$, is $\cM$-fully faithful.
\end{Crl} 
\begin{proof}
Let $\kappa$ be a cardinal satisfying~\ref{lem:kappa}.
A large category $X$ can be presented, after a change of 
universe, as a $\kappa$-filtered colimit of $\kappa$-compact 
categories $X_\alpha$, for which Propostion~\ref{prp:repr}
can be applied. We get 
$\Quiv_X(\cM)=\lim_\alpha\Quiv_{X_\alpha}(\cM)$ which implies
the analogous expression for $\Alg_\Ass(\Quiv_X(\cM))$. Denote
$i_\alpha:X_\alpha\to X$ the canonical map. By 
Lemma~\ref{lem:repr-comparison2} the endomorphism objects
of $F\circ i_\alpha$ are compatible, which gives the endomorphism object of $F$.
\end{proof}

\subsubsection{The case $\cM=\cS$}
\label{sss:miss}
Let now $\cM=\cS$. For a  category $\cB$ with colimits
we apply \ref{sss:prp-repr} to the identity functor $F:=\id_\cB$.
The endomorphism object of $\id_\cB$,
$E:=\End_{\Quiv_\cB(\cS)}(\id)$ is  an 
$\cS$-precategory with the category of objects $\cB$.
The unit map $\one_\cB\to E$ in $\Quiv_{\cB}(\cS)$ is 
an equivalence by the explicit description of $E(x,y)$
given in the proof of~\ref{prp:repr}. 
One has a canonical map $i:\one_\cB\to\cB$ defined by the 
$\one_\cB$-module structure on $\id_\cB$. Functoriality
(\ref{sss:functoriality-fun-ab}) defines for any 
$\cA\in\Quiv_X(\cM)$ a map
\begin{equation}
\label{eq:A1AB}
\Map_{\PCat(\cS)}(\cA,\one_\cB)\to\Map_\cS(\cA,\cB)=
\Fun_\cS(\cA,\cB)^\eq
\end{equation}
fibered over $\Map_\Cat(X,\cB)$. It is an equivalence as it 
induces equivalences of the fibers.

Let us now assume that $X$ is a space. Then the embedding
$i:\cB^\eq\to\cB$ induces an equivalence $\Map(X,\cB^\eq)=
\Map(X,\cB)$ and the equivalence (\ref{eq:A1AB}) can be rewritten as
\begin{equation}
\label{eq:A1AB2}
\Map_{\PCat(\cS)}(\cA,i^!\one_\cB)=\Map_\cS(\cA,\cB).
\end{equation}
Note that $i^!\one_\cB$ is precisely the $\cS$-enriched precategory we assigned to $\cB$ in 
\ref{sss:from-cat-to-senriched}.
As $\PCat(\cS)$ is equivalent to the category of Segal spaces,
we deduce that $\Map_\cS(\cA,\cB)$ is the space of maps between the corresponding Segal spaces.

\subsubsection{} 
Let $\cM\in\Alg_\Ass(\Cat^L)$, 
$\cA\in\PCat(\cM)$, $\cB\in\LMod_\cM(\Cat^L)$. Assume that
for all $b,c\in\cB$ the presheaf $\hom_\cB(b,c)$ is representable. 

Denote $\overline\cB=\End_{\Quiv_\cB(\cM)}(\one_\cB)$.
By~\ref{sss:functoriality-fun-ab}, one has a map of
spaces 
\begin{equation}
\label{eq:funAB}
\gamma:\Map_{\PCat(\cM)}(\cA,\overline\cB)\to
\Fun_\cM(\cA,\cB)^\eq,
\end{equation}
deduced by functoriality from the canonical action of
$\overline\cB$ on $\one_\cB\in\Fun(\cB,\cB)$.
\begin{Lem}
The map (\ref{eq:funAB}) is an equivalence.
\end{Lem}
\begin{proof}
Let $X$ be the category of objects of $\cA$. The source and the target of (\ref{eq:funAB}) are fibered over 
$\Map(X,\cB)$.
It is sufficient to verify that for any $f:X\to\cB$ the
fiber $\gamma_f$ at $f$ of (\ref{eq:funAB})
is an equivalence. By Lemma~\ref{lem:repr-comparison2},
the source of $\gamma_f$ identifies with
$\Map_{\Alg(\Quiv_X(\cM))}(\cA,\End_{\Quiv_X(\cM)}(f))$,
which, by the universal property of the endomorphism object
\cite{L.HA}, 4.7.1.41,
identifies with the target of $\gamma_f$.
\end{proof}

\subsubsection{}
\label{sss:funM-internal}
 The equivalence (\ref{eq:funAB}) can be 
upgraded to a description of $\Fun_\cM(\cA,\cB)$, 
provided $\cM$
is a presentably $\cE_2$-monoidal category, that is a 
$\cE_2$-algebra in the category of presentable categories
and colimit preserving functors.

In this case, according to \cite{GH}, 4.3.5 and 4.3.16,
$\PCat(\cM)$ is presentably monoidal. In particular,
the tensor product in $\PCat(\cM)$ admits a right adjoint
$$\cA,\cA'\mapsto\FUN(\cA,\cA').$$ 
For any  presentable category $\cB$ left-tensored
over $\cM$ we denote 
$\overline\cB=\End_{\Quiv_\cB(\cM)}(\one_\cB)$.
\begin{Prp}
Let $\cM$ be presentably $\cE_2$-monoidal category, 
$\cA$ a $\cM$-enriched precategory, $\cB$ presentable
category left-tensored over $\cM$. Then one has an equivalence
$$\overline{\Fun_\cM(\cA,\cB)}=\FUN(\cA,\overline\cB).$$
\end{Prp}
\begin{proof}
For any $\cA'\in\PCat(\cM)$ one has 
\begin{eqnarray}
\Map_{\PCat(\cM)}(\cA',\overline{\Fun_\cM(\cA,\cB)})=
\Fun_\cM(\cA',\Fun_\cM(\cA,\cB))^\eq= \\
\nonumber\Fun_{\cM\otimes\cM}(\cA'\boxtimes\cA,\cB)^\eq= 
\Fun_\cM(\cA'\otimes\cA,\cB)^\eq=\Map_{\PCat}(\cA'\otimes\cA,\bar\cB)=\\
\nonumber\Map_{\PCat}(\cA',\FUN(\cA,\bar\cB)).
\end{eqnarray}
\end{proof}
\begin{rem}In particular, for $\cM$ presentably $\cE_2$-monoidal the category of enriched presheaves $P_\cM(\cA)$ can be defined
entirely in the world of $\cM$-categories as 
$\FUN(\cA^\op,\cM)$.
We are grateful to the referee for pointing this out.
\end{rem}
\

We will now use Proposition~\ref{prp:repr} to 
prove Proposition~\ref{sss:funop} describing
the operad $\Funop(\cC,\cD)$ in the case where $\cC,\cD$ 
are symmetric monoidal categories.

\subsubsection{Proof of Proposition~\ref{sss:funop}}
\label{sss:proof-funop}

Recall that the bifibration 
$p=(p_1,p_0):\Ar\to\Cat\times\Cat$ is 
the composition of the embedding $\Ar\to\Cat_{/[1]}$
with the restriction to the ends of $[1]$.

Denote by $q:U\to\Cat$ the universal cocartesian fibration. This is the cocartesian fibration classified by the identity
functor $\id:\Cat\to\Cat$. There is a canonical evaluation map
\begin{equation}
\label{eq:ev}
\ev:U\times_\Cat\Ar\to U
\end{equation}
of cocartesian fibrations over $\Cat$, defined by the 
family of evaluation maps 
$\cP\times\Fun(\cP,\cQ)\to\cQ$.
We construct this map as follows~\footnote{Construction
of the map~\ref{eq:ev} is the only reason for which we delayed the proof of Proposition~\ref{sss:funop} till 
Section~\ref{sec:yoneda}.}.

We apply Corollary~\ref{sss:prp-repr} to $\cM=\Cat$,
$X=\Cat$ and $f=\id$. We deduce the existence of
$\cA=\End_{\Quiv_\Cat(\Cat)}(\id)$, with 
$\cA:\Cat^\op\times\Cat\to\Cat$ given by the formula
$\cA(\cP,\cQ)=\Fun(\cP,\cQ)$. This yields the evaluation
map~\ref{eq:ev} as the cocartesian fibration 
$q:U\to\Cat$ is classified by $\id:\Cat\to\Cat$ and the
bifibration $p:\Ar\to\Cat\times\Cat$ is classified by
$\cA$.
 
Note that (\ref{eq:ev}) preserves products.
Therefore, it induces a SM functor
\begin{equation}
\label{eq:evtimes}
\ev^\times:U^\times\times_{\Cat^\times}\Ar^\times\to U^\times.
\end{equation}
Let now $\cC$ be a strong approximation of an operad
and let $\cP,\cQ$ be $\cC$-monoidal categories.
Making the base change of~(\ref{eq:evtimes}) with respect 
to the map $\cC\to\Cat^\times\times\Cat^\times$ defined
by the pair $(\cP,\cQ)$, we get
$$\cP\times\cF_\cC^{\cP,\cQ}\to\cQ,$$
which yields the map of $\cC$-operads
\begin{equation}
\label{eq:tofunop}
\eta_\cC:\cF_\cC^{\cP,\cQ}\to\Funop_\cC(\cP,\cQ).
\end{equation}
Note that $\eta_\cC$ commutes with the base change:
for any $q:\cD\to\cC$ $q^*(\eta_\cC)$ is a composition
of $\eta_\cD$ with the equivalence described 
in~\ref{sss:basechange}.

It remains to prove that $\eta_\cC$ is an equivalence of 
$\cC$-operads. 
 The fiber of~(\ref{eq:evtimes})
at $\langle 1\rangle\in\Com$ is~(\ref{eq:ev}). This gives,
for any $x\in\cC_1$,
an equivalence of both
$(\cF_\cC^{\cP,\cQ})_x$ and $\Funop_\cC(\cP,\cQ)_x$ with
$\Fun(\cP_x,\cQ_x)$. Thus, $\eta_\cC$ induces an
equivalence of the categories of colors.

It remains to verify the spaces of active maps in 
$\cF_\cC^{\cP,\cQ}$ and $\Funop_\cC(\cP,\cQ)$. For $n$-ary operations, 
we will use the approximation $Q_n$ of the free operad
$\bC_n$, as in~\ref{sss:Qn}. Fix a map $\phi:\Q_n\to\cC$.
We have to verify that $\phi^*(\phi_\cC)$ induces an equivalence of the spaces of $Q_n$-algebras. Since 
$\eta_\cC$ commutes with the base change, the claim immediately reduces to the case $\cC=Q_n$ and $\phi=\id$.

$\Q_n$-monoidal category $\cP$ is just a collection 
$\cP_0,\ldots,\cP_n$ of categories endowed with a functor
$\cP_1\times\ldots\times\cP_n\to\cP_0$. Both spaces
of algebras, $\Map_{\Op_{Q_n}}(Q_n,\cF_{Q_n}^{\cP,\cQ})$ and
$\Map_{\Op_{Q_n}}(Q_n,\Funop_{Q_n}(\cP,\cQ))=
\Map_{\Op_{Q_n}}(\cP,\cQ)$, identify with the space
$$\prod_{i=1}^n\Map(\cP_i,\cQ_i)\times_{\Map(\prod\cP_i,\cQ_0)}(\Map(\prod\cP_i,\cQ_0)\times[1])\times_{\Map(\prod\cP_i,\cQ_0)}\Map(\cP_0,\cQ_0)$$
describing the 2-diagrams
\begin{equation}
\label{eq:spaceofdiagrams}
\xymatrix{
&{\cP_1\times\ldots\times\cP_n}\ar[r]\ar^{\{f_i\}}[d]&\cP_0
\ar@{=>}[ld]\ar^{f_0}[d]\\
&{\cQ_1\times\ldots\times\cQ_n}\ar[r] &\cQ_0,
}
\end{equation}
with the horizontal arrows defined by the operations 
in the $Q_n$-monoidal categories $\cP$ and $\cQ$.
 
\section{Completeness} 
\label{sec:completeness}
In this section $\cM$ is a monoidal category in $\Cat^L$.
We will define here $\cM$-enriched categories as
$\cM$-enriched precategories satisfying a completeness 
condition. 

This material is not new as our notion of 
enriched precategory is equivalent, for $X$ a space, to that 
of~\cite{GH} where the completeness issue has already been 
addressed, see~\cite{GH}, 5.6. We present a new, very easy, 
proof of~\cite{GH}, 5.6, based on Proposition~
\ref{prp:repr}.

\subsection{}
According to Section~\ref{sec:cart}, the category of 
$\cS$-enriched precategories is equivalent to the category of 
Segal spaces. This justifies the following definition.

\begin{dfn}
\label{dfn:enrichedcat}
\begin{itemize}
\item[1.] An $\cS$-enriched precategory $A\in\Quiv_X(\cS)$ 
is called {\sl an $\cS$-enriched category} if $X$ is a 
space and the Segal space defined by $A$ is complete.
\item[2.] 
An $\cM$-enriched precategory $\cA$ is called
{\sl an $\cM$-enriched category} if its image 
$j_!(\cA)\in\Quiv_X(\cS)$ with respect to a functor
$j:\cM\to\cS$ defined as $j(M)=\Map(\one,M)$, is an $\cS$-enriched category.
\end{itemize}
\end{dfn}

\subsubsection{}
\label{sss:complete2}

In this section we denote $\PCat(\cM)=\Alg_\Ass(\Quiv(\cM))$
the category of $\cM$-enriched precategories, and 
$\Cat(\cM)$ the category of $\cM$-enriched categories.

In \ref{ss:localization} below we prove that the full embedding 
$\Cat(\cM)\to\PCat(\cM)$ admits a left adjoint 
localization functor. The construction of localization makes use of Yoneda embedding.

Let $\cJ$ denote the (conventional) contractible groupoid on two objects. We will denote by the same letter the corresponding simplicial space and the $\cS$-enriched
precategory. We will also denote by $*$ the terminal
object in $\PCat(\cS)$ (the singleton).

We keep the notation $i:\cS\rlarrows\cM:j$
for the adjoint pair of functors. For an $\cM$-enriched precategory $\cA$ completeness of $j_!(\cA)$ means that
the canonical map
\begin{equation}
\Map_\PCat(\cS)(*,j_!(\cA))\to\Map_\PCat(\cS)(\cJ,j_!(\cA))
\end{equation}
is an equivalence. The source of this map is $X$, the space of objects of $\cA$, whereas the target can be rewritten
using the adjunction. Defining $\cJ_\cM:=i_!(\cJ)$, we  get a standard characterization of completeness.
\begin{Prp}
An $\cM$-precategory $\cA$ with the space of objects $X$ 
is complete iff the natural map
\begin{equation}
X\to\Map_\PCat(\cM)(\cJ_\cM,\cA)
\end{equation}
is an equivalence.
\end{Prp}

\subsection{Localization functor}
\label{ss:localization}

Given $\cA\in\Alg_\Ass(\Quiv_X(\cM))$, we define $X'$ to be the 
subspace of representable functors in $P_\cM(\cA)^\eq$.

The embedding $Y':X'\to P_\cM(\cA)$ is tautological.
The localization $L(\cA)$ is defined as the endomorphism
object $\cA':=\End_{\Quiv_{X'}(\cM)}(Y')$ whose existence
is guaranteed by Proposition~\ref{prp:repr}.

By definition, the Yoneda embedding $Y:X\to P_\cM(\cA)$ factors
through $Y'$, yielding $y:X\to X'$. 
The universality of $\cA'$ yields a unique map $a:\cA\to\cA'$
over $y:X\to X'$.

To prove the completeness of $\cA'$, we use criterion~\ref{sss:complete2}. Let $i:P_\cM(\cA)^\eq\to P_\cM(\cA)$
be the obvious embedding. The space  
$\Map_{\PCat(\cS)}(\cJ,j_!\cA')$ identifies with the subspace of
$$\Map_{\PCat(\cS)}(\cJ, j_!i^!\one_{P_\cM(\cA)})=
\Map_\Seg(\cJ,P_\cM(\cA)),$$
$\Seg$ being the category of Segal spaces,
spanned by the functors with values in representable objects.
Thus, completeness of $P_\cM(\cA)$ considered as a Segal space implies completeness of $\cA'$. Note that the functor
$a:\cA\to\cA'$ induces an equivalence
$$\Map_{\PCat(\cM)}(\cJ_\cM,\cA)=\Map_{\PCat(\cM)}(\cJ_\cM,\cA'),$$
as both identify with the same subspace $X'$ of $P_\cM(\cA)$.
Therefore, if $\cA$ is complete, $X=X'$ and $\cA=\cA'$.
This proves universality of the map $\cA\to\cA'$.

\begin{crl}
\label{crl:functorsascorr}
Let $\cA,\cB$ be $\cM$-enriched precategories,
so that $\cB$ is complete. Then the composition with Yoneda embedding $\cB\to P_\cM(\cB)$ identifies the space
$\Map_{\PCat(\cM)}(\cA,\cB)$ with the subspace of $\Fun_\cM(\cA,P_\cM(\cB))=\Fun_{\cM\times\cM^\rev}(\cA\boxtimes\cB^\op,\cM)$ consisting of the maps $F:\cA\to P_\cM(\cB)$ whose essential image is in the space of representable presheaves.
\end{crl}
\begin{proof}
Let $\cA\in\Quiv_X(\cM)$, $\cB\in\Quiv_{X'}(\cM)$. Since 
$\cB$ is complete, the Yoneda embedding $Y:\cB\to P_\cM(\cB)$
induces an equivalence of $X'$ with the space of representable presheaves in  $P_\cM(\cB)$.
One has a commutative diagram of spaces
\begin{equation}
\xymatrix{
&{\Map(\cA,\cB)}\ar[dr]\ar[rr]& &{\Fun_\cM(\cA,P_\cM(\cB))^\rep\ar[dl]}\\
&{}&{\Map(X,X')} &&
},
\end{equation} 
where the superscript $^\rep$ denotes the space of $\cM$-functors with essential image in representable presheaves.  
In order to prove the horizontal arrow is an equivalence, 
it is enough to verify the equivalence of fibers over any
$f:X\to X'$. The fiber of the left-hand side is 
$\Map(\cA,f^!\cB)$, whereas the fiber of the right-hand side is the space of $\cA$-module structures on the composition $F: X\stackrel{f}{\to}X'\to P_\cM(\cB)$.
Since $\cB$ is the endomorphism object of 
$X'\to P_\cM(\cB)$, $f^!(\cB)$ is the endomorphism object of $F$, so that the right-hand side also identifies with 
$\Map(\cA,f^!\cB)$.
\end{proof}

\section{Correspondences}
\label{sec:corr}

In the conventional category theory, a correspondence 
from $C$ to $D$ can be defined as a functor 
$C^\op\times D\to\Set$.
Equivalently, it can be defined as a category $X$ over $[1]$, with a pair of equivalences $C\cong X_0,\ D\cong X_1$.

A similar description for $\infty$-categories is also
well-known, see~\cite{L.T}, 2.3.1, or~\cite{AF}, 4.1. 
See also  D. Stevenson, \cite{S}.

In this Section we prove this result, using the techniques of Section~\ref{sec:yoneda}, in a greater 
generality, including, for instance, 
$(\infty,n)$-categories.

\subsection{Correspondences, I}

Let $\cM$ be a monoidal category with colimits, and let 
$\cC,\cD\in\Cat(\cM)$. A correspondence from $\cC$ to 
$\cD$ is, by definition, an $\cM$-functor 
$\cD\to P_\cM(\cC)$.

Equivalently, a correspondence can be defined as an
$\cM\otimes\cM^\rev$-functor from $\cD\boxtimes\cC^\op$ to $\cM$, or as an $\cM^\rev$-functor from $\cC^\op$ to
$P_{\cM^\rev}(\cD^\op)$.

A correspondence is right-representable if the essential image of the corresponding 
$\cM$-functor $\cD\to P_\cM(\cC)$ belongs to the space
of representable presheaves. According to \ref{crl:functorsascorr}, a right-representable correspondence gives rise to a uniquely defined functor $\cD\to\cC$. Left-representable correspondences are defined similarly; they lead to functors $\cC\to\cD$. An adjoint pair of functors in $\Cat(\cM)$ can be defined as a
correspondence which is simultaneously left and right representable.

\subsubsection{} 

The category of $\cM$-enriched correspondences,
$\Cor(\cM)$, classifies pairs $\cC,\cD\in\Cat(\cM)$, endowed with an object of $\Fun_{\cM\otimes\cM^\rev}(\cD\boxtimes\cC^\op,\cM)$. This leads to the following definition.
\begin{equation}
\label{eq:cor1}
\Cor(\cM)=\Cat(\cM)^2\times_{\Alg_\Ass(\Quiv(\cM\otimes\cM^\rev))}\Alg_\LM(\Quiv^\LM(\cM\otimes\cM^\rev,\cM)),
\end{equation}
where the map $\Cat(\cM)^2\to\Cat(\cM\otimes\cM^\rev)
\subset\Alg_\Ass(\Quiv(\cM\otimes\cM^\rev))$ carries a pair
$(\cC,\cD)$ to $\cD\boxtimes\cC^\op$.

\

We will prove later that, for $\cM$ 
prototopos~\ref{dfn:prototopos}, the described above category $\Cor(\cM)$ can be equivalently
described as $\Cat(\cM)_{/[1]}$. We will start with a digression about Kan extensions, with the aim 
of proving Corollary~\ref{crl:bm-ass-1}.

\subsection{Kan extensions}

Let $f:Y\to X$ be a cocartesian fibration and let $\cM$ be a category with colimits. In \ref{lem:kanext} below we
present a convenient way to describe the left Kan extension $f_!:\Fun(Y,\cM)\to\Fun(X,\cM)$.

\subsubsection{}
 
We start with a cartesian fibration 
$\pi:\cM^\bullet\to\Cat$ classifying the functor 
$e_\cM:\Cat^\op\to\Cat$ defined by the formula 
$e_\cM(B)=\Fun(B,\cM)$.We compare it to the trivial 
cartesian fibration $\cM\times\Cat$ over $\Cat$ 
classifying the constant functor $c_\cM:\Cat^\op\to\Cat$ 
with value $\cM$. For each $B\in\Cat$ the diagonal functor
$\delta_B:\cM\to\Fun(B,\cM)$ has a left adjoint, 
given by the colimit.

The diagonal $c_\cM\to e_\cM$ yields
a functor $\delta:\cM\times\Cat\to\cM^\bullet$ which has a 
(relative) left adjoint by~\cite{L.HA}, 7.3.2.6.
We denote the functor left adjoint to $\delta$
by $\colim:\cM^\bullet\to\cM\times\Cat$, for an obvious 
reason.

Denote $\cM^\bullet_X=X\times_\Cat\cM^\bullet$ the base change, where the map $X\to\Cat$ classifies $f$. The category $\cM^\bullet_X$ over $X$ can be described as an internal mapping object,  
$\cM^\bullet_X=\Fun_{\Cat_{/X}}(Y,\cM\times X)$,
see \cite{GHN}, 7.3.

One has a canonical equivalence $\theta_f$ defined as
the composition
\begin{equation}
\label{eq:funxfuny}
\Fun_\Cat(X,\cM^\bullet)\stackrel{\sim}{\to}
\Fun_X(X,\cM^\bullet_X)\stackrel{\sim}{\to}
\Fun_X(Y,\cM\times X)\stackrel{\sim}{\to}
\Fun(Y,\cM).
\end{equation}

\begin{lem}\label{lem:kanext}
The composition
$$ \Fun(Y,\cM)=\Fun_\Cat(X,\cM^\bullet)\stackrel{\colim}{\to}
\Fun(X,\cM)$$
is left adjoint
to $f^*:\Fun(X,\cM)\to\Fun(Y,\cM)$.
\end{lem}
\begin{proof}The functor 
$\colim:\Fun_\Cat(X,\cM^\bullet)\to
\Fun(X,\cM)$ is left adjoint to the diagonal 
$\delta:\Fun(X,\cM)\to
\Fun_\Cat(X,\cM^\bullet)$. Now the claim follows 
from the fact that the composition of $\delta$ with 
$\theta_f$ is equivalent to 
$f^*:\Fun(X,\cM)\to\Fun(Y,\cM)$. 
\end{proof}

\begin{lem}
\label{lem:kan-rf}
Assume $f:Y\to X$ is a left fibration and that $\cM$
is a prototopos. Then the left Kan extension functor $f_!:\Fun(Y,\cM)\to\Fun(X,\cM)$
is a right fibration.
\end{lem}
\begin{proof}
Since $f$ is a right fibration, we can replace the base of
the cartesian fibration $\cM^\bullet$ with $\cS\subset\Cat$. Moreover, since $\cM$ is a prototopos, $\cM^\bullet$ can be equivalently described as classifying the functor
$\cS^\op\to\Cat$ carrying $B\to \cM_{/B}$. The functor 
$(\colim,\pi):\cM^\bullet\to\cM\times\cS$ in this interpretation
becomes just the forgetful functor which is a right fibration as all $\cM_{/B}\to\cM$ are right fibrations.

Now the functor $f_!$ is the composition
$$
\Fun(Y,\cM)=\Fun_\cS(X,\cM^\bullet)\to
\Fun_\cS(X,\cM\times\cS)=\Fun(X,\cM)
$$ 
which is a right fibration as a base change of a right fibration.
\end{proof}

We will apply the above result to $X=\Delta^\op$, $Y=(\Delta_{/[1]})^\op$ and the forgetful functor $f:Y\to X$.
Recall that $\Ass=\Delta^\op$ and $\BM=(\Delta_{/[1]})^\op$.
We have

\begin{crl}
\label{crl:bm-ass-1}
For a prototopos $\cM$, one has a natural equivalence
$$\Fun(\BM,\cM)=\Fun(\Ass,\cM)_{/[1]},$$
where $[1]$ in the right hand side of the formula
is interpreted as a composition $\Delta^\op=\Ass\to\cS\to\cM$.
\end{crl}
\begin{proof}
According to \ref{lem:kan-rf}, the left Kan extension
$$f_!:\Fun(\BM,\cM)\to\Fun(\Ass,\cM)$$
is a right fibration. Since the left-hand side has a terminal object $*$, it is equivalent to the overcategory
over $f_!(*)=[1]$.
\end{proof}

\subsection{Categories over $[1]$} We now assume that $\cM$ is a prototopos,
in the sense of Definition~\ref{dfn:prototopos}.

The category of $\cM$-categories, $\Cat(\cM)$, is 
equivalent to the full subcategory of $\Fun(\Ass,\cM)$,
spanned by  the simplicial objects $X_\bullet$ 
satisfying the following properties
\begin{itemize}
\item $X_0$ is a space (that is, it belongs to the essential image of the unique colimit-preserving functor $\cS\to\cM$ preserving the terminal object.
\item $X_\bullet$ is complete and satisfies the Segal condition.
\end{itemize}

Applying~\ref{crl:bm-ass-1}, we get the following.
 
\begin{lem}
\label{lem:cat1-via-bm}
Let $\cM$ be a prototopos.
The category $\Cat(\cM)_{/[1]}$ is naturally equivalent
to the full subcategory of $\Fun(\BM,\cM)$ spanned by the functors $F:\BM\to\cM$ satisfying the following properties.
\begin{itemize}
\item Segal conditions.
\item Completeness: the restrictions of $F$ to $\Ass_\pm$
are complete.
\end{itemize}
\end{lem}\qed

We now have everything we need to prove the main result of this subsection.

\begin{prp}
Let $\cM$ be a prototopos. The the category of $\cM$-correspondences $\Cor{\cM}$ by the formula (\ref{eq:cor1}),
has an equivalent description as $\Cat(\cM)_{/[1]}$
where $[1]\in\Cat(\cM)$ is the image of the ``genuine''
segment $[1]\in\Cat$ under the canonical map
$\Cat=\Cat(\cS)\to\Cat(\cM)$. 
\end{prp}

\begin{proof}

We use Lemma~\ref{lem:cat1-via-bm} to present $\Cat(\cM)_{/[1]}$ as the full subcategory of $\Fun^\laxNC(\BM,\cM)$
spanned by the functors satisfying the completeness condition.

According to Proposition~\ref{prp:alg-monoids} applied to 
the operad $\BM$, $\Fun^\laxNC(\BM,\cM)$ 
identifies with $\Alg_\BM(\cM^\odot)$. 

Using folding functor and Proposition~\ref{sss:algfolding}, 
we deduce that $\Fun^\laxNC(\BM,\cM)$ is naturally equivalent to $\Alg_\LM(\phi(\cM^\odot)))$.

We will now present a description (\ref{eq:cor1})
in terms of $\LM$-algebras.

Using the fact that $\cM$ is symmetric monoidal, we replace
$\Fun_{\cM\otimes\cM^\rev}(\cD\boxtimes\cC^\op,\cM)$
with $\Fun_{\cM}(\cD\otimes\cC^\op,\cM)$, where
$\cD\otimes\cC^\op$ is the pushforward of $\cD\boxtimes\cC^\op$ with respect to SM functor $\cM\otimes\cM^\rev\to\cM$.

We get the following symmetric monoidal version of
(\ref{eq:cor1}).
\begin{equation}
\label{eq:cor1-SM}
\Cor(\cM)=\Cat(\cM)^2\times_{\Alg_\Ass(\Quiv(\cM))}\Alg_\LM(\Quiv^\LM(\cM,\cM)),
\end{equation}
where the map $\Cat(\cM)^2\to\Alg_\Ass(\Quiv(\cM))$ carries
a pair of algebras $(\cC,\cD)$ to $\cD\otimes\cC^\op$.
Denote as $i:\Ass\to\LM$ and $j:\LM\to\BM$ the standard embeddings. The functor
$i^*:\Op_\LM\to\Op_\Ass$ has a right adjoint $i_*$ carrying a planar operad $\cO$ to the pair $(\cO,[0])$. This functor
carries monoidal categories to $\LM$-monoidal categories.
This simple trick allows one to describe $\Alg_\Ass(\cO)$
as $\Alg_\LM(i_*(\cO))$. Therefore, the expression
 (\ref{eq:cor1-SM}) for $\Cor(\cM)$ can be rewritten as the full subcategory \footnote{spanned by the objects for which the algebras $\cC,\cD$ in $\Quiv(\cM)$ are complete.} of 
$\LM$-algebras in the family of $\LM$-monoidal categories
$$
i_*\Quiv(\cM)^2\times_{i_*\Quiv(\cM)}\Quiv^\LM(\cM,\cM).
$$
According to Proposition~\ref{prp-quiv-bm-proto},
$\Quiv(\cM)=\cM^\Delta$ and $\Quiv^\LM(\cM,\cM)=j^*\cM^\og$.
Therefore, the above family of $\LM$-monoidal categories is equivalent to
\begin{equation}
\label{eq:LMmonoidal-cor}
i_*(\cM^\Delta)^2\times_{i_*\cM^\Delta}j^*\cM^\og.
\end{equation}

The above formula may require explanation. The category
$\Fam\Mon_\LM$ admits limits which commute with the forgetful functor $\Fam\Mon_\LM\to\Cat$. Thus, the fiber
product \ref{eq:LMmonoidal-cor} is a family of monoidal categories over $\cS\times\cS$, with the fiber over $(X,Y)$ being 
$$
i_*(\cM^\Delta_X\times\cM^\Delta_Y)\times
_{i_*\cM^\Delta_{Y\times X^\op}}j^*\cM^\og_{Y\times X^\op}.$$
In Subsection~ \ref{ss:identify}  below we construct, 
for $\cM$ a prototopos, a canonical equivalence of 
(\ref{eq:LMmonoidal-cor}) with $\phi(\cM^\odot)$.
$\cM$-correspondences are precisely $\LM$-algebras in 
either of these categories, satisfying completeness
property. This proves the equivalence of two descriptions.
\end{proof}

\subsection{An equivalence of $\phi(\cM^\odot)$ with
(\ref{eq:LMmonoidal-cor})}
\label{ss:identify} 
We will construct a compatible pair of $\LM$-monoidal functors. 

$$\phi(\cM^\odot)\to i_*(\cM^\Delta)^2
\textrm{ and }
\phi(\cM^\odot)\to j^*(\cM^\og).$$
This will give an $\LM$-monoidal functor from $\phi(\cM^\odot)$
to the fiber product (\ref{eq:LMmonoidal-cor}) which will
be shown to be an equivalence.

\subsubsection{}
The map $\phi(\cM^\odot)\to i_*(\cM^\Delta)^2$ is adjoint to the equivalence
$(\cM^\Delta)^2=i^*(\phi(\cM^\odot))$, see Remark 
\ref{rem:phimon}. 

The construction of the map
$\phi(\cM^\odot)\to j^*(\cM^\og)$ is more tricky.

Recall that the family $\cM^\odot$ is constructed using
a functor $\cE:\BM^\op\to\Cat$, with the values in conventional categories, see~\ref{sss:E}. For each 
$s\in\BM$ the category $\cE(s)$ is endowed with a collection of distinguished diagrams; 
the family $\cM^\odot$ is classified by the functor
$F:\BM\to\Cat$ carrying $s\in\BM$ to $\Fun'(\cE(s),\cM)$,
where $\Fun'$ denotes the full subcategory of functors 
carrying distinguished diagrams in $\cE(s)$ to cartesian 
diagrams in $\cM$. We will present a similar description
for $\phi(\cM^\odot)$ and for $j^*(\cM^\og)$.
We will define two functors $\cE_\phi$ and $\cE_j$ from
$\LM^\op$ to $\Cat$, with values in conventional categories, and with families of distinguished diagrams,
so that $\phi(\cM^\odot)$ classifies the functor
$F_\phi$ carrying $s\in\LM$ to $\Fun'(\cE_\phi(s),\cM)$,
and, similarly, $j^*(\cM^\og)$ classifies the functor
$F_j$ carrying $s\in\LM$ to $\Fun'(\cE_j(s),\cM)$.
Then the functor $\phi(\cM^\odot)\to j^*(\cM^\og)$ will be 
induced by a map $\cE_j\to\cE_\phi$ of functors.

By definition, $\cE_j:\LM^\op\to\Cat$ is just the composition $\cE\circ j$. The functor $\cE_\phi$ will be constructed in two steps; we will first present an obvious
choice $\cE'_\phi$ for it, and then we will replace it with
a more appropriate version $\cE_\phi$.

We keep the notation of Remark \ref{rem:phimon}.
The family $\phi(\cM^\odot)$ classifies the functor
$F_\phi:\LM\to\Cat$, $F_\phi(s)=\Fun'(\cE'_\phi(s),\cM)$,
for the following choice of $\cE'_\phi$.
\begin{equation}
\cE'_\phi(s)=\begin{cases}
\cE(s)\sqcup^{L\sqcup R}\cE(s^\op),
& \textrm{if}\  s\in\Ass\subset\LM,\\
\cE(s^*) & \textrm{if}\ s\in\LM^-
\end{cases}
\end{equation}
Here the colimit is meant to be in the naive sense, 
identifying the pairs of objects $L$ and $R$.

Distinguished squares in $\cE'_\phi(s)$ come from distinguished squares in $\cE$. The formula 
$F_\phi(s)=\Fun'(\cE'_\phi(s),\cM)$ is immediate.
The choice of $\cE'_\phi$ as a representing functor for
$F_\phi$ has, however, a drawback: there is no obvious
morphism of functors $\cE_j\to\cE'_\phi$ required for the construction to work. This is why we make a minor amendment.

We will define a modified functor $\cE_\phi:\LM^\op\to\Cat$ 
and a collection of distinguished diagrams in $\cE_\phi(s)$
and a pair of full embeddings $i_0:\cE_j\to\cE_\phi$ and
$i_1:\cE'_\phi\to\cE_\phi$ so that the  map $i_1$ induces an equivalence
$$\Fun'(\cE_\phi(s),\cM)\to\Fun'(\cE'_\phi(s),\cM).$$

The  functor $i_0:\cE_j\to\cE_\phi$ will induce then the map $F_\phi\to F_j$ we are going to construct.
 
For $s\in\Ass$ we denote by $(i,j)'$ the object of $\cE'_\phi(s)$ coming from $(i,j)\in\cE(s)$ and
$(i,j)''$ the one coming from $\cE(s^\op)$.

\subsubsection{}
The categories $\cE_\phi$ will be defined as total categories of correspondences~\footnote{We are working now with conventional categories for which the equivalence
between correspondences and categories over $[1]$ is not a problem!}. Let us recall a general setup. Given a pair of (conventional) categories $A$, $B$, and a functor
$\Phi:A^\op\to P(B^\op)$,  a category $C$ over $[1]$
is defined, 
with fibers $A$ and $B$ over $0$ and $1$ respectively,
and with the Hom-sets $\Hom(X,Y)=\Phi(X)(Y)$ for $X\in A$,
$Y\in B$. We will apply this construction for
$A:=\cE(s)$ and $B:=\cE'_\phi(s)$ and to the functor
$\Phi$  defined by the formulas below.
In what follows we denote by $h^Y$ the copresheaf on $B$ corepresented by $Y\in B$.

{\sl Case I: $s\in\Ass\subset\LM$}.
The functor $\Phi$ is defined by the following formulas.
 \begin{itemize}
\item $\Phi(R)$ is an empty copresheaf.
\item $\Phi(L)=h^L\sqcup h^R$. 
\item For $0\leq i\leq j\leq n$, $n=|s|$,
$\Phi(i,j)=h^{(i,j)'}\sqcup h^{(n-j,n-i)''}$. 
\end{itemize}

{\sl Case II: $s\in\LM^-$}.  
The correspondence is given by a functor $\Phi:\cE(s)^\op
\to P(\cE(s^*)^\op)$ assigning to $e\in\cE(s)$ the copresheaf
on $\cE(s^*)$ defined as follows.  
\begin{itemize}
\item $\Phi(R)=\Phi(|s|,|s|)$ is an empty copresheaf.
\item $\Phi(L)=h^L\sqcup h^R$. 
\item For $(i,j),\ j<|s|$ $\Phi(i,j)=h^{(i,j}\sqcup h^{(n-j,n-i)}$, $n=|s^*|$. 
\item For $j=|s|$, $\Phi(i,j)=h^{(i,2j-i-1)}$.
\end{itemize}
Note that the formulas for $\Phi$ presented above, specifying the value of $\Phi$ one the objects, uniquely
extend to the arrows of $\cE(s)$.

The formulas above are functorial in $s$: a map $s\to s'$
defines a map of correspondences, and, therefore, a map of the total categories $\cE_\phi(s')\to\cE_\phi(s)$.
Thus, we have constructed $\cE_\phi:\LM^\op\to\Cat$.

\subsubsection{}
We will now define distinguished diagrams in $\cE_\phi(s)$.
They are of two types: 
\begin{itemize}
\item[(1)] Distinguished diagrams in 
$\cE'_\phi(s)\subset\cE_\phi(s)$.
\item[(2)] For each $X\in\cE(s)$ such that $\Phi(e)=\sqcup h^{Y_i}$, the diagram $X\to Y_i$ in $\cE_\phi(s)$
~\footnote{the set of indices is allowed to be empty.}.
\end{itemize}

We define $\Fun'(\cE_\phi(s),\cM)$ as the full subcategory
of $\Fun(\cE_\phi(s),\cM)$ spanned by the functors carrying
distinguished diagrams to cartesian diagrams in $\cM$.
This means, for instance, that any distinguished
diagram $X\to Y_i$ induces an equivalence $F(X)=\prod F(Y_i)$.

\begin{lem}\label{lem:fun=fun}
The restriction map $\Fun'(\cE_\phi(s),\cM)\to
\Fun'(\cE'_\phi(s),\cM)$ is an equivalence.
\end{lem}
\begin{proof}
A functor $F:\cE_\phi(s)\to\cM$ is a right Kan extension
of its restriction to $\cE'_\phi(s)$ if and only if the
distinguished diagrams of second kind are sent to
cartesian diagrams in $\cM$. The rest follows from
\cite{L.T}, 4.3.2.15.
\end{proof}

\subsubsection{}
We now define the morphism 
$\phi(\cM^\odot)\to j^*(\cM^\og)$ as classifying the morphism of functors $F_\phi\to F_j$ induced by the
embedding $\cE_j\to\cE_\phi$. Note that if a functor
$F:\cE_\phi(s)\to\cM$ carries distinguished diagrams
to cartesian diagrams in $\cM$, the composition
$\cE(s)\to\cE_\phi(s)\to\cM$ satisfies the same property.

Therefore, the restriction map $\Fun'(\cE_\phi(s),\cM)\to
\Fun(\cE(s),\cM)$ has image in $\Fun'(\cE(s),\cM)$.

Let us construct an equivalence between the two 
compositions $\phi(\cM^\odot)\to i_*\cM^\Delta$.
Using the adjunction between $i^*$ and $i_*$, we can deduce 
it from an equivalence of two adjoint maps 
$i^*\phi(\cM^\odot)\to\cM^\Delta$. Both are maps of cocartesian fibrations on $\Ass$ classifying functors to 
$\Cat$. Let us describe these functors. The restriction
$i^*\phi(\cM^\odot)$ has a more convenient presentation
than the one via $\cE_\phi:\LM^\op\to\Cat$. Let
$\cE^\circ(s)\subset\cE(s)$, $s\in\Ass$, be the full subcategory spanned by all objects except for $R$.

One has an isomorphism $\cE(s)\to\cE(s^\op)$, $e\mapsto e^\op$, interchanging $L$ with $R$ and carrying $(i,j)$ to
$(|s|-j,|s|-i)$.

We define $\cE^\circ_\phi(s)=\Lambda^2_0\times\cE^\circ(s)$
(recall that $\Lambda^2_0$ has three objects, $0,1,2$ and arrows from $0$ to $1$ and $2$).
 One has  a map $f:\cE^\circ(s)\to\cE_\phi(s)$ carrying $(0,e)$ to $e\in\cE(s)\subset\cE_\phi(s)$, $(1,e)$ to $e\in\cE(s)\subset\cE'_\phi(s)$, and $(2,e)$ to $e^\op\in\cE(s^\op)$.
The functior $f$ is fully faithful, with the image consisting of all objects except $R\in\cE(s)\subset\cE_\phi(s)$. Thus, with the obvious collection of distinguished diagrams, the map $\Fun'(\cE_\phi(s),\cM)
\to\Fun'(\cE^\circ_\phi(s),\cM)$ is an equivalence.
Finally, $\Fun'(\cE^\circ_\phi(s),\cM)=\Fun'(\Lambda^2_0,
\Fun'(\cE^\circ,\cM))$, where now $\Fun'(\Lambda^2_0,\_)$ 
denotes the full subcategory of functors, carrying 
$\Lambda^2_0$ to a product diagram. In these terms the left 
composition restricts $F:\Lambda^2_0\to\Fun'(\cE(s),\cM)$
to the pair $(1,2)$ and then takes a product, and the right composition just restricts $F$ to $0\in\Lambda^2_0$.

\subsubsection{}
\label{sss:identify-end}
It remains to verify that the map 
$$
\phi(\cM^\odot)\to i_*(\cM^\Delta)^2\times_{i_*\cM^\Delta} 
j^*\cM^\og$$
is an equivalence. This is an $\LM$-monoidal functor of
cartesian families of $\LM$-monoidal categories, so it
is sufficient to fix a pair of spaces $(X,Y)$ and restrict oneself to the fibers at $s=(00)$ and $(01)$.
In both cases the fiber at $(0,0)$ is $\cM^\Delta_X\times\cM^\Delta_Y$ and the fiber at $(01)$ is $\cM_{/X\times Y}$.

\end{document}